\documentclass[a4paper, 11pt]{article}
\usepackage[utf8]{inputenc}
\usepackage{amsmath}
\usepackage{mathtools}
\usepackage{amssymb}
\usepackage[breakable]{tcolorbox}
\usepackage{amsthm}
\usepackage{graphicx}
\usepackage{hyperref}
\usepackage{tikz}
\usepackage{pgfplots}
\usepackage{fancyhdr}
\usepackage{lastpage}
\usepackage[normalem]{ulem}
\usepackage[mathscr]{euscript}
\usepackage{geometry}
\usepackage{enumerate}
\usepackage{bm}
\usepackage{bbm}

\newtcolorbox{ntcolorbox}[1][]{%
    breakable
}

\allowdisplaybreaks

\newcommand\restr{\upharpoonright}

\newcommand{\blackdiamond}{%
  \mathord{%
    \sbox0{$\diamond$}%
    \resizebox{!}{1.125\ht0}{%
      \raisebox{\depth}{\rotatebox[origin=c]{45}{$\blacksquare$}}%
    }%
  }%
}
\makeatletter
\providecommand*{\shuffle}{%
  \mathbin{\mathpalette\shuffle@{}}%
}
\newcommand*{\shuffle@}[2]{%
  \sbox0{$#1\vcenter{}$}%
  \kern .15\ht0 
  \rlap{\vrule height .25\ht0 depth 0pt width 2.5\ht0}%
  \raise.1\ht0\hbox to 2.5\ht0{%
    \vrule height 1.75\ht0 depth -.1\ht0 width .17\ht0 %
    \hfill
    \vrule height 1.75\ht0 depth -.1\ht0 width .17\ht0 %
    \hfill
    \vrule height 1.75\ht0 depth -.1\ht0 width .17\ht0 %
  }%
  \kern .15\ht0 
}
\makeatother
\newcommand{\RN}[1]{%
  \textup{\uppercase\expandafter{\romannumeral#1}}%
}
\begin{document}
\title{\vspace{-1cm}\Large \textbf{Mass erasure on infinitely measured $\mathbb{R}$-trees $\RN{1}$:\\ construction and limit theorems}}
\author{Mie Glückstad\thanks{Mathematical Institute \& Department of Statistics, University of Oxford, \emph{mie.gluckstad@exeter.ox.ac.uk}}}
\date{}
\maketitle
\newtheorem{thm}{Theorem}[section]
\newtheorem{lemma}[thm]{Lemma}
\newtheorem{cor}[thm]{Corollary}
\newtheorem{prop}[thm]{Proposition}
\newtheorem{defin}[thm]{Definition}
\newtheorem{example}[thm]{Example}
\newtheorem{calc}[thm]{Calculation}
\newtheorem{choice}[thm]{Choice}
\newtheorem{rmk}[thm]{Remark}
\newtheorem{conj}[thm]{Conjecture}
\newtheorem{ques}[thm]{Question}
\newtheorem{algo}[thm]{Algorithm}
\let\oldcalc\calc
\renewcommand{\calc}{\oldcalc\normalfont}
\let\oldexample\example
\renewcommand{\example}{\oldexample\normalfont}
\let\oldrmk\rmk
\renewcommand{\rmk}{\oldrmk\normalfont}
\numberwithin{equation}{section}

\newcommand\xqed[1]{%
  \leavevmode\unskip\penalty9999 \hbox{}\nobreak\hfill
  \quad\hbox{#1}}
\newcommand\demo{\xqed{$\circ$}}
\setcounter{page}{1}
\pagenumbering{arabic}
\fancyfoot[R]{Page \thepage{} of \pageref{LastPage}}
\vspace{-.5cm}
\begin{abstract}
This paper is the first in a two-part series, seeking to i) extend the theory of mass erasure developed by Duquesne and Winkel \cite{DW26} to a setting with infinite measures, and ii) apply this theory to obtain invariance principles for Galton--Watson and Lévy forests in the most general setting where the trees may be supercritical and not satisfy Grey's condition.

We consider $\mathbb{R}$-trees equipped with a suitable subclass of boundedly finite measures, which we call discretely infinite measures. Let $(T, d, \rho)$ be a complete and separable (rooted) $\mathbb{R}$-tree, let $\mu$ be a discretely infinite measure on $T$, and let $h > 0$. As in \cite{DW26} the $h$-mass-erased subtree is obtained by removing all fringe subtrees with $\mu$-mass strictly less than $h$, and we prove that all $h$-mass-erased subtrees have a discrete branching structure. The mass-erased subtrees may be equipped with associated measures, which gives rise to a family of operators $(\mathscr{E}_h)_{h \geq 0}$ which forms a semigroup on the space of discretely infinite measures. The operators may be lifted to an appropriate space of Gromov-vague isometry classes.

Given a sequence $\boldsymbol{\mu}_n = [T_n, d_n, \rho_n, \mu_n]$, $n \in \mathbb{N}$, of suitable isometry classes, we say that they converge vaguely in the sense of mass erasure if their mass erasures $(\mathscr{E}_h \boldsymbol{\mu}_n)_{n \in \mathbb{N}}$ converge in the Gromov-vague topology for all $h > 0$. We prove that by imposing suitable tightness conditions on $(\boldsymbol{\mu}_n)_{n \in \mathbb{N}}$, Gromov-vague convergence implies vague convergence in the sense of mass erasure. Under appropriate tightness, we identify conditions which relate the two notions of convergence, and prove that either mode of convergence implies local Gromov--Hausdorff convergence of the mass-erased subtrees. The applications to Galton--Watson and Lévy forests are considered in part two.



\noindent \textbf{Comment:} No AI, LLM, or other generative technology was used in this work. The author does \textbf{not} consent to any parts of this work being used for training, testing, or feeding any AI, LLM or similar technology, without the author's prior express written consent.
\end{abstract}
\section{Introduction}
\label{sec:Intro}
In this paper, we extend the theory of \emph{mass erasure of $\mathbb{R}$-trees} developed by Duquesne and Winkel \cite{DW26} to a setting with suitable classes of infinite measures. One of the main goals of extending the theory is to obtain a general invariance principle for sequences of Galton--Watson $\mathbb{R}$-forests which are \emph{supercritical}, but which don't satisfy the so-called \emph{Grey condition}, meaning that their limiting continuous-state branching processes have probability zero of going extinct in finite time. This paper is part one in this overarching project, and will purely deal with developing the formal theory of mass erasure with infinite measures, focusing in particular on its construction and limit theorems. In part two \cite{Draft2}, the theory will first be extended to a setting which allows the study of appropriate \emph{random} infinitely measured $\mathbb{R}$-trees, and then be applied to obtain an invariance principle for Galton--Watson $\mathbb{R}$-forests. In some sense, one may broadly characterize this first part as developing the underlying \emph{deterministic} theory, and part two as developing the subsequent \emph{probabilistic theory} of mass erasure with infinite measures.

In the early 1990s Aldous introduced the \emph{Brownian continuum random tree} in three seminal papers \cite{A91a, A91b, A93}, where he constructed it i) as the scaling limit of discrete random trees, ii) via an encoding of the Brownian excursion, and iii) via a line-breaking construction. In later work by Duquesne, Le Gall and Le Jan \cite{DL02, LL98} a more general type of continuum random tree, the \emph{Lévy tree}, was introduced using encodings through height processes. Since then, the analytical foundations to study continuum random trees have been developed using tools from metric geometry (in particular the Gromov--Hausdorff topology and its variants, see \cite{P89}), pioneered by Evans, Pitman, Greven, Pfaffelhuber, and Winter in the early 2000s \cite{EPW06, EW06, GPW09}, with applications to the more general study of random metric (measure) spaces, often arising as scaling limits in the context of random graphs \cite{AG24} or random planar maps \cite{L07, MM07, MM06, M09}. In this context we view continuum random trees as \emph{$\mathbb{R}$-trees} -- formally, these are connected 0-hyperbolic metric spaces, but one may just think of them as being tree-like metric spaces.

The theory of finite mass erasure, introduced in \cite{DW26}, develops a notion of convergence weaker than Gromov--Prokhorov convergence, which allows for the study of continuum random trees arising as limits in contexts that fall outside the scope of what may be captured in classical topologies. This framework is useful for studying (sub-)critical Galton--Watson forests and their scaling limits, as subtrees obtained under finite mass erasure are \emph{bounded}. With infinite mass erasure we seek to generalize the theory in \cite{DW26}, to study $\mathbb{R}$-trees equipped with infinite measures. This is more suitable for handling supercritical Galton--Watson forests, as infinite mass may be allocated such that the \emph{unbounded} backbone/skeleton of the tree survives the mass erasure operation. Infinite mass erasure is developed by extending the ideas from \cite{DW26} in a framework building on the Gromov-vague topology of Athreya, Löhr and Winter \cite{ALW16}. 

Throughout this work we will be considering $\mathbb{R}$-trees which are complete and separable, and call such $\mathbb{R}$-trees \emph{Polish}.\footnote{Note here the slight abuse of terminology. We will call $(\mathscr{X}, d)$ a \emph{Polish} space, if it is complete and separable with respect to the metric $d$. This differs from the more standard notion of a Polish space, as being a separable topological space which is completely metrizable.} The $\mathbb{R}$-trees arising in our setting may not be (locally) compact, making their geometry complicated to study. To address this difficulty, mass erasure provides a technique for erasing parts of a measured $\mathbb{R}$-tree such that the mass-erased subtree obtained is (locally) compact with a discrete branching structure, and may be easily dealt with using the classical Gromov-type topologies. This is done by erasing all fringe subtrees of mass strictly less than $h > 0$ from the tree, with $h$ being an erasure parameter. The mass-erased subtree is reminiscent to the \emph{$h$-leaf-length erased subtree} studied by Evans, Pitman and Winter \cite{EPW06} (and applied in the context of Lévy forests by Duquesne and Winkel \cite{DW19}), or similar trimming techniques formalized in earlier work by Neveu and Pitman \cite{N86, NP89a, NP89b} in different tree formalisms. The main difference here being whether one erases fringe subtrees based on their mass or their height/leaf-length.

In \cite{DW26} mass erasure was developed for Polish $\mathbb{R}$-trees equipped with a finite Borel measure, whereas mass erasure in this paper will be developed for a suitable class of infinite Borel measures. In either setting, the mass-erased subtree may be equipped with a measure which relocates the mass of the fringe subtrees in a suitable way such that the resulting $h$-mass erasure operators $(\mathscr{E}_h)_{h \geq 0}$ form a semigroup. In \cite{DW26} the mass erasure operator was considered on the Gromov--Prokhorov space of (measure-preserving isometry classes) of \emph{minimal} Polish finitely measured $\mathbb{R}$-trees, with minimality ensuring that the span of the support of the measure is the entire $\mathbb{R}$-tree. In that setting, the semigroup $(\mathscr{E}_h)_{h \geq 0}$ was shown to be Gromov--Prokhorov-continuous. See also \cite{J20, L13, LVW15} for useful resources on the Gromov--Prokhorov topology. For mass erasure with infinite measures, we consider a similar space of (measure-preserving isometry classes) of minimal Polish infinitely measured $\mathbb{R}$-trees, with respect to the Gromov-vague topology. In this framework, the semigroup $(\mathscr{E}_h)_{h \geq 0}$ is \emph{not} immediately Gromov-vaguely continuous, but we may still establish sequential continuity under additional tightness conditions imposed on the sequence of isometry classes considered. A significant part of this paper will deal with establishing this result, using \emph{cut-off methods} which allow us to utilize the results from \cite{DW26}.

Mass erasure gives rise to a natural notion of convergence, namely \emph{convergence in the sense of mass erasure}. In the finitely measured setting of \cite{DW26}, we say that a sequence of isometry classes converges in the sense of mass erasure, if the mass erasures converge in the Gromov--Prokhorov sense for all $h > 0$. It is shown in \cite{DW26} that convergence in the sense of mass erasure is strictly weaker than Gromov--Prokhorov convergence, and it may indeed be metrized in order to obtain a Polish topology on an appropriate space of isometry classes. In the setting with infinite measures, we similarly introduce the notion of \emph{(Gromov-)vague convergence in the sense of mass erasure}, meaning that the mass erasures converge for all $h > 0$ in the Gromov-vague sense. One may metrize this notion of convergence on an appropriate space of isometry classes. This space is seen to be separable, but \emph{not} complete, and it seems difficult to provide a completion -- even proving that it is a Lusin space is an open question that remains to be addressed. Fortunately, for our probabilistic applications in \cite{Draft2}, the metrization is not strictly needed, as one may just directly study the convergence in distribution of the mass erasures on an appropriate Lusin space of isometry classes equipped with the Gromov-vague topology.

As noted in \cite{DW26} related notions and ideas to that of mass erasure have appeared in existing literature, for instance in \cite{B02, B06, BM05, FKM10, G26, HKK10, HM04, R23, TXCP23}. In recent work by Dubach, Thévenin and Wagner \cite{DTW26+} a similar problem to that of studying scaling limits which cannot be captured in the classical Gromov--Prokhorov setting, is addressed using the theory of \emph{dendrons} \cite{ET22, S21, TW23}. This is done in the context of scaling limits for descent-biased trees, and the approach builds on the sampling-based characterization of Gromov--Prokhorov convergence developed in \cite{A93, GPW09}. An open line of inquiry would be to examine the relationship between convergence of dendrons and convergence in the sense of mass erasure, and in particular to study how the limits obtained under these two notions of convergence may relate.

This paper is organized as follows: In Section \ref{sec:Main} we summarize our main definitions and results, in Section \ref{sec:3} we cover some preliminary material on infinite measures, projections, and topologies, in Section \ref{sec:4} we develop the theory of mass erasure with infinite measures on fixed $\mathbb{R}$-trees, and in Section \ref{sec:5} we extend the theory to the Gromov-vague setting.
\section{Main definitions and results}
\label{sec:Main}
\subsection{Vague convergence and discretely infinite measures on $\mathbb{R}$-trees}
\label{subsec:MainVague}
\emph{Boundedly finite measures and vague convergence:} Let $(\mathscr{X}, d, \rho)$ be a pointed metric space. We denote by $\mathscr{B}(\mathscr{X})$ the Borel $\sigma$-algebra induced by the metric $d$. For each $r > 0$, we denote by $B_{\mathscr{X}}(\rho, r)$ the open ball of radius $r$ around $\rho$ in $\mathscr{X}$. We let $\overline{B}_{\mathscr{X}}(\rho, r)$ be its closure and $\partial B_{\mathscr{X}}(\rho, r)$ its boundary. If it is obvious from the context which space we are working on, we will sometimes just write $B(\rho, r)$, $\overline{B}(\rho, r)$ and $\partial B(\rho, r)$. A Borel measure $\mu$ is said to be \emph{boundedly finite} if it is finite on all bounded sets, and we denote by
$$ \mathscr{M}_{\rm{prob}}(\mathscr{X}) \subseteq \mathscr{M}_{\rm{fin}}(\mathscr{X}) \subseteq \mathscr{M}_{\rm{bf}}(\mathscr{X}) \subseteq \mathscr{M}(\mathscr{X}), $$
the collections of probability measures, finite measures, boundedly finite measures and (non-negative) measures, respectively, on the measurable space $(\mathscr{X}, \mathscr{B}(\mathscr{X}))$. The support of a measure $\mu \in \mathscr{M}(\mathscr{X})$ is denoted by ${\rm{supp}}(\mu) := \{ x \in \mathscr{X} ~|~ \forall A \subseteq \mathscr{X} \text{ open with } x \in A: ~ \mu(A) > 0\}$.

We equip $\mathscr{M}_{\rm{fin}}(\mathscr{X})$ with the topology of weak convergence. For $\mu \in \mathscr{M}_{\rm{bf}}(\mathscr{X})$, we write $\mu^{\restr r} := \mu|_{\overline{B}(\rho, r)} \in \mathscr{M}_{\rm{fin}}(\mathscr{X})$ for the restriction of $\mu$ to the ball of radius $r$ around $\rho$ in $\mathscr{X}$. We say that a sequence of boundedly finite measures $(\mu_n)_{n \in \mathbb{N}}$ converges vaguely to some $\mu \in \mathscr{M}_{\rm{bf}}(T)$, and write $\mu_n \stackrel{\rm{vg}}{\rightarrow} \mu$, if $\mu_n^{\restr r_k} \stackrel{\rm{wk}}{\rightarrow} \mu^{\restr r_k}$ as $n \rightarrow \infty$ for all $k$ in some increasing sequence of radii $r_k \rightarrow \infty$. One may metrize this notion of convergence as in Definition \ref{def:BFMeas}, which gives rise to a Polish space $(\mathscr{M}_{\rm{bf}}(\mathscr{X}), d_{\rm{V}})$.

It is easy to show that for finite measures, weak convergence implies vague convergence. For the opposite implication, it is well-known (see \cite{HLS22} Section 2), that tightness has to be imposed. Indeed, under vague convergence, mass can \emph{escape to infinity} due to the lack of tightness. As this observation will be essential in our further work, let us illustrate it with a standard example.
\begin{example}
\label{ex:BFMeas6}
Let $\mathscr{X} = [0, \infty)$ be equipped with the standard Euclidean distance, and let $\rho =0$. Let for each $n \in \mathbb{N}$, $\delta_n$ be the Dirac point-measure in $n$. Then $\delta_n \stackrel{\rm{vg}}{\rightarrow} \mathbf{0}$ (with $\mathbf{0}$ denoting the null measure), but $(\delta_n)_{n \in \mathbb{N}}$ is clearly not tight, and thus not weakly convergent in $\mathscr{M}_{\rm{fin}}([0,\infty))$.
\demo
\end{example}
An easy application of Prokhorov's theorem yields that a sequence of (finite) measures $(\mu_n)_{n \in \mathbb{N}}$ defined on a pointed Polish space is weakly convergent if and only if it is vaguely convergent and tight. As vague convergence is defined only in terms of closed balls, and these may not be compact on a general Polish space, we can even loosen our assumption of tightness to a condition that we will refer to as \emph{b-tightness} (understood as a form of \emph{tightness on \underline{b}alls}).
\begin{defin}[b-tightness]
\label{def:btight}
Let $(\mathscr{X}, d, \rho)$ be a pointed Polish space, and let $\mu_n \in \mathscr{M}_{\rm{bf}}(\mathscr{X})$, $n \in \mathbb{N}$. We say that $(\mu_n)_{n \in \mathbb{N}}$ is b-tight on $\mathscr{X}$ if $ \lim_{R \rightarrow \infty} \limsup_{n \rightarrow \infty} \mu_n(\mathscr{X} \setminus \overline{B}_{\mathscr{X}}(\rho, R)) = 0$.
\end{defin}
\begin{lemma}
\label{lemma:BFMeas7}
Let $(\mathscr{X}, d, \rho)$ be a pointed Polish space, and let $\mu_n \in \mathscr{M}_{\rm{bf}}(\mathscr{X})$, $n \in \mathbb{N}$. Then $(\mu_n)_{n \in \mathbb{N}}$ is vaguely convergent and b-tight if and only if there exists some $N \in \mathbb{N}$ such that $(\mu_n)_{n \geq N}$ is a sequence of weakly convergent finite measures.
\end{lemma}
The proof of Lemma \ref{lemma:BFMeas7} and other results on vague convergence may be found in Section \ref{subsec:BFmeas}.\\

\noindent \emph{$\mathbb{R}$-trees:} Let $(\mathscr{X}, d)$ be a metric space. For any compact real interval $[a,b] \subseteq \mathbb{R}$ and any isometry $\phi\colon [a,b] \rightarrow \mathscr{X}$, we refer to the image $\phi([a,b]) \subseteq \mathscr{X}$ as a \emph{segment} in $\mathscr{X}$. We call $\phi(a)$ and $\phi(b)$ the \emph{end points} of the segment. If for some $x,y \in \mathscr{X}$, there exists a \emph{unique} segment in $\mathscr{X}$ with end points $x$ and $y$, we denote this segment by $[x,y]$.
\begin{defin}[$\mathbb{R}$-tree]
A pointed metric space $(T, d, \rho)$ is called an $\mathbb{R}$-tree with root $\rho$ if
\begin{enumerate}[$(i)$]
\item for any $\sigma_1, \sigma_2 \in T$ there exists a (unique\footnote{This actually follows by property $(ii)$, see \cite{E08}.}) segment $[\sigma_1, \sigma_2]$ in $T$ with end points $\sigma_1, \sigma_2$;
\item whenever $[\sigma_1, \sigma_0]$, $[\sigma_0, \sigma_2]$ are two segments in $T$ with $[\sigma_1, \sigma_0] \cap [\sigma_0, \sigma_2] = \{\sigma_0\}$, then $[\sigma_1, \sigma_0] \cup [\sigma_0, \sigma_2]$ is a segment with end points $\sigma_1, \sigma_2$.
\end{enumerate}
\end{defin}
There are a multitude of equivalent definitions of $\mathbb{R}$-trees. A pointed metric space is an $\mathbb{R}$-tree if and only if it is 0-hyperbolic and connected (see \cite{B74}). Alternatively, $T$ is an $\mathbb{R}$-tree if there is a unique path between any two points $\sigma_1, \sigma_2 \in T$, and this path is isometric to the compact interval $[0, d(\sigma_1, \sigma_2)]$. $\mathbb{R}$-trees satisfy the properties one would expect them to satisfy; for instance, if $\sigma_0, \sigma_1, \sigma_2 \in T$, then $[\sigma_0, \sigma_1] \cap [\sigma_0, \sigma_2] = [\sigma_0, \sigma]$ for some unique point $\sigma \in T$. When $\sigma_0 = \rho$, we denote the unique point $\sigma$ by $\sigma_1 \wedge \sigma_2$, and refer to this point as the most recent common ancestor of $\sigma_1$ and $\sigma_2$. For a detailed discussion of $\mathbb{R}$-trees, see \cite{B99, E08}.

Let $(T, d, \rho)$ be an $\mathbb{R}$-tree. We denote the \emph{degree} of each point in $T$ by
\begin{align*}
n(\sigma, T) := \#\{\text{connected components of } T \setminus \{\sigma\}\} \in \mathbb{N}_0 \cup \{+\infty\} \quad \text{for every } \sigma \in T,
\end{align*}
We moreover define the collections of leaves and branch points in $T$ by
\begin{align*}
{\rm{Lf}}(T) &:= \{\sigma \in T \setminus \{\rho\} ~|~ n(\sigma,T) = 1 \}, \quad {\rm{Bp}}(T) := \{\sigma \in T \setminus \{\rho\} ~|~ n(\sigma,T) \geq 3 \}.
\end{align*}
We denote the \emph{height} of an $\mathbb{R}$-tree by $ {\rm{Ht}}(T) := \sup_{\sigma \in T} d(\rho, \sigma)$.

Suppose that $T' \subseteq T$ is a subset of the $\mathbb{R}$-tree $(T, d, \rho)$. We say that $T'$ is a \emph{subtree} of $T$ if it is connected. If in addition, $\rho \in T'$ we say that the subtree is \emph{rooted}. An easy way of generating a \emph{rooted} subtree is by letting it be spanned by a non-empty subset of $T$.
\begin{defin}[Subtree spanned by a set]
\label{def:Span}
Let $(T, d, \rho)$ be an $\mathbb{R}$-tree, and let $A \subseteq T$ be a non-empty subset. The \emph{(rooted) subtree spanned by $A$} is the smallest connected subset of $T$ containing $A$ and $\rho$, i.e. obtained by setting $ {\rm{Span}}(A) := \bigcup_{\sigma \in A} [\rho, \sigma]$.
\end{defin} 
Given some fixed point $\sigma \in T$, we denote by $\theta_{\sigma}T$ the subtree above the point $\sigma$ by setting
$$ \theta_{\sigma}T := \{ \sigma' \in T ~|~ \sigma \in [\rho, \sigma'] \}. $$
It is then easily checked that $\theta_{\sigma}T$ is a closed subtree of $T$. It is also easily seen that if $\sigma \in [\rho, \sigma']$, then $\theta_{\sigma'}T \subseteq \theta_{\sigma}T$, and if $\sigma \wedge \sigma' \notin \{\sigma, \sigma'\}$, then $\theta_{\sigma}T \cap \theta_{\sigma'}T = \emptyset$. Whenever $\sigma \neq \rho$, we write $\theta_{\sigma}T^{\circ} = \theta_{\sigma}T \setminus \{\sigma\}$ for the interior of the subtree above $\sigma$.\\

\noindent \emph{Discretely infinite measures on $\mathbb{R}$-trees:} Let $(T, d, \rho)$ be a Polish $\mathbb{R}$-tree, and let $\mu \in \mathscr{M}_{\rm{bf}}(T)$. In our later discussions of projections (see Section \ref{subsec:Proj}) and mass erasure (see Section \ref{subsec:MESubtrees}), we will quickly discover that some restrictions will need to be imposed on the measure $\mu$. We define
$$ {\rm{Skel}}(T, \mu) := \{\rho\} \cup \left\{ \sigma \in T ~|~ \mu(\theta_{\sigma}T) = +\infty \right\},$$
to be the set of points in $T$ with infinite mass above them. We will call this set the \emph{infinite-mass skeleton} or just \emph{skeleton} of $T$ with respect to $\mu$, and indeed ${\rm{Skel}}(T, \mu)$ is a closed rooted subtree of $T$. We denote the set of points in $T$ with finite mass above them by
$$ F_{\mu}(T) := T \setminus {\rm{Skel}}(T, \mu) = \left\{ \sigma \in T \setminus \{\rho\} ~|~ \mu(\theta_{\sigma}T) < \infty \right\}. $$
Recall from \cite{DW19} Definition 2.1 that $(T, d, \rho)$ is called \emph{discrete} if it is complete, and $n(\rho, T) + \sum_{\sigma \in {\rm{Bp}}(T) \cap \overline{B}_T(\rho, r)} n(\sigma, T) < \infty$ for all $r > 0$.
\begin{defin}[Discretely infinite measures]
\label{def:DIMeas}
Let $(T, d, \rho)$ be a Polish $\mathbb{R}$-tree, and let $\mu \in \mathscr{M}_{\rm{bf}}(T)$. We say that $\mu$ is \emph{discretely infinite}, and write $\mu \in \mathscr{M}_{\rm{di}}(T)$, if $\mu(T) = +\infty$, and
\begin{enumerate}[$(i)$]
\item ${\rm{Skel}}(T, \mu)$ is a discrete $\mathbb{R}$-tree, and
\item $\mu\left( \bigcup_{\sigma \in \partial B_T(\rho, r) \cap F_{\mu}(T)} \theta_{\sigma}T \right) < \infty$ for all $r > 0$.
\end{enumerate}
\end{defin}
The second condition ensures that when we consider all of the subtrees above level $r$, infinite mass will only be found on the subtrees belonging to points on the skeleton, and \emph{cannot} be diffused across infinitely many subtrees belonging to points in $F_{\mu}(T)$. We denote by $\mathscr{M}_{\rm{di}}(T)$ the space of discretely infinite \emph{or finite} measures on $T$, such that
$$ \mathscr{M}_{\rm{prob}}(T) \subseteq \mathscr{M}_{\rm{fin}}(T) \subseteq \mathscr{M}_{\rm{di}}(T) \subseteq \mathscr{M}_{\rm{bf}}(T) \subseteq \mathscr{M}(T). $$
Observe also as a consequence of condition $(ii)$, that for a discretely infinite measure $\mu$, the skeleton ${\rm{Skel}}(T, \mu)$ will have no leaves. Indeed, ${\rm{Skel}}(T, \mu)$ is the union of a collection of \emph{rays}\footnote{A subset of $T$ is called a \emph{ray} if it contains the root and is isometric to $[0, \infty)$. Visually one can think of a ray as an infinitely long branch in $T$, starting at the root.} in $T$. We see in Proposition \ref{prop:DIMeas1} that the space $(\mathscr{M}_{\rm{di}}(T), d_{\rm{V}})$ is \emph{Lusin}, but \emph{not} complete in general.
\subsection{Mass erasure on infinitely measured $\mathbb{R}$-trees}
\label{subsec:MainME}
Let $(T, d, \rho)$ be a Polish $\mathbb{R}$-tree. Given some $\mu \in \mathscr{M}(T)$, we define for every $h \geq 0$, the \emph{$h$-mass-erased subtree of $T$} by
\begin{equation} 
\label{eq:MEFix}
R_{\mu,h}(T) = \{\rho\} \cup \{ \sigma \in T ~|~ \mu(\theta_{\sigma}T) \geq h \}.
\end{equation}
In \cite{DW26} it was shown that if $\mu \in \mathscr{M}_{\rm{fin}}(T)$, then $R_{\mu,h}(T)$ is a compact discrete $\mathbb{R}$-tree.\footnote{In \cite{DW26} this was done by showing that $R_{\mu,h}(T)$ is of \emph{finite type}, but by Lemma \ref{lemma:Rtree1} this is equivalent to being compact and discrete.}
\begin{lemma}
\label{lemma:MEFix2}
Let $(T, d, \rho)$ be a Polish $\mathbb{R}$-tree, and let $\mu \in \mathscr{M}_{\rm{di}}(T)$. Then $R_{\mu,h}(T)$ is a discrete $\mathbb{R}$-tree for all $h > 0$.
\end{lemma}
$R_{\mu,h}(T)$ is clearly a closed rooted subtree of $T$. We also observe that ${\rm{Skel}}(T, \mu) \subseteq R_{\mu,h}(T)$ for all $h > 0$, so if $\mu \in \mathscr{M}_{\rm{di}}(T) \setminus \mathscr{M}_{\rm{fin}}(T)$, then $R_{\mu,h}(T)$ is unbounded. The subtrees $\left( R_{\mu, h}(T) \right)_{h \geq 0}$ are seen to be decreasing as $h \rightarrow \infty$, and it will be useful to consider their right limits 
\begin{equation}
\label{eq:RightLim}
R_{\mu,h+}(T) := \overline{\bigcup_{h' > h} R_{\mu,h'}(T)}.
\end{equation}

One may ask why the assumption that $\mu \in \mathscr{M}_{\rm{di}}(T)$ is necessary in order to ensure that the $h$-mass-erased subtree is discrete, and whether it would not suffice to consider $\mu \in \mathscr{M}_{\rm{bf}}(T)$. We illustrate why this condition is necessary in the following example.
\begin{example}
\label{ex:DI}
Let $(T, d, \rho)$ be the (infinitely large) infinite star, consisting of countably many rays which only intersect at the root. Order the rays by $\mathbf{r}_1, \mathbf{r}_2, \dots$. 

Let $\mu_1$ be the boundedly finite measure on $T$ defined by acting on every ray $\mathbf{r}_i$ as the Lebesgue measure $\mathbf{m}_{[i, \infty)}$ started at level $i$. Then $\mu_1 \notin \mathscr{M}_{\rm{di}}(T)$ as its skeleton is not discrete. Moreover, $R_{\mu_1,h}(T) = T$, which is clearly not in any way a nicer subtree than what we started with.

Let $\mu_2$ be the boundedly finite measure on $T$ defined by associating to each ray $\mathbf{r}_i$ a Dirac point mass $\delta_i$ at level $i$. Then $\mu_2 \notin \mathscr{M}_{\rm{di}}(T)$, as above any level $r > 0$, the total mass of its finite-mass subtrees is infinite. Moreover, erasing with any parameter $h < 1$ we see that $R_{\mu_2,h}(T)$ is the infinite star with branches of integer lengths, which is clearly not a discrete $\mathbb{R}$-tree.
\demo
\end{example}
In order to define a notion of \emph{convergence in the sense of mass erasure} for Polish measured $\mathbb{R}$-trees, we may define mass erasure as an operator which acts on measures. Let $(T, d, \rho)$ be a fixed Polish $\mathbb{R}$-tree. In \cite{DW26} mass erasure is defined as an operator $\mathscr{E}_h \colon \mathscr{M}_{\rm{fin}}(T) \rightarrow \mathscr{M}_{\rm{fin}}(T)$ \emph{uniquely} defined by satisfying the mass erasure property
\begin{equation}
\label{eq:MEProperty}
\mathscr{E}_h\mu(\theta_{\sigma}T) = \left( \mu(\theta_{\sigma}T) - h \right)_+ \quad \text{for all } \sigma \in T.
\end{equation}
This operator will, in a heuristic sense, project any mass outside of the mass-erased subtree onto its nearest point in $R_{\mu,h}(T)$ in order to evaluate $\mu(\theta_{\sigma}T)$ correctly. To subtract $h$ mass appropriately from subtrees, mass $h$ is removed from each leaf in $R_{\mu,h}(T)$ and at the same time point masses of suitable multiples of $h$ are added to each branch point of $R_{\mu,h}(T)$. See Figure \ref{fig:MEOperator}. The theory of projections is discussed in Section \ref{subsec:Proj}.

\begin{figure}[t]
\center
\includegraphics[width=.85\textwidth]{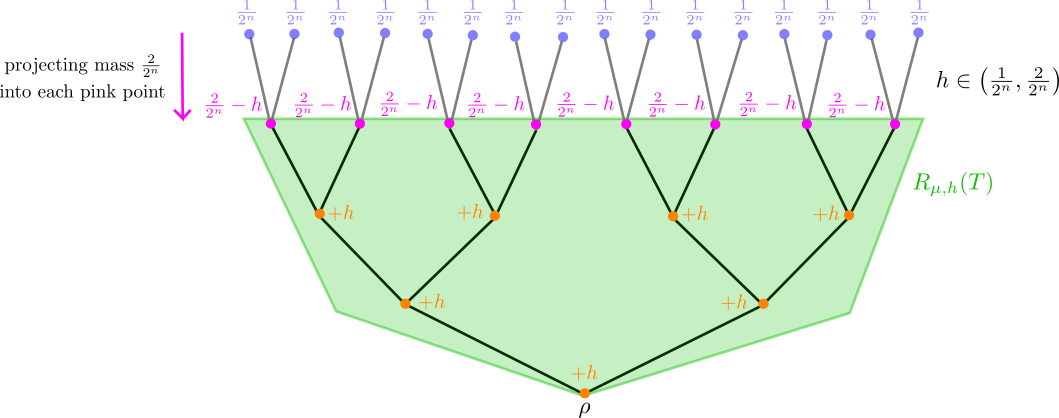}
\caption{Example of how the mass erasure operator acts when $(T, d, \rho)$ is the level-$n$ binary tree, and $\mu \in \mathscr{M}_{\rm{fin}}(T)$ allocates point masses of $\frac{1}{2^n}$ in each of the $2^n$ leaves of $T$.}
\label{fig:MEOperator}
\end{figure}

In the setting of this paper, we may similarly define mass erasure as an operator $\mathscr{E}_h \colon \mathscr{M}_{\rm{di}}(T) \rightarrow \mathscr{M}_{\rm{di}}(T)$ satisfying \eqref{eq:MEProperty}. It immediately follows that ${\rm{Skel}}(T, \mathscr{E}_h\mu) = {\rm{Skel}}(T, \mu)$. When considering discretely infinite measures, the property \eqref{eq:MEProperty} is however not enough to uniquely determine $\mathscr{E}_h$ as mass may be moved around on the skeleton without affecting subtree masses. We will thus define $\mathscr{E}_h$ explicitly to behave like in the finite-mass setting, ensuring the consistency of the operators $(\mathscr{E}_h)_{h \geq 0}$ on $\mathscr{M}_{\rm{fin}}(T)$ and $\mathscr{M}_{\rm{di}}(T)$. See Section \ref{subsec:MEOp} for more details. The mass erasures $(\mathscr{E}_h)_{h \geq 0}$ satisfy similar properties on $\mathscr{M}_{\rm{di}}(T)$ as in the setting of finite mass erasure from \cite{DW26}.
\begin{prop}[Semigroup property]
\label{prop:MEOp3}
Let $(T, d, \rho)$ be a Polish $\mathbb{R}$-tree. Then the family of operators $(\mathscr{E}_h)_{h \geq 0}$ forms a semigroup on $\mathscr{M}_{\rm{di}}(T)$, i.e. $ \mathscr{E}_h \circ \mathscr{E}_{h'} = \mathscr{E}_{h+h'}$ for all $h,h' \geq 0$.
\end{prop}
\begin{lemma}[Approximation via mass erasures]
\label{lemma:MEOp5b}
Let $(T, d, \rho)$ be a Polish $\mathbb{R}$-tree, and let $\mu \in \mathscr{M}_{\rm{di}}(T)$. If $h_p \searrow 0$ is a decreasing sequence of numbers in $(0, \infty)$, then $\mathscr{E}_{h_p}\mu \stackrel{\rm{vg}}{\rightarrow} \mu$ as $p \rightarrow \infty$.
\end{lemma}
Let $(T, d, \rho)$ be a fixed Polish $\mathbb{R}$-tree. Recall from Duquesne and Winkel \cite{DW26} Definition 3.22, that \emph{(weak) convergence in the sense of mass erasure} for finite measures $(\mu_n)_{n \in \mathbb{N}}$ and $\mu$ on $T$ is defined to mean that $\mathscr{E}_h\mu_n \stackrel{\rm{wk}}{\rightarrow} \mathscr{E}_h\mu$ as $n \rightarrow \infty$ for all $h > 0$, and we will denote this convergence by $\mu_n \stackrel{\rm{me}}{\rightarrow} \mu$. This notion of convergence may be metrized and yields a suitable topology on $\mathscr{M}_{\rm{fin}}(T)$ which is strictly weaker than the topology of weak convergence. The space $\mathscr{M}_{\rm{fin}}(T)$ equipped with the topology of (weak) convergence in the sense of mass erasure is \emph{not} complete, and so a completion space is defined in \cite{DW26} using so-called \emph{bordifications}. We refer to \cite{DW26} Sections 2.3 and 3.4 for a more in-depth discussion of this. In the setting of mass erasure with discretely infinite measures, we may now define the notion of \emph{vague convergence in the sense of mass erasure}.
\begin{defin}[Vague convergence in the sense of mass erasure]
\label{def:MEConv}
Let $(T, d, \rho)$ be a Polish $\mathbb{R}$-tree, and let $\mu, \mu_n \in \mathscr{M}_{\rm{di}}(T)$, $n \in \mathbb{N}$. We say that \emph{$(\mu_n)_{n \in \mathbb{N}}$ converges to $\mu$ vaguely in the sense of mass erasure}, and write $\mu_n \stackrel{\rm{vme}}{\rightarrow} \mu$, if $\mathscr{E}_h\mu_n \stackrel{\rm{vg}}{\rightarrow} \mathscr{E}_h\mu$ as $n \rightarrow \infty$ for all $h > 0$.
\end{defin}
If $\mu$ is a finite measure, $\mu_n \stackrel{\rm{vme}}{\rightarrow} \mu$, and $(\mathscr{E}_h\mu_n)_{n \in \mathbb{N}}$ is b-tight on $T$ for all $h > 0$, then it clearly follows that $\mu_n \stackrel{\rm{me}}{\rightarrow} \mu$ as $n \rightarrow \infty$ by Lemma \ref{lemma:BFMeas7}.

When considering applications in probability theory, we will usually be interested in studying \emph{random} Polish measured $\mathbb{R}$-trees, meaning that not only the measures but also the underlying metric space structures are chosen under some form of randomness. Studying the convergence of sequences of such random measured metric spaces requires us to allow for comparison of measures which are defined on \emph{different} $\mathbb{R}$-trees. The natural setting for considering such convergence is to look at appropriate spaces of \emph{isometry classes} of Polish measured $\mathbb{R}$-trees, equipped with appropriate Gromov-like topologies (as introduced and discussed in Section \ref{subsec:Gromov}). We may then define a notion of \emph{(Gromov--Prokhorov/Gromov-vague) convergence in the sense of mass erasure} in the extended setting of isometry classes, similarly to what we did here in the setting of a fixed $\mathbb{R}$-tree. For the remainder of this section, all of our remaining results will be stated in the more general `Gromov'-setting. Note, however, that corresponding statements in the setting of fixed $\mathbb{R}$-trees may be found in Section \ref{sec:4}. 
\subsection{Gromov-vague convergence in the sense of mass erasure}
\label{subsec:MainConvME}
Two Polish measured $\mathbb{R}$-trees $(T, d, \rho, \mu)$ and $(T', d', \rho', \mu')$ are said to be isometric if there exists a bijective isometry $\varphi \colon T \rightarrow T'$ with $\varphi(\rho) = \rho'$ and $\mu \circ \varphi^{-1} = \mu'$, where $\mu \circ \varphi^{-1}$ denotes the pushforward of $\mu$ under $\varphi$. We will be considering the isometry classes of \emph{minimal} Polish measured $\mathbb{R}$-trees.
\begin{defin}[Minimal $\mathbb{R}$-trees]
\label{def:MinRTree}
Let $(T, d, \rho, \mu)$ be a Polish $\mathbb{R}$-tree equipped with a Borel measure. We say that it is \emph{minimal} if $ T = \overline{\text{\emph{Span}}(\text{\emph{supp}}(\mu))}$.
\end{defin}
In other words, minimality ensures that the tree does not contain fringe subtrees with mass zero, which are indistinguishable under the Gromov--Prokhorov distance. We will denote by $\mathbb{T}_{\rm{min}}^{\rm{fin}}$, $\mathbb{T}_{\rm{min}}^{\rm{di}}$ and $\mathbb{T}_{\rm{min}}^{\rm{bf}}$ the spaces of isometry classes of minimal Polish $\mathbb{R}$-trees $(T, d, \rho)$ equipped with measures from $\mathscr{M}_{\rm{fin}}(T)$, $\mathscr{M}_{\rm{di}}(T)$, and $\mathscr{M}_{\rm{bf}}(T)$, respectively. We denote the isometry classes in these spaces by $\boldsymbol{\mu} = [T, d, \rho, \mu]$. Letting $d_{\rm{GP}}$ denote the Gromov--Prokhorov distance and $d_{\rm{GV}}$ the Gromov-vague distance as recalled and discussed in Section \ref{subsec:Gromov}, the spaces $(\mathbb{T}_{\rm{min}}^{\rm{fin}}, d_{\rm{GP}})$ and $(\mathbb{T}_{\rm{min}}^{\rm{bf}}, d_{\rm{GV}})$ are Polish, whereas $(\mathbb{T}_{\rm{min}}^{\rm{di}}, d_{\rm{GV}})$ is Lusin.

Let $\boldsymbol{\mu} = [T, d, \rho, \mu]$ be an isometry class in $\mathbb{T}_{\rm{min}}^{\rm{di}}$. We may want to characterize the \emph{minimal} Polish $\mathbb{R}$-tree generated by the mass erasure $\mathscr{E}_h\mu$. Recall the definition of the right limit of a mass-erased subtree in \eqref{eq:RightLim}. The following result follows from \cite{DW26} Proposition 3.19$(iv)$.
\begin{lemma}
\label{lemma:MEOp2b}
Let $(T, d, \rho)$ be a Polish $\mathbb{R}$-tree, $\mu \in \mathscr{M}_{\rm{di}}(T)$, and $h > 0$. Then $\text{\emph{Span}}(\text{\emph{supp}}(\mathscr{E}_h\mu)) = R_{\mu,h+}(T)$.
\end{lemma}
As $R_{\mu,h+}(T)$ is closed by construction, we must have that $[R_{\mu,h+}(T), d, \rho, \mathscr{E}_h\mu] \in \mathbb{T}_{\rm{min}}^{\rm{di}}$. We note moreover, that if $(T, d, \rho, \mu)$ and $(T', d', \rho', \mu')$ are representatives of the same isometry class in $\mathbb{T}_{\rm{min}}^{\rm{di}}$, then the representatives $(R_{\mu,h+}(T), d, \rho, \mathscr{E}_h\mu)$ and $(R_{\mu',h+}(T), d', \rho', \mathscr{E}_h\mu')$ also belong to the same isometry class in $\mathbb{T}_{\rm{min}}^{\rm{di}}$. This allows us to make sense of the following definition.
\begin{defin}[Mass erasure on $\mathbb{T}_{\rm{min}}^{\rm{di}}$]
\label{def:MEIsom}
For every $h > 0$, we define the mass erasure operator $\mathscr{E}_h \colon \mathbb{T}_{\rm{min}}^{\rm{di}} \rightarrow \mathbb{T}_{\rm{min}}^{\rm{di}}$ by setting $\mathscr{E}_h\boldsymbol{\mu} := [R_{\mu,h+}(T), d, \rho, \mathscr{E}_h\mu]$ for every $\boldsymbol{\mu} = [T, d, \rho, \mu] \in \mathbb{T}_{\rm{min}}^{\rm{di}}$.
\end{defin}
The semigroup property in $\mathbb{T}_{\rm{min}}^{\rm{di}}$, $\mathscr{E}_h \circ \mathscr{E}_{h'} = \mathscr{E}_{h + h'}$ for all $h,h' \geq 0$, now follows immediately from Proposition \ref{prop:MEOp3}.

In \cite{DW26} \emph{(Gromov--Prokhorov) convergence in the sense of mass erasure} is defined on $\mathbb{T}_{\rm{min}}^{\rm{fin}}$ to mean that $\mathscr{E}_h \boldsymbol{\mu}_n \stackrel{\rm{GP}}{\rightarrow} \mathscr{E}_h \boldsymbol{\mu}$ for all $h > 0$, and we write $\boldsymbol{\mu}_n \stackrel{\rm{GPme}}{\rightarrow} \boldsymbol{\mu}$. It is shown in \cite{DW26} that the mass erasure operator $\mathscr{E}_h \colon \mathbb{T}_{\rm{min}}^{\rm{fin}} \rightarrow \mathbb{T}_{\rm{min}}^{\rm{fin}}$ is Gromov--Prokhorov continuous, and moreover (Gromov--Prokhorov) convergence in the sense of mass erasure is shown to be strictly weaker than Gromov--Prokhorov convergence on $\mathbb{T}_{\rm{min}}^{\rm{fin}}$. Note that this also holds under variations in the erasure parameter $h$ -- indeed, the map $(h, \boldsymbol{\mu}) \mapsto \mathscr{E}_h\boldsymbol{\mu}$ is sequentially continuous in the Gromov--Prokhorov sense. We can similarly define a notion of \emph{(Gromov-vague) convergence in the sense of mass erasure} on $\mathbb{T}_{\rm{min}}^{\rm{di}}$. 
\begin{defin}[Gromov-vague convergence in the sense of mass erasure]
Let $\boldsymbol{\mu} = [T, d, \rho, \mu]$ and $\boldsymbol{\mu}_n = [T_n, d_n, \rho_n, \mu_n]$, $n \in \mathbb{N}$, be isometry classes in $\mathbb{T}_{\rm{min}}^{\rm{di}}$. We say that $(\boldsymbol{\mu}_n)_{n \in \mathbb{N}}$ converges to $\boldsymbol{\mu}$ \emph{Gromov-vaguely in the sense of mass erasure}, and write $\boldsymbol{\mu}_n \stackrel{\rm{GVme}}{\rightarrow} \boldsymbol{\mu}$, if $\mathscr{E}_h\boldsymbol{\mu}_n \stackrel{\rm{GV}}{\rightarrow} \mathscr{E}_h\boldsymbol{\mu}$ as $n \rightarrow \infty$ for all $h > 0$.
\end{defin}
Contrary to the Gromov--Prokhorov setting, the mass erasure operator $\mathscr{E}_h \colon \mathbb{T}_{\rm{min}}^{\rm{di}} \rightarrow \mathbb{T}_{\rm{min}}^{\rm{di}}$ is \emph{not} Gromov-vaguely continuous. Gromov-vague convergence is also not enough to ensure the Gromov--Hausdorff convergence of mass-erased subtrees. This is illustrated in the example below, which is for simplicity stated in the setting of a fixed $\mathbb{R}$-tree.
\begin{example}
\label{ex:MEConv}
Let $T = [0, \infty)$ be equipped with the standard Euclidean norm, and let $\rho = 0$. Set $\mu_n = \delta_1 + \delta_n$. Then clearly $\mu_n \stackrel{\rm{vg}}{\rightarrow} \mu$ where $\mu = \delta_1$. Let $h \in (0,1)$. Then $\mathscr{E}_h\mu_n = \delta_1 + (1-h)\delta_n \stackrel{\rm{vg}}{\rightarrow} \delta_1$, yet $\mathscr{E}_h\mu = (1-h)\delta_1$. So clearly $\mathscr{E}_h\mu_n \stackrel{\rm{vg}}{\nrightarrow} \mathscr{E}_h\mu$.

We moreover have for all $h \in (0,1)$ that $R_{\mu_n, h}(T) = [0,n]$ but clearly $R_{\mu,h}(T) = [0,1]$, and so $d_H(R_{\mu_n, h}(T), R_{\mu,h}(T)) = n-1$ for all $n \in \mathbb{N}$.
\demo
\end{example}
To overcome this lack of sequential continuity for the mass erasure operator $\mathscr{E}_h$, we will need to impose an appropriate b-tightness condition on the sequence considered, to rule out examples such as the above. For a sequence of isometry classes $\boldsymbol{\mu}_n = [T_n, d_n, \rho_n, \mu_n] \in \mathbb{T}_{\rm{min}}^{\rm{di}}$, $n \in \mathbb{N}$, we will say that $(\mu_n)_{n \in \mathbb{N}}$ is b-tight if $\lim_{R \rightarrow \infty} \limsup_{n \rightarrow \infty} \mu_n\left( T_n \setminus \overline{B}_{T_n}(\rho_n, R) \right) = 0$. This notion extends to transformations and restrictions of the representative measures $(\mu_n)_{n \in \mathbb{N}}$. I.e. we are looking for a statement of the form \emph{``if $\boldsymbol{\mu}_n \stackrel{\rm{GV}}{\rightarrow} \boldsymbol{\mu}$ and appropriate b-tightness conditions are imposed, then $\boldsymbol{\mu}_n \stackrel{\rm{GVme}}{\rightarrow} \boldsymbol{\mu}$"}. Our main method of addressing this will be using the \emph{cut-off method}, as discussed in Section \ref{subsec:MainTechnique}. For some isometry class $\boldsymbol{\mu} = [T, d, \rho, \mu] \in \mathbb{T}_{\rm{min}}^{\rm{di}}$ and some $r > 0$, the `$\mu$-intrinsic' \emph{cut-off points}, $\bullet^r = \bullet^r(\mu)$, at level $r$, will be the points in $\partial B_T(\rho, r)$ which have infinite $\mu$-mass above them. We then obtain a \emph{cut-off tree}, ${\rm{Cut}}(T, \bullet^r)$, by erasing all subtrees above points in $\bullet^r$ from $T$. The associated \emph{cut-off measure}, ${\rm{cut}}_{\bullet^r}\mu = {\rm{cut}}_{\bullet^r}^{h_0} \mu$ is obtained by restricting $\mu$ to ${\rm{Cut}}(T, \bullet^r)$, and assigning point masses of size $h_0$ to each point in $\bullet^r$. One may replace the `$\mu$-intrinsic' cut-off points $\bullet^r = \bullet^r(\mu)$ by more \emph{general} cut-off points, i.e. finite collections of points in $\partial B_T(\rho, r)$, typically denoted by $\blackdiamond^r$. This terminology may of course similarly be applied to the isometry classes $\boldsymbol{\mu}_n = [T_n, d_n, \rho_n, \mu_n] \in \mathbb{T}_{\rm{min}}^{\rm{di}}$.

For some sequence $\boldsymbol{\mu}_n = [T_n, d_n, \rho_n, \mu_n] \in \mathbb{T}_{\rm{min}}^{\rm{di}}$, $n \in \mathbb{N}$, we say for each $n$ that a collection $\blackdiamond_n^r$ of (general) cut-off points in $\partial B_{T_n}(\rho_n, r)$ is \emph{$\bullet^r$-compatible} if $|{\blackdiamond_n^r}| = |{\bullet^r}| \wedge |{\partial B_{T_n}(\rho_n, r)}|$. The use of compatible cut-off points is useful in cases where finite mass in $\boldsymbol{\mu}_n$ grows to infinity.
\begin{lemma}[Approximation via mass erasures]
\label{lemma:MEGV18}
Let $\boldsymbol{\mu} = [T, d, \rho, \mu] \in \mathbb{T}_{\rm{min}}^{\rm{di}}$. Then for a decreasing sequence of numbers $h_p \searrow 0$,
\begin{enumerate}[$(i)$]
\item $\mathscr{E}_{h_p}\boldsymbol{\mu} \stackrel{\rm{GV}}{\rightarrow} \boldsymbol{\mu}$ as $p \rightarrow \infty$, and $\left( \left(\mathscr{E}_{h_p}\mu\right)|_{{\rm{Cut}}(T, \bullet^r(\mu))} \right)_{p \in \mathbb{N}}$ is b-tight.
\item whenever $\boldsymbol{\nu} \in \mathbb{T}_{\rm{min}}^{\rm{di}}$ with $\mathscr{E}_{h_p}\boldsymbol{\mu} = \mathscr{E}_{h_p}\boldsymbol{\nu}$ for all $p$, it follows that $\boldsymbol{\mu} = \boldsymbol{\nu}$.
\end{enumerate}
\end{lemma}
Note that the b-tightness condition in Lemma \ref{lemma:MEGV18}$(i)$ may be checked for \emph{any} representative of the isometry class. We will often choose the radius $r > 0$ such that $\mu(\partial B_T(\rho, r)) = 0$ and $\partial B_T(\rho, r) \cap {\rm{Bp}}(T) = \emptyset$, i.e. such that nothing significant happens on the boundary $\partial B_T(\rho, r)$. If $r$ satisfies these two properties, we say that it is $\boldsymbol{\mu}$-nice. Indeed, Lebesgue-almost all choices of radii are $\boldsymbol{\mu}$-nice due to separability. For a fixed Polish $\mathbb{R}$-tree we similarly say for some $\mu \in \mathscr{M}_{\rm{di}}(T)$ that $r$ is $\mu$-nice if the above conditions are satisfied.
\begin{thm}[Convergence of mass erasures]
\label{thm:MEGV17}
Let $\boldsymbol{\mu} = [T, d, \rho, \mu]$ and $\boldsymbol{\mu}_n = [T_n, d_n, \rho_n, \mu_n]$, $n \in \mathbb{N}$, be isometry classes in $\mathbb{T}_{\rm{min}}^{\rm{di}}$. Let $r_k \rightarrow \infty$ be an increasing sequence of $\boldsymbol{\mu}$-nice radii. Let for each $n$ and $k$, $\blackdiamond_n^{r_k}$ be $\bullet^{r_k}(\mu)$-compatible cut-off points in $\partial B_{T_n}(\rho_n, r_k)$. Let $h, h_n \in (0, \infty)$, $n \in \mathbb{N}$, and suppose that
\begin{enumerate}[$(i)$]
\item $\boldsymbol{\mu}_n \stackrel{\rm{GV}}{\rightarrow} \boldsymbol{\mu}$, and $\left( \mu_n|_{{\rm{Cut}}(T_n, \blackdiamond_n^{r_k})} \right)_{n \in \mathbb{N}}$ is b-tight for all $k$,
\item $h_n \rightarrow h$ as $n \rightarrow \infty$.
\end{enumerate}
Then  $\mathscr{E}_{h_n}\boldsymbol{\mu}_n \stackrel{\rm{GV}}{\rightarrow} \mathscr{E}_h\boldsymbol{\mu}$, and $\left( \left(\mathscr{E}_h \mu_n \right)|_{{\rm{Cut}}(T_n, \blackdiamond_n^{r_k})} \right)_{n \in \mathbb{N}}$ is b-tight for all $k$.
\end{thm}
As an immediate consequence of Theorem \ref{thm:MEGV17} we see that under assumption $(i)$ it follows that $\boldsymbol{\mu}_n \stackrel{\rm{GVme}}{\rightarrow} \boldsymbol{\mu}$. In particular, we may think of Gromov-vague convergence in the sense of mass erasure as being a weaker mode of convergence than Gromov-vague convergence, provided that we impose b-tightness on appropriate choices of cut-off trees.
\begin{thm}[Convergence of mass-erased subtrees]
\label{thm:MEGV23b}
Let for all $n$, $\boldsymbol{\mu} = [T, d, \rho, \mu]$ and $\boldsymbol{\mu}_n = [T_n, d_n, \rho_n, \mu_n]$ be isometry classes in $\mathbb{T}_{\rm{min}}^{\rm{di}}$. Fix some $\boldsymbol{\mu}$-nice $r > 0$. Let for each $n$, $\blackdiamond_n^r$ be $\bullet^r(\mu)$-compatible cut-off points in $\partial B_{T_n}(\rho_n, r)$. Let $h, h_n \in (0,\infty)$, $n \in \mathbb{N}$, and let $h_0 \geq \sup_{n \in \mathbb{N}} h_n$. Suppose that
\begin{enumerate}[$(i)$]
\item $\boldsymbol{\mu}_n \stackrel{\rm{GVme}}{\rightarrow} \boldsymbol{\mu}$, and $\left( \left(\mathscr{E}_h\mu_n\right)|_{{\rm{Cut}}(T_n, \blackdiamond_n^r)} \right)$ is b-tight for all $h > 0$,
\item $h_n \rightarrow h$, and the map $h' \mapsto R_{{\rm{cut}}_{\bullet^r(\mu)}^{h_0}\mu, h'}(T)$ is $d_H$-continuous at $h$.\footnote{This condition may be removed if $h_n = h$ eventually.}
\end{enumerate}
Then $\lim_{n \rightarrow \infty} d_{\rm{GH}}\left( R_{\mu_n,h_n}(T_n) \cap {\rm{Cut}}(T_n, \blackdiamond_n^r), R_{\mu,h}(T) \cap {\rm{Cut}}(T, \bullet^r(\mu)) \right) = 0$. If $(i)$--$(ii)$ are true for some increasing sequence of radii $r_k \rightarrow \infty$, then $\lim_{n \rightarrow \infty} d_{\rm{GH}}^{\rm{loc}}(R_{\mu_n, h_n}(T_n), R_{\mu,h}(T)) = 0$.
\end{thm}
One may ask in which cases Gromov-vague convergence in the sense of mass erasure can be strengthened to Gromov-vague convergence (under sufficient b-tightness). As in \cite{DW26}, this boils down to a question of convergence of \emph{height measures}. Let $(T, d, \rho)$ be a Polish $\mathbb{R}$-tree, and let $\mu \in \mathscr{M}_{\rm{di}}(T)$. Recall from \cite{DW26} Definition 3.26$(b)$ the \emph{height measure} $\lambda_{\mu}$, which is the pushforward of $\mu$ under the map $T \ni \sigma \mapsto d(\rho, \sigma) \in [0, \infty)$. It is easily checked that $\lambda_{\mu} \in \mathscr{M}_{\rm{bf}}([0,\infty))$. If $\boldsymbol{\mu} = [T, d, \rho, \mu] \in \mathbb{T}_{\rm{min}}^{\rm{di}}$, then all representatives of $\boldsymbol{\mu}$ will have the same pushforward measure under the map $\sigma \mapsto d(\rho, \sigma)$. So there is a unique measure $\lambda_{\boldsymbol{\mu}} \in \mathscr{M}_{\rm{bf}}([0, \infty))$ which is the height measure for all representatives of $\boldsymbol{\mu}$. Note that if $\boldsymbol{\mu} \in \mathbb{T}_{\rm{min}}^{\rm{fin}}$, then also $\lambda_{\boldsymbol{\mu}} \in \mathscr{M}_{\rm{fin}}([0,\infty))$. 
\begin{thm}
\label{thm:MEGV24}
Let for all $n$, $\boldsymbol{\mu} = [T, d, \rho, \mu]$ and $\boldsymbol{\mu}_n = [T_n, d_n, \rho_n, \mu_n]$ be isometry classes in $\mathbb{T}_{\rm{min}}^{\rm{di}}$. Let $r_k \rightarrow \infty$ be an increasing sequence of $\boldsymbol{\mu}$-nice radii. Let for each $n$ and $k$, $\blackdiamond_n^{r_k}$ be $\bullet^{r_k}(\mu)$-compatible cut-off points in $\partial B_{T_n}(\rho_n, r_k)$. Let $h_0 > 0$. Then the following are equivalent:
\begin{itemize}
\item[$(a)$] $\boldsymbol{\mu}_n \stackrel{\rm{GV}}{\rightarrow} \boldsymbol{\mu}$ and $\left( \mu_n|_{{\rm{Cut}}(T_n, \blackdiamond_n^{r_k})} \right)_{n \in \mathbb{N}}$ is b-tight for all $k$.
\item[$(b)$] $\boldsymbol{\mu}_n \stackrel{\rm{GVme}}{\rightarrow} \boldsymbol{\mu}$, $\left( \left(\mathscr{E}_h\mu_n \right)|_{{\rm{Cut}}(T_n, \blackdiamond_n^{r_k})} \right)_{n \in \mathbb{N}}$ is b-tight for all $h > 0$ and $k$, and $\lambda_{{\rm{cut}}_{\blackdiamond_n^{r_k}}^{h_0}\boldsymbol{\mu}_n} \stackrel{\rm{wk}}{\rightarrow} \lambda_{{\rm{cut}}_{\bullet^{r_k}(\mu)}^{h_0}\boldsymbol{\mu}}$ as $n \rightarrow \infty$ for all $k$.
\end{itemize}
\end{thm}
In applications, when proving Gromov-vague convergence in the sense of mass erasure for some specific sequence of isometry classes $\boldsymbol{\mu}_n = [T_n, d_n,\rho_n, \mu_n]$, $n \in \mathbb{N}$, in $\mathbb{T}_{\rm{min}}^{\rm{di}}$, we may want to check for each $h > 0$ that the sequence $(\mathscr{E}_h \boldsymbol{\mu}_n)_{n \in \mathbb{N}}$ converges Gromov-vaguely to some limit, $\boldsymbol{\nu}_h$, which will naturally depend on the erasure parameter $h$. The question now, is whether we may construct a `global' limit $\boldsymbol{\mu} \in \mathbb{T}_{\rm{min}}^{\rm{di}}$ such that $\mathscr{E}_h\boldsymbol{\mu} = \boldsymbol{\nu}_h$ for all $h > 0$. To address this question, we start by noting that if $\boldsymbol{\mu}_n \in \mathbb{T}_{\rm{min}}^{\rm{di}}$, $n \in \mathbb{N}$, then \emph{any} possible Gromov-vague limit of $\left( \mathscr{E}_h \boldsymbol{\mu}_n \right)_{n \in \mathbb{N}}$ must necessarily be contained in $\mathbb{T}_{\rm{min}}^{\rm{di}}$. This observation is non-trivial, as we recall that the space $(\mathbb{T}_{\rm{min}}^{\rm{di}}, d_{\rm{GV}})$ is Lusin, but \emph{not} complete.

\begin{lemma}
\label{lemma:MEGV24b}
Fix some $h > 0$. Then $ \overline{\mathscr{E}_h(\mathbb{T}_{\rm{min}}^{\rm{di}})}^{d_{\rm{GV}}} \subseteq \mathbb{T}_{\rm{min}}^{\rm{di}}$.
\end{lemma}
We may finally state the following existence result, which allows us to construct global limits in $\mathbb{T}_{\rm{min}}^{\rm{di}}$, just by considering limits of the $h_p$-mass erasures for some decreasing sequence of erasure parameters $h_p \searrow 0$, under appropriate b-tightness conditions on both the mass erasures $(\mathscr{E}_{h_p}\boldsymbol{\mu}_n)_{n \in \mathbb{N}}$ with respect to their limits for each $p$, and across the limits considered as a sequence indexed by $p \in \mathbb{N}$.

\begin{thm}[Existence of a limit]
\label{thm:MEGV27}
Let $\boldsymbol{\mu}_n = [T_n, d_n, \rho_n, \mu_n]$, $n \in \mathbb{N}$, be a sequence of isometry classes in $\mathbb{T}_{\rm{min}}^{\rm{di}}$, and let $h_p \searrow 0$ be a strictly decreasing sequence of numbers. Suppose for every $p \in \mathbb{N}$ that there exists some $\boldsymbol{\nu}_p = [T_p', d_p', \rho_p', \nu_p] \in \mathbb{T}_{\rm{min}}^{\rm{di}}$ such that $\mathscr{E}_{h_p}\boldsymbol{\mu}_n \stackrel{\rm{GV}}{\rightarrow} \boldsymbol{\nu}_p$ as $n \rightarrow \infty$. 

For any given $r > 0$ and $p$, set as usual $\bullet^r(\nu_p) = \partial B_{T_p'}(\rho_p', r) \cap {\rm{Skel}}(T_p', \nu_p)$, and let for each $n$, $\blackdiamond_n^r(\nu_p)$ be $\bullet^r(\nu_p)$-compatible cut-off points in $\partial B_{T_n}(\rho_n, r)$. Suppose for some increasing sequence of radii $r_k \rightarrow \infty$, that
\begin{enumerate}[$(i)$]
\item $\left( \left( \mathscr{E}_{h_p}\mu_n \right)|_{{\rm{Cut}}(T_n, \blackdiamond_n^{r_k}(\nu_p))} \right)_{n \in \mathbb{N}}$ is b-tight for all $p$, and
\item $\left( \nu_p|_{{\rm{Cut}}(T_p', \bullet^{r_k}(\nu_p))} \right)_{p \in \mathbb{N}}$ is b-tight.
\end{enumerate}
Then there exists a unique isometry class $\boldsymbol{\mu} = [T, d, \rho, \mu] \in \mathbb{T}_{\rm{min}}^{\rm{di}}$ such that $\boldsymbol{\mu}_n \stackrel{\rm{GVme}}{\rightarrow} \boldsymbol{\mu}$ as $n \rightarrow \infty$ and $\mathscr{E}_{h_p}\boldsymbol{\mu} = \boldsymbol{\nu}_p$ for all $p \in \mathbb{N}$.
\end{thm}
\begin{rmk}
i) By Lemma \ref{lemma:MEGV24b} it is sufficient to check that $(\mathscr{E}_{h_p}\boldsymbol{\mu}_n)_{n \in \mathbb{N}}$ is Cauchy in the Gromov-vague topology to get the existence of a limit $\boldsymbol{\nu}_p \in \mathbb{T}_{\rm{min}}^{\rm{di}}$. ii) If $\mathscr{E}_{h_p}\boldsymbol{\mu} = \boldsymbol{\nu}_p$, then $|{\bullet^r(\nu_p)}| = |{\bullet^r(\mathscr{E}_{h_p}\mu)}| = |{\bullet^r(\mu)}|$, so the cut-off points $\blackdiamond_n^r(\nu_p) = \blackdiamond_n^r$ from Theorem \ref{thm:MEGV27} may be considered as $\bullet^r(\mu)$-compatible cut-off points in $\partial B_{T_n}(\rho_n, r)$.
\demo
\end{rmk}
In \cite{DW26}, Duquesne and Winkel metrized (Gromov--Prokhorov) convergence in the sense of mass erasure, and identified the completion of $\mathbb{T}_{\rm{min}}^{\rm{fin}}$ under this metric as a \emph{bordification} space. We may similarly metrize the notion of Gromov-vague convergence in the sense of mass erasure.

\begin{prop}
\label{prop:MEGV25}
Let $h_p \searrow 0$ be a decreasing sequence of numbers in $(0, \infty)$, and set
$$ \delta_{\rm{GV}}^{\rm{era}}(\boldsymbol{\mu}, \boldsymbol{\nu}) := \sum_{p=1}^{\infty} 2^{-p}\left(1 \wedge d_{\rm{GV}}(\mathscr{E}_{h_p}\boldsymbol{\mu}, \mathscr{E}_{h_p}\boldsymbol{\nu}) \right) \quad \text{for all } \boldsymbol{\mu}, \boldsymbol{\nu} \in \mathbb{T}_{\rm{min}}^{\rm{di}}. $$
Then $(\mathbb{T}_{\rm{min}}^{\rm{di}}, \delta_{\rm{GV}}^{\rm{era}})$ is a metric space, and $\delta_{\rm{GV}}^{\rm{era}}$ metrizes Gromov-vague convergence in the sense of mass erasure on $\mathbb{T}_{\rm{min}}^{\rm{di}}$. The space $(\mathbb{T}_{\rm{min}}^{\rm{di}}, \delta_{\rm{GV}}^{\rm{era}})$ is separable but \emph{not} complete.
\end{prop}
\begin{rmk}
\label{rmk:OpenQuestion}
It remains an open question whether the space $(\mathbb{T}_{\rm{min}}^{\rm{di}}, \delta_{\rm{GV}}^{\rm{era}})$ is Lusin. A natural idea to consider in this context is whether a bordification as in \cite{DW26} may be utilized to characterize a completion of the space $(\mathbb{T}_{\rm{min}}^{\rm{di}}, \delta_{\rm{GV}}^{\rm{era}})$. This is however not straight-forward, even in the simpler setting where one considers a fixed $\mathbb{R}$-tree $(T, d, \rho)$. Recall from \cite{DW26} Theorem 2.13 that the bordification $T^*$ is equipped with a metric $d^*$ which distorts distances in a way such that all points in $(T^*, d^*)$ are at most distance $2$ away from each other. If one considers a discretely infinite measure $\mu$ on $T$ the distortion of the metric under the bordification will mean that infinite mass is squeezed into a closed ball of radius $2$ when considering the measure $\mu$ on $(T^*, d^*)$. As such $\mu \notin \mathscr{M}_{\rm{bf}}(T^*)$, and so the setting collapses. A similar issue naturally persists when one considers the more general setting of isometry classes.
\demo
\end{rmk}
Although it remains an open question whether the space $(\mathbb{T}_{\rm{min}}^{\rm{di}}, \delta_{\rm{GV}}^{\rm{era}})$ is Lusin, we may still make sense of Gromov-vague convergence in the sense of mass erasure for \emph{random} isometry classes. Indeed, the mass erasures may be considered as well-defined random variables on the Lusin space $(\mathbb{T}_{\rm{min}}^{\rm{di}}, d_{\rm{GV}})$. A sequence of random isometry classes in $\mathbb{T}_{\rm{min}}^{\rm{di}}$ will then be said to converge \emph{in distribution} in the Gromov-vague sense of mass erasure, if for all $h > 0$ the associated mass erasures convergence in distribution in the Lusin space $(\mathbb{T}_{\rm{min}}^{\rm{di}}, d_{\rm{GV}})$. In our companion paper \cite{Draft2}, we extend the results outlined here to a probabilistic setting, in order to prove an invariance principle for supercritical Galton--Watson $\mathbb{R}$-forests which do not satisfy Grey's condition, analogously to what was done by Duquesne and Winkel in \cite{DW26}.
\subsection{Techniques and cut-off methods}
\label{subsec:MainTechnique}
\emph{Local compactness and leaf-length-erasure:} As the notion of (local) compactness plays a central role in the difficulties we are trying to address with mass erasure, let us heuristically describe how this is expressed in $\mathbb{R}$-trees. Let $(T, d, \rho)$ be a Polish $\mathbb{R}$-tree. If it is (locally) compact, it may still be somewhat ill-behaved in the sense that we may observe i) infinite degree branch points, i.e. points $\sigma \in T$ with $n(\sigma, T) = +\infty$, and ii) accumulations of branch points, i.e. convergent sequences $(\sigma_n)_{n \in \mathbb{N}}$ of distinct points in ${\rm{Bp}}(T)$. In either of the cases i) and ii), (local) compactness will ensure that the heights of the countably many subtrees associated to either i) or ii) will converge to zero. In a general Polish $\mathbb{R}$-tree, these subtrees may not have heights tending to zero -- indeed, one may encounter the phenomena i) and ii) in a setting where all of the countably many subtrees are e.g. of fixed or infinite height. See Figure \ref{fig:LC}.

\begin{figure}[b]
\center
\includegraphics[width=\textwidth]{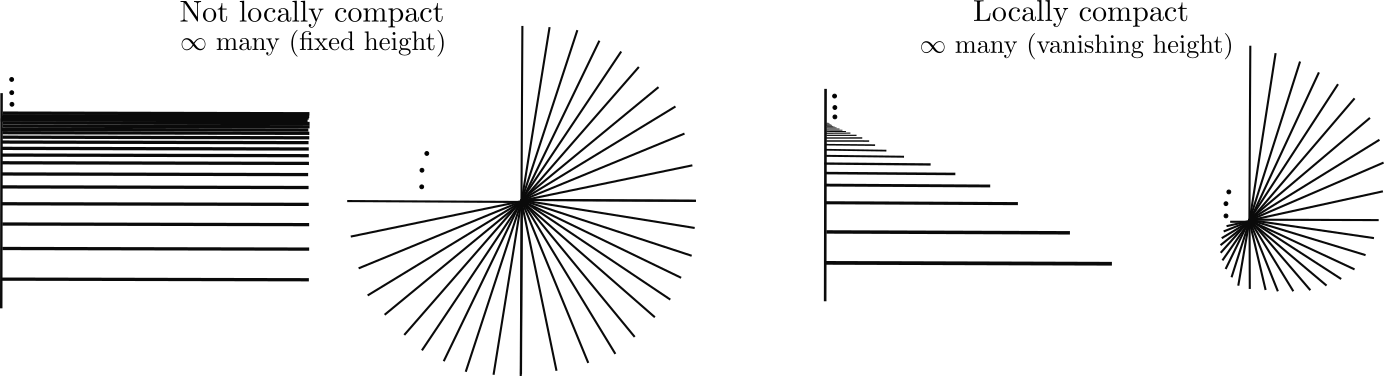}
\caption{Illustration of how i) accumulations of branch points, and ii) infinite degree branch points may appear in, respectively, Polish $\mathbb{R}$-trees and locally compact $\mathbb{R}$-trees.}
\label{fig:LC}
\end{figure}

Let $(T, d, \rho)$ be a complete locally compact $\mathbb{R}$-tree. A useful technique for examining the convergence of a sequence of such $\mathbb{R}$-trees is \emph{leaf-length erasure}, also known as \emph{trimming} \cite{DW19, EPW06, N86, NP89a, NP89b}. Given some erasure parameter $h > 0$, one constructs the $h$-leaf-length erased subtree by setting $ R_h(T) := \{\rho\} \cup \{\sigma \in T ~|~ {\rm{Ht}}(\theta_{\sigma}T) \geq h \}$, i.e. by trimming away all points which have above them a subtree of height strictly less than $h$. As seen in the illustration of local compactness in Figure \ref{fig:LC}, this will get rid of countably many of the small fringe subtrees that may appear in infinite degree branch points or accumulations of branch points, and one may indeed prove that the $h$-leaf-length erased subtree is a discrete $\mathbb{R}$-tree \cite{DW19}. In \cite{DW19} Duquesne and Winkel moreover show that a sequence of complete locally compact $\mathbb{R}$-trees will converge in the local Gromov--Hausdorff topology (see Section \ref{subsec:Gromov} for an overview of the various Gromov-type topologies) if and only if the $h$-leaf-length erased subtrees converge in the local Gromov--Hausdorff topology across all $h > 0$. In other words, local Gromov--Hausdorff convergence is induced by convergence of the much simpler $h$-leaf-length erased subtrees. This has been used to prove invariance principles for Galton--Watson forests, \emph{assuming Grey's condition} (which ensures local compactness of the limiting $\mathbb{R}$-tree), as seen in \cite{DW19}.

Suppose now instead that $(T, d, \rho)$ is a Polish $\mathbb{R}$-tree, which may not be locally compact. By Figure \ref{fig:LC} erasing fringe subtrees according to their height is not a technique particularly well-suited to dealing with general Polish $\mathbb{R}$-trees. Mass erasure thus provides a suitable alternative to this technique by instead removing fringe subtrees according to their \emph{mass}, allocated by some measure $\mu \in \mathscr{M}_{\rm{di}}(T)$. As noted in Lemma \ref{lemma:MEFix2} the mass-erased subtrees obtained are discrete $\mathbb{R}$-trees, and may be studied in the local Gromov-Hausdorff topology, as in Theorem \ref{thm:MEGV23b}.\\

\noindent \emph{Cut-off methods:} As a general theme throughout this paper, we will be considering the \emph{cut-off points}, \emph{cut-off trees} and \emph{cut-off measures} of $(T, d, \rho, \mu)$.

Let $(T, d, \rho)$ be a Polish $\mathbb{R}$-tree, and let $\mu \in \mathscr{M}_{\rm{di}}(T)$. Let $\blackdiamond^r$ be a finite collection of points on the level-$r$ boundary $\partial B_T(\rho, r)$. We may construct a \emph{cut-off tree} with respect to the \emph{cut-off points} $\blackdiamond^r$, by cutting off any subtrees above the points $\blackdiamond^r$ in $T$. We define the cut-off tree with respect to $\blackdiamond^r$ by
$$ {\rm{Cut}}(T, \blackdiamond^r) := \overline{B}_T(\rho, r) \cup \bigcup_{\boldsymbol{\sigma} \in \partial B_T(\rho, r) \setminus \blackdiamond^r} \theta_{\boldsymbol{\sigma}}T, $$
and note that ${\rm{Cut}}(T, \blackdiamond^r)$ is a closed rooted subtree of $T$ with boundary $\partial {\rm{Cut}}(T, \blackdiamond^r) = \blackdiamond^r$.
We may for any $h_0 > 0$ define an associated \emph{cut-off measure}, which removes any mass outside of ${\rm{Cut}}(T, \blackdiamond^r)$, and then associates a point mass of $h_0$ to every cut-off point $\blackdiamond^r$, i.e. by setting
\begin{equation} 
\label{eq:CutMeas}
{\rm{cut}}_{\blackdiamond^r} \mu = {\rm{cut}}_{\blackdiamond^r}^{h_0} \mu := \mu|_{{\rm{Cut}}(T, \blackdiamond^r)} + h_0\sum_{\boldsymbol{\sigma} \in \blackdiamond^r} \delta_{\boldsymbol{\sigma}}.
\end{equation}
It is easily checked that ${\rm{cut}}_{\blackdiamond^r} \mu \in \mathscr{M}_{\rm{di}}(T)$. The most natural cut-off points to consider will be the points on $\partial B_T(\rho, r)$ which have infinite $\mu$-mass above them. We denote these `$\mu$-intrinsic' cut-off points by $\bullet^r = \bullet^r(\mu) := \partial B_T(\rho, r) \cap {\rm{Skel}}(T, \mu)$, and note that $|{\bullet^r(\mu)}| < \infty$ for all $r > 0$ due to the assumption that $\mu \in \mathscr{M}_{\rm{di}}(T)$. Using this notation, we moreover note that condition $(ii)$ in Definition \ref{def:DIMeas} may be replaced by the condition $\mu\left( {\rm{Cut}}(T, \bullet^r(\mu)) \right) < \infty$ for all $r > 0$. See Figure \ref{fig:MEFix} for an illustration of the cut-off points $\bullet^r(\mu)$ and their associated cut-off tree.

\begin{figure}[t]
\center
\includegraphics[width=0.6\textwidth]{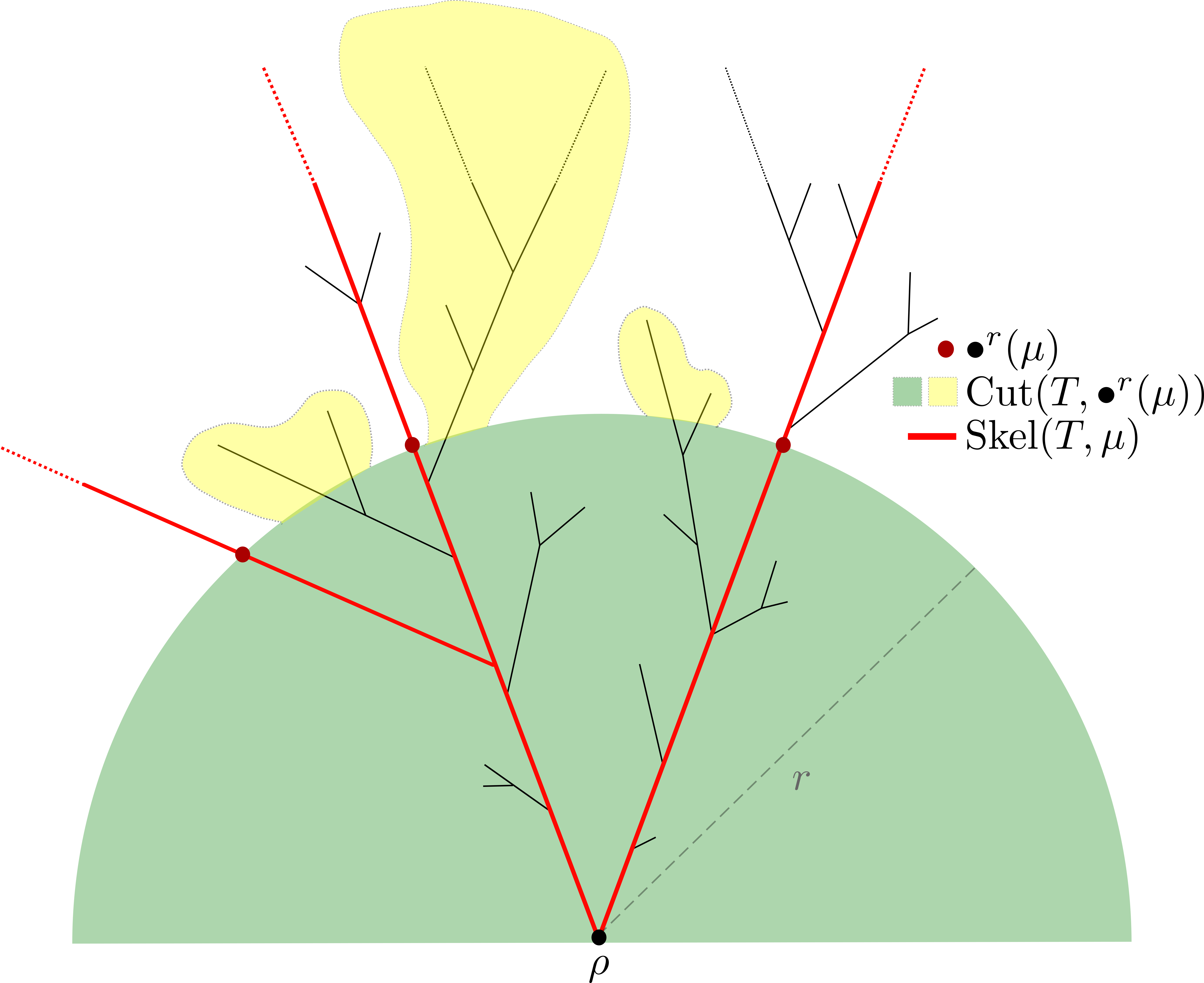}
\caption{The skeleton, intrinsic cut-off points and associated cut-off tree of a discretely infinitely measured Polish $\mathbb{R}$-tree $(T, d, \rho, \mu)$.}
\label{fig:MEFix}
\end{figure}

For the intrinsic cut-off points, clearly ${\rm{cut}}_{\bullet^r(\mu)}^{h_0} \mu$ is in $\mathscr{M}_{\rm{fin}}(T)$ with support on ${\rm{Cut}}(T, \bullet^r(\mu))$. When considering $h$-mass-erased subtrees, we will typically choose $h_0 \geq h$. As the specific choice of $h_0$ rarely matters beyond this, we will often suppress the parameter in our notation. 

Let us illustrate how the cut-off method is generally applied by outlining the proof-strategy for Lemma \ref{lemma:MEFix2}. Write $\bullet^r = \bullet^r(\mu)$, and let $h_0 \geq h$. To prove Lemma \ref{lemma:MEFix2}, one may make the crucial observation that $R_{\mu, h}(T) \cap {\rm{Cut}}(T, \bullet^r) = R_{{\rm{cut}}_{\bullet^r}^{h_0}\mu, h}(T)$ (see Lemma \ref{lemma:MEFix1}), in other words, when restricting to the cut-off tree at level $r$, the mass-erased subtree under the discretely infinite measure $\mu$ is precisely the mass-erased subtree under the \emph{finite} cut-off measure ${\rm{cut}}_{\bullet^r}^{h_0} \mu$. As the theory of mass erasure with finite measures has been developed in \cite{DW26} by Duquesne and Winkel, we may apply their results to the $h$-mass-erased subtree $R_{{\rm{cut}}_{\bullet^r}^{h_0}\mu, h}(T)$. This allows us to conclude that $R_{\mu, h}(T) \cap {\rm{Cut}}(T, \bullet^r(\mu))$ is a compact discrete $\mathbb{R}$-tree, and as this is the case for every $r > 0$ (in particular as $r \rightarrow \infty$), we may arrive at the conclusion from Lemma \ref{lemma:MEFix2}. 

Throughout this work, we will frequently make use of this general method, where we restrict our object of interest to some cut-off tree/measure at level $r$ and then translate it into something that may be expressed in terms of finite mass erasure. This allows us to apply relevant results from \cite{DW26} to obtain our desired properties \emph{locally} at level $r$. Doing this across different radii $r > 0$ as $r \rightarrow \infty$, may allow us to make global statements based on these local properties. Indeed, this local approach goes nicely together with the idea of vague convergence (see Section \ref{subsec:MainVague}).

Let us finally note that cut-off measures may also be appropriately considered in the setting of isometry classes. Let $\boldsymbol{\mu} = [T, d, \rho, \mu] \in \mathbb{T}_{\rm{min}}^{\rm{di}}$, and let $\blackdiamond^r$ be a collection of cut-off points in $\partial B_T(\rho, r)$. If $\left( \widetilde{T}, \widetilde{d}, \widetilde{\rho}, \widetilde{\mu} \right)$ is another representative of the isometry class $\boldsymbol{\mu}$, then there must exist a collection of points $\tilde{\blackdiamond}^r$ in $\partial B_{\widetilde{T}}(\widetilde{\rho}, r)$ with $|{\tilde{\blackdiamond}^r}| = |{\blackdiamond^r}|$ such that the minimal spaces generated by ${\rm{cut}}_{\tilde{\blackdiamond}^r}\widetilde{\mu}$ and ${\rm{cut}}_{\blackdiamond^r}\mu$ are isometric. Indeed, any measured space which is isometric to the minimal space generated by ${\rm{cut}}_{\blackdiamond^r}\mu$, will be possible to construct from a representative of $\boldsymbol{\mu}$ by picking $|{\blackdiamond^r}|$ appropriate points on the boundary at level $r$. With a slight abuse of notation, noting of course that $\blackdiamond^r$ depends on the specific representative of the isometry class $\boldsymbol{\mu}$, we set
\begin{align*}
{\rm{cut}}_{\blackdiamond^r}\boldsymbol{\mu} &:= \left[\overline{{\rm{Span}}({\rm{supp}}({\rm{cut}}_{\blackdiamond^r}\mu))}, d, \rho, {\rm{cut}}_{\blackdiamond^r}\mu\right].
\end{align*}
\section{Measures and topologies}
\label{sec:3}
For any subset $A \subseteq \mathscr{X}$ and any $\varepsilon > 0$, we denote the $\varepsilon$-neighbourhood of $A$ by $A^{\varepsilon} := \{ x \in \mathscr{X} ~|~ d(x,A) \leq \varepsilon\}$ where $d(x,A) := \inf\{ d(x,y) ~|~ y \in A\}$. Let $\mathfrak{K}(\mathscr{X})$ denote the collection of non-empty compact subsets of $\mathscr{X}$. We define the Hausdorff distance $d_{\rm{H}}$ on $\mathfrak{K}(\mathscr{X})$ by
$$ d_{\rm{H}}(A, B) := \inf\{ \varepsilon > 0 ~|~ A \subseteq B^{\varepsilon}, ~ B \subseteq A^{\varepsilon} \} \quad \text{for all } A,B \in \mathfrak{K}(\mathscr{X}).$$
It is well-known that $(\mathfrak{K}(\mathscr{X}), d_{\rm{H}})$ is a metric space which is complete whenever $(\mathscr{X}, d)$ is complete, and compact whenever $(\mathscr{X}, d)$ is compact (see e.g., \cite{BBI01} Section or \cite{E08} Section 4.1). Moreover, it is known from \cite{BBI01} Section 7.3.1 that if $A_n \supseteq A_{n+1}$ is a non-increasing sequence of sets in $\mathfrak{K}(\mathscr{X})$, then setting $A = \bigcap_{n \in \mathbb{N}} A_n$ it holds that $\lim_{n \rightarrow \infty}d_H(A_n, A) = 0$. Similarly, if $A_n \subseteq A_{n+1}$ is a non-decreasing sequence of sets in $\mathfrak{K}(\mathscr{X})$, then setting $A = \overline{\bigcup_{n \in \mathbb{N}} A_n}$ it holds that $\lim_{n \rightarrow \infty} d_H(A_n, A) = 0$. The Hausdorff distance can be localized to the set $\mathfrak{M}(\mathscr{X}) := \left\{ A \subseteq \mathscr{X} ~|~ A \neq \emptyset, ~ A \text{ is closed and locally compact} \right\}$, by setting
$$ d_{\rm{H}}^{\rm{loc}}(A, B) := \sum_{k=1}^{\infty} \frac{1}{2^k} d_{\rm{H}}(A \cap \overline{B}_{\mathscr{X}}(\rho, k), B \cap \overline{B}_{\mathscr{X}}(\rho, k)) \quad \text{for all } A, B \in \mathfrak{M}(\mathscr{X}).$$
Then $(\mathfrak{M}(\mathscr{X}), d_{\rm{H}}^{\rm{loc}})$ is a metric space, and the topology generated by $d_{\rm{H}}^{\rm{loc}}$ is independent of the choice of point $\rho$. It can be checked that $(\mathfrak{M}(\mathscr{X}), d_{\rm{H}}^{\rm{loc}})$ is Polish whenever $(\mathscr{X}, d)$ is Polish. This is done in an analogous manner to \cite{DW07} Proposition 3.3. It is moreover easily checked that if $A_n \supseteq A_{n+1}$ is a non-increasing sequence of sets in $\mathfrak{M}(\mathscr{X})$, then for $A = \bigcap_{n \in \mathbb{N}} A_n$ it holds that $\lim_{n \rightarrow \infty} d_{\rm{H}}^{\rm{loc}}(A_n, A) =0$, using the fact that the property holds in $(\mathfrak{K}(\mathscr{X}), d_{\rm{H}})$, and applying the dominated convergence theorem. Similarly, if $A_n \subseteq A_{n+1}$ is a non-decreasing sequence of sets in $\mathscr{M}(\mathscr{X})$, then for $A = \overline{\bigcup_{n \in \mathbb{N}} A_n}$ it holds that $\lim_{n \rightarrow \infty} d_{\rm{H}}^{\rm{loc}}(A_n, A) =0$.

We define the Prokhorov distance $d_{\rm{P}}$ on $\mathscr{M}_{\rm{fin}}(\mathscr{X})$ by
$$ d_{\rm{P}}(\mu, \nu) := \inf\left\{ \varepsilon > 0 ~\Big{|}~ \forall A \subseteq \mathscr{X} \text{ closed}: ~ \begin{array}{l}
\mu(A) \leq \nu(A^{\varepsilon}) + \varepsilon\\ 
\nu(A) \leq \mu(A^{\varepsilon}) + \varepsilon
\end{array} \right\} \quad \text{for all } \mu, \nu \in \mathscr{M}_{\rm{fin}}(\mathscr{X}),$$
such that $(\mathscr{M}_{\rm{fin}}(\mathscr{X}), d_{\rm{P}})$ and $(\mathscr{M}_{\rm{prob}}(\mathscr{X}), d_{\rm{P}})$ are metric spaces, which are known to be Polish whenever $(\mathscr{X}, d)$ is Polish, (see e.g., \cite{DVJ03} Appendix A2.5). We also recall from \cite{DVJ03} Appendix A2.5 that the Prokhorov distance metrizes weak convergence of finite measures, i.e. $d_{\rm{P}}(\mu_n, \mu) \rightarrow 0$ if and only if $\mu_n \stackrel{\rm{wk}}{\rightarrow} \mu$ for $(\mu_n)_{n \in \mathbb{N}}$ and $\mu$ measures in $\mathscr{M}_{\rm{fin}}(\mathscr{X})$.
\subsection{The vague topology}
\label{subsec:BFmeas}
In this section we discuss the notion of \emph{vague convergence} for \emph{boundedly finite measures}, as recalled in Section \ref{subsec:MainVague}. Vague convergence is often understood to be defined as a functional form of convergence for Radon measures on Heine--Borel spaces. Instead of taking the classical functional approach to vague convergence, we will be following the approach of Athreya, Löhr and Winter \cite{ALW16} in which they introduce a notion of \emph{Gromov-vague} and \emph{Gromov--Hausdorff-vague} convergence, not just in a Heine--Borel setting, but for isometry classes of general Polish spaces equipped with boundedly finite measures. Before discussing the `Gromov-type' convergences of \cite{ALW16} in Section \ref{subsec:Gromov}, we will follow their approach to get a similar and suitable metrization of vague convergence on the space of boundedly finite measures in a fixed pointed Polish space $(\mathscr{X}, d, \rho)$.

We note here for context that one might instead be inclined to suggest the \emph{weak$^{\#}$} topology (metrized through the \emph{weak$^{\#}$ distance}, $d^{\#}$) as known from e.g., Daley and Vere-Jones \cite{DVJ03} Appendix A2.6, when considering a suitable topology for the space $\mathscr{M}_{\rm{bf}}(\mathscr{X})$. However, as discussed by Morariu-Patrichi \cite{MP18}, there are some issues with the proof in \cite{DVJ03} that $(\mathscr{M}_{\rm{bf}}(\mathscr{X}), d^{\#})$ forms a Polish space, which to our knowledge have only been resolved in specific settings, such as the one in \cite{MP18} where only \emph{integer-valued} boundedly finite measures are considered. To avoid these issues, and provide more consistency with the metrizations provided in the Gromov-setting, we will base our results solely on the approach of \cite{ALW16}. To remain faithful to their naming of the Gromov-type topologies, we will refer to our topology on $\mathscr{M}_{\rm{bf}}(\mathscr{X})$ as the \emph{vague} topology. We note indeed that convergence in our vague topology will be equivalent to weak$^{\#}$ convergence, and so (as discussed in \cite{DVJ03, LR16}) be consistent with the traditional notion of vague convergence, when considering the special case where $(\mathscr{X}, d)$ is a Heine--Borel space (in which case any boundedly finite measure on $\mathscr{X}$ automatically becomes a Radon measure).
\begin{defin}[Vague topology]
\label{def:BFMeas}
We define a metric on $\mathscr{M}_{\rm{bf}}(\mathscr{X})$ by setting
$$ d_V(\mu, \nu) := \int_0^{\infty} e^{-r} (1 \wedge d_P(\mu^{\restr r}, \nu^{\restr r})) dr \quad \text{for all } \mu, \nu \in \mathscr{M}_{\rm{bf}}(\mathscr{X}), $$
and call the topology induced by this metric, the \emph{vague topology}.
\end{defin}
Convergence in the vague topology is denoted by $\mu_n \stackrel{\rm{vg}}{\rightarrow} \mu$, and it is immediately clear that $d_{\rm{V}} \leq d_{\rm{P}}$. The following proposition is analogous to \cite{ALW16} Lemma 2.6, and ensures that convergence under the vague topology coincides with our notion of vague convergence from Section \ref{subsec:MainVague}.
\begin{prop}
\label{prop:BFMeas1}
Let $(\mathscr{X}, d, \rho)$ be a pointed Polish space, and let $\mu, \mu_n \in \mathscr{M}_{\rm{bf}}(\mathscr{X})$, $n \in \mathbb{N}$. Then the following statements are equivalent:
\begin{itemize}
\item[$(a)$] $\mu_n^{\restr r} \stackrel{\rm{wk}}{\rightarrow} \mu^{\restr r}$ as $n \rightarrow \infty$ for all $r > 0$ with $\mu(\partial B_{\mathscr{X}}(\rho, r)) = 0$.
\item[$(b)$] $\mu_n^{\restr r} \stackrel{\rm{wk}}{\rightarrow} \mu^{\restr r}$ as $n \rightarrow \infty$ for Lebesgue-almost all $r > 0$.
\item[$(c)$] $\mu_n^{\restr r_k} \stackrel{\rm{wk}}{\rightarrow} \mu^{\restr r_k}$ as $n \rightarrow \infty$ for all $k \in \mathbb{N}$, and some increasing sequence of radii $r_k \rightarrow \infty$.
\item[$(d)$] $\mu_n \stackrel{\rm{vg}}{\rightarrow} \mu$ as $n \rightarrow \infty$.
\end{itemize}
\end{prop}
\begin{proof}
The implications $(a) \Rightarrow (b) \Rightarrow (c) \Rightarrow (d)$ are trivial. For the implication $(d) \Rightarrow (a)$, fix some $r > 0$ with $\mu(\partial B(\rho, r)) = 0$, and choose some $r_k \geq r$. Then $\mu_n^{\restr r_k} \stackrel{\rm{wk}}{\rightarrow} \mu^{\restr r_k}$, and so the Portmanteau theorem implies that $\mu_n^{\restr r} \stackrel{\rm{wk}}{\rightarrow} \mu^{\restr r}$ as $n \rightarrow \infty$, as wanted.
\end{proof}
As a consequence of the Portmanteau theorem for finite measures, we see that vague convergence is stable under restriction to \emph{$\mu$-continuity sets}, i.e. Borel sets $\Theta \in \mathscr{B}(\mathscr{X})$ with $\mu(\partial \Theta) = 0$. 
\begin{lemma}
\label{lemma:BFMeas2}
Let $(\mathscr{X}, d, \rho)$ be a pointed Polish space, and let $\mu, \mu_n \in \mathscr{M}_{\rm{bf}}(\mathscr{X})$, $n \in \mathbb{N}$. Suppose that $\mu_n \stackrel{\rm{vg}}{\rightarrow} \mu$ as $n \rightarrow \infty$, and let $\Theta$ be a $\mu$-continuity set in $\mathscr{X}$. Then
$$ \mu_n|_{\Theta} \stackrel{\rm{vg}}{\rightarrow} \mu|_{\Theta} \quad \text{as } n \rightarrow \infty \quad \text{in } \mathscr{M}_{\rm{bf}}(\mathscr{X}). $$
\end{lemma}
\begin{proof}
Take some $\mu$-continuity set $A$, and some $r > 0$ with $\mu(\partial B(\rho, r)) =0$. Then by vague convergence, and the Portmanteau theorem for finite measures we have that $ \mu_n|_{\Theta}^{\restr r}(A) = \mu_n^{\restr r}(\Theta \cap A) \rightarrow \mu^{\restr r}(\Theta \cap A ) = \mu|_{\Theta}^{\restr r}(A)$ as $n \rightarrow \infty$. As this holds for all $\mu$-continuity sets $A$, we conclude again by the Portmanteau theorem for finite measures that $\mu_n|_{\Theta}^{\restr r} \stackrel{\rm{wk}}{\rightarrow} \mu|_{\Theta}^{\restr r}$ as $n \rightarrow \infty$, as wanted.
\end{proof}
The following estimates allow us to compare the vague distance with the Prokhorov distance on fixed balls.
\begin{lemma}
\label{lemma:BFMeas3}
Let $(\mathscr{X}, d, \rho)$ be a pointed Polish space, and let $\mu, \nu \in \mathscr{M}_{\rm{bf}}(\mathscr{X})$. Let $R > 0$ and $\varepsilon \in (0, e^{-R})$. Then
\begin{enumerate}[$(i)$]
\item if $d_{\rm{P}}(\mu^{\restr R}, \nu^{\restr R}) \leq \varepsilon$, then $d_V(\mu, \nu) \leq \varepsilon\left( 1 + \mu(\overline{B}_T(\rho, R)) + \nu(\overline{B}_T(\rho, R)) \right) + e^{-R} $.
\item if $d_{\rm{V}}(\mu, \nu) \leq \varepsilon$, then there exists some $R' \geq R$ such that $d_{\rm{P}}\left(\mu^{\restr R'}, \nu^{\restr R'} \right) \leq \varepsilon (1 + e^R)$.
\end{enumerate}
\end{lemma}
\begin{proof} 
If $d_{\rm{P}}(\mu^{\restr R}, \nu^{\restr R}) \leq \varepsilon$, then $d_{\rm{V}}(\mu, \nu) \leq \int_0^R e^{-r}d_P(\mu^{\restr r}, \nu^{\restr r}) dr + \int_R^{\infty} e^{-r} dr$. The second term is bounded by $e^{-R}$. For the first term, set $A_{r, \varepsilon} := \overline{B}_T(\rho, r+\varepsilon) \setminus \overline{B}_T(\rho, r)$, and observe that then $d_P(\mu^{\restr r}, \nu^{\restr r}) \leq \varepsilon + \max\left\{ \mu^{\restr R}(A_{r, \varepsilon}), \nu^{\restr R}(A_{r, \varepsilon}) \right\}$. Now, Fubini's theorem yields
\begin{align*}
\int_0^R \mu^{\restr R}(A_{r, \varepsilon}) dr &= \int_0^R \int_{\overline{B}_T(\rho, R)} \mathbbm{1}_{( r < d(\rho, x) \leq r+\varepsilon)} d\mu(x) dr\\ 
&= \int_{\overline{B}_T(\rho, R)} \int_0^R \mathbbm{1}_{( r < d(\rho, x) \leq r+\varepsilon)} dr d\mu(x)\\ 
&\leq \varepsilon \mu(\overline{B}_T(\rho, R)),
\end{align*}
and so we may conclude that $d_V(\mu, \nu) \leq \varepsilon\left( 1 + \mu(\overline{B}_T(\rho, R)) + \nu(\overline{B}_T(\rho, R)) \right) + e^{-R} $.

If instead $d_{\rm{V}}(\mu, \nu) \leq \varepsilon$, then $\varepsilon \geq \int_R^{\infty} e^{-r}(1 \wedge d_P(\mu^{\restr r}, \nu^{\restr r})) dr \geq e^{-R} \inf_{R' \geq R} d_{\rm{P}}(\mu^{\restr R'}, \nu^{\restr R'})$.
We may find some $R' \geq R$ such that $d_{\rm{P}}(\mu^{\restr R'}, \nu^{\restr R'}) \leq \inf_{R' \geq R} d_{\rm{P}}(\mu^{\restr R'}, \nu^{\restr R'}) + \varepsilon$, and so we may conclude that $d_{\rm{P}}(\mu^{\restr R'}, \nu^{\restr R'}) \leq \varepsilon (1 + e^R)$.
\end{proof}
Completely analogously to \cite{ALW16} Corollary 4.3, we can characterize the relatively compact sets in $(\mathscr{M}_{\rm{bf}}(\mathscr{X}), d_{\rm{V}})$ in terms of their \emph{restrictions} being relatively compact in $(\mathscr{M}_{\rm{fin}}(\mathscr{X}), d_{\rm{P}})$.
\begin{lemma}
\label{lemma:BFMeas4}
Let $(\mathscr{X}, d, \rho)$ be a pointed Polish space, and let $\Pi \subseteq \mathscr{M}_{\rm{bf}}(\mathscr{X})$. Then the following statements are equivalent:
\begin{itemize}
\item[$(a)$] $\Pi$ is relatively compact in $(\mathscr{M}_{\rm{bf}}(\mathscr{X}), d_{\rm{V}})$.
\item[$(b)$] For all $r > 0$, $\Pi^{\restr r} := \{\mu^{\restr r} ~|~ \mu \in \Pi \}$ is relatively compact in $(\mathscr{M}_{\rm{fin}}(\mathscr{X}), d_{\rm{P}})$.
\item[$(c)$] $\Pi^{\restr r_k}$ is relatively compact in $(\mathscr{M}_{\rm{fin}}(\mathscr{X}), d_{\rm{P}})$ for an increasing sequence of radii $r_k \rightarrow \infty$.
\end{itemize}
\end{lemma}
\begin{proof}
The proof is completely analogous to that of \cite{ALW16} Corollary 4.3, and even simpler as our measures are already defined on the same pointed Polish space $(\mathscr{X}, d, \rho)$.
\end{proof}
As in \cite{ALW16} Proposition 4.8, we may now conclude that $(\mathscr{M}_{\rm{bf}}(\mathscr{X}), d_{\rm{V}})$ forms a Polish space.
\begin{prop}
\label{prop:BFMeas5}
Let $(\mathscr{X}, d, \rho)$ be a pointed Polish space. Then $(\mathscr{M}_{\rm{bf}}(\mathscr{X}), d_{\rm{V}})$ is a Polish space.
\end{prop}
\begin{proof}
Completeness is shown in the same way as in \cite{ALW16} Proposition 4.8, using the characterization of relatively compact sets from Lemma \ref{lemma:BFMeas4}. For separability, let $r_k \rightarrow \infty$ be some increasing sequence of radii. Then as $(\mathscr{M}_{\rm{fin}}(\mathscr{X}), d_{\rm{P}})$ is separable, we can for each $k \in \mathbb{N}$ find a $d_{\rm{P}}$-dense countable set $D_k \subseteq \mathscr{M}_{\rm{fin}}(\overline{B}_{\mathscr{X}}(\rho, r_k))$, and let $D = \bigcup_{k \in \mathbb{N}} D_k$. Then $D$ is clearly countable, and what remains is to check that it is dense in $(\mathscr{M}_{\rm{bf}}(\mathscr{X}), d_{\rm{V}})$. Take some $\mu \in \mathscr{M}_{\rm{bf}}(\mathscr{X})$ and $\varepsilon > 0$. We want to find some $\nu \in D$ such that $d_{\rm{V}}(\mu, \nu) < \varepsilon$. To do this, choose $k$ large enough that $e^{-r_k} < \frac{\varepsilon}{2}$ and let $\delta \in (0, e^{-r_k})$ be small enough that $\delta \left( 1 + \mu\left( \overline{B}_T(\rho, r_k) \right) + \nu\left( \overline{B}_T(\rho, r_k) \right) \right) < \frac{\varepsilon}{2}$. Then as $D_k$ is dense in $ \mathscr{M}_{\rm{fin}}(\overline{B}_{\mathscr{X}}(\rho, r_k))$ we can find some $\nu \in D_k \subseteq D$ (with $\nu = \nu^{\restr r_k}$) such that $d_{\rm{P}}(\mu^{\restr r_k}, \nu) \leq \delta$. By Lemma \ref{lemma:BFMeas3} we then conclude that $d_{\rm{V}}(\mu, \nu) < \varepsilon$, as wanted.
\end{proof}
As noted in Section \ref{subsec:MainVague}, it is immediately clear from Proposition \ref{prop:BFMeas1} that weak convergence implies vague convergence. However, to strengthen vague convergence to weak convergence, one needs an additional assumption of b-tightness (recall Definition \ref{def:btight}). This was the statement of Lemma \ref{lemma:BFMeas7}, which we prove below.
\begin{proof}[Proof of Lemma \ref{lemma:BFMeas7}]
The direction $(\Leftarrow)$ is obvious. For $(\Rightarrow)$, suppose that $\mu_n \stackrel{\rm{vg}}{\rightarrow} \mu$ for some $\mu \in \mathscr{M}_{\rm{bf}}(\mathscr{X})$ and that $(\mu_n)_{n \in \mathbb{N}}$ is b-tight. Suppose for contradiction that for all $N \in \mathbb{N}$ there exists $n \geq N$ such that $\mu_n(\mathscr{X}) = +\infty$. Then as each $\mu_n$ is boundedly finite, we must have for all $R > 0$ that $\mu_n(\mathscr{X} \setminus \overline{B}(\rho, R)) = +\infty$, and so $\limsup_{n \rightarrow \infty} \mu_n(\mathscr{X} \setminus \overline{B}(\rho, R)) = +\infty$, contradicting b-tightness. So there must exist some $N \in \mathbb{N}$ such that $\mu_n \in \mathscr{M}_{\rm{fin}}(\mathscr{X})$ for all $n \geq N$.

To see that also $\mu \in \mathscr{M}_{\rm{fin}}(\mathscr{X})$, take some $R > 0$ with $\mu(\partial B(\rho, R)) =0$. By vague convergence we have for all $R' \geq R$ with $\mu(\partial B(\rho, R')) =0$, that $\mu_n^{\restr R'} \stackrel{\rm{wk}}{\rightarrow} \mu^{\restr R'}$, and so in particular Prokhorov's theorem yields that $\mu_n(\overline{B}(\rho, R')) \rightarrow \mu(\overline{B}(\rho, R'))$ as $n \rightarrow \infty$. Now, observe that $ \mu_n(\overline{B}(\rho, R')) \leq \mu_n(\overline{B}(\rho, R)) + \mu_n(\mathscr{X} \setminus \overline{B}(\rho, R))$. Taking the $\limsup$ on both sides as $n \rightarrow \infty$ gives the inequality $ \mu(\overline{B}(\rho, R')) \leq \mu(\overline{B}(\rho, R)) + \limsup_{n \rightarrow \infty} \mu_n(\mathscr{X} \setminus \overline{B}(\rho, R))$, and so letting $R' \rightarrow \infty$ (noting that only the left-hand side depends on $R'$) yields
\begin{equation}
\label{eq:BFMeas7}
\mu(\mathscr{X}) \leq \mu(\overline{B}(\rho, R)) + \limsup_{n \rightarrow \infty} \mu_n(\mathscr{X} \setminus \overline{B}(\rho, R)) < \infty,
\end{equation}
using again b-tightness. Hence, $\mu \in \mathscr{M}_{\rm{fin}}(\mathscr{X})$, as wanted.

To get weak convergence, take some $\varepsilon > 0$ and choose $R > 0$ such that $\limsup_{n \rightarrow \infty} \mu_n(\mathscr{X} \setminus \overline{B}_{\mathscr{X}}(\rho, R)) < \varepsilon$. Then \eqref{eq:BFMeas7} gives that also $\mu(\mathscr{X} \setminus \overline{B}(\rho, R)) < \varepsilon$. As $\mu_n^{\restr R} \stackrel{\rm{wk}}{\rightarrow} \mu^{\restr R}$ there exists some $N \in \mathbb{N}$ such that for all $n \geq N$: $d_{\rm{P}}(\mu_n^{\restr R}, \mu^{\restr R}) < \varepsilon$, so for any $A \subseteq \mathscr{X}$ closed, and $n \geq N$ we have
\begin{align*}
\mu_n(A) &\leq \mu_n^{\restr R}(A) + \mu_n(\mathscr{X} \setminus \overline{B}(\rho, R)) \leq \mu^{\restr R}(A^{\varepsilon}) + \varepsilon + \mu_n(\mathscr{X} \setminus \overline{B}(\rho, R))  \qquad \text{and} \\
\mu(A) &\leq \mu^{\restr R}(A) + \varepsilon \leq \mu_n^{(R)}(A^{\varepsilon}) + 2\varepsilon \leq \mu_n(A^{2\varepsilon}) + 2\varepsilon.
\end{align*}
Hence, $d_P(\mu_n, \mu) \leq 2 \varepsilon + \mu_n(\mathscr{X} \setminus \overline{B}(\rho, R))$. Taking the $\limsup$ as $n \rightarrow \infty$ on both sides, thus yields $ \limsup_{n \rightarrow \infty} d_P(\mu_n, \mu) \leq 2 \varepsilon + \limsup_{n \rightarrow \infty} \mu_n(\mathscr{X} \setminus \overline{B}(\rho, R)) < 3\varepsilon$. As $\varepsilon > 0$ was chosen arbitrarily small, this allows us to conclude that $\mu_n \stackrel{\rm{wk}}{\rightarrow} \mu$, as desired.
\end{proof}
Let $(T, d, \rho)$ be a Polish $\mathbb{R}$-tree, and recall from Section \ref{subsec:MainVague} the space of \emph{discretely infinite} (or finite) measures $\mathscr{M}_{\rm{di}}(T) \subseteq \mathscr{M}_{\rm{bf}}(T)$. We shall see in the following proposition that $\mathscr{M}_{\rm{di}}(T)$ is a measurable subset of $\mathscr{M}_{\rm{bf}}(T)$ under the vague topology, but that it is not in general closed. This implies in particular that $(\mathscr{M}_{\rm{di}}(T), d_{\rm{V}})$ is a Lusin space, but will in general \emph{not} be complete.
\begin{prop}
\label{prop:DIMeas1}
Let $(T, d, \rho)$ be a Polish $\mathbb{R}$-tree. Then $(\mathscr{M}_{\rm{di}}(T), d_{\rm{V}})$ is a Lusin space, but \emph{not} necessarily complete.
\end{prop}
\begin{proof}
As $(\mathscr{M}_{\rm{bf}}(T), d_{\rm{V}})$ is Polish, it is enough to argue (by \cite{C13} Theorem 8.2.10) that $\mathscr{M}_{\rm{di}}(T) \subseteq \mathscr{M}_{\rm{bf}}(T)$ is a Borel-measurable subset. Note, that for some increasing sequence of (strictly positive) radii $r_k \rightarrow \infty$ with $\partial B_T(\rho, r_k) \cap {\rm{Bp}}(T) = \emptyset$, we can write $\mathscr{M}_{\rm{di}}(T)$ as 
\begin{align*}
\bigcap_{k \in \mathbb{N}}\underbrace{\left\{ \mu \in \mathscr{M}_{\rm{bf}}(T) ~\Bigg{|}~ |{\rm{Skel}}(T, \mu) \cap \partial B_T(\rho, r_k)| < \infty, ~ \mu\left( \bigcup_{\sigma \in \partial B_T(\rho, r_k) \cap F_{\mu}(T)} \theta_{\sigma}T^{\circ}\right) < \infty \right\}}_{=:A^{(k)}},
\end{align*}
and we can further decompose each $A^{(k)}$ into
$$ \bigcup_{m \in \mathbb{N}_0} \bigcup_{K \in \mathbb{N}} \underbrace{\left\{ \mu \in \mathscr{M}_{\rm{bf}}(T) ~\Bigg{|}~ |{\rm{Skel}}(T, \mu) \cap \partial B_T(\rho, r_k)| = m, ~ \mu\left( \bigcup_{\sigma \in \partial B_T(\rho, r_k) \cap F_{\mu}(T)} \theta_{\sigma}T^{\circ}\right) \leq K \right\}}_{=:A_{m,K}^{(k)}}. $$
Observe for each $(k, m, K)$ that $A_{m,K}^{(k)}$ may not be closed. This may for instance be seen when $T = [0, \infty)$ is equipped with the Euclidean distance and has its root at $0$, $m = 1$, and $K$ is arbitrary. In this case, letting $\mu_n$ be the Lebesgue measure on $[n, \infty)$, we have that $\mu_n \stackrel{\rm{vg}}{\rightarrow} \mathbf{0}$, but clearly $\mathbf{0} \notin A_{1,K}^{(k)}$ for any $k$ as ${\rm{Skel}}(T, \mathbf{0}) = \{\rho\}$. We may however easily check for the closures that $\overline{A_{m,K}^{(k)}} \subseteq \mathscr{M}_{\rm{di}}(T)$ for any $(k, m, K)$ (the main idea here being that infinite-mass subtrees may have the infinite-mass pushed so far out in the limit that only finite mass survives, corresponding to $m$ becoming smaller in the limit at the cost of a larger $K$). As we also trivially have that $A_{m,K}^{(k)} \subseteq \overline{A_{m,K}^{(k)}}$ we see from our decomposition of $\mathscr{M}_{\rm{di}}(T)$ that we may indeed write $ \mathscr{M}_{\rm{di}}(T) = \bigcap_{k \in \mathbb{N}} \bigcup_{m \in \mathbb{N}_0} \bigcup_{K \in \mathbb{N}} \overline{A_{m,K}^{(k)}}$. As $\mathscr{M}_{\rm{di}}(T)$ is thus decomposed into countable intersections and unions of closed sets, it must be a Borel-measurable subset of $(\mathscr{M}_{\rm{bf}}(T), d_{\rm{V}})$, as desired.

To see that $\mathscr{M}_{\rm{di}}(T)$ is \emph{not} closed in $(\mathscr{M}_{\rm{bf}}(T), d_{\rm{V}})$ (and thus, \emph{not} a complete space), consider the following example. Suppose that $T$ is the \emph{infinite star}, consisting of countably many rays which only intersect at the root. Order the rays in $T$ by $\mathbf{r}_1, \mathbf{r}_2, \dots$. Let for each $n \in \mathbb{N}$, $\mu_n$ be the measure on $T$ constructed such that
\begin{itemize}
\item for $1 \leq i \leq n$, $\mu_n$ acts as the Lebesgue measure $\mathbf{m}_{[i,\infty)}$ started at level $i$ on the ray $\mathbf{r}_i$;
\item for $i \geq n+1$, $\mu_n$ acts as the null measure on $\mathbf{r}_i$. 
\end{itemize}
Then $\mu_n \in \mathscr{M}_{\rm{di}}(T)$ for every $n \in \mathbb{N}$. Let on the other hand $\mu$ be the boundedly finite measure on $T$ defined such that for every $i \in \mathbb{N}$, $\mu$ acts as the Lebesgue measure $\mathbf{m}_{[i,\infty)}$ started at level $i$ on the ray $\mathbf{r}_i$. Then ${\rm{Skel}}(T, \mu) = T$ is clearly not a discrete $\mathbb{R}$-tree, so $\mu \notin \mathscr{M}_{\rm{di}}(T)$. On the other hand it is clear that $\mu_n^{\restr r} = \mu^{\restr r}$ whenever $n \geq \lfloor r \rfloor$, so in particular $\mu_n \stackrel{\rm{vg}}{\rightarrow} \mu$ as $n \rightarrow \infty$.
\end{proof}
\subsection{Discrete $\mathbb{R}$-trees and spanning sets}
\label{subsec:Rtree}
Let $(T, d, \rho)$ be an $\mathbb{R}$-tree, and recall the \emph{span} from Definition \ref{def:Span}. Given a rooted subtree $T' \subseteq T$, it is natural to ask what the smallest possible spanning set of the tree is. In particular, we may ask whether it is possible to find a \emph{finite} spanning set of the tree.
\begin{defin}[Subtrees of finite type]
\label{def:FinType}
A rooted subtree $T' \subseteq T$ is said to be of \emph{finite type} if there exists a \emph{finite} subset $A \subseteq T$, such that $T' = {\rm{Span}}(A)$.
\end{defin}
It is easily seen by Definition \ref{def:Span} that ${\rm{Span}}(A)$ is compact whenever $A$ is finite, so every subtree of finite type is indeed a compact subtree. Moreover, subtrees of finite type will have at most finitely many leaves, at most finitely many branch points and be of bounded height.

The notion of being a subtree of finite type, can easily be converted into a local property for unbounded subtrees.
\begin{defin}[Subtrees of boundedly finite type]
\label{def:BFType}
A rooted subtree $T' \subseteq T$ is said to be of \emph{boundedly finite type} if $T' \cap \overline{B}_T(\rho, r)$ is a rooted subtree of finite type for every $r > 0$.
\end{defin}
The notions in Definition \ref{def:FinType} and \ref{def:BFType} are closely linked to \emph{discrete $\mathbb{R}$-trees} (see Section \ref{subsec:MainVague}).
\begin{lemma}
\label{lemma:Rtree1}
Let $(T, d, \rho)$ be an $\mathbb{R}$-tree. Then
\begin{enumerate}[$(i)$]
\item $T$ is a discrete $\mathbb{R}$-tree if and only if it is of boundedly finite type.
\item $T$ is a compact discrete $\mathbb{R}$-tree if and only if it is of finite type.
\end{enumerate}
\end{lemma}
\begin{proof}
Note for every $r > 0$, that $\overline{B}_T(\rho, r)$ is of finite type (i.e. is spanned by at most finitely many leaves) if and only if the total degree of the root and branch points in $\overline{B}_T(\rho, r)$ is finite. So the only non-trivial argument we need to make is that $T$ is complete if it is of boundedly finite type.
Suppose that $T$ is of boundedly finite type, and let $(\sigma_n)_{n \in \mathbb{N}}$ be a Cauchy sequence in $T$. Take some $\varepsilon > 0$. Then there exists some $N \in \mathbb{N}$ such that $d(\sigma_n, \sigma_m) < \varepsilon$ for all $n,m \geq N$. Let $r > 0$ be chosen sufficiently large such that $\sigma_1, \dots, \sigma_N \in \overline{B}_T(\rho, r)$. Then we must have that $\{\sigma_n ~|~ n \in \mathbb{N}\} \subseteq \overline{B}_T(\rho, r+\varepsilon)$. So as $\overline{B}_T(\rho, r+\varepsilon)$ is of finite type, we can conclude that $(\sigma_n)_{n \in \mathbb{N}}$ must be convergent in $\overline{B}_T(\rho, r+\varepsilon)$ (and thus especially in $T$), as wanted.
\end{proof}
\subsection{Projections on closed subtrees}
\label{subsec:Proj}
\begin{defin}[Projections of points on closed subtrees]
Let $(T, d, \rho)$ be an $\mathbb{R}$-tree, and let $T' \subseteq T$ be a \emph{closed} rooted subtree of $T$. Recall that for every $\sigma \in T$, there exists a unique point $\sigma' \in T'$ such that $[\rho, \sigma] \cap T' = [\rho, \sigma']$. We define the \emph{projection of $\sigma$ on $T'$} by $f_{T'}(\sigma) := \sigma'$. This induces a map $f_{T'}\colon T \rightarrow T'$, which we  call the \emph{point-projection map on $T'$}.
\end{defin}
Given a \emph{Polish} $\mathbb{R}$-tree $(T, d, \rho)$ and a closed rooted subtree $T' \subseteq T$, we denote by $(C_i)_{i \in I}$ the countable\footnote{This is a direct consequence of separability.} family of connected components of the open set $T \setminus T'$. We recall the following properties from Duquesne and Winkel \cite{DW26} Proposition 3.5$(i)$--$(iv)$ and Proposition 2.4$(iv)$.
\begin{prop}[Properties of point-projections]
\label{prop:Proj1}
Let $(T, d, \rho)$ be a Polish $\mathbb{R}$-tree, let $T' \subseteq T$ be a closed rooted subtree, and let $(C_i)_{i \in I}$ be given as above. Then
\begin{enumerate}[$(i)$]
\item For every $i \in I$ there exists a point $\sigma_i \in T'$, such that $C_i$ is a connected component of $T \setminus \{\sigma_i\}$ with closure $\overline{C_i} = C_i \cup \{\sigma_i\}$.

\item For every $i \in I$, and every sequence of points $(\sigma_{n,i})_{n \in \mathbb{N}}$ in $C_i$ which converges to the closure point $\sigma_i$, it holds that $C_i = \bigcup_{n \in \mathbb{N}} \theta_{\sigma_{n,i}}T$.

\item $f_{T'}(\sigma) = \sigma$ if $\sigma \in T'$, and $f_{T'}(\sigma) = \sigma_i$ if $\sigma \in C_i$ for some $i \in I$.
\item The point-projection map $f_{T'}\colon T \rightarrow T'$ is 1-Lipschitz.\footnote{Meaning that $ d(f_{T'}(\sigma), f_{T'}(\sigma')) \leq d(\sigma, \sigma')$ for all $\sigma, \sigma' \in T$.}
\item Let $T_1, T_2 \subseteq T$ be two closed rooted subtrees. Then $T_1 \cap T_2$ is a closed rooted subtree of $T$, and $f_{T_1} \circ f_{T_2} = f_{T_1 \cap T_2}$.
\end{enumerate}
\end{prop}
Let $(T, d, \rho)$ be a Polish $\mathbb{R}$-tree, and let $T' \subseteq T$ be a closed rooted subtree. Then we may similarly consider projections of \emph{measures} onto the subtree $T'$.
\begin{defin}[Projections of measures on closed subtrees]
Let $(T, d, \rho)$ be an $\mathbb{R}$-tree, and let $T' \subseteq T$ be a closed rooted subtree of $T$. Let $\mu \in \mathscr{M}(T)$ be a fixed Borel measure. Then we define the \emph{projection of $\mu$ on $T'$} as the pushforward under $f_{T'}$, i.e. ${\rm{Pr}}_{T'}\mu := \mu \circ f_{T'}^{-1}$. This yields a map ${\rm{Pr}}_{T'}\colon \mathscr{M}(T) \rightarrow \mathscr{M}(T')$, which we call the \emph{measure-projection map on $T'$}.
\end{defin}
It is easily seen that on the space of finite Borel measures $\mathscr{M}_{\rm{fin}}(T)$, the projection onto a closed rooted subtree $T' \subseteq T$ yields a map ${\rm{Pr}}_{T'}\colon \mathscr{M}_{\rm{fin}}(T) \rightarrow \mathscr{M}_{\rm{fin}}(T')$.
\begin{prop}[Properties of measure-projections]
\label{prop:Proj2}
Let $(T, d, \rho)$ be a Polish $\mathbb{R}$-tree, and let $T' \subseteq T$ be a closed rooted subtree. Let $(C_i)_{i \in I}$ be given as before, and consider some fixed positive measure $\mu \in \mathscr{M}(T)$. Then the following statements are true:
\begin{enumerate}[$(i)$]
\item The projection of $\mu$ on $T'$ is explicitly given by ${\rm{Pr}}_{T'} \mu = \mu(\cdot \cap T') + \sum_{i \in I} \mu(C_i) \delta_{\sigma_i}$.

\item ${\rm{Pr}}_{T'}\mu(\theta_{\sigma}T) = \mu(\theta_{\sigma}T)$ if $\sigma \in T'$, and ${\rm{Pr}}_{T'}\mu(\theta_{\sigma}T) = 0$ otherwise.

\item Let $T'' \subseteq T' \subseteq T$ be closed rooted subtrees. Then ${\rm{Pr}}_{T''}\mu = {\rm{Pr}}_{T''}{\rm{Pr}}_{T'}\mu$.
\end{enumerate}
\end{prop}
The proof of Proposition \ref{prop:Proj2} is the same as in \cite{DW26} Proposition 3.5$(v)$--$(vi)$ and Lemma 3.7, but replacing any instance of $\mathscr{M}_{\rm{fin}}(T)$ with $\mathscr{M}(T)$. A direct application of Proposition \ref{prop:Proj2}$(ii)$ allows us to easily calculate the projected measure on a connected component outside some arbitrary closed rooted subtree.
\begin{lemma}
\label{lemma:Proj3}
Let $(T, d, \rho)$ be a Polish $\mathbb{R}$-tree, and let $T' \subseteq T$ and $T'' \subseteq T$ be two closed rooted subtrees. Let $C$ be an open connected component of $T \setminus T''$, and consider some fixed positive measure $\mu \in \mathscr{M}(T)$. Then ${\rm{Pr}}_{T'}\mu(C) = \mathbbm{1}_{(C \cap T' \neq \emptyset)} \mu(C)$.
\end{lemma}
\begin{proof}
As ${\rm{Pr}}_{T'}\mu$ has support on (a subset of) $T'$, we have ${\rm{Pr}}_{T'}\mu(C) = 0$ whenever $C \cap T' = \emptyset$. Suppose that there exists some $\sigma \in C \cap T'$. Then $[\rho, \sigma] \subseteq T'$. By Proposition \ref{prop:Proj1}$(i)$ we can find some point $\tilde{\sigma} \in T''$ such that $\overline{C} = C \cup \{\tilde{\sigma}\}$. In particular, we must have that $\tilde{\sigma} \in [\rho, \sigma)$. We can now choose some sequence of points $(\sigma_n)_{n \in \mathbb{N}}$ in $(\tilde{\sigma}, \sigma] \subseteq C$ with $\sigma_{n+1} \in [\rho, \sigma_n]$ and $d(\sigma_n, \tilde{\sigma}) \rightarrow 0$ as $n \rightarrow \infty$. Then $\theta_{\sigma_1}T \subseteq \theta_{\sigma_2}T \subseteq \dots$ with $\bigcup_{n \in \mathbb{N}} \theta_{\sigma_n}T = C$, and we conclude by Proposition \ref{prop:Proj2}$(ii)$, that $ {\rm{Pr}}_{T'}\mu(C) = \lim_{n \rightarrow \infty} {\rm{Pr}}_{T'}\mu(\theta_{\sigma_n}T) = \lim_{n \rightarrow \infty} \mu(\theta_{\sigma_n}T) = \mu(C)$.
\end{proof}
We now examine under which conditions the projection of a discretely infinite measure onto a closed rooted subtree is discretely infinite.
\begin{prop}
\label{prop:Proj4}
Let $(T, d, \rho)$ be a Polish $\mathbb{R}$-tree, let $\mu \in \mathscr{M}_{\rm{di}}(T)$, and let $T' \subseteq T$ be a closed rooted subtree. Then the following are equivalent:
\begin{itemize}
\item[$(a)$] ${\rm{Pr}}_{T'}\mu \in \mathscr{M}_{\rm{di}}(T)$.
\item[$(b)$] ${\rm{Pr}}_{T'}\mu \in \mathscr{M}_{\rm{bf}}(T)$.
\item[$(c)$] $\mu(\theta_{\sigma}T) < \infty$ for all $\sigma \in T \setminus T'$.
\end{itemize}
In the affirmative case, we in particular have that ${\rm{Skel}}(T, \mu) = {\rm{Skel}}(T, {\rm{Pr}}_{T'}\mu)$.
\end{prop}
\begin{proof}
Start by observing that
\begin{equation}
\label{eq:Proj4}
{\rm{Pr}}_{T'}\mu \in \mathscr{M}_{\rm{bf}}(T) \quad \text{if and only if} \quad \mu\left( f_{T'}^{-1}(\overline{B}(\rho, r)) \right) < \infty \text{ for all } r \geq 0.
\end{equation}
The implication $(a) \Rightarrow (b)$ is trivial. To see that $(b) \Rightarrow (c)$, suppose for contradiction that there exists some $\sigma \in T \setminus T'$ with $\mu(\theta_{\sigma}T) = +\infty$, and set $\tilde{\sigma} := f_{T'}(\sigma)$. As $T'$ is a closed rooted subtree of $T$, we have that $\theta_{\sigma}T \subseteq T \setminus T'$, and as all elements of $\theta_{\sigma}T$ belong to the same open connected component of $T \setminus T'$, we must have that $f_{T'}(\sigma') = \tilde{\sigma}$ for all $\sigma' \in \theta_{\sigma}T$. Hence, $\theta_{\sigma}T \subseteq f_{T'}^{-1}(\{\tilde{\sigma}\})$. Now, choose $r > 0$ sufficiently large such that $\tilde{\sigma} \in \overline{B}_T(\rho, r)$. Then $\mu(f_{T'}^{-1}(\overline{B}_T(\rho, r))) \geq \mu(f_{T'}^{-1}(\{\tilde{\sigma}\})) \geq \mu(\theta_{\sigma}T) = +\infty$. By \eqref{eq:Proj4} we conclude that ${\rm{Pr}}_{T'}\mu \notin \mathscr{M}_{\rm{bf}}(T)$, giving the contrapositive of $(b) \Rightarrow (c)$

Next, suppose that $(c)$ holds and fix some $r \geq 0$. Denote by $(C_i)_{i \in I}$ the open connected components of $T \setminus T'$, and let $(\sigma_i)_{i \in I}$ denote their closure points in $T'$ from Proposition \ref{prop:Proj1}$(i)$. Let $I_r \subseteq I$ denote the subset of indices for which $\sigma_i \in \overline{B}_T(\rho, r)$. Then by Proposition \ref{prop:Proj1}$(iii)$ we see that
$$ \mu(f_{T'}^{-1}(\overline{B}_T(\rho, r))) = \mu\left( T' \cap \overline{B}_T(\rho, r) \right) + \mu\left( \bigcup_{i \in I_r} C_i \right).$$
As $\mu \in \mathscr{M}_{\rm{di}}(T)$ (and so in particular $\mu \in \mathscr{M}_{\rm{bf}}(T)$), we immediately see that the first term is finite. For the second term, take some $\varepsilon > 0$, and note by $(c)$ that $C_i \cap \partial B_T(\rho, r+\varepsilon) \subseteq F_{\mu}(T)$ for every $i \in I_r$. This implies in particular that $\bigcup_{i \in I_r} C_i \subseteq \overline{B}_T(\rho, r+\varepsilon) \cup \bigcup_{\sigma \in \partial B_T(\rho, r+\varepsilon) \cap F_{\mu}(T)} \theta_{\sigma}T$, and so we get that $\mu\left( \bigcup_{i \in I_r} C_i \right) \leq \mu(\overline{B}_T(\rho, r+\varepsilon)) + \mu\left( \bigcup_{\sigma \in \partial B_T(\rho, r+\varepsilon) \cap F_{\mu}(T)} \theta_{\sigma}T \right) < \infty$ since $\mu \in \mathscr{M}_{\rm{di}}(T)$ by assumption. By \eqref{eq:Proj4} we conclude that ${\rm{Pr}}_{T'}\mu \in \mathscr{M}_{\rm{bf}}(T)$, taking care of the implication $(c) \Rightarrow (b)$

Finally let us prove that $(b)/(c) \Rightarrow (a)$. We immediately see that $(c)$ ensures that ${\rm{Skel}}(T, \mu) \subseteq T'$, and so we conclude by Proposition \ref{prop:Proj2}$(ii)$ that ${\rm{Skel}}(T, {\rm{Pr}}_{T'}\mu) = T' \cap {\rm{Skel}}(T, \mu) = {\rm{Skel}}(T, \mu)$. As $\mu \in \mathscr{M}_{\rm{di}}(T)$, we thus have that ${\rm{Skel}}(T, {\rm{Pr}}_{T'}\mu)$ is a discrete $\mathbb{R}$-tree. Moreover, we have by Proposition \ref{prop:Proj2}$(ii)$ that ${\rm{Pr}}_{T'}\mu(\theta_{\sigma}T) \leq \mu(\theta_{\sigma}T)$ for all $\sigma \in T$. Fix some $r > 0$ and let $D_r = \partial B_T(\rho, r) \setminus {\rm{Lf}}(T)$ be the countable collection of boundary points with $\theta_{\sigma}T \setminus \{\sigma\} \neq \emptyset$. Then we conclude that
\begin{align*}
{\rm{Pr}}_{T'}\mu \left( \bigcup_{\sigma \in \partial B_T(\rho, r) \cap F_{\mu}(T)} \theta_{\sigma} T \right) &\leq {\rm{Pr}}_{T'}\mu\left( \partial B_T(\rho, r) \right) + \sum_{\sigma \in \partial B_T(\rho, r) \cap F_{\mu}(T) \cap D_r} \mu(\theta_{\sigma}T)
\end{align*}
which is finite as ${\rm{Pr}}_{T'}\mu \in \mathscr{M}_{\rm{bf}}(T)$ and $\mu \in \mathscr{M}_{\rm{di}}(T)$.
\end{proof}
\begin{example}
\label{ex:Proj5}
In Proposition \ref{prop:Proj4}, the assumption that $\mu \in \mathscr{M}_{\rm{di}}(T)$ is absolutely essential for the argument to work. To illustrate why it will not suffice to assume that $\mu \in \mathscr{M}_{\rm{bf}}(T)$, let $(T, d, \rho)$ be the infinite star from Example \ref{ex:BFMeas6} with rays ordered by $\mathbf{r}_1, \mathbf{r}_2, \dots$, and let $\mu$ correspond to the Dirac point-measure $\delta_i$ on the ray $\mathbf{r}_i$ for every $i \in \mathbb{N}$. Then we clearly have that $\mu \in \mathscr{M}_{\rm{bf}}(T)$, but $\mu \notin \mathscr{M}_{\rm{di}}(T)$ as it violates the second condition of Definition \ref{def:DIMeas}. If we let $T' = \overline{B}_T(\rho, 1)$, we have that $\mu(\theta_{\sigma}T) < \infty$ for all $\sigma \in T \setminus T'$, but we immediately see that the countably many Dirac point-masses will all be projected onto the boundary $\partial B_T(\rho, 1)$. So in particular, ${\rm{Pr}}_{T'}\mu(\overline{B}_T(\rho, 1)) = +\infty$, implying that ${\rm{Pr}}_{T'}\mu \notin \mathscr{M}_{\rm{bf}}(T)$.
\demo
\end{example}
\subsection{Gromov-type topologies on isometry classes of $\mathbb{R}$-trees}
\label{subsec:Gromov}
Recall that two pointed metric spaces $(\mathscr{X}, d_{\mathscr{X}}, \rho_{\mathscr{X}})$ and $(\mathscr{Y}, d_{\mathscr{Y}}, \rho_{\mathscr{Y}})$ are said to be \emph{isometric} if there exists a bijective isometry $\varphi\colon \mathscr{X} \rightarrow \mathscr{Y}$ with $\varphi(\rho_{\mathscr{X}}) = \rho_{\mathscr{Y}}$. Two measured pointed metric spaces $(\mathscr{X}, d_{\mathscr{X}}, \rho_{\mathscr{X}}, \mu)$ and $(\mathscr{Y}, d_{\mathscr{Y}}, \rho_{\mathscr{Y}}, \nu)$ are said to be isometric if there exists a pointed bijective isometry which moreover satisfies that $\mu \circ \varphi^{-1} = \nu$. We will denote the isometry class of a pointed (measured) metric space by square brackets, i.e. $[\mathscr{X}, d, \rho]$ or $[\mathscr{X}, d, \rho, \mu]$. Returning to the setting of $\mathbb{R}$-trees, we denote by $\mathbb{T}_{\rm{c}}$ the space of isometry classes of compact $\mathbb{R}$-trees, and by $\mathbb{T}_{\rm{lc}}$ the space of isometry classes of complete locally compact $\mathbb{R}$-trees.

Given two compact $\mathbb{R}$-trees, $(T_1, d_1, \rho_1)$ and $(T_2, d_2, \rho_2)$, we define their (pointed) Gromov--Hausdorff distance by
\begin{align*}
d_{\rm{GH}}(T_1, T_2) := \inf\left\{ \delta_{\rm{H}}(\phi_1(T_1), \phi_2(T_2)) ~\Bigg{|} ~ \begin{array}{r}
(\mathscr{Z}, \delta, \varrho) \text{ is a pointed metric space}, \\
\phi_1 \colon T_1 \rightarrow \mathscr{Z} \text{ is a pointed isometry}, \\
\phi_2 \colon T_2 \rightarrow \mathscr{Z} \text{ is a pointed isometry},
\end{array} \right\}
\end{align*}
where we will usually take $\delta_H$ to denote the Hausdorff distance on an embedding space $(\mathscr{Z}, \delta)$. It is easily seen that $d_{\rm{GH}}(T_1, T_2) = 0$ if and only if $(T_1, d_1, \rho_1)$ and $(T_2, d_2, \rho_2)$ are isometric, and in particular it can be checked that $d_{\rm{GH}}$ is a well-defined metric on the space $\mathbb{T}_{\rm{c}}$. It is well-known from e.g., \cite{EPW06} Theorem 2 that $(\mathbb{T}_{\rm{c}}, d_{\rm{GH}})$ forms a Polish space.
\begin{rmk}
\label{rmk:Gromov}
Note that we have used a slightly different definition of the pointed Gromov--Hausdorff distance here, compared to that of \cite{EPW06} and most other literature. This difference lies in the fact that we require the isometries $\phi_1$ and $\phi_2$ to be \emph{pointed} in the sense that they must identify the roots, $\phi_1(\rho_1) = \varrho = \phi_2(\rho_2)$. It is an easy exercise to show that these two different definitions of the pointed Gromov--Hausdorff distance generate the same topology -- indeed, if we denote the distance from \cite{EPW06} Section 2.3 by
\begin{align*}
\widetilde{d}_{\rm{GH}}(T_1, T_2) := \inf\left\{ \widetilde{\delta}_{\rm{H}}(\widetilde{\phi}_1(T_1), \widetilde{\phi}_2(T_2)) \vee \widetilde{\delta}(\widetilde{\phi}_1(\rho_1), \widetilde{\phi}_2(\rho_2)) ~\Bigg{|} ~ \begin{array}{r}
(\widetilde{\mathscr{Z}}, \widetilde{\delta}) \text{ is a metric space}, \\
\widetilde{\phi}_1 \colon T_1 \rightarrow \mathscr{Z} \text{ is an isometry}, \\
\widetilde{\phi}_2 \colon T_2 \rightarrow \mathscr{Z} \text{ is an isometry}.
\end{array} \right\},
\end{align*}
then $\widetilde{d}_{\rm{GH}} \leq d_{\rm{GH}} \leq 2 \widetilde{d}_{\rm{GH}}$. The lower bound holds automatically as $\widetilde{d}_{\rm{GH}}$ allows for the possibility that $\widetilde{\phi}_1$ and $\widetilde{\phi}_2$ are pointed. The upper bound is shown by noting that the embedding space $\widetilde{\mathscr{Z}}$ in $\widetilde{d}_{\rm{GH}}$ can always be chosen as the disjoint union $T_1 \sqcup T_2$, and that $\widetilde{\delta}$ is in this case chosen as a metric which coincides with $d_1$ on $T_1$ and $d_2$ on $T_2$. For $d_{\rm{GH}}$, one then considers the quotient space $\mathscr{Z} = T_1 \sqcup T_2 / \{\rho_1, \rho_2\}$ which identifies the roots, and sets $\varrho := [\rho_1]$. The metric $\delta$ can then be defined by letting it coincide with $d_1$ and $d_2$ on $(T_1\setminus \{\rho_1\}) \cup \{\rho\}$ and $(T_2 \setminus \{\rho_2\}) \cup \{\rho\}$, respectively, and setting for $\sigma_1 \in T_1$ and $\sigma_2 \in T_2$: $ \delta(\sigma_1, \sigma_2) = \min\left\{ d_1(\sigma_1, \rho_1) + d_2(\rho_2, \sigma_2), \widetilde{\delta}(\sigma_1, \sigma_2) + \widetilde{\delta}(\rho_1, \rho_2) \right\}$ (noting by convention that we suppress the isometries in our notation, when choosing the embedding space to be the disjoint union).
\demo
\end{rmk}

The (pointed) Gromov--Hausdorff distance can easily be localized to obtain a way of measuring distances between complete locally compact $\mathbb{R}$-trees. Recall from the Hopf--Rinow theorem that any closed and bounded subset of a locally compact space is compact, so in particular $\overline{B}_T(\rho, r)$ is compact for every $r \geq 0$ if $(T, d, \rho)$ is locally compact. We can thus define the \emph{local} (pointed) Gromov--Hausdorff distance of two complete locally compact $\mathbb{R}$-trees $(T_1, d_1, \rho_1)$ and $(T_2, d_2, \rho_2)$ by
$$ d_{\rm{GH}}^{\rm{loc}}(T_1, T_2) := \sum_{k=1}^{\infty} \frac{1}{2^k} d_{\rm{GH}}(\overline{B}_{T_1}(\rho_1, k), \overline{B}_{T_2}(\rho_2, k)).$$
It is then easily checked that $d_{\rm{GH}}^{\rm{loc}}$ is a well-defined metric on $\mathbb{T}_{\rm{lc}}$, and it is well-known from e.g., \cite{DW07} Proposition 3.4, that $(\mathbb{T}_{\rm{lc}}, d_{\rm{GH}}^{\rm{loc}})$ forms a Polish space.

\bigskip
\noindent \emph{Minimal $\mathbb{R}$-trees and isometry classes:} Our next objective will be to introduce the Gromov--Prokhorov and Gromov-vague distances on suitable spaces of isometry classes of Polish measured $\mathbb{R}$-trees. Consider two Polish finitely measured $\mathbb{R}$-trees $(T_1, d_1, \rho_1, \mu_1)$ and $(T_2, d_2, \rho_2, \mu_2)$. We define their \emph{Gromov--Prokhorov distance} by
\begin{align*}
d_{\rm{GP}}(\mu_1, \mu_2) &:= \inf\left\{ \delta_{\rm{P}}( \mu_1 \circ \phi_1^{-1}, \mu_2 \circ \phi_2^{-1}) ~\Bigg{|}~ \begin{array}{r}
(\mathscr{Z}, d, \rho) \text{ is a pointed metric space}, \\
\phi_1 \colon T_1 \rightarrow \mathscr{Z} \text{ is a pointed isometry}, \\
\phi_2 \colon T_2 \rightarrow \mathscr{Z} \text{ is a pointed isometry}.
\end{array} \right\}.
\end{align*}
Recall the notion of a minimal $\mathbb{R}$-tree from Definition \ref{def:MinRTree}. 
As discussed by Duquesne and Winkel \cite{DW26}, if $d_{\rm{GP}}(\mu_1, \mu_2) = 0$, then $\left( \overline{\text{Span}(\text{supp}(\mu_1))}, d_1, \rho_1, \mu_1 \right)$ and $\left( \overline{\text{Span}(\text{supp}(\mu_2))}, d_2, \rho_2, \mu_2 \right)$ are isometric. If we denote by $\mathbb{T}_{\rm{min}}^{\rm{fin}}$ the space of isometry classes of \emph{minimal} Polish finitely measured  $\mathbb{R}$-trees, it can easily be checked that $d_{\rm{GP}}$ is a well-defined metric on $\mathbb{T}_{\rm{min}}^{\rm{fin}}$ and as seen in \cite{LVW15} Proposition 2.6 the space $(\mathbb{T}_{\rm{min}}^{\rm{fin}}, d_{\rm{GP}})$ is indeed Polish. Similarly as in Remark \ref{rmk:Gromov} we note the slight change from \cite{DW26} that we identify the roots of the trees in our embedding space (as usual, this does not impact the topology generated, and just means that we might have to multiply some estimates by two). We denote isometry classes in $\mathbb{T}_{\rm{min}}^{\rm{fin}}$ by $\boldsymbol{\mu} = [T, d, \rho, \mu]$. We say that a sequence $(\boldsymbol{\mu}_n)_{n \in \mathbb{N}}$ converges in the \emph{Gromov--Prokhorov sense} to $\boldsymbol{\mu}$, and write $\boldsymbol{\mu}_n \stackrel{\rm{GP}}{\rightarrow} \boldsymbol{\mu}$, if $d_{\rm{GP}}(\boldsymbol{\mu}_n, \boldsymbol{\mu}) \rightarrow 0$ as $n \rightarrow \infty$.

Denote now by $\mathbb{T}_{\rm{min}}^{\rm{bf}}$ the space of isometry classes of minimal Polish boundedly finitely measured $\mathbb{R}$-trees. We will equip this space with the Gromov-vague topology, as constructed by Athreya, Löhr and Winter \cite{ALW16}.\footnote{Note as in \cite{ALW16}, that these results also hold in a non-minimal setting where one considers more general metric spaces than just $\mathbb{R}$-trees. For our purposes we will only state results on the space $\mathbb{T}_{\rm{min}}^{\rm{bf}}$.} Consider an isometry class $\boldsymbol{\mu} = [T, d, \rho, \mu] \in \mathbb{T}_{\rm{min}}^{\rm{bf}}$, and fix some radius $r > 0$. We can define by $\boldsymbol{\mu}^{\restr r}$ the isometry class of the restriction to the closed ball of radius $r$ around the root, by setting
$$ \boldsymbol{\mu}^{\restr r} := \left[ \overline{{\rm{Span}}({\rm{supp}}(\mu^{\restr r}))}, d, \rho, \mu^{\restr r} \right] \in \mathbb{T}_{\rm{min}}^{\rm{fin}}. $$
Given two isometry classes $\boldsymbol{\mu}_1, \boldsymbol{\mu}_2 \in \mathbb{T}_{\rm{min}}^{\rm{bf}}$ we define their Gromov-vague distance by
$$ d_{\rm{GV}}(\boldsymbol{\mu}_1, \boldsymbol{\mu}_2) := \int_0^{\infty} e^{-r} \left( 1 \wedge d_{\rm{GP}}(\boldsymbol{\mu}_1^{\restr r}, \boldsymbol{\mu}_2^{\restr r} \right) dr.$$
If $d_{\rm{GV}}(\boldsymbol{\mu}_1, \boldsymbol{\mu}_2) = 0$, we must have for Lebesgue-almost all $r > 0$ that $d_{\rm{GP}}(\boldsymbol{\mu}_1^{\restr r}, \boldsymbol{\mu}_2^{\restr r}) = 0$, and so we know that $\boldsymbol{\mu}_1^{\restr r} = \boldsymbol{\mu}_2^{\restr r}$. Given representatives $(T_1, d_1, \rho_1, \mu_1)$ and $(T_2, d_2, \rho_2, \mu_2)$ for the minimal isometry classes, we thus have for almost all $r > 0$ that there must exist some bijective isometry $\phi^{(r)}\colon \overline{{\rm{Span}}({\rm{supp}}(\mu_1^{\restr r}))} \rightarrow \overline{{\rm{Span}}({\rm{supp}}(\mu_2^{\restr r}))}$ with $\phi^{(r)}(\rho_1) = \rho_2$ and $\mu_1^{\restr r} \circ (\phi^{(r)})^{-1} = \mu_2^{\restr r}$. If $r' \leq r$ we see that the restriction of $\phi^{(r)}$ to $\overline{{\rm{Span}}({\rm{supp}}(\mu_1^{\restr r'}))}$ will precisely be a bijective isometry onto $\overline{{\rm{Span}}({\rm{supp}}(\mu_2^{\restr r'}))}$ with $\mu_1^{\restr r'} \circ \left(\phi^{(r)}|_{\overline{{\rm{Span}}({\rm{supp}}(\mu_1^{\restr r'}))}}\right)^{-1} = \mu_2^{\restr r'}$, so we can without loss of generality assume that $\phi^{(r')} = \phi^{(r)}|_{\overline{{\rm{Span}}({\rm{supp}}(\mu_1^{\restr r'}))}}$. Noting that $T_1$ and $T_2$ are minimal, we can let $r \rightarrow \infty$ in order to uniquely extend to a bijective isometry $\phi\colon T_1 \rightarrow T_2$ with $\phi(\rho_1) = \rho_2$ and $\mu_1 \circ \phi^{-1} = \mu_2$, and so we conclude that $\boldsymbol{\mu}_1 = \boldsymbol{\mu}_2$, as wanted. It is clear that $d_{\rm{GV}}$ is symmetric, and it can also easily be checked that it satisfies the triangle inequality, hence $(\mathbb{T}_{\rm{min}}^{\rm{bf}}, d_{\rm{GV}})$ forms a well-defined metric space, and we say that a sequence $(\boldsymbol{\mu}_n)_{n \in \mathbb{N}}$ converges \emph{Gromov-vaguely} to $\boldsymbol{\mu}$, and write $\boldsymbol{\mu}_n \stackrel{\rm{GV}}{\rightarrow} \boldsymbol{\mu}$, if $d_{\rm{GV}}(\boldsymbol{\mu}_n, \boldsymbol{\mu}) \rightarrow 0$ as $n \rightarrow \infty$. Proposition \ref{prop:GV1} is analogous to \cite{ALW16} Lemma 2.6.
\begin{prop}
\label{prop:GV1}
Let $\boldsymbol{\mu}, \boldsymbol{\mu}_n \in \mathbb{T}_{\rm{min}}^{\rm{bf}}$, $n \in \mathbb{N}$. Then the following statements are equivalent:
\begin{itemize}
\item[$(a)$] $\boldsymbol{\mu}_n^{\restr r} \stackrel{\rm{GP}}{\rightarrow} \boldsymbol{\mu}^{\restr r}$ as $n \rightarrow \infty$ for all $r > 0$ with $\mu(\partial B_T(\rho, r)) = 0$.
\item[$(b)$] $\boldsymbol{\mu}_n^{\restr r} \stackrel{\rm{GP}}{\rightarrow} \boldsymbol{\mu}^{\restr r}$ as $n \rightarrow \infty$ for Lebesgue-almost all $r > 0$.
\item[$(c)$] $\boldsymbol{\mu}_n^{\restr r_k} \stackrel{\rm{GP}}{\rightarrow} \boldsymbol{\mu}^{\restr r_k}$ as $n \rightarrow \infty$ for all $k \in \mathbb{N}$, and some increasing sequence of radii $r_k \rightarrow \infty$.
\item[$(d)$] $\boldsymbol{\mu}_n \stackrel{\rm{GV}}{\rightarrow} \boldsymbol{\mu}$ as $n \rightarrow \infty$.
\end{itemize}
\end{prop}
Next, observe that Lemma \ref{lemma:BFMeas3} generalizes immediately to the Gromov setting.
\begin{lemma}
\label{lemma:GV2}
Let $\boldsymbol{\mu}$ and $\boldsymbol{\nu}$ be isometry classes in $\mathbb{T}_{\rm{min}}^{\rm{bf}}$, and let $R > 0$ and $\varepsilon \in (0, e^{-R})$. Then
\begin{enumerate}[$(i)$]
\item if $d_{\rm{GP}}(\boldsymbol{\mu}^{\restr R}, \boldsymbol{\nu}^{\restr R}) \leq \varepsilon$, then $d_{\rm{GV}}(\boldsymbol{\mu}, \boldsymbol{\nu}) \leq \varepsilon \left( 1 + \mu\left( \overline{B}_T(\rho, R) \right) + \nu\left( \overline{B}_T(\rho, R) \right) \right) + e^{-R}$.
\item if $d_{\rm{GV}}(\boldsymbol{\mu}, \boldsymbol{\nu}) \leq \varepsilon$, then there exists some $R' \geq R$ such that $d_{\rm{GP}}(\boldsymbol{\mu}^{\restr R'}, \boldsymbol{\nu}^{\restr R'}) \leq \varepsilon (1 + e^R)$.
\end{enumerate}
\end{lemma}
In the same way as the Gromov--Prokhorov distance was defined in terms of minimizing the (rooted) Prokhorov-distance in a common embedding space, Gromov-vague convergence can similarly be characterized in terms of isometric embeddings, as in \cite{ALW16} Proposition 4.1.
\begin{prop}[Characterization via isometric embeddings]
\label{prop:GV3}
Let $\boldsymbol{\mu}_n = [T_n, d_n, \rho_n, \mu_n]$, $n \in \mathbb{N}$, and $\boldsymbol{\mu} = [T, d, \rho, \mu]$ be isometry classes in $\mathbb{T}_{\rm{min}}^{\rm{bf}}$. Then $\boldsymbol{\mu}_n \stackrel{\rm{GV}}{\rightarrow} \boldsymbol{\mu}$ as $n \rightarrow \infty$ if and only if there exists a Polish pointed metric space $(\mathscr{Z}, \delta, \varrho)$ and \emph{pointed} isometries $\phi_n\colon T_n \rightarrow \mathscr{Z}$ and $\phi\colon T \rightarrow \mathscr{Z}$, such that $\left( \mu_n \circ \phi_n^{-1} \right)^{\restr r} \stackrel{\rm{wk}}{\rightarrow} \left( \mu \circ \phi^{-1} \right)^{\restr r}$ as $n \rightarrow \infty$ for all $r > 0$ with $\mu(\partial B_T(\rho, r)) = 0$.
\end{prop}
Proposition \ref{prop:GV3} especially ensures that for two isometry classes $\boldsymbol{\mu}_1 = [T_1, d_1, \rho_1, \mu_1]$ and $\boldsymbol{\mu}_2=[T_2, d_2, \rho_2, \mu_2]$, we can set
\begin{align*}
d_{\rm{GV}}^{\rm{emb}}(\boldsymbol{\mu}_1, \boldsymbol{\mu}_1) = \inf\left\{ \delta_{\rm{V}}( \mu_1 \circ \phi_1^{-1}, \mu_2 \circ \phi_2^{-1}) ~ \Bigg{|} ~ \begin{array}{l}
(\mathscr{Z}, \delta, \varrho) \text{ is a pointed metric space} \\
\phi_1\colon T_1 \rightarrow \mathscr{Z} \text{ is a pointed isometry} \\
\phi_2\colon T_2 \rightarrow \mathscr{Z} \text{ is a pointed isometry}
\end{array}  \right\},
\end{align*}
such that $d_{\rm{GV}}^{\rm{emb}}$ is an alternative metrization generating the Gromov-vague topology on $\mathbb{T}_{\rm{min}}^{\rm{bf}}$.

The characterization of relatively compact sets in $(\mathbb{T}_{\rm{min}}^{\rm{bf}}, d_{\rm{GV}})$ and the fact that this forms a Polish space is proved in the exact same way as in \cite{ALW16} Corollary 4.3 and Proposition 4.8.
\begin{prop}
\label{prop:GV4}
Let $\boldsymbol{\Pi} \subseteq \mathbb{T}_{\rm{min}}^{\rm{bf}}$. Then the following statements are equivalent:
\begin{itemize}
\item[$(a)$] $\boldsymbol{\Pi}$ is relatively compact in $(\mathbb{T}_{\rm{min}}^{\rm{bf}}, d_{\rm{GV}})$.
\item[$(b)$] For all $r > 0$, $\boldsymbol{\Pi}^{\restr r} := \left\{ \boldsymbol{\mu}^{\restr r} ~|~ \boldsymbol{\mu} \in \boldsymbol{\Pi} \right\}$ is relatively compact in $(\mathbb{T}_{\rm{min}}^{\rm{fin}}, d_{\rm{GP}})$. 
\item[$(c)$] $\boldsymbol{\Pi}^{\restr r_k}$ is relatively compact in $(\mathbb{T}_{\rm{min}}^{\rm{fin}}, d_{\rm{GP}})$ for an increasing sequence of radii $r_k \rightarrow \infty$.
\end{itemize}
\end{prop}
\begin{prop}
\label{prop:GV5}
$(\mathbb{T}_{\rm{min}}^{\rm{bf}}, d_{\rm{GV}})$ is a Polish space.
\end{prop}
Finally, let $\mathbb{T}_{\rm{min}}^{\rm{di}}$ denote the space of isometry classes of minimal Polish discretely infinitely (or finitely) measured $\mathbb{R}$-trees. Then naturally, $\mathbb{T}_{\rm{min}}^{\rm{di}} \subseteq \mathbb{T}_{\rm{min}}^{\rm{bf}}$. The proof of Proposition \ref{prop:DIMeas1} is easily generalized to the Gromov-setting, and so we have that $(\mathbb{T}_{\rm{min}}^{\rm{di}}, d_{\rm{GV}})$ is a Lusin space, which is \emph{not} complete.
\section{Mass erasure in fixed $\mathbb{R}$-trees}
\label{sec:4}
\subsection{Mass-erased subtrees}
\label{subsec:MESubtrees}
Recall the definition and discussion of mass-erased subtrees from Section \ref{subsec:MainME}. Our first goal will be to formally prove Lemma \ref{lemma:MEFix2}. Recall that $\bullet^r = \bullet^r(\mu)$.
\begin{lemma}
\label{lemma:MEFix1}
Let $(T, d, \rho)$ be a Polish $\mathbb{R}$-tree, let $\mu \in \mathscr{M}_{\rm{di}}(T)$, and let $h_0 \geq h > 0$. Then for all $r > 0$, $ R_{\mu,h}(T) \cap {\rm{Cut}}(T, \bullet^r) = R_{{\rm{cut}}_{\bullet^r}^{h_0}\mu, h}(T)$, and it is a subtree of finite type.
\end{lemma}
\begin{proof}
By Duquesne and Winkel \cite{DW26} Lemma 3.12, $R_{{\rm{cut}}_{\bullet^r}\mu, h}(T)$ is of finite type since ${\rm{cut}}_{\bullet^r}\mu \in \mathscr{M}_{\rm{fin}}(T)$. As ${\rm{cut}}_{\bullet^r}\mu(\theta_{\sigma}T) \leq \mu(\theta_{\sigma}T)$ for all $\sigma \in T$, and ${\rm{cut}}_{\bullet^r}\mu$ has support on ${\rm{Cut}}(T, \bullet^r)$, we immediately have that $R_{{\rm{cut}}_{\bullet^r}\mu, h}(T) \subseteq R_{\mu,h}(T) \cap {\rm{Cut}}(T, \bullet^r)$. Suppose that $\sigma \in R_{\mu,h}(T) \cap {\rm{Cut}}(T, \bullet^r)$. Then ${\rm{cut}}_{\bullet^r}\mu(\theta_{\sigma}T) \geq h_0 \geq h$ if $\sigma \in {\rm{Skel}}(T, \mu)$, and ${\rm{cut}}_{\bullet^r}\mu(\theta_{\sigma}T) = \mu(\theta_{\sigma}T) \geq h$ if $\sigma \in F_{\mu}(T)$, so in either case, $\sigma \in R_{{\rm{cut}}_{\bullet^r}\mu, h}(T)$.
\end{proof}
We may now prove Lemma \ref{lemma:MEFix2}, stating that $R_{\mu,h}(T)$ is a discrete $\mathbb{R}$-tree (i.e. boundedly finite) whenever $\mu \in \mathscr{M}_{\rm{di}}(T)$ and $h > 0$.
\begin{proof}[Proof of Lemma \ref{lemma:MEFix2}]
Fix some $r > 0$, and let $h_0 \geq h$. By Lemma \ref{lemma:MEFix1}, $R_{\mu,h}(T) \cap {\rm{Cut}}(T, \bullet^r)$ is of finite type, and so its restriction $R_{\mu,h}(T) \cap \overline{B}_T(\rho, r)$ is also of finite type. As the argument holds for all $r > 0$, we conclude the desired.
\end{proof}
Consider some Polish $\mathbb{R}$-tree $(T, d, \rho)$, and let $\mu \in \mathscr{M}_{\rm{di}}(T)$. Consider the right limit from \eqref{eq:RightLim}. We define $R_{\mu,h+}^{\circ}(T) := \bigcup_{h' > h} R_{\mu, h'}(T)$ such that $R_{\mu,h+}(T) = \overline{R_{\mu,h+}^{\circ}(T)}$, and note contrary to the notation that $R_{\mu,h+}^{\circ}(T)$ may be a closed subtree of $T$. Analogously to \cite{DW26} Proposition 3.13, we establish some basic properties for mass-erased subtrees. 
\begin{prop}[Properties of mass-erased subtrees]
\label{prop:MEFix3}
Let $(T, d, \rho)$ be a Polish $\mathbb{R}$-tree, let $\mu \in \mathscr{M}_{\rm{di}}(T)$, and let $h > 0$. Then the following statements are true.
\begin{enumerate}[$(i)$]
\item If $T' \subseteq T$ is a closed rooted subtree with $\mu(T \setminus T') < h$, then $ R_{\text{Pr}_{T'}\mu, h}(T) = R_{\mu,h}(T)$.
\item Let $\nu \in \mathscr{M}_{\rm{di}}(T)$. If $R_{\mu,h}(T) = R_{\nu,h}(T)$ for all $h > 0$, then ${\rm{Skel}}(T, \mu) = {\rm{Skel}}(T, \nu)$ and $\mu(\theta_{\sigma}T) = \nu(\theta_{\sigma}T)$ for all $\sigma \in T$. If additionally $\mu|_{{\rm{Skel}}(T, \mu)} = \nu|_{{\rm{Skel}}(T, \nu)}$, then $\mu = \nu$.

\item For all $h' > h > 0$ it holds that $R_{\mu,h'}(T) \subseteq R_{\mu,h}(T)$ and $\bigcap_{h'' \in (0,h)} R_{\mu,h''}(T) = R_{\mu,h}(T)$. Moreover, the map $(0, \infty) \rightarrow \mathfrak{M}(T)$ given by $h \mapsto R_{\mu,h}(T)$ is left-continuous with right-limits (càglàd) with respect to the local Hausdorff distance $d_{\rm{H}}^{\rm{loc}}$ on $\mathfrak{M}(T)$, its set of discontinuity points is at most countable, and $R_{\mu,h+}(T) = \lim_{h' \searrow h} R_{\mu,h'}(T)$.

\item Let as usual $f_{R_{\mu,h+}(T)}$ be the point projection onto the closed rooted subtree $R_{\mu,h+}(T)$. Then
\begin{align*}
\bigcup_{\sigma \in {\rm{Lf}}(R_{\mu,h}(T))} ]f_{R_{\mu,h+}(T)}(\sigma), \sigma] &= R_{\mu,h}(T) \setminus R_{\mu,h+}(T) \\
&\subseteq R_{\mu,h}(T) \setminus R_{\mu,h+}^{\circ}(T) \\
&= \{\sigma \in T ~|~ \mu(\theta_{\sigma}T) = h\}.
\end{align*}

\item If $\text{\emph{supp}}(\mu)$ is a rooted subtree, then $ R_{\mu,h}(T) = R_{\mu,h+}(T)$ for all $h > 0$.

\item It holds that $\overline{\bigcup_{h > 0} R_{\mu,h}(T)} = \overline{\text{\emph{Span}}(\text{\emph{supp}}(\mu))}$.
\end{enumerate}
\end{prop}
\begin{proof}
The proofs of $(i)$, $(iv)$, $(v)$ and $(vi)$ are identical to the corresponding proofs in \cite{DW26} Proposition 3.13. Statement $(iii)$ also follows analogously to \cite{DW26} Proposition 3.13$(iii)$ but using the fact that $R_{\mu,h}(T)$ is of boundedly finite type (and thus closed and locally compact) together with the properties of the local Hausdorff distance outlined at the beginning of Section \ref{sec:3}.

For $(ii)$, the fact that $\mu(\theta_{\sigma}T) = \nu(\theta_{\sigma}T)$ for all $\sigma \in T$ is shown in the same way as in \cite{DW26} Proposition 3.13$(ii)$, and we can immediately conclude from this that ${\rm{Skel}}(T, \mu) = {\rm{Skel}}(T, \nu)$, and so in particular for all $r > 0$: ${\rm{Cut}}(T, \bullet^r(\mu)) = {\rm{Cut}}^r(T, \bullet^r(\nu))$. Write ${\rm{Skel}}(T)$ for the common skeleton and $\bullet^r$ for the common cut-off points at level $r$ for $\mu$ and $\nu$. Then as ${\rm{Skel}}(T)$ is a closed rooted subset of $T$, every connected component $C$ of $F_{\mu}(T)$ has $\mu(C) = \nu(C) < \infty$ by Proposition \ref{prop:Proj1}$(ii)$. So if $\mu$ and $\nu$ agree on ${\rm{Skel}}(T)$, and $\sigma \in {\rm{Skel}}(T) \cap {\rm{Cut}}(T, \bullet^r)$, then it is easily checked that $\mu(\theta_{\sigma}T \cap {\rm{Cut}}(T, \bullet^r)) = \nu(\theta_{\sigma}T \cap {\rm{Cut}}(T, \bullet^r))$. The same of course automatically holds true for $\sigma \in F_{\mu}(T) \cap {\rm{Cut}}(T, \bullet^r)$ by construction of the cut-off tree. As $\mu|_{{\rm{Cut}}(T, \bullet^r)}$ and $\nu|_{{\rm{Cut}}(T, \bullet^r)}$ are finite measures, we get from \cite{DW26} Lemma 2.4$(vi)$ that $\mu|_{{\rm{Cut}}(T, \bullet^r)} = \nu|_{{\rm{Cut}}(T, \bullet^r)}$. As this holds for all $r > 0$, and $\bigcup_{r > 0} {\rm{Cut}}(T, \bullet^r) = T$, we conclude that $\mu = \nu$, as desired.
\end{proof}
\subsection{Local Hausdorff-convergence of mass-erased subtrees}
\label{subsec:LocHaus}
Let $(T, d, \rho)$ be a Polish $\mathbb{R}$-tree. Suppose that $h_n \rightarrow h$ in $(0, \infty)$ and that $\mu_n \stackrel{\rm{vg}}{\rightarrow} \mu$ in $\mathscr{M}_{\rm{di}}(T)$ as $n \rightarrow \infty$. As in \cite{DW26} Proposition 3.16, we ask whether $d_{\rm{H}}^{\rm{loc}}(R_{\mu_n, h_n}(T), R_{\mu,h}(T)) \rightarrow 0$ as $n \rightarrow \infty$? As seen in Example \ref{ex:MEConv}, the answer to this is clearly \emph{no}, as we may have (finite or infinite) mass escaping to infinity under vague convergence. To address this problem, we impose a b-tightness condition for the sequence $(\mu_n)_{n \in \mathbb{N}}$ on the cut-off tree ${\rm{Cut}}(T, \bullet^r(\mu))$ associated to the limit measure $\mu$ at different levels $r > 0$. This ensures that i) whatever parts of the tree are supposed to be finite in the limit, have to satisfy b-tightness (ruling out the possibility of finite or infinite mass from $(\mu_n)_{n \in \mathbb{N}}$ escaping out in the tree), and ii) whatever parts of the tree are supposed to be infinite in the limit, are allowed to grow arbitrarily large in their mass.

Let $\bullet^r = \bullet^r(\mu)$. We say that $(\mu_n)_{n \in \mathbb{N}}$ is b-tight on ${\rm{Cut}}(T, \bullet^r)$ if the restricted measures $(\mu_n|_{{\rm{Cut}}(T, \bullet^r)})_{n \in \mathbb{N}}$ are b-tight on $T$, i.e. if
$$\lim_{R \rightarrow \infty} \limsup_{n \rightarrow \infty} \mu_n\left( {\rm{Cut}}(T, \bullet^r) \setminus \overline{B}_T(\rho, R) \right) = 0.$$
\begin{rmk}
Letting $\bullet^r = \bullet^r(\mu)$ for all $r > 0$, we clearly have that ${\rm{Cut}}(T, \bullet^{\tilde{r}}) \subseteq {\rm{Cut}}(T, \bullet^r)$ for all $\tilde{r} \in (0,r)$. Hence, b-tightness on ${\rm{Cut}}(T, \bullet^r)$ automatically implies b-tightness on ${\rm{Cut}}(T, \bullet^{\tilde{r}})$ for all $\tilde{r} \in (0,r)$.
\demo
\end{rmk}

We start by examining what happens on the skeleton ${\rm{Skel}}(T, \mu)$.
\begin{lemma}
\label{lemma:MEFix5}
Let $(T, d, \rho)$ be a Polish $\mathbb{R}$-tree and suppose that $\mu_n \stackrel{\rm{vg}}{\rightarrow} \mu$ in $\mathscr{M}_{\rm{di}}(T)$. Suppose that $\sigma \in {\rm{Skel}}(T, \mu)$ with $\mu(\{\sigma\}) = 0$. Then for all $K > 0$ we have that $ \mu_n(\theta_{\sigma}T) > K$ eventually\footnote{meaning that there exists $N \in \mathbb{N}$ (depending on $K$) such that for all $n \geq N$: $\mu_n(\theta_{\sigma}T) \geq K$.}.
\end{lemma}
\begin{proof}
By Lemma \ref{lemma:BFMeas2} we have that $\mu_n|_{\theta_{\sigma}T} \stackrel{\rm{vg}}{\rightarrow} \mu|_{\theta_{\sigma}T}$, so for an increasing sequence of $\mu$-nice radii $r_k \rightarrow \infty$ we get that $\mu_n(\theta_{\sigma}T \cap \overline{B}_T(\rho, r_k)) \rightarrow \mu(\theta_{\sigma}T \cap \overline{B}_T(\rho, r_k))$ as $n \rightarrow \infty$. As $\sigma \in {\rm{Skel}}(T, \mu)$ we can choose $k_0$ large enough that $\mu(\theta_{\sigma}T \cap \overline{B}_T(\rho, r_k)) > K$  for all $k \geq k_0$. Hence, we conclude for all $n$ large enough that $\mu_n(\theta_{\sigma}T) \geq \mu_n(\theta_{\sigma}T \cap \overline{B}_T(\rho, r_k)) > K$, as wanted.
\end{proof}
\begin{lemma}
\label{lemma:MEFix6}
Let $(T, d, \rho)$ be a Polish $\mathbb{R}$-tree, and suppose that $h_n \rightarrow h$ in $(0, \infty)$ and $\mu_n \stackrel{\rm{vg}}{\rightarrow} \mu$ in $\mathscr{M}_{\rm{di}}(T)$. Then for every $\mu$-nice $r > 0$, $ {\rm{Skel}}(T, \mu) \cap \overline{B}_T(\rho, r) \subseteq R_{\mu_n,h_n}(T)$ eventually.
\end{lemma}
\begin{proof}
As ${\rm{Skel}}(T, \mu)$ is a discrete $\mathbb{R}$-tree, ${\rm{Skel}}(T, \mu) \cap \overline{B}_T(\rho, r)$ is spanned by finitely many leaves. Suppose that $\sigma$ is one such leaf. Then $\mu(\theta_{\sigma}T) = +\infty$, so choosing $K = \sup_{n \in \mathbb{N}} h_n$ and applying Lemma \ref{lemma:MEFix5}, we get that $\sigma \in R_{\mu_n,h_n}(T)$ eventually. As there are only finitely many leaves for which this argument needs to be applied, we conclude that for all $n$ sufficiently large, ${\rm{Lf}}({\rm{Skel}}(T, \mu) \cap \overline{B}_T(\rho, r)) \subseteq R_{\mu_n,h_n}(T)$. As $R_{\mu_n,h_n}(T)$ is a closed rooted subtree of $T$, this yields the desired.
\end{proof}
Going forward we will often see, as in Lemma \ref{lemma:MEFix5} and \ref{lemma:MEFix6} that certain relations only hold \emph{eventually} (in the sense that they hold for all $n$ sufficiently large). We adopt the convention, that if a quantity $x_n$ is only well-defined \emph{eventually}, we may still write $x_n \rightarrow x$ as $n \rightarrow \infty$. An example of this could be the Hausdorff-distance between sets which are only \emph{eventually} compact, or weak convergence of measures (metrized through the Prokhorov distance) which are only \emph{eventually} finite. With this convention in mind, we establish convergence in the local Hausdorff sense of the mass-erased subtrees under vague convergence and appropriate b-tightness assumptions.
\begin{thm}
\label{thm:MEFix7}
Let $(T, d, \rho)$ be a Polish $\mathbb{R}$-tree, and suppose that $h_n \rightarrow h$ in $(0, \infty)$ and $\mu_n \stackrel{\rm{vg}}{\rightarrow} \mu$ in $\mathscr{M}_{\rm{di}}(T)$. Let $h_0 \geq \sup_{n \in \mathbb{N}} h_n$, fix some $\mu$-nice $r > 0$, and let $\bullet^r = \bullet^r(\mu)$. Then
\begin{equation} 
\label{eq:MEFix7}
R_{{\rm{cut}}_{\bullet^r}^{h_0}\mu_n, h_n}(T) = R_{\mu_n, h_n}(T) \cap {\rm{Cut}}(T, \bullet^r) \quad \text{eventually.}
\end{equation}
If additionally, $(\mu_n)_{n \in \mathbb{N}}$ is b-tight on ${\rm{Cut}}(T, \bullet^r)$, then
\begin{enumerate}[$(i)$]
\item ${\rm{cut}}_{\bullet^r}^{h_0}\mu_n$ is eventually a finite measure, and ${\rm{cut}}_{\bullet^r}^{h_0}\mu_n \stackrel{\rm{wk}}{\rightarrow} {\rm{cut}}_{\bullet^r}^{h_0}\mu$ in $\mathscr{M}_{\rm{fin}}(T)$, and
\item $R_{\mu_n,h_n}(T) \cap {\rm{Cut}}(T, \bullet^r)$ is eventually of finite type, and if the map $h' \mapsto R_{{\rm{cut}}_{\bullet^r}^{h_0}\mu,h'}(T)$ is $d_H$-continuous at $h$, then $ d_{\rm{H}}\left( R_{\mu_n,h_n}(T) \cap {\rm{Cut}}(T, \bullet^r), R_{\mu,h}(T) \cap {\rm{Cut}}(T, \bullet^r) \right) \rightarrow 0$.
\end{enumerate}
\end{thm}
\begin{proof}
For every $\sigma \in {\rm{Skel}}(T, \mu) \cap \overline{B}_T(\rho, r)$, we have ${\rm{cut}}_{\bullet^r}\mu_n(\theta_{\sigma}T) \geq h_0 \geq h_n$, so ${\rm{Skel}}(T, \mu) \cap \overline{B}_T(\rho, r) \subseteq R_{{\rm{cut}}_{\bullet^r}\mu_n, h_n}(T)$. Combining with Lemma \ref{lemma:MEFix6} gives that eventually
\begin{align*}
R_{{\rm{cut}}_{\bullet^r}\mu_n, h_n}(T) &= \left({\rm{Skel}}(T, \mu) \cap \overline{B}_T(\rho, r)\right) \cup \left\{ \sigma \in F_{\mu}(T) \cap {\rm{Cut}}(T, \bullet^r) ~|~ {\rm{cut}}_{\bullet^r}\mu_n(\theta_{\sigma}T) \geq h_n \right\} \\
&= \left({\rm{Skel}}(T, \mu) \cap \overline{B}_T(\rho, r)\right) \cup \left\{ \sigma \in F_{\mu}(T) \cap {\rm{Cut}}(T, \bullet^r) ~|~ \mu_n(\theta_{\sigma}T) \geq h_n \right\} \\
&= R_{\mu_n, h_n}(T) \cap {\rm{Cut}}(T, \bullet^r).
\end{align*}

Now, apply the additional assumptions. Note that $\partial {\rm{Cut}}(T, \bullet^r) = \bullet^r \subseteq \partial B_T(\rho, r)$, so indeed $ \mu(\partial {\rm{Cut}}(T, \bullet^r)) = 0$. Hence, by Lemma \ref{lemma:BFMeas2}, $\mu_n|_{{\rm{Cut}}(T, \bullet^r)} \stackrel{\rm{vg}}{\rightarrow} \mu|_{{\rm{Cut}}(T, \bullet^r)}$ as $n \rightarrow \infty$. As $(\mu_n)_{n \in \mathbb{N}}$ is b-tight on ${\rm{Cut}}(T, \bullet^r)$ we conclude by Lemma \ref{lemma:BFMeas7} that $\mu_n|_{{\rm{Cut}}(T, \bullet^r)} \stackrel{\rm{wk}}{\rightarrow} \mu|_{{\rm{Cut}}(T, \bullet^r)}$ in $\mathscr{M}_{\rm{fin}}(T)$ as $n \rightarrow \infty$. As the $h_0$-point-masses in \eqref{eq:CutMeas} are constant across all $n$, we immediately get $(i)$, and \cite{DW26} Lemma 3.12 ensures that $R_{{\rm{cut}}_{\bullet^r}\mu_n, h_n}(T)$ is eventually a subtree of finite type. Finally, we get that $d_{\rm{H}}\left( R_{{\rm{cut}}_{\bullet^r}\mu_n, h_n}(T), R_{{\rm{cut}}_{\bullet^r}\mu, h}(T) \right) \rightarrow 0$ as $n \rightarrow \infty$ by \cite{DW26} Proposition 3.16, and so \eqref{eq:MEFix7} together with Lemma \ref{lemma:MEFix1} give $(ii)$.
\end{proof}
It is easily seen from Theorem \ref{thm:MEFix7}$(ii)$ that if the assumptions of the theorem hold for some increasing sequence of $\mu$-nice radii $r_k \rightarrow \infty$, then $d_{\rm{H}}^{\rm{loc}}(R_{\mu_n,h_n}(T), R_{\mu,h}(T)) \rightarrow 0$ as $n \rightarrow \infty$.

If $\mu \in \mathscr{M}_{\rm{di}}(T)$, then as $\mu(\theta_{\sigma}T) < h < \infty$ for all $\sigma \in T \setminus R_{\mu,h}(T)$, we have that ${\rm{Pr}}_{R_{\mu,h}(T)}\mu \in \mathscr{M}_{\rm{di}}(T)$ by Proposition \ref{prop:Proj4}. Instead of just studying the local Hausdorff convergence of mass-erased subtrees, we can extend Theorem \ref{thm:MEFix7} to also obtain vague convergence of the projected measures on mass-erased subtrees. We start by proving a lemma which allows us to compare the projection of $\mu_n$ onto $R_{\mu_n,h_n}(T)$ with the projection of ${\rm{cut}}_{\bullet^r}\mu_n$ onto $R_{{\rm{cut}}_{\bullet^r}\mu_n, h_n}(T)$ (and similarly for $\mu$ and ${\rm{cut}}_{\bullet^r}\mu$) on a ball of radius strictly less than $r$.
\begin{lemma}
\label{lemma:MEFix8}
Let $(T, d, \rho)$ be a Polish $\mathbb{R}$-tree, and suppose that $h_n \rightarrow h$ in $(0, \infty)$ and $\mu_n \stackrel{\rm{vg}}{\rightarrow} \mu$ in $\mathscr{M}_{\rm{di}}(T)$. Let $h_0 \geq \sup_{n \in \mathbb{N}} h_n$, fix some $\mu$-nice $r > 0$, and let $\bullet^r = \bullet^r(\mu)$. Then for all $\tilde{r} \in (0, r)$:
\begin{enumerate}[$(i)$]
\item $ \left({\rm{Pr}}_{R_{\mu,h}(T)}\mu\right)^{\restr \tilde{r}} = \left( {\rm{Pr}}_{R_{{\rm{cut}}_{\bullet^r}^{h_0}\mu, h}(T)} {\rm{cut}}_{\bullet^r}^{h_0}\mu \right)^{\restr \tilde{r}}$, and 
\item $\left({\rm{Pr}}_{R_{\mu_n,h_n}(T)}\mu_n\right)^{\restr \tilde{r}} = \left( {\rm{Pr}}_{R_{{\rm{cut}}_{\bullet^r}^{h_0}\mu_n, h_n}(T)} {\rm{cut}}_{\bullet^r}^{h_0}\mu_n \right)^{\restr \tilde{r}}$ eventually.
\end{enumerate}
\end{lemma}
\begin{proof}
It will suffice to prove $(ii)$, as $(i)$ follows by a similar (but easier) argument. By Lemma \ref{lemma:MEFix6} and \eqref{eq:MEFix7}, we can choose $N$ sufficiently large such that for all $n \geq N$, ${\rm{Skel}}(T, \mu) \cap \overline{B}_T(\rho, r) \subseteq R_{\mu_n,h_n}(T)$ and $R_{{\rm{cut}}_{\bullet^r}\mu_n, h_n}(T) = R_{\mu_n,h_n}(T) \cap {\rm{Cut}}(T, \bullet^r)$. Going forward, fix some $n \geq N$. Then we have $R_{{\rm{cut}}_{\bullet^r}\mu_n, h_n}(T) \cap \overline{B}_T(\rho, r) = R_{\mu_n,h_n}(T) \cap \overline{B}_T(\rho, r)$. So if we denote by $(C_i)_{i \in I(\tilde{r})}$ the countable collection of connected components of $T \setminus R_{\mu_n,h_n}(T)$ with closure points $(\sigma_i)_{i \in I(\tilde{r})}$ in $\overline{B}_T(\rho, \tilde{r})$, then as all closure points are contained in the interior of $\overline{B}_T(\rho, r)$, we must have that $(C_i)_{i \in I(\tilde{r})}$ are also precisely the connected components of $T \setminus R_{{\rm{cut}}_{\bullet^r}\mu_n, h_n}(T)$ with closure points in $\overline{B}_T(\rho, \tilde{r})$. Moreover, $\left(\bigcup_{i \in I(\tilde{r})} C_i \right) \cap {\rm{Skel}}(T, \mu) = \emptyset$, so indeed $\bigcup_{i \in I(\tilde{r})} C_i \subseteq {\rm{Cut}}^r(T, \mu) \cap F_{\mu}(T)$ and so by definition ${\rm{cut}}_{\bullet^r}\mu_n(C_i) = \mu_n(C_i)$ for every $i \in I(\tilde{r})$. We then finally get that
\begin{align*}
\left( {\rm{Pr}}_{R_{{\rm{cut}}_{\bullet^r}\mu_n, h_n}(T)} {\rm{cut}}_{\bullet^r}\mu_n \right)^{\restr \tilde{r}} &= {\rm{cut}}_{\bullet^r}\mu_n\left( \cdot \cap R_{{\rm{cut}}_{\bullet^r}\mu_n, h_n}(T) \cap \overline{B}_T(\rho, \tilde{r}) \right) + \sum_{i \in I(\tilde{r})} {\rm{cut}}_{\bullet^r}\mu_n(C_i) \delta_{\sigma_i} \\
&= \mu_n\left( \cdot \cap R_{\mu_n,h_n}(T) \cap \overline{B}_T(\rho, \tilde{r})\right) + \sum_{i \in I(\tilde{r})} \mu_n(C_i) \delta_{\sigma_i} \\
&= \left({\rm{Pr}}_{R_{\mu_n,h_n}(T)}\mu_n\right)^{\restr \tilde{r}}
\end{align*}
by applying Proposition \ref{prop:Proj2}$(i)$.
\end{proof}
\begin{prop}
\label{prop:MEFix9}
Let $(T, d, \rho)$ be a Polish $\mathbb{R}$-tree, and suppose that $h_n \rightarrow h$ in $(0, \infty)$ and $\mu_n \stackrel{\rm{vg}}{\rightarrow} \mu$ in $\mathscr{M}_{\rm{di}}(T)$. Let $r_k \rightarrow \infty$ be an increasing sequence of $\mu$-nice radii, let $h_0 \geq \sup_{n \in \mathbb{N}} h_n$, and set $\bullet^{r_k} = \bullet^{r_k}(\mu)$. If $(\mu_n)_{n \in \mathbb{N}}$ is b-tight on ${\rm{Cut}}(T, \bullet^{r_k})$ for all $k$, then
\begin{enumerate}[$(i)$]
\item ${\rm{Skel}}(T, \mu) = {\rm{Skel}}(T, {\rm{Pr}}_{R_{\mu,h}(T)}\mu)$,

\item $\left( {\rm{Pr}}_{R_{\mu_n,h_n}(T)}\mu_n\right)_{n \in \mathbb{N}}$ is b-tight on ${\rm{Cut}}(T, \bullet^{r_k})$ for all $k$, and
\item if $h' \mapsto R_{{\rm{cut}}_{\bullet^{r_k}}^{h_0}\mu, h'}(T)$ is $d_H$-continuous at $h$ for all $k$, then ${\rm{Pr}}_{R_{\mu_n,h_n}(T)}\mu_n \stackrel{\rm{vg}}{\rightarrow} {\rm{Pr}}_{R_{\mu,h}(T)}\mu$.
\end{enumerate}
\end{prop}
\begin{proof}
It is easily seen that ${\rm{Skel}}(T, \mu) = {\rm{Skel}}(T, {\rm{Pr}}_{R_{\mu,h}(T)}\mu)$ by Proposition \ref{prop:Proj2}$(ii)$. Fix some $r \in \{r_k ~|~ k \in \mathbb{N} \}$. For b-tightness, set $\bullet^r = \bullet^r(\mu)$, and consider some $R > r$. Then it is straightforward to check that $f_{R_{\mu_n,h_n}(T)}^{-1}\left( {\rm{Cut}}(T, \bullet^r) \setminus \overline{B}_T(\rho, R) \right) \subseteq {\rm{Cut}}(T, \bullet^r) \setminus \overline{B}_T(\rho, R)$. Hence, we conclude that ${\rm{Pr}}_{R_{\mu_n,h_n}(T)}\mu_n({\rm{Cut}}(T, \bullet^r) \setminus \overline{B}_T(\rho, R)) \leq \mu_n\left( {\rm{Cut}}(T, \bullet^r) \setminus \overline{B}_T(\rho, R) \right)$ since ${\rm{Pr}}_{R_{\mu_n,h_n}(T)}\mu_n$ is the pushforward measure under $f_{R_{\mu_n,h_n}(T)}$, and so b-tightness of $(\mu_n)_{n \in \mathbb{N}}$ on ${\rm{Cut}}(T, \bullet^r)$ gives $(ii)$.

For $(iii)$, we know from Theorem \ref{thm:MEFix7}$(i)$ that ${\rm{cut}}_{\bullet^r}\mu_n \stackrel{\rm{vg}}{\rightarrow} {\rm{cut}}_{\bullet^r}\mu$ in $\mathscr{M}_{\rm{fin}}(T)$, and so \cite{DW26} Proposition 3.16 gives ${\rm{Pr}}_{R_{{\rm{cut}}_{\bullet^r}\mu_n, h_n}(T)}{\rm{cut}}_{\bullet^r}\mu_n \stackrel{\rm{wk}}{\rightarrow} {\rm{Pr}}_{R_{{\rm{cut}}_{\bullet^r}\mu, h}(T)}{\rm{cut}}_{\bullet^r}\mu$ as $n \rightarrow \infty$. Given some $\varepsilon > 0$ choose $\tilde{r} \in (r-\varepsilon, r)$ to be ${\rm{Pr}}_{R_{{\rm{cut}}_{\bullet^r}\mu, h}(T)}{\rm{cut}}_{\bullet^r}\mu$-nice. Then using the characterization from Lemma \ref{lemma:MEFix8}, we conclude that $\left({\rm{Pr}}_{R_{\mu_n,h_n}(T)}\mu_n\right)^{\restr \tilde{r}} \stackrel{\rm{wk}}{\rightarrow} \left({\rm{Pr}}_{R_{\mu,h}(T)}\mu\right)^{\restr \tilde{r}}$ as $n \rightarrow \infty$. Repeating the argument for every $r \in \{r_k ~|~ k \in \mathbb{N}\}$ produces an appropriate sequence $\tilde{r}_k \rightarrow \infty$, and we conclude by Proposition \ref{prop:BFMeas1} that ${\rm{Pr}}_{R_{\mu_n,h_n}(T)}\mu_n \stackrel{\rm{vg}}{\rightarrow} {\rm{Pr}}_{R_{\mu,h}(T)}\mu$ as $n \rightarrow \infty$.
\end{proof}
\subsection{Mass erasure as an operator on measures}
\label{subsec:MEOp}
Given a \emph{finitely} measured Polish $\mathbb{R}$-tree, $(T, d, \rho, \mu)$, we represent $h$-mass erasure via a unique operator $\mathscr{E}_h \colon \mathscr{M}_{\rm{fin}}(T) \rightarrow \mathscr{M}_{\rm{fin}}(T)$ satisfying the $h$-mass erasure property $\mathscr{E}_h\mu(\theta_{\sigma}T) = \left( \mu(\theta_{\sigma}T) - h \right)_+$ for every $\sigma \in T$, as in \cite{DW26} Section 3.3. In \cite{DW26} Proposition 3.18, the $h$-mass erasure operator is more explicitly characterized by setting for each $h > 0$
\begin{equation} 
\label{eq:MEOpFin}
\mathscr{E}_h\mu := {\rm{Pr}}_{R_{\mu,h}(T)} \mu + h \varpi_{R_{\mu,h}(T)},
\end{equation}
where given any subtree of \emph{finite type} $T' \subseteq T$, the signed measure $\varpi_{T'}$ is defined by
$$ \varpi_{T'} := \sum_{\sigma \in T'} \left( \deg_{\rm{out}}(\sigma, T') - 1 \right) = (n(\rho, T') - 1)\delta_{\rho} + \sum_{\sigma \in {\rm{Bp}}(T')} (n(\sigma, T') - 2)\delta_{\sigma} - \sum_{\sigma \in {\rm{Lf}}(T')} \delta_{\sigma}, $$
where $\deg_{\rm{out}}(\sigma, T') = n(\sigma, T') - 1$ for any $\sigma \neq \rho$, and $\deg_{\rm{out}}(\rho, T') = n(\rho, T')$.

If we instead consider some $\mu \in \mathscr{M}_{\rm{di}}(T)$, the mass-erased subtree $R_{\mu,h}(T)$ need only be of \emph{boundedly finite type}, and so it can potentially contain infinitely many leaves and branch points. In that case, the signed measure $\varpi_{R_{\mu,h}(T)}$ becomes ill-defined, and so it is clear that we cannot directly reuse the definition \eqref{eq:MEOpFin}. Instead, we define the mass erasure operator explicitly, by setting for every $h > 0$
\begin{align}
\label{eq:MEOp}
\begin{split}
\mathscr{E}_h\mu &:= {\rm{Pr}}_{R_{\mu,h}(T)}\mu(\cdot \cap (R_{\mu,h}(T) \setminus {\rm{Lf}}(R_{\mu,h}(T)))) + \sum_{\sigma \in {\rm{Lf}}(R_{\mu,h}(T))} (\mu(\theta_{\sigma}T) - h) \delta_{\sigma}\\ 
&\quad + h(n(\rho, R_{\mu,h}(T)) - 1) \delta_{\rho} + h \sum_{\sigma \in {\rm{Bp}}(R_{\mu,h}(T))} (n(\sigma, R_{\mu,h}(T)) - 2) \delta_{\sigma}.
\end{split}
\end{align}
We call $\mathscr{E}_h\mu$ the \emph{$h$-mass erasure of $\mu$}, and set by convention $\mathscr{E}_0\mu := \mu$. 
Although this definition may seem rather complicated, it ensures that the mass erasure operator behaves as in the setting with finite measures from \cite{DW26} Section 3.3. Indeed, the only difference from \eqref{eq:MEOpFin} is that the negative point-masses assigned to the leaves of the mass-erased subtree have been moved into the second term of \eqref{eq:MEOp}, in a way such that the term always has a non-negative contribution. Before proving that the mass erasure operator on $\mathscr{M}_{\rm{di}}(T)$ satisfies the mass erasure property, note that $R_{\mu,h}(\theta_{\sigma}T) = \theta_{\sigma}T \cap R_{\mu,h}(T)$ for any $\sigma \in R_{\mu,h}(T)$, where the left-hand side denotes the $h$-mass-erased subtree of $(\theta_{\sigma}T, d, \sigma)$ with respect to the restriction of $\mu$ to $\theta_{\sigma}T$. In particular, ${\rm{Lf}}(R_{\mu,h}(\theta_{\sigma}T)) = {\rm{Lf}}(R_{\mu,h}(T)) \cap \theta_{\sigma}T$ and ${\rm{Bp}}(R_{\mu,h}(\theta_{\sigma}T)) = {\rm{Bp}}(R_{\mu,h}(T)) \cap (\theta_{\sigma}T \setminus \{\sigma\})$. 
\begin{prop}[Mass erasure]
\label{prop:MEOp1}
Let $(T, d, \rho)$ be a Polish $\mathbb{R}$-tree, let $\mu \in \mathscr{M}_{\rm{di}}(T)$, and let $h > 0$. Then $\mathscr{E}_h \mu \in \mathscr{M}_{\rm{di}}(T)$ with ${\rm{Skel}}(T, \mathscr{E}_h\mu) = {\rm{Skel}}(T, \mu)$, and it satisfies the $h$-mass erasure property
$$ \mathscr{E}_h\mu(\theta_{\sigma}T) = \left( \mu(\theta_{\sigma}T) - h \right)_+ \quad \text{for all } \sigma \in T. $$
\end{prop}
\begin{proof}
It is easily checked that $\mathscr{E}_h\mu \in \mathscr{M}_{\rm{bf}}(T)$ by noting that ${\rm{Pr}}_{R_{\mu,h}(T)}\mu \in \mathscr{M}_{\rm{bf}}(T)$ and $R_{\mu,h}(T)$ is of boundedly finite type. In particular, the measure is positive as $\mu(\theta_{\sigma}T) - h \geq 0$ for every $\sigma \in {\rm{Lf}}(R_{\mu,h}(T))$. The fact that ${\rm{Skel}}(T, \mathscr{E}_h\mu) = {\rm{Skel}}(T, \mu)$ and $\mathscr{E}_h\mu \in \mathscr{M}_{\rm{di}}(T)$ will follow immediately if we establish the $h$-mass erasure property.

For the mass erasure property, suppose first that $\sigma \in T \setminus R_{\mu,h}(T)$. Then $\mu(\theta_{\sigma}T) < h < +\infty$, and it is easily seen by Proposition \ref{prop:Proj2}$(ii)$ that $\mathscr{E}_h\mu(\theta_{\sigma}T) = 0 = (\mu(\theta_{\sigma}T) - h)_+$. 

Suppose instead that $\sigma \in R_{\mu,h}(T) \cap F_{\mu}(T)$. If $\sigma = \rho$, then $\mu$ is a finite measure on $T$, and we easily verify that $\mathscr{E}_h\mu$ corresponds to the $h$-mass erasure from \cite{DW26} Section 3.3, which satisfies the $h$-mass erasure property. So we can assume without loss of generality that $\sigma \neq \rho$. As $\mu|_{\theta_{\sigma}T}$ is a finite measure by assumption, $R_{\mu,h}(\theta_{\sigma}T)$ is of finite type, and \cite{DW26} Proposition 3.18 gives that $\mathscr{E}_h(\mu|_{\theta_{\sigma}T}) \in \mathscr{M}_{\rm{fin}}(T)$ with $\mathscr{E}_h(\mu|_{\theta_{\sigma}T})(\theta_{\sigma}T) = (\mu(\theta_{\sigma}T) - h)_+$. Now, since $\rho \notin \theta_{\sigma}T$, we have
\begin{align*}
(\mathscr{E}_h\mu)|_{\theta_{\sigma}T} &= {\rm{Pr}}_{R_{\mu,h}(T)}\mu( \cdot \cap \theta_{\sigma}T \cap (R_{\mu,h}(T) \setminus {\rm{Lf}}(R_{\mu,h}(T)))) + \sum_{\sigma' \in {\rm{Lf}}(R_{\mu,h}(T)) \cap \theta_{\sigma}T} (\mu(\theta_{\sigma'}T) - h) \delta_{\sigma'}\\
&\quad + h \sum_{\sigma' \in {\rm{Bp}}(R_{\mu,h}(T)) \cap \theta_{\sigma}T} (n(\sigma', R_{\mu,h}(T)) - 2) \delta_{\sigma'} \\
&= {\rm{Pr}}_{R_{\mu,h}(T)}\mu(\cdot \cap (R_{\mu,h}(\theta_{\sigma}T) \setminus {\rm{Lf}}(R_{\mu,h}(\theta_{\sigma}T))) + \sum_{\sigma' \in {\rm{Lf}}(R_{\mu,h}(\theta_{\sigma}T))} (\mu(\theta_{\sigma'}T) - h)\delta_{\sigma'}\\
&\quad + h (n(\sigma, R_{\mu,h}(\theta_{\sigma}T)) - 1)\delta_{\sigma} + h\sum_{\sigma' \in {\rm{Bp}}(R_{\mu,h}(\theta_{\sigma}T))} (n(\sigma', R_{\mu,h}(\theta_{\sigma}T)) - 2)\delta_{\sigma'}\\
&= \mathscr{E}_h(\mu|_{\theta_{\sigma}T}),
\end{align*}
and so in particular $\mathscr{E}_h\mu(\theta_{\sigma}T) = (\mathscr{E}_h\mu)|_{\theta_{\sigma}T}(\theta_{\sigma}T) = \mathscr{E}_h(\mu|_{\theta_{\sigma}T})(\theta_{\sigma}T) = (\mu(\theta_{\sigma}T) - h)_+$.

Finally, suppose that $\sigma \in {\rm{Skel}}(T, \mu)$. If $R_{\mu,h}(\theta_{\sigma}T)$ has infinitely many branch points, then $\mathscr{E}_h\mu(\theta_{\sigma}T) = +\infty$ by \eqref{eq:MEOp}. If on the contrary, $R_{\mu,h}(\theta_{\sigma}T)$ has at most finitely many branch points, then it must also have at most finitely many leaves, but then considering the first two terms of \eqref{eq:MEOp} gives by Proposition \ref{prop:Proj2}$(ii)$,
$$ \mathscr{E}_h\mu(\theta_{\sigma}T) \geq {\rm{Pr}}_{R_{\mu,h}(T)}\mu(\theta_{\sigma}T) - h |{\rm{Lf}}(R_{\mu,h}(\theta_{\sigma}T))| = \mu(\theta_{\sigma}T) - h |{\rm{Lf}}(R_{\mu,h}(\theta_{\sigma}T))| = +\infty.$$
We thus conclude that $\mathscr{E}_h\mu$ satisfies the $h$-mass erasure property, as desired.
\end{proof}
The fact that $\mathscr{E}_h\mu$ satisfies the $h$-mass erasure property ensures that whenever $\mu$ is a finite measure, the notion is consistent with that from \cite{DW26} Proposition 3.18. The proof of Proposition \ref{prop:MEOp2} is as in \cite{DW26} Proposition 3.19.
\begin{prop}[Properties of mass erasure]
\label{prop:MEOp2}
Let $(T, d, \rho)$ be a Polish $\mathbb{R}$-tree, let $\mu \in \mathscr{M}_{\rm{di}}(T)$, and let $h > 0$. Suppose that $T' \subseteq T$ is a closed rooted subtree. Then
\begin{enumerate}[$(i)$]
\item if $\mu(T \setminus T') < h$, then $\mathscr{E}_h\mu(\theta_{\sigma}T) = \mathscr{E}_h \text{\emph{Pr}}_{T'}\mu(\theta_{\sigma}T)$ for every $\sigma \in T$.
\item for every connected component $C$ of $T \setminus T'$, we have that $\mathscr{E}_h\mu(C) = (\mu(C) - h)_+$.
\item $\mathscr{E}_h\mu(T \setminus T') \leq \mu(T \setminus T')$.
\end{enumerate}
\end{prop}
We next prove the semigroup property from Proposition \ref{prop:MEOp3}.
\begin{proof}[Proof of Proposition \ref{prop:MEOp3}]
Suppose that $h,h' > 0$ (otherwise the proof is trivial). Then using the mass erasure property from Proposition \ref{prop:MEOp1} we easily check that $R_{\mathscr{E}_{h'}\mu, h}(T) = R_{\mu, h+h'}(T)$. Applying this to \eqref{eq:MEOp} thus gives
\begin{align*}
\mathscr{E}_h[\mathscr{E}_{h'}\mu] 
&= {\rm{Pr}}_{R_{\mu, h+h'}(T)}[\mathscr{E}_{h'}\mu](\cdot \cap (R_{\mu, h+h'}(T) \setminus {\rm{Lf}}(R_{\mu, h+h'}(T)))) \\
&\quad + \sum_{\sigma \in {\rm{Lf}}(R_{\mu,h+h'}(T))} (\mu(\theta_{\sigma}T) - (h+h')) \delta_{\sigma}\\
&\quad + h(n(\rho, R_{\mu,h+h'}(T)) - 1) \delta_{\rho} + h \sum_{\sigma \in {\rm{Bp}}(R_{\mu,h+h'}(T))} (n(\sigma, R_{\mu,h+h'}(T)) - 2)\delta_{\sigma}.
\end{align*}
For the purposes of this proof, set $R_{\mu,h+h'}^{\circ}(T) := R_{\mu,h+h'}(T) \setminus {\rm{Lf}}(R_{\mu,h+h'}(T))$. What remains in order to conclude the desired is then to show that
\begin{align}
\label{eq:MEOp3a}
\begin{split}
{\rm{Pr}}_{R_{\mu,h+h'}(T)}[\mathscr{E}_{h'}\mu]( \cdot \cap R_{\mu, h+h'}^{\circ}(T) ) &= {\rm{Pr}}_{R_{\mu,h+h'}(T)}\mu( \cdot \cap R_{\mu, h+h'}^{\circ}(T))\\ 
&\quad + h'(n(\rho, R_{\mu,h+h'}(T)) - 1)\delta_{\rho}\\ 
&\quad + h' \sum_{\sigma \in {\rm{Bp}}(R_{\mu,h+h'}(T))} (n(\sigma, R_{\mu,h+h'}(T)) - 2)\delta_{\sigma}.
\end{split}
\end{align}
Let $(C_i)_{i \in I^{\circ}}$ denote the countable collection of open connected components in $T \setminus R_{\mu,h+h'}(T)$ with closure points $(\sigma_i)_{i \in I^{\circ}}$ in $R_{\mu,h+h'}^{\circ}(T)$. Our strategy will then be to argue that i) if $\sigma^* \in \{\sigma_i ~|~ i \in I^{\circ}\}$, then ${\rm{Pr}}_{R_{\mu,h+h'}(T)}[\mathscr{E}_{h'}\mu](\{\sigma^*\})$ has the correct form, and ii) if $A \subseteq R_{\mu,h+h'}^{\circ}(T)$ is a Borel subset with $A \cap \{\sigma_i ~|~ i \in I^{\circ}\} = \emptyset$, then ${\rm{Pr}}_{R_{\mu,h+h'}(T)}[\mathscr{E}_{h'}\mu](A)$ has the correct form. We can then use $\sigma$-additivity to get an explicit expression for ${\rm{Pr}}_{R_{\mu,h+h'}(T)}[\mathscr{E}_{h'}\mu](A \cap R_{\mu,h+h'}^{\circ}(T))$ for every Borel set $A \in \mathscr{B}(T)$, and so we conclude the desired.

To see i), take some $\sigma^* \in \{\sigma_i ~|~ i \in I^{\circ}\}$. Denote by $I^* \subseteq I^{\circ}$ the subset of indices such that $\overline{C}_i = C_i \cup \{\sigma^*\}$ for every $i \in I^*$. Denote the subsubset of indices for which the associated component intersects $R_{\mu,h'}(T)$ by $ I_{h'}^* := \{i \in I^* ~|~ C_i \cap R_{\mu,h'}(T) \neq \emptyset \}$. Now, for every $i \in I^* \setminus I_{h'}^*$ we have that $\mu(\theta_{\sigma}T) < h'$ for all $\sigma \in C_i$, and so in particular it follows that $\mu(C_i) \leq h'$. Whenever $i \in I_{h'}^*$ there must be some $\sigma \in C_i$ with $\mu(\theta_{\sigma}T) \geq h'$, and so $\mu(C_i) \geq h'$. We also note by construction that $ n(\sigma^*, R_{\mu,h'}(T)) = n(\sigma^*, R_{\mu,h+h'}(T)) + |I_{h'}^*|$, and as $R_{\mu,h'}(T)$ is of boundedly finite type we naturally have $n(\sigma^*, R_{\mu,h'}(T)) < \infty$, which implies in particular that $|I_{h'}^*| < \infty$. Now observe by Proposition \ref{prop:Proj2}$(i)$ and Proposition \ref{prop:MEOp2}$(ii)$ that
\begin{align}
\label{eq:MEOp3b}
\begin{split}
{\rm{Pr}}_{R_{\mu,h+h'}(T)}[\mathscr{E}_{h'}\mu]( \{\sigma^*\}) &= \mathscr{E}_{h'}\mu(\{\sigma^*\}) + \sum_{i \in I^*} \mathscr{E}_{h'}\mu(C_i)= \mathscr{E}_{h'}\mu(\{\sigma^*\}) + \sum_{i \in I_{h'}^*}  (\mu(C_i) - h').
\end{split}
\end{align}
To calculate the first term, observe that if $\sigma^* \notin {\rm{Lf}}(R_{\mu,h+h'}(T))$, then it automatically follows that also $\sigma^* \notin {\rm{Lf}}(R_{\mu,h'}(T))$. So in particular $n(\sigma^*, R_{\mu,h'}(T)) \geq 2$ whenever $\sigma^* \neq \rho$. So \eqref{eq:MEOp} gives that $\mathscr{E}_{h'}\mu(\{\sigma^*\}) = {\rm{Pr}}_{R_{\mu,h'}(T)}\mu(\{\sigma^*\}) + h'(n(\sigma^*, R_{\mu,h'}(T)) - 1 - \mathbbm{1}_{(\sigma^* \neq \rho)})$. We now see by Proposition \ref{prop:Proj2}$(i,iii)$ and Lemma \ref{lemma:Proj3}, that
\begin{align*}
{\rm{Pr}}_{R_{\mu,h+h'}(T)}\mu(\{\sigma^*\}) &= {\rm{Pr}}_{R_{\mu,h+h'}(T)} \left[ {\rm{Pr}}_{R_{\mu,h'}(T)} \mu\right](\{\sigma^*\}) \\
&= {\rm{Pr}}_{R_{\mu,h'}(T)}\mu(\{\sigma^*\}) + \sum_{i \in I^*} {\rm{Pr}}_{R_{\mu,h'}(T)}\mu(C_i) \\
&={\rm{Pr}}_{R_{\mu,h'}(T)} \mu(\{\sigma^*\}) + \sum_{i \in I_{h'}^*} \mu(C_i),
\end{align*}
using in the last step that $R_{\mu,h'}(T) \cap C_i \neq \emptyset$ if and only if $i \in I_{h'}^*$. Rearranging and plugging into \eqref{eq:MEOp3b} thus yields that
\begin{align*}
{\rm{Pr}}_{R_{\mu,h+h'}(T)}[\mathscr{E}_{h'}\mu](\{\sigma^*\}) &= {\rm{Pr}}_{R_{\mu,h+h'}(T)}\mu(\{\sigma^*\}) + h'\left( n(\sigma^*, R_{\mu,h'}(T)) - |I_{h'}^*| - 1 - \mathbbm{1}_{(\sigma^* \neq \rho)} \right) \\
&= {\rm{Pr}}_{R_{\mu,h+h'}(T)}\mu(\{\sigma^*\}) + h'\left( n(\sigma^*, R_{\mu,h+h'}(T)) - 1 - \mathbbm{1}_{(\sigma^* \neq \rho)} \right),
\end{align*}
which is precisely the desired form from \eqref{eq:MEOp3a}.

To see ii), take some $A \subseteq R_{\mu,h+h'}^{\circ}(T)$ with $A \cap \{\sigma_i ~|~ i \in I^{\circ}\} = \emptyset$. Then by Proposition \ref{prop:Proj2}$(i)$ and \eqref{eq:MEOp} we have, since no points are projected into $A$ under the map $f_{R_{\mu,h+h'}(T)}$, that
\begin{align}
\label{eq:MEOp3c}
\begin{split}
{\rm{Pr}}_{R_{\mu,h+h'}(T)}[\mathscr{E}_{h'}\mu](A) = \mathscr{E}_{h'}\mu(A) &= {\rm{Pr}}_{R_{\mu,h'}(T)}\mu(A) + h'(n(\rho, R_{\mu,h'}(T)) - 1) \mathbbm{1}_{(\rho \in A)}\\
&\quad + h' \sum_{\sigma \in {\rm{Bp}}(R_{\mu,h'}(T)) \cap A} (n(\sigma, R_{\mu,h'}(T)) - 2)
\end{split}
\end{align}
using that ${\rm{Lf}}(R_{\mu,h'}(T)) \cap R_{\mu,h+h'}^{\circ}(T) = \emptyset$. Observe moreover, again as no points are projected into $A$ under $f_{R_{\mu,h+h'}(T)}$ and using Proposition \ref{prop:Proj2}$(i,iii)$, that ${\rm{Pr}}_{R_{\mu,h+h'}(T)}\mu(A) = {\rm{Pr}}_{R_{\mu,h+h'}(T)}\left[ {\rm{Pr}}_{R_{\mu,h'}(T)} \mu \right] (A) = {\rm{Pr}}_{R_{\mu,h'}(T)}\mu(A)$. As $A \cap \{\sigma_i ~|~ i \in I^{\circ}\} = \emptyset$ we must also have for all $\sigma \in A$ that $n(\sigma, R_{\mu,h+h'}(T)) = n(\sigma, T)$. As $R_{\mu,h+h'}(T) \subseteq R_{\mu,h'}(T) \subseteq T$, we thus get that $ n(\sigma, R_{\mu,h+h'}(T)) = n(\sigma, R_{\mu,h'}(T))$ for all $\sigma \in A$, and so we moreover conclude that ${\rm{Bp}}(R_{\mu,h+h'}(T)) \cap A = {\rm{Bp}}(R_{\mu,h'}(T)) \cap A$. Inserting into \eqref{eq:MEOp3c} thus finally yields
\begin{align*}
{\rm{Pr}}_{R_{\mu,h+h'}(T)}[\mathscr{E}_{h'}\mu](A) &= {\rm{Pr}}_{R_{\mu,h+h'}(T)}\mu(A)  + h'(n(\rho, R_{\mu,h+h'}(T)) - 1)\mathbbm{1}_{(\rho \in A)}\\ 
&\quad + h' \sum_{\sigma \in {\rm{Bp}}(R_{\mu,h+h'}(T)) \cap A} (n(\sigma, R_{\mu,h+h'}(T)) - 2),
\end{align*}
which is the desired form from \eqref{eq:MEOp3a}.
\end{proof}
Our final goal of this section will be to prove a convergence result for mass erasures. We start by noting that when restricted to its cut-off tree at height $r$, the mass erasure has a form reminiscent to that of \cite{DW26} Proposition 3.18.
\begin{lemma}
\label{lemma:MEOp4}
Let $(T, d, \rho)$ be a Polish $\mathbb{R}$-tree, let $\mu \in \mathscr{M}_{\rm{di}}(T)$ and let $h > 0$. Fix some $r > 0$, and let $\bullet^r = \bullet^r(\mu)$. Then for any arbitrary choice of $\varepsilon > 0$, letting $\bullet^{r+\varepsilon} = \bullet^{r+\varepsilon}(\mu)$, we have
\begin{align*}
\mathscr{E}_h\mu|_{{\rm{Cut}}(T, \bullet^r)} &= \left({\rm{Pr}}_{R_{\mu,h}(T)} \mu\right)|_{{\rm{Cut}}(T, \bullet^r)} + h \varpi_{R_{\mu,h}(T) \cap {\rm{Cut}}(T, \bullet^{r+\varepsilon})} |_{{\rm{Cut}}(T, \bullet^r)} \\
&= \left({\rm{Pr}}_{R_{\mu,h+}(T)} \mu\right)|_{{\rm{Cut}}(T, \bullet^r)} + h \varpi_{R_{\mu,h+}(T) \cap {\rm{Cut}}(T, \bullet^{r+\varepsilon})} |_{{\rm{Cut}}(T, \bullet^r)}.
\end{align*}
\end{lemma}

\begin{proof}
The first line follows by a straight-forward calculation, using the fact that $R_{\mu,h}(T) \cap {\rm{Cut}}(T, \bullet^r)$ is a subtree of finite type for all $r > 0$ by Lemma \ref{lemma:MEFix1}, and applying the definition of the signed measure. The second line is a direct consequence of Proposition \ref{prop:MEFix3}$(iv)$. Indeed, it is easily checked that $\mu(R_{\mu,h}(T) \setminus R_{\mu,h+}(T)) = 0$, and that every $\sigma \in {\rm{Lf}}(R_{\mu,h}(T)) \setminus R_{\mu,h+}(T)$ must precisely satisfy ${\rm{Pr}}_{R_{\mu,h}(T)}\mu(\{\sigma\}) = h$. The signed measure then subtracts this $h$-mass from $\sigma$, leaving zero mass on the interval $]f_{R_{\mu,h+}(T)}(\sigma), \sigma]$ in the first line. For the second line, a mass of precisely $h$ is projected into the point $f_{R_{\mu,h+}(T)}(\sigma)$ and then subtracted again by the signed measure. So the two lines coincide.
\end{proof}
\begin{lemma}
\label{lemma:MEOp5}
Let $(T, d, \rho)$ be a Polish $\mathbb{R}$-tree, and suppose that $h_n \rightarrow h$ in $(0, \infty)$ and $\mu_n \stackrel{\rm{vg}}{\rightarrow} \mu$ in $\mathscr{M}_{\rm{di}}(T)$. Let $h_0 \geq \sup_{n \in \mathbb{N}} h_n$, fix some $\mu$-nice $r > 0$, and set $\bullet^r = \bullet^r(\mu)$. Then for all $\tilde{r} \in (0,r)$ it holds that $\left( \mathscr{E}_h\mu \right)^{\restr \tilde{r}} = \left( \mathscr{E}_h {\rm{cut}}_{\bullet^r}^{h_0}\mu \right)^{\restr \tilde{r}}$ and $\left( \mathscr{E}_{h_n}\mu_n \right)^{\restr \tilde{r}} = \left( \mathscr{E}_{h_n}{\rm{cut}}_{\bullet^r}^{h_0}\mu_n \right)^{\restr \tilde{r}}$ eventually.
\end{lemma}
\begin{proof}
This easily follows by using Lemma \ref{lemma:MEOp4} in combination with Lemma \ref{lemma:MEFix8}.
\end{proof}
Using Lemma \ref{lemma:MEOp5}, we may prove the approximation property from Lemma \ref{lemma:MEOp5b}.
\begin{proof}[Proof of Lemma \ref{lemma:MEOp5b}]
Fix some $\varepsilon > 0$ and $r > \varepsilon$, and set $\bullet^r = \bullet^r(\mu)$. As ${\rm{cut}}_{\bullet^r} \mu \in \mathscr{M}_{\rm{fin}}(T)$ it is known from \cite{DW26} Proposition 3.20 that $\mathscr{E}_{h_p} {\rm{cut}}_{\bullet^r}\mu \stackrel{\rm{wk}}{\rightarrow} {\rm{cut}}_{\bullet^r}\mu$ as $p \rightarrow \infty$. Now, restricting to some $\mu$-nice  $\tilde{r} \in (r-\varepsilon, r)$, we get from Lemma \ref{lemma:MEOp5} that $\left( \mathscr{E}_{h_p}\mu \right)^{\restr \tilde{r}} = \left( \mathscr{E}_{h_p}{\rm{cut}}_{\bullet^r}\mu \right)^{\restr \tilde{r}} \stackrel{\rm{wk}}{\rightarrow} \left( {\rm{cut}}_{\bullet^r}\mu \right)^{\restr \tilde{r}} = \mu^{\restr \tilde{r}}$ as $n \rightarrow \infty$. As the argument can be repeated for $(r, \tilde{r})$ arbitrarily large, we get the desired.
\end{proof}
\begin{prop}
\label{prop:MEOp6}
Let $(T, d, \rho)$ be a Polish $\mathbb{R}$-tree, and suppose that $h_n \rightarrow h$ in $(0, \infty)$ and $\mu_n \stackrel{\rm{vg}}{\rightarrow} \mu$ in $\mathscr{M}_{\rm{di}}(T)$. Let $r_k \rightarrow \infty$ be an increasing sequence of $\mu$-nice radii, and set $\bullet^{r_k} = \bullet^{r_k}(\mu)$. If $(\mu_n)_{n \in \mathbb{N}}$ is b-tight on ${\rm{Cut}}(T, \bullet^{r_k})$ for all $k$, then
\begin{enumerate}[$(i)$]
\item $\mathscr{E}_{h_n}\mu_n \stackrel{\rm{vg}}{\rightarrow} \mathscr{E}_h\mu$ as $n \rightarrow \infty$, and
\item $(\mathscr{E}_{h_n}\mu_n)_{n \in \mathbb{N}}$ is b-tight on ${\rm{Cut}}(T, \bullet^{r_k})$ for all $k$.
\end{enumerate}
\end{prop}
\begin{proof}
Fix some $r \in \{r_k ~|~ k \in \mathbb{N}\}$.
For $(ii)$, consider some $R > r$. Then ${\rm{Cut}}(T, \bullet^r) \setminus \overline{B}_T(\rho, R)$ is the union over a subcollection $(C_i)_{i \in I}$ of the connected components of $T \setminus \overline{B}_T(\rho, R)$. By Proposition \ref{prop:MEOp2}$(ii)$ we have that $\mathscr{E}_{h_n}\mu_n(C_i) \leq \mu_n(C_i)$ for every $i \in I$, and so
$$\mathscr{E}_{h_n}\mu_n\left( {\rm{Cut}}(T, \bullet^r) \setminus \overline{B}_T(\rho, R) \right) = \sum_{i \in I} \mathscr{E}_{h_n}\mu_n(C_i) \leq \sum_{i \in I} \mu_n(C_i) = \mu_n\left( {\rm{Cut}}(T, \bullet^r) \setminus \overline{B}_T(\rho, R) \right).$$ 
Letting $n \rightarrow \infty$ and then $R \rightarrow \infty$, b-tightness easily follows.

For $(i)$, fix some $\varepsilon > 0$. By Theorem \ref{thm:MEFix7}$(i)$ we know that ${\rm{cut}}_{\bullet^r}\mu_n \stackrel{\rm{wk}}{\rightarrow} {\rm{cut}}_{\bullet^r}\mu$ in $\mathscr{M}_{\rm{fin}}(T)$, and so \cite{DW26} Proposition 3.20 yields that $\mathscr{E}_{h_n} {\rm{cut}}_{\bullet^r}\mu_n \stackrel{\rm{wk}}{\rightarrow} \mathscr{E}_h {\rm{cut}}_{\bullet^r}\mu$. Choosing some $\tilde{r} \in (r- \varepsilon, r)$ which is $\mathscr{E}_h {\rm{cut}}_{\bullet^r}\mu$-nice, and applying Lemma \ref{lemma:MEOp5} then gives that $\left( \mathscr{E}_{h_n} \mu_n \right)^{\restr \tilde{r}} \stackrel{\rm{wk}}{\rightarrow} \left( \mathscr{E}_h\mu \right)^{\restr \tilde{r}}$. As this argument may be repeated for every $r \in \{r_k ~|~ k \in \mathbb{N}\}$, we conclude that $\mathscr{E}_{h_n}\mu_n \stackrel{\rm{vg}}{\rightarrow} \mathscr{E}_h\mu$.
\end{proof}
\subsection{Vague convergence in the sense of mass erasure}
\label{subsec:vagueME}
Recall the notion of vague convergence in the sense of mass erasure from Definition \ref{def:MEConv}. From Proposition \ref{prop:MEOp6} we see that if $\mu_n \stackrel{\rm{vg}}{\rightarrow} \mu$ in $\mathscr{M}_{\rm{di}}(T)$ as $n \rightarrow \infty$, and $(\mu_n)_{n \in \mathbb{N}}$ is b-tight on an increasing sequence of cut-off trees ${\rm{Cut}}(T, \bullet^{r_k}(\mu))$ with $r_k \rightarrow \infty$, then $\mu_n \stackrel{{\rm{vme}}}{\rightarrow} \mu$ as $n \rightarrow \infty$. The b-tightness condition here is crucial, as otherwise the statement fails, as illustrated in Example \ref{ex:MEConv}. Note also that vague convergence in the sense of mass erasure does not imply vague convergence, even if the mass erasures $(\mathscr{E}_h\mu_n)_{n \in \mathbb{N}}$ are b-tight on ${\rm{Cut}}(T, \bullet^r(\mu))$ for $r$ arbitrarily large. This is illustrated in \cite{DW26} Example 1.9. 

We also observe as a consequence of the semigroup property and Proposition \ref{prop:MEOp6} that vague convergence in the sense of mass erasure still holds, even if we just check that $\mathscr{E}_{h_p}\mu_n \stackrel{\rm{vg}}{\rightarrow} \mathscr{E}_{h_p}\mu$ as $n \rightarrow \infty$, for any decreasing sequence of erasure parameters $h_p \searrow 0$.

The convergence results from Section \ref{subsec:LocHaus} can be generalized to the case where one replaces the assumption $\mu_n \stackrel{\rm{vg}}{\rightarrow} \mu$ with the assumption $\mu_n \stackrel{\rm{vme}}{\rightarrow} \mu$. As ${\rm{Skel}}(T, \mu) = {\rm{Skel}}(T, \mathscr{E}_h\mu)$ for all $h > 0$, Lemma \ref{lemma:MEFix5}, Lemma \ref{lemma:MEFix6} and Lemma \ref{lemma:MEFix8} follow immediately in this setting. As \eqref{eq:MEFix7} is a direct consequence of  Lemma \ref{lemma:MEFix6}, it is also seen to hold when replacing the assumption of vague convergence with that of vague convergence in the sense of mass erasure.
\begin{prop}
\label{prop:MEOp9}
Let $(T, d, \rho)$ be a Polish $\mathbb{R}$-tree, and suppose that $h_n \rightarrow h$ in $(0, \infty)$ and $\mu_n \stackrel{\rm{vme}}{\rightarrow} \mu$ in $\mathscr{M}_{\rm{di}}(T)$. Let $h_0 \geq \sup_{n \in \mathbb{N}} h_n$, let $r_k \rightarrow \infty$ be an increasing sequence of $\mu$-nice radii, and set $\bullet^{r_k} = \bullet^{r_k}(\mu)$. Suppose that $(\mathscr{E}_h\mu_n)_{n \in \mathbb{N}}$ is b-tight on ${\rm{Cut}}(T, \bullet^{r_k})$ for all $h > 0$ and $k$. 

Then for all $k$, $ {\rm{cut}}_{\bullet^{r_k}}^{h_0}\mu_n \stackrel{\rm{me}}{\rightarrow} {\rm{cut}}_{\bullet^{r_k}}^{h_0}\mu$ in $\mathscr{M}_{\rm{fin}}(T)$, and $R_{\mu_n,h_n}(T) \cap {\rm{Cut}}(T, \bullet^{r_k})$ is eventually a subtree of finite type. Moreover, if the map $h' \mapsto R_{{\rm{cut}}_{\bullet^{r_k}}^{h_0}\mu, h'}(T)$ is $d_H$-continuous at $h$ for all $k$, it follows that
\begin{enumerate}[$(i)$]
\item $d_H\left( R_{\mu_n,h_n}(T) \cap {\rm{Cut}}(T, \bullet^{r_k}), R_{\mu,h}(T) \cap {\rm{Cut}}(T, \bullet^{r_k}) \right) \rightarrow 0$ as $n \rightarrow \infty$,
\item ${\rm{Pr}}_{R_{\mu_n,h_n}(T)}\mu_n \stackrel{\rm{vg}}{\rightarrow} {\rm{Pr}}_{R_{\mu,h}(T)}\mu$ as $n \rightarrow \infty$, and
\item $\left( {\rm{Pr}}_{R_{\mu_n,h_n}(T)}\mu_n \right)_{n \in \mathbb{N}}$ is b-tight on ${\rm{Cut}}(T, \bullet^{r_k})$ for all $k$.
\end{enumerate}
\end{prop}
\begin{proof}
Fix some $r \in \{r_k ~|~ k \in \mathbb{N}\}$, let $\delta \in (0, h)$ and consider some $h' \in (0, h - \delta]$. Choose $\varepsilon > 0$ such that $r+\varepsilon$ is $\mu$-nice, and suppose that $n$ is chosen sufficiently large such that, 1. ${\rm{Skel}}(T, \mu) \cap \overline{B}_T(\rho, r+\varepsilon) \subseteq R_{\mu_n,h_n}(T)$ and ${\rm{cut}}_{\bullet^r}\mu_n$ is a finite measure\footnote{It is possible to choose such $n$ due to our assumption that (for any arbitrary choice of $h$) $(\mathscr{E}_h\mu_n)_{n \in \mathbb{N}}$ is b-tight on ${\rm{Cut}}(T, \bullet^r)$, and using the fact that ${\rm{Skel}}(T, \mu) = {\rm{Skel}}(T, \mathscr{E}_h\mu)$.}, 2. $R_{{\rm{cut}}_{\bullet^r}\mu_n, h'}(T) = R_{\mu_n,h'}(T) \cap {\rm{Cut}}(T, \bullet^r)$, and 3. $h_n \geq h - \delta$. Using \eqref{eq:MEOpFin}, we then see that 
\begin{align*}
\mathscr{E}_{h'}{\rm{cut}}_{\bullet^r}\mu_n &= {\rm{Pr}}_{R_{\mu_n,h'}(T) \cap {\rm{Cut}}(T, \bullet^r)} {\rm{cut}}_{\bullet^r}\mu_n + h' \varpi_{R_{\mu_n,h'}(T) \cap {\rm{Cut}}(T, \bullet^r)} \\
&= \left({\rm{Pr}}_{R_{\mu_n,h'}(T)}\mu_n\right)|_{{\rm{Cut}}(T, \bullet^r)} + (h_0 - h') \sum_{\boldsymbol{\sigma} \in \bullet^r} \delta_{\boldsymbol{\sigma}} + h' \varpi_{R_{\mu_n,h'}(T) \cap {\rm{Cut}}(T, \bullet^{r+\varepsilon})}|_{{\rm{Cut}}(T, \bullet^r) \setminus \bullet^r}\\
&= \left( {\rm{Pr}}_{R_{\mu_n,h'}(T)}\mu_n + h' \varpi_{R_{\mu_n,h'}(T) \cap {\rm{Cut}}(T, \bullet^{r+\varepsilon})} \right)|_{{\rm{Cut}}(T, \bullet^r)} + h_0 \sum_{\boldsymbol{\sigma} \in \bullet^r} \delta_{\boldsymbol{\sigma}}\\ 
&\quad - h' \left(\varpi_{R_{\mu_n,h'}(T) \cap {\rm{Cut}}(T, \bullet^{r+\varepsilon})}|_{\bullet^r} + \sum_{\boldsymbol{\sigma} \in \bullet^r} \delta_{\boldsymbol{\sigma}} \right) \\
&= \left(\mathscr{E}_{h'}\mu_n\right)|_{{\rm{Cut}}(T,\bullet^r)} + h_0 \sum_{\boldsymbol{\sigma} \in \bullet^r} \delta_{\boldsymbol{\sigma}} - h' \sum_{\boldsymbol{\sigma} \in \bullet^r} \deg_{\rm{out}}(\boldsymbol{\sigma}, R_{\mu_n,h'}(T) \cap {\rm{Cut}}(T, \bullet^{r+\varepsilon})) \delta_{\boldsymbol{\sigma}},
\end{align*}
where Lemma \ref{lemma:MEOp4} has been applied in the last step. For the first term, note that $\mathscr{E}_{h'}\mu(\bullet^r) = 0$ since $r$ is $\mu$-nice\footnote{Here we use the fact that $\bullet^r \cap {\rm{Bp}}(T) = \emptyset$, which ensures that $\mathscr{E}_{h'}\mu$ assigns no point masses to $\bullet^r$.}. Hence, Lemma \ref{lemma:BFMeas2} gives $\left( \mathscr{E}_{h'}\mu_n \right)|_{{\rm{Cut}}(T, \bullet^r)} \stackrel{\rm{vg}}{\rightarrow} \left( \mathscr{E}_{h'}\mu \right)|_{{\rm{Cut}}(T, \bullet^r)}$. The second term is constant in $n$, and for the third term we have for every $\boldsymbol{\sigma} \in \bullet^r$ that $\boldsymbol{\sigma}$ cannot be a branch point in $T$, and so in particular as ${\rm{Skel}}(T, \mu) \cap \overline{B}_T(\rho, r+\varepsilon) \subseteq R_{\mu_n,h_n}(T) \subseteq R_{\mu_n, h'}(T)$ by 2. and 3., we must have that $\deg_{\rm{out}}(\boldsymbol{\sigma}, R_{\mu_n,h'}(T) \cap {\rm{Cut}}(T, \bullet^{r+\varepsilon})) = 1 = \deg_{\rm{out}}(\boldsymbol{\sigma}, R_{\mu,h}(T) \cap {\rm{Cut}}(T, \bullet^{r+\varepsilon}))$. Our previous calculation can be repeated for $\mathscr{E}_{h'}{\rm{cut}}_{\bullet^r}\mu$, and so we conclude that $\mathscr{E}_{h'}{\rm{cut}}_{\bullet^r}\mu_n \stackrel{\rm{vg}}{\rightarrow} \mathscr{E}_{h'}{\rm{cut}}_{\bullet^r}\mu$. For b-tightness, note that $\mathscr{E}_{h'}{\rm{cut}}_{\bullet^r}\mu_n\left( T \setminus \overline{B}_T(\rho, R) \right) = \mathscr{E}_{h'}\mu_n\left( {\rm{Cut}}(T, \bullet^r) \setminus \overline{B}_T(\rho, R) \right)$ for all $R > r$ by our previous calculation. Hence, b-tightness of $(\mathscr{E}_{h'}\mu_n)_{n \in \mathbb{N}}$ on ${\rm{Cut}}(T, \bullet^r)$ implies b-tightness of $\left( \mathscr{E}_{h'}{\rm{cut}}_{\bullet^r}\mu_n \right)_{n \in \mathbb{N}}$. We thus have $\mathscr{E}_{h'}{\rm{cut}}_{\bullet^r}\mu_n \stackrel{\rm{wk}}{\rightarrow} \mathscr{E}_{h'} {\rm{cut}}_{\bullet^r}\mu$. As $h' \in (0, h - \delta]$ was chosen arbitrarily, we conclude that ${\rm{cut}}_{\bullet^r}\mu_n \stackrel{\rm{me}}{\rightarrow} {\rm{cut}}_{\bullet^r}\mu$ in $\mathscr{M}_{\rm{fin}}(T)$.

Statement $(i)$ now follows directly by \cite{DW26} Lemma 3.23 and \eqref{eq:MEFix7}, and \cite{DW26} Lemma 3.23 moreover gives that ${\rm{Pr}}_{R_{{\rm{cut}}_{\bullet^r}\mu_n,h_n}(T)} {\rm{cut}}_{\bullet^r}\mu_n \stackrel{\rm{wk}}{\rightarrow} {\rm{Pr}}_{R_{{\rm{cut}}_{\bullet^r}\mu,h}(T)} {\rm{cut}}_{\bullet^r}\mu$ as $n \rightarrow \infty$. The rest of the proof proceeds identically to that of Proposition \ref{prop:MEFix9}.
\end{proof}
We finally ask in which cases vague convergence in the sense of mass erasure can be strengthened to vague convergence (under sufficient b-tightness). Recall the definition of the height measure from Section \ref{subsec:MainConvME}.
\begin{prop}
\label{prop:MEOp10}
Let $(T, d, \rho)$ be a Polish $\mathbb{R}$-tree, and let $\mu, \mu_n \in \mathscr{M}_{\rm{di}}(T)$, $n \in \mathbb{N}$. Let $r_k \rightarrow \infty$ be an increasing sequence of $\mu$-nice radii, and set $\bullet^{r_k} = \bullet^{r_k}(\mu)$. Then the following two statements are equivalent:
\begin{itemize}
\item[$(a)$] $\mu_n \stackrel{\rm{vg}}{\rightarrow} \mu$ and $(\mu_n)_{n \in \mathbb{N}}$ is b-tight on ${\rm{Cut}}(T, \bullet^{r_k})$ for all $k$.
\item[$(b)$] $\mu_n \stackrel{\rm{vme}}{\rightarrow} \mu$, $(\mathscr{E}_h\mu_n)_{n \in \mathbb{N}}$ is b-tight on ${\rm{Cut}}(T, \bullet^{r_k})$ for all $h > 0$ and $k$, and $\lambda_{{\rm{cut}}_{\bullet^{r_k}}^{h_0}\mu_n} \stackrel{\rm{wk}}{\rightarrow} \lambda_{{\rm{cut}}_{\bullet^{r_k}}^{h_0}\mu}$ in $\mathscr{M}_{\rm{fin}}([0,\infty))$ for some $h_0 > 0$ and all $k$.
\end{itemize}
\end{prop}
\begin{proof}
If $(a)$ is true, then it immediately follows that $\mu_n \stackrel{\rm{vme}}{\rightarrow} \mu$ and $(\mathscr{E}_h\mu_n)_{n \in \mathbb{N}}$ is b-tight on ${\rm{Cut}}(T, \bullet^{r_k})$ for all $h > 0$ and $k$. Moreover, Theorem \ref{thm:MEFix7} gives that ${\rm{cut}}_{\bullet^{r_k}}\mu_n \stackrel{\rm{wk}}{\rightarrow} {\rm{cut}}_{\bullet^{r_k}} \mu$ for all $k$. As it can easily be checked that the map $\sigma \mapsto d(\rho, \sigma)$ is continuous, the continuous mapping theorem gives the desired, taking care of the direction $(a) \Rightarrow (b)$.

Suppose instead that $(b)$ is true. Then by Proposition \ref{prop:MEOp9}, ${\rm{cut}}_{\bullet^{r_k}}\mu_n \stackrel{\rm{me}}{\rightarrow} {\rm{cut}}_{\bullet^{r_k}}\mu$ in $\mathscr{M}_{\rm{fin}}(T)$ for all $k$. Applying \cite{DW26} Proposition 3.28 gives that ${\rm{cut}}_{\bullet^{r_k}}\mu_n \stackrel{\rm{wk}}{\rightarrow} {\rm{cut}}_{\bullet^{r_k}}\mu$ for all $k$. Restricting to balls of $\mu$-nice radius $\tilde{r}_k < r_k$, then allows us to conclude that $\mu_n \stackrel{\rm{vg}}{\rightarrow} \mu$, as desired. The b-tightness property is immediately inherited, as weak convergence of the cut-off measures implies b-tightness of $(\mu_n)_{n \in \mathbb{N}}$ on ${\rm{Cut}}(T, \bullet^{r_k})$.
\end{proof}
\begin{rmk}
\label{rmk:MEOp10b}
Fix some $r > 0$. Then $(\mathscr{E}_h\mu_n)_{n \in \mathbb{N}}$ is b-tight on ${\rm{Cut}}(T, \bullet^r)$ for all $h > 0$ if and only if $(\mathscr{E}_{h_p}\mu_n)_{n \in \mathbb{N}}$ is b-tight on ${\rm{Cut}}(T, \bullet^r)$ for some decreasing sequence $h_p \searrow 0$. This follows immediately by utilizing the semigroup property together with an analogous argument to that in Proposition \ref{prop:MEOp6}$(i)$.
\demo
\end{rmk}
Before moving on to the next section, we observe that for any fixed $h > 0$, a convergent sequence of $h$-mass erasures $(\mathscr{E}_h\mu_n)_{n \in \mathbb{N}}$ must have its limit in $\mathscr{M}_{\rm{di}}(T)$.
\begin{lemma}
\label{lemma:MEOp10b}
Let $(T, d, \rho)$ be a Polish $\mathbb{R}$-tree, and let $h > 0$. Then $ \overline{\mathscr{E}_h(\mathscr{M}_{\rm{di}}(T))}^{d_{\rm{V}}} \subseteq \mathscr{M}_{\rm{di}}(T)$.
\end{lemma}
\begin{proof}
Let $\mu_n \in \mathscr{M}_{\rm{di}}(T)$, $n \in \mathbb{N}$, and suppose that $(\mathscr{E}_h\mu_n)_{n \in \mathbb{N}}$ is a Cauchy sequence with respect to the vague topology. Then as $(\mathscr{M}_{\rm{bf}}(T), d_{\rm{V}})$ is complete, there exists some $\nu \in \mathscr{M}_{\rm{bf}}(T)$ such that $\mathscr{E}_h\mu_n \stackrel{\rm{vg}}{\rightarrow} \nu$ as $n \rightarrow \infty$. We wish to argue that indeed, $\nu \in \mathscr{M}_{\rm{di}}(T)$.

Suppose for contradiction that either i) ${\rm{Skel}}(T, \nu) \cap \overline{B}_T(\rho, r)$ is \emph{not} of finite type or ii) $\nu\left( \bigcup_{\sigma \in \partial B_T(\rho, r) \setminus {\rm{Skel}}(T, \nu)} \theta_{\sigma}T \right) = +\infty$, for some $\nu$-nice $r > 0$. Then for any $M \in \mathbb{N}$ we may choose some $\nu$-nice $R > r$ large enough that $\nu(\theta_{\sigma_i}T \cap \overline{B}_T(\rho, R)) > 0$ for at least $M$ points in $\partial B_T(\rho, r)$. Denote these $M$ points by $\sigma_1, \dots, \sigma_M$. Then by vague convergence and the Portmanteau theorem, we get that $\left(\mathscr{E}_h\mu_n\right)^{\restr R}(\theta_{\sigma_i}T) \rightarrow \nu^{\restr R}(\theta_{\sigma_i}T)$ as $n \rightarrow \infty$ for $1 \leq i \leq M$. Hence, we must have for all $1 \leq i \leq M$ that $\mathscr{E}_h\mu_n(\theta_{\sigma_i}T) > 0$ eventually, which implies by Proposition \ref{prop:MEOp1} that $\mu_n(\theta_{\sigma_i}T) \geq h$ for $1 \leq i \leq M$ eventually, and so $R_{\mu_n,h}(T) \cap \overline{B}_T(\rho, r)$ must eventually have at least $M$ leaves on $\partial B_T(\rho, r)$. But then $(n(\rho, R_{\mu_n,h}(T)) - 1) + \sum_{\sigma \in {\rm{Bp}}(R_{\mu_n,h}(T)) \cap B_T(\rho, r)} (n(\sigma, R_{\mu_n,h}(T)) - 2) \geq M-1$, and so $\mathscr{E}_h\mu_n(\overline{B}_T(\rho, r)) \geq h(M-1)$ eventually by \eqref{eq:MEOp}. As this argument can be repeated for $M$ arbitrarily large, we conclude that $\mathscr{E}_h\mu_n(\overline{B}_T(\rho, r)) \rightarrow \infty$, contradicting our assumptions.
\end{proof}
\subsection{Existence of a limit}
Let $(T, d, \rho)$ be a Polish $\mathbb{R}$-tree, and consider an increasing sequence $T_1 \subseteq T_2 \subseteq \dots \subseteq T$ of closed rooted subtrees of $T$. Suppose that $\mu_n \in \mathscr{M}_{\rm{di}}(T)$, $n \in \mathbb{N}$ with ${\rm{Pr}}_{T_n}\mu_{n+1} = \mu_n$ for all $n \in \mathbb{N}$.

As $\mu_n \in \mathscr{M}_{\rm{di}}(T)$ for each $n$, we must have that each projection ${\rm{Pr}}_{T_n}\mu_{n+1} \in \mathscr{M}_{\rm{di}}(T)$, and so $\mu_{n+1}(\theta_{\sigma}T) < \infty$ for every $\sigma \in T \setminus T_n$ by Proposition \ref{prop:Proj4},  so in particular infinite mass is not projected anywhere. Using Proposition \ref{prop:Proj2}$(ii)$ inductively on $n$ it is thus easily shown that ${\rm{Skel}}(T, \mu_1) = {\rm{Skel}}(T, \mu_2) = \dots = {\rm{Skel}}(T, \mu_n) = \dots$. We can therefore denote by ${\rm{Skel}}(T) := {\rm{Skel}}(T, \mu_1)$ the infinite-mass skeleton with respect to $(\mu_n)_{n \in \mathbb{N}}$, and by $\bullet^r = \bullet^r(\mu_1)$ the cut-off points at level $r$. In particular ${\rm{Skel}}(T) \subseteq T_n$, and so we can write ${\rm{Cut}}(T_n, \bullet^r) := T_n \cap {\rm{Cut}}(T, \bullet^r)$ for all $n \in \mathbb{N}$, in which case we of course have that ${\rm{Cut}}(T_1, \bullet^r) \subseteq {\rm{Cut}}(T_2, \bullet^r) \subseteq \dots \subseteq {\rm{Cut}}(T, \bullet^r)$.
\begin{prop}[Monotone convergence of projections]
\label{prop:MEOp12}
Let $(T, d, \rho)$ be a Polish $\mathbb{R}$-tree, and let $T_1 \subseteq T_2 \subseteq \dots \subseteq T$ be closed rooted subtrees of $T$. Suppose that $\mu_n \in \mathscr{M}_{\rm{di}}(T)$, $n \in \mathbb{N}$, satisfies
\begin{enumerate}[$(i)$]
\item ${\rm{Pr}}_{T_n}\mu_{n+1} = \mu_n$ for all $n \in \mathbb{N}$, and
\item $(\mu_n)_{n \in \mathbb{N}}$ is b-tight on ${\rm{Cut}}(T, \bullet^{r_k})$ for an increasing sequence of radii $r_k \rightarrow \infty$.
\end{enumerate}
Then there exists a unique $\mu \in \mathscr{M}_{\rm{di}}(T)$ with ${\rm{Skel}}(T, \mu) = {\rm{Skel}}(T)$ and ${\rm{supp}}(\mu) \subseteq \overline{\bigcup_{n \in \mathbb{N}} T_n}$ such that ${\rm{Pr}}_{T_n}\mu = \mu_n$ for all $n \in \mathbb{N}$ and $\mu_n \stackrel{\rm{vg}}{\rightarrow} \mu$ in $\mathscr{M}_{\rm{di}}(T)$.
\end{prop}
\begin{proof}
Fix some $r > 0$, and set $\nu_n^r := \mu_n|_{{\rm{Cut}}(T, \bullet^r)} \in \mathscr{M}_{\rm{fin}}(T)$ for all $n \in \mathbb{N}$. As we saw earlier that infinite mass will not be projected anywhere, any mass outside of $T \setminus {\rm{Cut}}(T, \bullet^r)$ will never be projected into ${\rm{Cut}}(T, \bullet^r)$, and so ${\rm{Pr}}_{{\rm{Cut}}(T_n, \bullet^r)}\nu_{n+1}^r = \left( {\rm{Pr}}_{T_n}\mu_{n+1} \right)|_{{\rm{Cut}}(T, \bullet^r)} = \nu_n^r$ for all $n \in \mathbb{N}$. By assumption $(ii)$ we moreover have that $(\nu_n^r)_{n \in \mathbb{N}}$ is b-tight on $T$, and so we conclude by \cite{DW26} Proposition 3.9 that there exists some $\nu^r \in \mathscr{M}_{\rm{fin}}(T)$ such that ${\rm{supp}}(\nu^r) \subseteq \overline{\bigcup_{n \in \mathbb{N}} {\rm{Cut}}(T_n, \bullet^r)}$, ${\rm{Pr}}_{{\rm{Cut}}(T_n, \bullet^r)} \nu^r = \nu_n^r$ for all $n \in \mathbb{N}$, and $\nu_n^r \stackrel{\rm{wk}}{\rightarrow} \nu^r$ in $\mathscr{M}_{\rm{fin}}(T)$. The argument can of course be repeated for all choices of $r > 0$, giving rise to a family $\{\nu^r ~|~ r > 0\}$ of measures in $\mathscr{M}_{\rm{fin}}(T)$.

Now, observe if $r' > r > 0$ is chosen such that $\nu^{r'}(\bullet^r) = 0$, then as $\nu_n^{r'} \stackrel{\rm{wk}}{\rightarrow} \nu^{r'}$ it holds that also $\nu_n^r = \nu_n^{r'}|_{{\rm{Cut}}(T, \bullet^r)} \stackrel{\rm{wk}}{\rightarrow} \nu^{r'}|_{{\rm{Cut}}(T, \bullet^r)}$ in $\mathscr{M}_{\rm{fin}}(T)$. As we however also know that $\nu_n^r \stackrel{\rm{wk}}{\rightarrow} \nu^r$, we conclude by uniqueness of limits that $\nu^{r'}|_{{\rm{Cut}}(T, \bullet^r)} = \nu^r$ for Lebesgue-almost all $r' > r > 0$. Thus, there is a unique Borel measure $\mu$ on $T$ that satisfies $\mu|_{{\rm{Cut}}(T, \bullet^r)} = \nu^r$ for all $r > 0$. What remains is to check that the measure $\mu$ has the desired properties.

By construction we immediately see that $\mu \in \mathscr{M}_{\rm{bf}}(T)$, and moreover we have for all $\mu$-nice $r > 0$ that $\mu_n^{\restr r} = (\nu_n^r)^{\restr r} \stackrel{\rm{wk}}{\rightarrow} (\nu^r)^{\restr r} = \mu^{\restr r}$ in $\mathscr{M}_{\rm{fin}}(T)$, so indeed $\mu_n \stackrel{\rm{vg}}{\rightarrow} \mu$ in $\mathscr{M}_{\rm{bf}}(T)$. To see that $\mu \in \mathscr{M}_{\rm{di}}(T)$, first observe that ${\rm{Skel}}(T, \mu) \subseteq {\rm{Skel}}(T)$. This is easily seen by noting that if $\sigma \in T \setminus {\rm{Skel}}(T)$, then at height $r = d(\rho, \sigma)$ we will by construction have that $\theta_{\sigma}T \subseteq {\rm{Cut}}(T, \bullet^r)$, but then $\mu(\theta_{\sigma}T) = \nu^r(\theta_{\sigma}T) < \infty$, and so $\sigma \in T \setminus {\rm{Skel}}(T, \mu)$. This directly implies that ${\rm{Skel}}(T, \mu)$ is a discrete $\mathbb{R}$-tree, and moreover $\bullet^r(\mu) \subseteq \bullet^r$. Moreover,
$$\mu\left({\rm{Cut}}(T, \bullet^r(\mu)\right) = \mu\left( {\rm{Cut}}(T, \bullet^r) \cup \bigcup_{\boldsymbol{\sigma} \in \bullet^r \setminus \bullet^r(\mu)} \theta_{\boldsymbol{\sigma}}T \right) \leq \nu^r(T) + \sum_{\boldsymbol{\sigma} \in \bullet^r \setminus \bullet^r(\mu)} \mu\left(\theta_{\boldsymbol{\sigma}}T\right) < \infty,$$
and so $\mu \in \mathscr{M}_{\rm{di}}(T)$ as wanted. Finally, for any fixed $n \in \mathbb{N}$ we observe that $\mu_n|_{{\rm{Cut}}(T, \bullet^r)} = \nu_n^r = {\rm{Pr}}_{T_n}\nu^r = \left({\rm{Pr}}_{T_n}\mu \right)|_{{\rm{Cut}}(T, \bullet^r)}$ for all $r > 0$, and so by upwards continuity of measures we conclude that ${\rm{Pr}}_{T_n}\mu = \mu_n$. From this Proposition \ref{prop:Proj4} yields ${\rm{Skel}}(T, \mu) = {\rm{Skel}}(T)$.
\end{proof}
Next, let $h_n \searrow 0$ be a strictly decreasing sequence of numbers, and let $\mu_n \in \mathscr{M}_{\rm{di}}(T)$, $n \in \mathbb{N}$, with $\mathscr{E}_{h_n - h_{n+1}}\mu_{n+1} = \mu_n$ for all $n \in \mathbb{N}$. Then we again see that ${\rm{Skel}}(T, \mu_1) = {\rm{Skel}}(T, \mu_2) = \dots$, and so similarly to before we can write ${\rm{Skel}}(T) := {\rm{Skel}}(T, \mu_1)$ and $\bullet^r = \bullet^r(\mu_1)$ for all $r > 0$.
\begin{prop}[Monotone convergence of mass erasures]
\label{prop:MEOp13}
Let $(T, d, \rho)$ be a Polish $\mathbb{R}$-tree, let $h_n \searrow 0$ be strictly decreasing in $(0, \infty)$, and let $\mu_n \in \mathscr{M}_{\rm{di}}(T)$, $n \in \mathbb{N}$. Suppose that
\begin{enumerate}[$(i)$]
\item $\mathscr{E}_{h_n - h_{n+1}}\mu_{n+1} = \mu_n$ for all $n \in \mathbb{N}$, and
\item $(\mu_n)_{n \in \mathbb{N}}$ is b-tight on ${\rm{Cut}}(T, \bullet^{r_k})$ for an increasing sequence of radii $r_k \rightarrow \infty$. 
\end{enumerate}
Then there exists a unique $\mu \in \mathscr{M}_{\rm{di}}(T)$ with ${\rm{Skel}}(T, \mu) = {\rm{Skel}}(T)$ such that $\mathscr{E}_{h_n}\mu = \mu_n$ for all $n \in \mathbb{N}$ and $\mu_n \stackrel{\rm{vg}}{\rightarrow} \mu$ in $\mathscr{M}_{\rm{di}}(T)$.
\end{prop}
\begin{proof}
Using Proposition \ref{prop:MEOp2}$(iv)$, set $T_n := {\rm{Span}}({\rm{supp}}(\mu_n)) = R_{\mu_{n+1}, (h_n - h_{n+1})_+}(T)$ for each $n \in \mathbb{N}$. Then $T_n$ is of boundedly finite type (and of course, more generally, a closed rooted subtree of $T$). Moreover, it is easily checked that $T_1 \subseteq T_2 \subseteq \dots$, since $R_{\mu_{n+1}, (h_n - h_{n+1})+}(T) \subseteq \text{Span}(\text{supp}(\mu_{n+1}))$ for all $n \in \mathbb{N}$. Now, set as before ${\rm{Cut}}(T_n, \bullet^r) := T_n \cap {\rm{Cut}}(T, \bullet^r)$ for each $n \in \mathbb{N}$. Then by standard arguments it follows that ${\rm{Cut}}(T_n, \bullet^r)$ is a subtree of finite type.

Fix some $r > 0$ and $\varepsilon > 0$, and set 
\begin{equation}
\label{eq:MEOp13a}
\nu_n^r := \mu_n|_{{\rm{Cut}}(T, \bullet^r)} - h_n \varpi_{{\rm{Cut}}(T_n, \bullet^{r+\varepsilon})}|_{{\rm{Cut}}(T, \bullet^r)} \quad \text{for every } n \in \mathbb{N},
\end{equation}
with $\varpi$ being the signed measure from the beginning of Section \ref{subsec:MEOp}. We wish to prove that
\begin{itemize}
\item[(a)] $\nu_n^r \in \mathscr{M}_{\rm{fin}}(T)$ for every $n \in \mathbb{N}$ and $r > 0$,
\item[(b)] for every $n \in \mathbb{N}$ there exists some $\nu_n \in \mathscr{M}_{\rm{di}}(T)$ such that $\nu_n|_{{\rm{Cut}}(T, \bullet^r)} = \nu_n^r$, $\nu_n^r \stackrel{\rm{vg}}{\rightarrow} \nu_n$ as $r \rightarrow \infty$, and ${\rm{Skel}}(T, \nu_n) = {\rm{Skel}}(T)$,
\item[(c)] there exists a unique measure $\mu \in \mathscr{M}_{\rm{di}}(T)$ with ${\rm{Skel}}(T, \mu) = {\rm{Skel}}(T)$ such that ${\rm{Pr}}_{T_n}\mu = \nu_n$ for every $n \in \mathbb{N}$ and $\nu_n \stackrel{\rm{vg}}{\rightarrow} \mu$ in $\mathscr{M}_{\rm{di}}(T)$ as $n \rightarrow \infty$, and 
\item[(d)] $\mathscr{E}_{h_n}\mu = \mu_n$ for all $n \in \mathbb{N}$ and $\mu_n \stackrel{\rm{vg}}{\rightarrow} \mu$ as $n \rightarrow \infty$.
\end{itemize}

To see (a), first note that $\nu_n^r$ is clearly a (finite) signed measure. So we just need to argue that it is non-negative. Note by a similar argument to the one of \cite{DW26} Lemma 3.21$(ii)$ using the semigroup property in Proposition \ref{prop:MEOp3}, that $\mathscr{E}_{h_n - h_{n+m}}\mu_{n+m} = \mu_n$ and $T_n = R_{\mu_{n+m}, (h_n - h_{n+m})_+}(T)$ for every $n,m \in \mathbb{N}$. We thus see by Lemma \ref{lemma:MEOp4}, that 
\begin{equation}
\label{eq:MEOp13b}
\mu_n|_{{\rm{Cut}}(T, \bullet^r)} = \left({\rm{Pr}}_{T_n}\mu_{n+m}\right)|_{{\rm{Cut}}(T, \bullet^r)} + (h_n - h_{n+m}) \varpi_{{\rm{Cut}}(T_n, \bullet^{r+\varepsilon})}|_{{\rm{Cut}}(T, \bullet^r)}.
\end{equation}
Plugging into \eqref{eq:MEOp13a} thus yields that $\nu_n^r = {\rm{Pr}}_{T_n}\mu_{n+m}|_{{\rm{Cut}}(T, \bullet^r)} - h_{n+m}\varpi_{{\rm{Cut}}(T_n, \bullet^{r+\varepsilon})}|_{{\rm{Cut}}(T, \bullet^r)}$, and noting that the signed part tends to zero as $m \rightarrow \infty$, we conclude that $\nu_n^r$ must be non-negative.

For (b), observe that the exact value of $\varepsilon > 0$ has no influence on \eqref{eq:MEOp13a} as long as it is strictly larger than zero. Hence, we get for all $r' > r > 0$ that $\nu_n^{r'}|_{{\rm{Cut}}(T, \bullet^r)} = \nu_n^r$, and so we can construct a unique Borel measure $\nu_n$ on $T$ satisfying that $ \nu_n|_{{\rm{Cut}}(T, \bullet^r)} = \nu_n^r$ for all $r > 0$. As we know from (a) that $\nu_n^r \in \mathscr{M}_{\rm{fin}}(T)$, we indeed see that $\nu_n$ must be boundedly finite. Now, fix some $\tilde{r} > 0$. Then for all $r \geq \tilde{r}$ we have that $\left(\nu_n^r\right)^{\restr \tilde{r}} = \left( \nu_n|_{{\rm{Cut}}(T, \bullet^r)} \right)^{\restr \tilde{r}} = \nu_n^{\restr \tilde{r}}$. As this argument holds for all $\tilde{r} > 0$ arbitrarily large, we get that $\nu_n^r \stackrel{\rm{vg}}{\rightarrow} \nu_n$ as $r \rightarrow \infty$. What remains is to check that ${\rm{Skel}}(T, \nu_n) = {\rm{Skel}}(T)$ (as this will also immediately imply that $\nu_n \in \mathscr{M}_{\rm{di}}(T)$). Suppose that $\sigma \in {\rm{Skel}}(T, \nu_n)$, such that $\nu_n(\theta_{\sigma}T) = +\infty$. As $\nu_n^r \in \mathscr{M}_{\rm{fin}}(T)$ for all $r > 0$, this is only possible if $\theta_{\sigma}T \cap (T \setminus {\rm{Cut}}(T, \bullet^r)) \neq \emptyset$ for all $r > 0$. But this is the case only if $\sigma \in {\rm{Skel}}(T)$, taking care of the inclusion ${\rm{Skel}}(T, \nu_n) \subseteq {\rm{Skel}}(T)$. For the opposite inclusion, suppose that $\sigma \in {\rm{Skel}}(T)$. Then $\mu_n(\theta_{\sigma}T) = +\infty$, and so for any $K > 0$ we can always choose $r > 0$ sufficiently large such that $\mu_n(\theta_{\sigma}T \cap {\rm{Cut}}(T, \bullet^r)) > K$. Noting by \cite{DW26} Lemma 3.17$(i)$ that $\varpi_{{\rm{Cut}}(T_n, \bullet^{r+\varepsilon})}(\theta_{\sigma}T) \in \{-1,0\}$, we see that $\nu_n^r(\theta_{\sigma}T)$ can also be made arbitrarily large by choosing $r$ sufficiently large, and so we must have that $\nu_n(\theta_{\sigma}T) = +\infty$, i.e. $\sigma \in {\rm{Skel}}(T, \nu_n)$, as wanted.

For (c), start by fixing some $r > 0$. As the infinite-mass skeleton is fixed in our set-up, mass from outside ${\rm{Cut}}(T, \bullet^r)$ will never be projected into ${\rm{Cut}}(T, \bullet^r)$, and so we have in particular for any $\nu \in \mathscr{M}_{\rm{di}}(T)$ with ${\rm{Skel}}(T, \nu) \subseteq {\rm{Skel}}(T)$, that $\left( {\rm{Pr}}_{T_n}\nu \right)|_{{\rm{Cut}}(T, \bullet^r)} = {\rm{Pr}}_{{\rm{Cut}}(T_n, \bullet^r)}\left[ \nu|_{{\rm{Cut}}(T, \bullet^r)}\right]$. This also holds if $\nu$ is a (finite) signed measure. Hence, using \eqref{eq:MEOp13a}, \eqref{eq:MEOp13b} and \cite{DW26} Lemma 3.17,
\begin{align*}
\left( {\rm{Pr}}_{T_n}\nu_{n+1} \right)|_{{\rm{Cut}}(T, \bullet^r)} &= {\rm{Pr}}_{{\rm{Cut}}(T_n, \bullet^r)} \nu_{n+1}^r \\
&= {\rm{Pr}}_{{\rm{Cut}}(T_n, \bullet^r)}\left[ \mu_{n+1}|_{{\rm{Cut}}(T, \bullet^r)} \right] - h_{n+1} {\rm{Pr}}_{{\rm{Cut}}(T_n, \bullet^r)} \left[ \varpi_{{\rm{Cut}}(T_{n+1}, \bullet^{r+\varepsilon})}|_{{\rm{Cut}}(T, \bullet^r)} \right] \\
&= \left({\rm{Pr}}_{T_n}\mu_{n+1} \right)|_{{\rm{Cut}}(T, \bullet^r)} - h_{n+1} \left( {\rm{Pr}}_{T_n} \varpi_{{\rm{Cut}}(T_{n+1}, \bullet^{r+\varepsilon})} \right)|_{{\rm{Cut}}(T, \bullet^r)}\\
&= \mu_n|_{{\rm{Cut}}(T, \bullet^r)} - (h_n - h_{n+1}) \varpi_{{\rm{Cut}}(T_n, \bullet^{r+\varepsilon})}|_{{\rm{Cut}}(T, \bullet^r)}\\ 
&\quad - h_{n+1} \varpi_{{\rm{Cut}}(T_n, \bullet^{r+\varepsilon})}|_{{\rm{Cut}}(T, \bullet^r)} \\
&= \nu_n^r.
\end{align*}
As this holds for all $r > 0$ arbitrarily large, we easily conclude by upwards continuity of measures that ${\rm{Pr}}_{T_n}\nu_{n+1} = \nu_n$ for all $n \in \mathbb{N}$. Next, fix some $r > 0$ and $\varepsilon > 0$. Then ${\rm{Cut}}(T_n, \bullet^{r+\varepsilon})$ is a subtree of finite type, so in particular we can find some $R > r$ sufficiently large such that ${\rm{Cut}}(T_n, \bullet^{r+\varepsilon}) \subseteq \overline{B}_T(\rho, R)$, in which case $\varpi_{{\rm{Cut}}(T_n, \bullet^{r+\varepsilon})}(T \setminus \overline{B}_T(\rho, R)) = 0$. Hence, \eqref{eq:MEOp13a} yields for all $R$ sufficiently large that $\nu_n({\rm{Cut}}(T, \bullet^r) \setminus \overline{B}_T(\rho, R)) = \nu_n^r(T \setminus \overline{B}_T(\rho, R)) = \mu_n({\rm{Cut}}(T, \bullet^r) \setminus \overline{B}_T(\rho, r))$, and so assumption $(ii)$ allows us to conclude that $(\nu_n)_{n \in \mathbb{N}}$ is b-tight on ${\rm{Cut}}(T, \bullet^r)$. As $r > 0$ was chosen arbitrarily, we apply Proposition \ref{prop:MEOp12} to get the existence of a limit measure $\mu$ satisfying the properties from (c).

For (d), we start by checking that for all $\mu$-nice $r > 0$,
\begin{equation}
\label{eq:MEOp13c}
\mathscr{E}_{h_n}\mu(\theta_{\sigma}T) = \mu_n(\theta_{\sigma}T) \quad \text{for all }\sigma \in {\rm{Cut}}(T, \bullet^r).
\end{equation}
Suppose first that $\mu(\theta_{\sigma}T) = 0$, such that also $\mathscr{E}_{h_n}\mu(\theta_{\sigma}T) = 0$. Then $\theta_{\sigma}T \subseteq {\rm{Cut}}(T, \bullet^r)$, and by Proposition \ref{prop:Proj2}$(ii)$ and \eqref{eq:MEOp13a} we have that $ 0 = {\rm{Pr}}_{T_n}\mu(\theta_{\sigma}T) = \nu_n(\theta_{\sigma}T) = \nu_n^r(\theta_{\sigma}T) = \mu_n(\theta_{\sigma}T) - h_n \varpi_{{\rm{Cut}}(T_n, \bullet^{r+\varepsilon})}(\theta_{\sigma}T)$. As $ \varpi_{{\rm{Cut}}(T_n, \bullet^{r+\varepsilon})}(\theta_{\sigma}T) \in \{-1, 0\}$ by \cite{DW26} Lemma 3.17$(i)$ and $\mu_n$ is non-negative, we must necessarily have that $\mu_n(\theta_{\sigma}T) = 0$, and so \eqref{eq:MEOp13c} is satisfied. 

Suppose next that $0 < \mu(\theta_{\sigma}T) < \infty$ with $\mu(\{\sigma\}) = 0$ if $\sigma \neq \rho$. As $\nu_m \stackrel{\rm{vg}}{\rightarrow} \mu$ as $m \rightarrow \infty$ and $(\nu_m)_{m \in \mathbb{N}}$ is known to be b-tight on ${\rm{Cut}}(T, \bullet^r)$ by the proof of c), we have that $\nu_m^r \stackrel{\rm{wk}}{\rightarrow} \mu|_{{\rm{Cut}}(T, \bullet^r)}$. So the Portmanteau theorem yields that $\lim_{m \rightarrow \infty} \nu_m(\theta_{\sigma}T) = \mu(\theta_{\sigma}T) > 0$. In particular for all $m$ sufficiently large, ${\rm{Pr}}_{T_m}\mu(\theta_{\sigma}T) = \nu_m(\theta_{\sigma}T) > 0$, which by Proposition \ref{prop:Proj2}$(ii)$ can only be the case if $\sigma \in T_m$ and $\nu_m(\theta_{\sigma}T) = \mu(\theta_{\sigma}T)$. Now, if $m > n$, we moreover get by \eqref{eq:MEOp13a} that
\begin{align*}
\mathscr{E}_{h_n}\mu(\theta_{\sigma}T) &= \left( \nu_m(\theta_{\sigma}T)  - h_n \right)_+ \\
&= \left( \mu_m(\theta_{\sigma}T) - h_m\varpi_{{\rm{Cut}}(T_m, \bullet^{r+\varepsilon})}(\theta_{\sigma}T) - h_n \right)_+ \\
&= \left( \mu_m(\theta_{\sigma}T) - (h_n - h_m) \right)_+ \\
&= \mathscr{E}_{h_n - h_m}\mu_m(\theta_{\sigma}T) \\
&= \mu_n(\theta_{\sigma}T),
\end{align*}
using that $\varpi_{{\rm{Cut}}(T_m, \bullet^{r+\varepsilon})}(\theta_{\sigma}T) = -\mathbbm{1}_{{\rm{Cut}}(T_m, \bullet^{r+\varepsilon}) \cap \theta_{\sigma}T \neq \emptyset} = -1$ by \cite{DW26} Lemma 3.17$(i)$. 

The identity is easily generalized to the case where $0 < \mu(\theta_{\sigma}T) < \infty$ with $\mu(\{\sigma\}) > 0$ and $\sigma \neq \rho$, by noting since $d(\rho, \sigma) > 0$ that there must exist a sequence $(\sigma_m)_{m \in \mathbb{N}}$ of points in $[\rho, \sigma]$ with $\mu(\{\sigma_m\}) = 0$ such that $d(\sigma_m, \sigma) \rightarrow 0$ as $m \rightarrow \infty$. Hence, $\theta_{\sigma}T = \bigcap_{m \in \mathbb{N}} \theta_{\sigma_m}T$, and so \eqref{eq:MEOp13c} generalizes by downwards continuity of measures.

Finally, if $\mu(\theta_{\sigma}T) = + \infty$, \eqref{eq:MEOp13c} just becomes $+ \infty = + \infty$, as both measures share the same infinite-mass skeleton ${\rm{Skel}}(T)$. This shows \eqref{eq:MEOp13c}, and as $r > 0$ was chosen arbitrarily, we conclude that $\mathscr{E}_{h_n}\mu(\theta_{\sigma}T) = \mu_n(\theta_{\sigma}T) = \mathscr{E}_{h_n - h_{n+1}}\mu_{n+1}(\theta_{\sigma}T)$ for all $\sigma \in T$, where we have moreover applied assumption $(i)$. So in particular we must have that $T_n = R_{\mu_{n+1}, (h_n - h_{n+1})+}(T) = R_{\mu,h_n+}(T)$, and so we conclude by Lemma \ref{lemma:MEOp4} and \eqref{eq:MEOp13a} that for any $\mu$-nice $r > 0$,
\begin{align*}
\mathscr{E}_{h_n}\mu|_{{\rm{Cut}}(T, \bullet^r)} &= \left( {\rm{Pr}}_{T_n}\mu + h_n \varpi_{{\rm{Cut}}(T_n, \bullet^{r+\varepsilon})} \right)|_{{\rm{Cut}}(T, \bullet^r)}\\ 
&= \nu_n^r + h_n \varpi_{{\rm{Cut}}(T_n, \bullet^{r+\varepsilon})}|_{{\rm{Cut}}(T, \bullet^r)}\\
&= \mu_n|_{{\rm{Cut}}(T, \bullet^r)}.
\end{align*}
We conclude that $\mathscr{E}_{h_n}\mu = \mu_n$ for all $n \in \mathbb{N}$, implying by Lemma \ref{lemma:MEOp5b} that $\mu_n = \mathscr{E}_{h_n}\mu \stackrel{\rm{vg}}{\rightarrow} \mu$.
\end{proof}
Let $(T, d, \rho)$ be a Polish $\mathbb{R}$-tree. Let $h_p \searrow 0$ be strictly decreasing in $(0, \infty)$, and let $\mu_n \in \mathscr{M}_{\rm{di}}(T)$, $n \in \mathbb{N}$. If there exists some $\nu_p \in \mathscr{M}_{\rm{di}}(T)$ such that $\mathscr{E}_{h_p}\mu_n \stackrel{\rm{vg}}{\rightarrow} \nu_p$ as $n \rightarrow \infty$, and $(\mathscr{E}_{h_p}\mu_n)_{n \in \mathbb{N}}$ is b-tight on ${\rm{Cut}}(T, \bullet^r(\nu_p))$ for all $p$ and $r$ arbitrarily large, then by the semigroup property and Proposition \ref{prop:MEOp6}, $\mathscr{E}_{h_p}\mu_n = \mathscr{E}_{h_p - h_{p+1}}\mathscr{E}_{h_{p+1}}\mu_n \stackrel{\rm{vg}}{\rightarrow} \mathscr{E}_{h_p - h_{p+1}}\nu_{p+1}$ as $n \rightarrow \infty$. By uniqueness of limits we thus get that $\mathscr{E}_{h_p - h_{p+1}}\nu_{p+1} = \nu_p$, and so as seen before, the sequence $(\nu_p)_{p \in \mathbb{N}}$ must share a common infinite-mass skeleton, which we shall denote by ${\rm{Skel}}(T) = {\rm{Skel}}(T, \nu_1)$. We write $\bullet^r = \bullet^r(\nu_1)$ for the associated cut-off points at level $r > 0$.
\begin{thm}
\label{thm:MEOp14}
Let $(T, d, \rho)$ be a Polish $\mathbb{R}$-tree, let $h_p \searrow 0$ be strictly decreasing in $(0, \infty)$, and let $\mu_n \in \mathscr{M}_{\rm{di}}(T)$, $n \in \mathbb{N}$. Let $r_k \rightarrow \infty$ be an increasing sequence of radii, and suppose for all $p$ and $k$ that
\begin{enumerate}[$(i)$]
\item there exists some $\nu_p \in \mathscr{M}_{\rm{di}}(T)$ such that $\mathscr{E}_{h_p} \mu_n \stackrel{\rm{vg}}{\rightarrow} \nu_p$ as $n \rightarrow \infty$,
\item $(\mathscr{E}_{h_p}\mu_n)_{n \in \mathbb{N}}$ is b-tight on ${\rm{Cut}}(T, \bullet^{r_k})$, and
\item $(\nu_p)_{p \in \mathbb{N}}$ is b-tight on ${\rm{Cut}}(T, \bullet^{r_k})$.
\end{enumerate}
Then there exists a unique $\mu \in \mathscr{M}_{\rm{di}}(T)$ with ${\rm{Skel}}(T, \mu) = {\rm{Skel}}(T)$, such that $\mu_n \stackrel{\rm{vme}}{\rightarrow} \mu$ and $\mathscr{E}_{h_p}\mu = \nu_p$ for all $p \in \mathbb{N}$.
\end{thm}
\begin{proof}
As seen above, we have that $\mathscr{E}_{h_p - h_{p+1}}\nu_{p+1} = \nu_p$ for all $p \in \mathbb{N}$ by assumptions $(i)$ and $(ii)$. By Proposition \ref{prop:MEOp13} we are thus able to conclude, using $(iii)$, that there exists a unique measure $\mu \in \mathscr{M}_{\rm{di}}(T)$ such that ${\rm{Skel}}(T, \mu) = {\rm{Skel}}(T)$ and $\mathscr{E}_{h_p}\mu = \nu_p$ for all $p \in \mathbb{N}$. Replacing $\nu_p$ by $\mathscr{E}_{h_p}\mu$ in assumption $(i)$ and using Proposition \ref{prop:MEOp11} thus shows that $\mu_n \stackrel{\rm{vme}}{\rightarrow} \mu$ as $n \rightarrow \infty$.
\end{proof}
Vague convergence in the sense of mass erasure can be metrized to obtain a topology on $\mathscr{M}_{\rm{di}}(T)$ in an analogous way to how (weak) convergence in the sense of mass erasure is metrized in \cite{DW26} equation (1.10). For any strictly decreasing sequence $h_p \searrow 0$ in $(0, \infty)$, we denote the metric for (weak) convergence in the sense of mass erasure on $\mathscr{M}_{\rm{fin}}(T)$ by
$$ d_{\rm{P}}^{\rm{era}}\left(\mu, \nu \right) = \sum_{p=1}^{\infty} 2^{-p}\left(1 \wedge d_{\rm{P}}\left( \mathscr{E}_{h_p}\mu, \mathscr{E}_{h_p}\nu \right) \right) \quad \text{for all } \mu, \nu \in \mathscr{M}_{\rm{fin}}(T), $$
and note that the topology generated by this metric is invariant under different choices of $h_p \searrow 0$ by \cite{DW26} Theorem 3.24. It is also known from \cite{DW26} Theorem 3.24 that $(\mathscr{M}_{\rm{fin}}(T), d_{\rm{P}}^{\rm{era}})$ is separable but not complete (indeed, in order to complete the space, one can consider finite measures defined on the \emph{bordification} of $T$, see \cite{DW26} for a discussion of this approach).
\begin{prop}
\label{prop:MEOp11}
Let $h_p \searrow 0$ be a decreasing sequence in $(0, \infty)$, and set
$$ d_{\rm{V}}^{\rm{era}}(\mu, \nu) := \sum_{p=1}^{\infty} 2^{-p}\left(1 \wedge d_{\rm{V}}(\mathscr{E}_{h_p}\mu, \mathscr{E}_{h_p}\nu) \right) \quad \text{for all } \mu, \nu \in \mathscr{M}_{\rm{di}}(T). $$
Then $(\mathscr{M}_{\rm{di}}(T), d_{\rm{V}}^{\rm{era}})$ is a metric space, and $d_{\rm{V}}^{\rm{era}}$ metrizes vague convergence in the sense of mass erasure on $\mathscr{M}_{\rm{di}}(T)$. The space $(\mathscr{M}_{\rm{di}}(T), d_{\rm{V}}^{\rm{era}})$ is separable but \emph{not} complete.
\end{prop}
\begin{proof}
The fact that $d_{\rm{V}}^{\rm{era}}$ is symmetric and satisfies the triangle inequality is easily checked using the corresponding properties for $d_V$. Suppose that $d_{\rm{V}}^{\rm{era}}(\mu, \nu) = 0$. Then $\mathscr{E}_{h_p}\mu = \mathscr{E}_{h_p}\nu$ for all $p \in \mathbb{N}$. By Lemma \ref{lemma:MEOp5b} we have that $\mathscr{E}_{h_p}\mu \stackrel{\rm{vg}}{\rightarrow} \mu$ and $\mathscr{E}_{h_p}\nu \stackrel{\rm{vg}}{\rightarrow} \nu$ as $p \rightarrow \infty$, and so uniqueness of limits ensures that $\mu = \nu$, as wanted. Hence, $(\mathscr{M}_{\rm{di}}(T), d_{\rm{V}}^{\rm{era}})$ is indeed a metric space. It is easily checked that $d_{\rm{V}}^{\rm{era}}$ metrizes vague convergence in the sense of mass erasure, irrespective of the choice of sequence $h_p \rightarrow 0$, using the semigroup property from Proposition \ref{prop:MEOp3}. 

To see that $(\mathscr{M}_{\rm{di}}(T), d_{\rm{V}}^{\rm{era}})$ is separable, let $D \subseteq \mathscr{M}_{\rm{fin}}(T)$ be a $d_{\rm{P}}^{\rm{era}}$-dense subset. Fix some $\varepsilon > 0$ and $\mu \in \mathscr{M}_{\rm{di}}(T)$, and let $R > 0$ be chosen large enough that $\int_R^{\infty} e^{-r} dr < \frac{\varepsilon}{2}$. Let $\bullet^R = \bullet^R(\mu)$. Then from a similar calculation to that in Proposition \ref{prop:MEOp9} we get for $h_0 \geq h > 0$ that $\mathscr{E}_h {\rm{cut}}_{\bullet^R}^{h_0}\mu = \left( \mathscr{E}_h \mu \right)|_{{\rm{Cut}}(T, \bullet^R)} + (h_0 - h) \sum_{\sigma \in \bullet^R} \delta_{\sigma}$, and so in particular for every $r \in (0, R)$ we have that $\left( \mathscr{E}_h {\rm{cut}}_{\bullet^R}^{h_0}\mu \right)^{\restr r} = \left(\mathscr{E}_h\mu \right)^{\restr r}$. Hence, for any $h > 0$, $ d_{\rm{V}}\left(\mathscr{E}_h\mu,\mathscr{E}_h {\rm{cut}}_{\bullet^R}^{h_0}\mu\right) \leq \int_R^{\infty} e^{-r} dr < \frac{\varepsilon}{2}$, and so we conclude that also $d_{\rm{V}}^{\rm{era}}\left( \mu, {\rm{cut}}_{\bullet^R}^{h_0}\mu \right) < \frac{\varepsilon}{2}$. Now, as ${\rm{cut}}_{\bullet^R}^{h_0}\mu \in \mathscr{M}_{\rm{fin}}(T)$ there exists $\nu \in D$ such that $d_{\rm{P}}\left( {\rm{cut}}_{\bullet^R}^{h_0}\mu, \nu \right) < \frac{\varepsilon}{2}$. As $d_{\rm{V}} \leq d_{\rm{P}}$ we immediately get that also $d_{\rm{V}}^{\rm{era}} \leq d_{\rm{P}}^{\rm{era}}$, and so by the triangle inequality we conclude that $d_{\rm{V}}^{\rm{era}}(\mu, \nu) < \varepsilon$, showing that $D$ is dense, as desired.

To see that $(\mathscr{M}_{\rm{di}}(T), d_{\rm{V}}^{\rm{era}})$ is \emph{not} complete, recall that the space $(\mathscr{M}_{\rm{fin}}(T), d_{\rm{P}}^{\rm{era}})$ is also \emph{not} complete (indeed, a completion of this space is characterized in \cite{DW26} Section 3.4 as a space of finite measures on the \emph{bordification} of $T$ which are diffuse on its boundary, denoted $\mathscr{M}_{\rm{fin}}^{\rm{era}}(T)$ -- see also \cite{DW26} Example 1.8). Hence, one may construct a Cauchy sequence $(\mu_n)_{n \in \mathbb{N}}$ on $(\mathscr{M}_{\rm{fin}}(T), d_{\rm{P}}^{\rm{era}})$ which is \emph{not} convergent, more specifically satisfying for all $p$ that $\mathscr{E}_{h_p}\mu_n \stackrel{\rm{wk}}{\rightarrow} \nu_p$ as $n \rightarrow \infty$ for some $\nu_p \in \mathscr{M}_{\rm{fin}}(T)$, but where the only possible measure $\mu$ which could satisfy $\mathscr{E}_{h_p}\mu = \nu_p$ is one which is non-zero on the boundary of $T$, and so is neither an element of $\mathscr{M}_{\rm{fin}}(T)$ nor $\mathscr{M}_{\rm{di}}(T)$. From this we clearly also get that $\mathscr{E}_{h_p}\mu_n \stackrel{\rm{vg}}{\rightarrow} \nu_p$ as $n \rightarrow \infty$ for all $p$, and so $(\mu_n)_{n \in \mathbb{N}}$ is clearly also a Cauchy sequence in $(\mathscr{M}_{\rm{di}}(T), d_{\rm{V}}^{\rm{era}})$ which is not convergent.
\end{proof}
Recall the discussion from Remark \ref{rmk:OpenQuestion}. In the setting here, where we consider a fixed Polish $\mathbb{R}$-tree, $(T, d, \rho)$, it remains an open question to show that the space $(\mathscr{M}_{\rm{di}}(T), d_{\rm{V}}^{\rm{era}})$ is Lusin. 
\section{Mass erasure in the Gromov setting}
\label{sec:5}
In Section \ref{sec:4} we established limit theorems for sequences of discretely infinite measures defined on a fixed Polish $\mathbb{R}$-tree. In this context it was necessary to impose a b-tightness condition on the cut-off tree associated to the limit measure, to avoid (finite or infinite) mass escaping to infinity in the tree. Our goal in this section will be to establish analogous limit theorems for sequences of \emph{isometry classes}, $\boldsymbol{\mu}_n = [T_n, d_n, \rho_n, \mu_n] \in \mathbb{T}_{\rm{min}}^{\rm{di}}$, $n \in \mathbb{N}$, tending to some limit $\boldsymbol{\mu} = [T, d, \rho, \mu] \in \mathbb{T}_{\rm{min}}^{\rm{di}}$. As $(\boldsymbol{\mu}_n)_{n \in \mathbb{N}}$ and $\boldsymbol{\mu}$ may not be embedded into a common \emph{$\mathbb{R}$-tree}, the question of how to identify which parts of the trees in $(\boldsymbol{\mu}_n)_{n \in \mathbb{N}}$ have mass tending to infinity, becomes more complicated. A large part of this section will be spent dealing with this challenge.

Suppose that $\boldsymbol{\mu}_n \stackrel{\rm{GV}}{\rightarrow} \boldsymbol{\mu}$. The challenge of identifying the right cut-off points in $T_n$, $n \in \mathbb{N}$, is initially addressed by considering an appropriate embedding space capturing the Gromov-vague convergence as in Proposition \ref{prop:GV3}. Heuristically, we identify for each $r > 0$, the points in $\partial B_{T_n}(\rho_n, r)$ whose subtrees move within some $\eta$-distance of the infinite-mass subtrees belonging to $\boldsymbol{\mu}$ in the embedding space. This being with the technical caveats that we consider the infinite-mass subtrees above some level $r+\kappa$ for $\boldsymbol{\mu}$, and require that the subtrees need to be $\eta$-close above level $r+2\kappa$ in the embedding space. We denote the points in $\partial B_{T_n}(\rho_n, r)$ identified by this procedure by $\bullet_n^{r,\kappa,\eta} = \bullet_n^{r,\kappa,\eta}(\mu)$, and prove, among other things, that $|{\bullet_n^{r,\kappa,\eta}}(\mu)| = |{\bullet^r(\mu)}|$ eventually in Section \ref{subsec:CutOffEmbed}. The geometry in the embedding space may be utilized to prove limit theorems under the b-tightness condition, $\lim_{R \rightarrow \infty} \limsup_{n \rightarrow \infty} \mu_n\left( {\rm{Cut}}(T_n, \bullet_n^{r_k,\kappa,\eta}) \setminus \overline{B}_{T_n}(\rho_n, R) \right) = 0$ for all $k$ in some increasing sequence of radii $r_k \rightarrow \infty$.

Other than being an intricate construction, the main downside of working with the cut-off points $\bullet_n^{r,\kappa,\eta}$ is the inherent dependence on the choice of embedding space in their construction. This dependence is in particular unsuitable for studying random isometry classes, as we will do in subsequent work \cite{Draft2}. For some fixed $r > 0$, we may instead think of alternative ways of choosing or representing the cut-off points. It turns out that no matter how one chooses the cut-off points, as long as i) Gromov-vague convergence holds, ii) the number of points coincides with $|{\bullet^r}(\mu)|$, and iii) a b-tightness condition is satisfied on the cut-off trees associated to the chosen cut-off points, the cut-off points will eventually coincide with $\bullet_n^{r,\kappa,\eta}$. For applications, this means that we can construct cut-off points in any way that we like, as long as they satisfy conditions i)--iii), and that this will always produce cut-off points which coincide with $\bullet_n^{r,\kappa,\eta}$.

The idea behind the three conditions essentially boils down to an argument which roughly goes as follows. Suppose that $|{\bullet^r}(\mu)| = m$. Then Gromov-vague convergence will force $m$ points in $T_n$ to eventually have large mass above them tending to infinity (with these points precisely being $\bullet_n^{r,\kappa,\eta}$). But if one chooses $m$ cut-off points in $T_n$ in a way which does not coincide with $\bullet_n^{r,\kappa,\eta}$, then by the pigeonhole principle, assumption iii) will imply that a b-tightness estimate will necessarily have to be imposed on one of the points in $\bullet_n^{r,\kappa,\eta}$. This will contradict the fact that the mass above this point has to tend to infinity. 
\subsection{Cut-off points in embedding systems}
\label{subsec:CutOffEmbed}
Let $\boldsymbol{\mu} = [T, d, \rho, \mu]$ and $\boldsymbol{\mu}_n = [T_n, d_n, \rho_n, \mu_n]$, $n \in \mathbb{N}$, be isometry classes in $\mathbb{T}_{\rm{min}}^{\rm{di}}$, and let as usual $\bullet^r = \bullet^r(\mu)$. Let $(\mathscr{Z}, \delta, \varrho)$ be a pointed Polish space, and let $\phi \colon T \rightarrow \mathscr{Z}$, $\phi_n \colon T_n \rightarrow \mathscr{Z}$, $n \in \mathbb{N}$, be pointed isometries. Writing $\Phi = (\phi_n)_{n \in \mathbb{N}}$, we refer to the triplet $(\mathscr{Z}, \Phi, \phi)$ as an \emph{embedding system}, and we call $(\mathscr{Z}, \delta, \varrho)$ the \emph{embedding space}. We will want to define a contour area around each infinite-mass subtree of $(T, d, \rho, \mu)$ above a certain level in the embedding system in order to identify any subtrees of $(T_n, d_n, \rho_n, \mu_n)$ `moving close' to the infinite $\mu$-mass. Fix some $r > 0$ and $\kappa > 0$. For technical purposes we shall distinguish between `cut-off points' at height $r$ and $r+\kappa$, respectively. We refer to the cut-off points at height $r+\kappa$ as \emph{check points}, as they will be used to identify which points should be cut off at height $r$, and we denote them by $\circ^{r,\kappa} = \circ^{r,\kappa}(\mu) := \bullet^{r+\kappa}(\mu)$. It is easily seen that $f_{\overline{B}_T(\rho, r)}\left( \circ^{r,\kappa} \right) = \bullet^r$, with $f_{\overline{B}_T(\rho, r)}$ denoting the point-projection map onto $\overline{B}_T(\rho, r)$, so in particular $|{\bullet^r}| \leq |{\circ^{r,\kappa}}|$. For any given $r > 0$, $\kappa > 0$ and $\eta > 0$, we define the contour areas:
\begin{align*}
J_{\sigma}^{\kappa, \eta} &:= \phi(\theta_{\sigma}T)^{\eta} \setminus B_{\mathscr{Z}}(\varrho, r+2\kappa) &&\text{for every } \sigma \in \circ^{r,\kappa}, \\
\mathbf{J}_{\boldsymbol{\sigma}}^{\kappa, \eta} &:= \bigcup_{\sigma \in \circ^ {r,\kappa}: ~ f_{\overline{B}_T(\rho, r)}(\sigma) = \boldsymbol{\sigma}} J_{\sigma}^{\kappa,\eta} &&\text{for every } \boldsymbol{\sigma} \in \bullet^r.
\end{align*}
The entire contour area at level $r$ is defined by 
$$ \mathscr{J}_r^{\kappa,\eta} := \bigcup_{\sigma \in \circ^{r,\kappa}} J_{\sigma}^{\kappa,\eta} = \phi\left( T \setminus {\rm{Cut}}(T, \circ^{r,\kappa}) \right)^{\eta} \setminus B_{\mathscr{Z}}(\varrho, r+2\kappa).$$
See Figure \ref{fig:Contour} for an illustration. As ${\rm{Cut}}(T, \bullet^r) \subseteq {\rm{Cut}}(T, \bullet^{r'})$, it is easily checked that $\mathscr{J}_r^{\kappa,\eta} \subseteq \mathscr{J}_{r'}^{\kappa',\eta'}$ whenever $r \geq r'$, $\eta \leq \eta'$ and $\kappa \geq \kappa'$. Now, to determine the appropriate check- and cut-off points for $(T_n, d_n, \rho_n, \mu_n)$, $n \in \mathbb{N}$, in the embedding system, we set:
\begin{align*}
\circ_n^{r,\kappa,\eta} = \circ_n^{r,\kappa,\eta}(\mu) = \circ_n^{r,\kappa,\eta}(T_n, \mu) &:= \left\{ \sigma_n \in \partial B_{T_n}(\rho_n, r+\kappa) ~|~ \phi_n(\theta_{\sigma_n}T_n) \cap \mathscr{J}_r^{\kappa,\eta} \neq \emptyset \right\}, \\
\bullet_n^{r,\kappa,\eta} = \bullet_n^{r,\kappa,\eta}(\mu) = \bullet_n^{r,\kappa,\eta}(T_n, \mu) &:= f_{\overline{B}_{T_n}(\rho_n, r)}(\circ_n^{r,\kappa,\eta}(T_n, \mu)).
\end{align*}
Then, as before, it is easily seen that $|{\bullet_n^{r,\kappa,\eta}}| \leq |{\circ_n^{r,\kappa,\eta}}|$, noting that one or both of these quantities may in principle be infinite.

\begin{figure}[t]
\center
\includegraphics[width=0.6\textwidth]{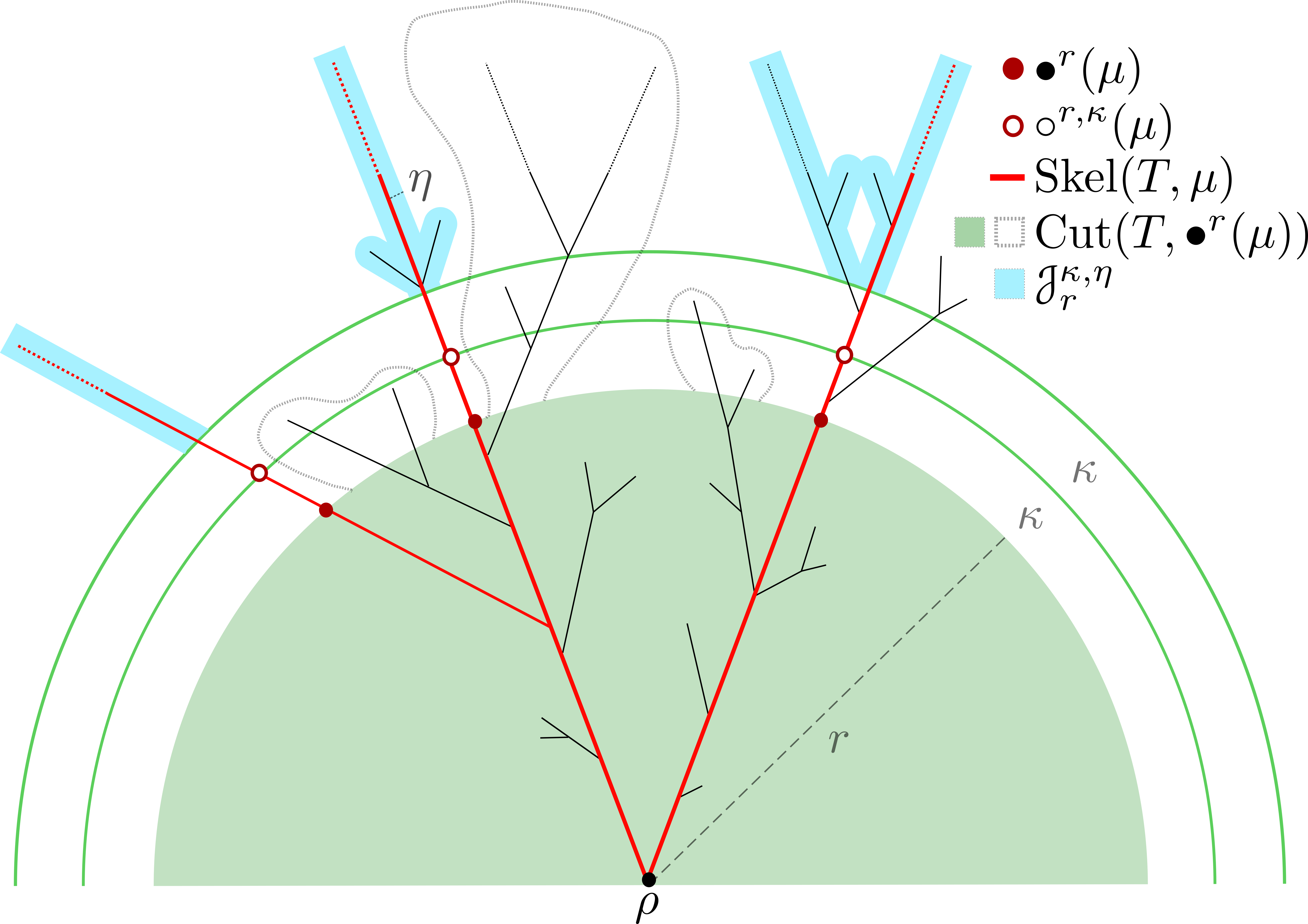}
\caption{Illustration of the cut-off points, check points and the contour area $\mathscr{J}_r^{\kappa,\eta}$ for the embedding of $(T, d, \rho, \mu)$ in some suitable space.}
\label{fig:Contour}
\end{figure}

Fix some $\boldsymbol{\mu}$-nice $r > 0$, and some $\kappa > 0$. Our first goal will be to show that under the assumption $\boldsymbol{\mu}_n \stackrel{\rm{GV}}{\rightarrow} \boldsymbol{\mu}$, whenever $\eta$ is sufficiently small and $n$ is sufficiently large it will follow in a suitable embedding system that i) there is a one-to-one correspondence between the cut-off points $\bullet_n^{r,\kappa,\eta}$ and $\bullet^r$, and ii) this correspondence is well-approximating. As ${\rm{Skel}}(T, \mu)$ is a discrete $\mathbb{R}$-tree, we can throughout this section let $\eta_0 = \eta_0(r, \kappa) > 0$ be chosen small enough that $\left(\overline{B}_T(\rho, r+2\eta_0) \setminus B_T(\rho, r)\right) \cap {\rm{Bp}}({\rm{Skel}}(T, \mu)) = \emptyset$ and 
$$ \eta_0 \leq  \frac{\kappa}{2} \wedge \left\{ \begin{array}{rl}
1 \wedge \frac{1}{4} \min_{\substack{\sigma, \sigma' \in {\rm{Cut}}(T, \bullet^r) \\ \sigma \neq \sigma'}} d(\sigma, \sigma') & \text{if } |{\rm{Cut}}(T,\bullet^r)| \geq 2, \\
1 & \text{if } |{\rm{Cut}}(T,\bullet^r)| \leq 1.
\end{array} \right. $$
Note throughout this section that our results will often be stated for some fixed choice of $n \in \mathbb{N}$, and indeed, our results hold true when comparing any suitable isometry class in $\mathbb{T}_{\rm{min}}^{\rm{di}}$ against $\boldsymbol{\mu}$. Our first result shows the existence of a correspondence map between the points $\bullet_n^{r,\kappa,\eta}$ and $\bullet^r$.
\begin{prop}[The map $u_n$]
\label{prop:MEGV1}
Let $\boldsymbol{\mu} = [T, d, \rho, \mu]$ and $\boldsymbol{\mu}_n = [T_n, d_n, \rho_n, \mu_n]$ be isometry classes in $\mathbb{T}_{\rm{min}}^{\rm{di}}$, and let $(\mathscr{Z}, \Phi, \phi)$ be an embedding system. Fix some $\boldsymbol{\mu}$-nice $r > 0$, and let $\kappa > 0$ and $\eta \in (0, \eta_0)$. Then for every $\boldsymbol{\sigma}_n \in \bullet_n^{r,\kappa,\eta}$ there exists a unique $\boldsymbol{\sigma} \in  \bullet^r$ such that 
$\phi_n(\theta_{\boldsymbol{\sigma}_n}T_n) \cap \mathbf{J}_{\boldsymbol{\sigma}}^{\kappa, \eta} \neq \emptyset$. This yields a unique map $u_n := u_n^{r, \kappa, \eta} \colon \bullet_n^{r,\kappa,\eta} \rightarrow \bullet^r$.
\end{prop}
After proving Proposition \ref{prop:MEGV1}, we will show that the map $u_n \colon \bullet_n^{r,\kappa,\eta} \rightarrow \bullet^r$ is eventually \emph{bijective} if the underlying isometry classes converge in the Gromov-vague sense. This will be proven through a series of lemmas, leading up to Proposition \ref{prop:MEGV2} -- indeed, some of these lemmas will prove useful later on. The lemmas are written out in terms of explicit estimates, and indeed Proposition \ref{prop:MEGV2} is immediately obtained from the estimates in Corollary \ref{cor:MEGV5}.
\begin{prop}[$u_n$ is bijective]
\label{prop:MEGV2}
Let $\boldsymbol{\mu} = [T, d, \rho, \mu]$ and $\boldsymbol{\mu}_n = [T_n, d_n, \rho_n, \mu_n]$, $n \in \mathbb{N}$, be isometry classes in $\mathbb{T}_{\rm{min}}^{\rm{di}}$. Fix some $\boldsymbol{\mu}$-nice $r > 0$, and let $\kappa > 0$ and $\eta \in (0, \eta_0)$. If $\boldsymbol{\mu}_n \stackrel{\rm{GV}}{\rightarrow} \boldsymbol{\mu}$, then there exists an embedding system $(\mathscr{Z}, \Phi, \phi)$ such that $u_n \colon \bullet_n^{r,\kappa,\eta} \rightarrow \bullet^r$ is eventually bijective.
\end{prop}
Finally, we will prove in Lemma \ref{lemma:MEGV6} that the embedding system $(\mathscr{Z}, \Phi, \phi)$ can be constructed in a way such that the cut-off points $\bullet_n^{r,\kappa,\eta}$ eventually approximate $\bullet^r$ well.
\begin{figure}[t]
\center
\includegraphics[width=.8\textwidth]{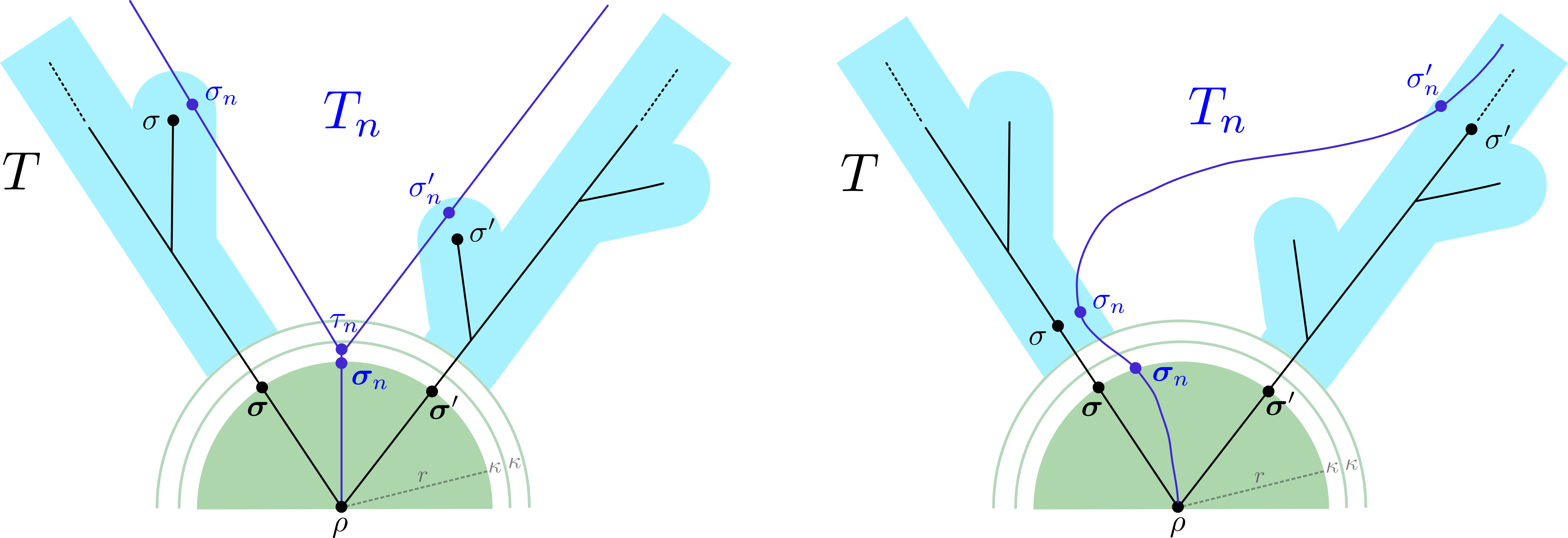}
\caption{$\tau_n \notin \{\sigma_n, \sigma_n'\}$ (left) and $\tau_n \in \{\sigma_n, \sigma_n'\}$ (right) in the proof of Proposition \ref{prop:MEGV1}.}
\label{fig:Fig3}
\end{figure}
\begin{proof}[Proof of Proposition \ref{prop:MEGV1}]
If $\bullet^r = \emptyset$, then also $\bullet_n^{r,\kappa,\eta} = \emptyset$, and so $u_n$ can be chosen to be the empty map. If $|{\bullet^r}| = 1$, then either $\bullet_n^{r,\kappa,\eta} = \emptyset$ in which case $u_n$ is chosen to be the empty map, or $\bullet_n^{r,\kappa,\eta} \neq \emptyset$ in which case $u_n$ is the constant map which maps everything into the single point of $\bullet^r$. We can thus assume for the rest of the proof that $|{\bullet^r}| \geq 2$.

Take some $\boldsymbol{\sigma}_n \in \bullet_n^{r,\kappa,\eta}$ and suppose for contradiction that there exist $\boldsymbol{\sigma}, \boldsymbol{\sigma}' \in \bullet^r$ with $\boldsymbol{\sigma} \neq \boldsymbol{\sigma}'$ such that $\phi_n(\theta_{\boldsymbol{\sigma}_n}T_n) \cap \mathbf{J}_{\boldsymbol{\sigma}}^{\kappa,\eta} \neq \emptyset$, and $\phi_n(\theta_{\boldsymbol{\sigma}_n}T_n) \cap \mathbf{J}_{\boldsymbol{\sigma}'}^{\kappa,\eta} \neq \emptyset$. Then there exist points $\sigma_n, \sigma_n' \in \theta_{\boldsymbol{\sigma}_n}T_n$, $\sigma \in \theta_{\boldsymbol{\sigma}}T \setminus B_T(\rho, r+2\kappa)$ and $\sigma' \in \theta_{\boldsymbol{\sigma}'}T \setminus B_T(\rho, r+2\kappa)$ such that $\delta(\phi_n(\sigma_n), \phi(\sigma)) \leq \eta$ and $\delta(\phi_n(\sigma_n'), \phi(\sigma')) \leq \eta$. Set $R := d(\rho, \sigma)$ and $R' := d(\rho, \sigma')$. Then we immediately have (since the roots are identified in the embedding system) that $|d_n(\rho_n, \sigma_n) - R| \leq \eta$ and $|d_n(\rho_n, \sigma_n') - R'| \leq \eta$.

Now, observe in $T$ that $d(\sigma, \sigma') = d(\sigma, \boldsymbol{\sigma}) + d(\boldsymbol{\sigma}, \boldsymbol{\sigma}') + d(\boldsymbol{\sigma}', \sigma') = R + R' - 2r + d(\boldsymbol{\sigma}, \boldsymbol{\sigma}')$. We must moreover have in the embedding space that $d(\sigma, \sigma') \leq \delta(\phi(\sigma), \phi_n(\sigma_n)) + d_n(\sigma_n, \sigma_n') + \delta(\phi_n(\sigma_n'), \phi(\sigma')) \leq 2\eta + d_n(\sigma_n, \sigma_n')$. Combining the two observations gives 
\begin{equation}
\label{eq:MEGV1}
2\eta + d_n(\sigma_n, \sigma_n') \geq R + R' - 2r + d(\boldsymbol{\sigma}, \boldsymbol{\sigma}').
\end{equation}
Let $\tau_n := \sigma_n \wedge \sigma_n'$ (see Figure \ref{fig:Fig3}). Suppose first that $\tau_n = \sigma_n$ (the argument for the case $\tau_n = \sigma_n'$ is symmetric). Then $d_n(\sigma_n, \sigma_n') = d_n(\rho_n, \sigma_n') - d_n(\rho_n, \sigma_n) \leq R' - R + 2\eta$. Hence, \eqref{eq:MEGV1} gives that $4\eta + R' - R \geq R + R' - 2r + d(\boldsymbol{\sigma}, \boldsymbol{\sigma}')$, which we can rewrite as $4\eta \geq 2(R - r) + d(\boldsymbol{\sigma}, \boldsymbol{\sigma}') \geq 4\kappa + 4\eta_0$. This violates the assumption that $\eta < \eta_0$, and so we obtain a contradiction. 

Suppose instead that $\tau_n \notin \{\sigma_n, \sigma_n'\}$. Set $r_n = d_n(\rho_n, \tau_n)$. Then $d_n(\sigma_n, \sigma_n') = d_n(\sigma_n, \tau_n) + d_n(\tau_n, \sigma_n') = d_n(\rho_n, \sigma_n) + d_n(\rho_n, \sigma_n') - 2r_n \leq R + R' - 2r_n + 2\eta$. So \eqref{eq:MEGV1} gives that $4\eta + R + R' - 2r_n \geq R + R' - 2r + d(\boldsymbol{\sigma}, \boldsymbol{\sigma}')$ which can be rewritten as $4\eta \geq 2(r_n - r) + d(\boldsymbol{\sigma}, \boldsymbol{\sigma}') \geq d(\boldsymbol{\sigma}, \boldsymbol{\sigma}') \geq 4\eta_0$ using that $r_n \geq r$. But as $\eta < \eta_0$, this of course again yields a contradiction.

We thus conclude by contradiction that there can at most exist one point $\boldsymbol{\sigma} \in \bullet^r$ such that $\phi_n(\theta_{\boldsymbol{\sigma}_n}T_n) \cap \mathbf{J}_{\boldsymbol{\sigma}}^{\kappa,\eta} \neq \emptyset$. As we also have by construction that there must exist at least one such $\boldsymbol{\sigma}$, this takes care of our statement.
\end{proof}
\begin{lemma}[Injectivity of $u_n$]
\label{lemma:MEGV3}
Let $\boldsymbol{\mu} = [T, d, \rho, \mu]$ and $\boldsymbol{\mu}_n = [T_n, d_n, \rho_n, \mu_n]$ be isometry classes in $\mathbb{T}_{\rm{min}}^{\rm{di}}$, and let $(\mathscr{Z}, \Phi, \phi)$ be an embedding system. Fix some $\boldsymbol{\mu}$-nice $r > 0$, and let $\kappa > 0$ and $\eta \in (0, \eta_0)$. Consider some $\boldsymbol{\sigma} \in \bullet^r$. Then there exists at most one $\boldsymbol{\sigma}_n \in \partial B_{T_n}(\rho_n, r)$ with $\phi_n(\theta_{\boldsymbol{\sigma}_n}T_n) \cap \mathbf{J}_{\boldsymbol{\sigma}}^{\kappa,\eta} \neq \emptyset$. In particular, the map $u_n \colon \bullet_n^{r, \kappa, \eta} \rightarrow \bullet^r$ is injective.
\end{lemma}
\begin{proof}
Suppose for contradiction that $\phi_n(\theta_{\boldsymbol{\sigma}_n}T_n) \cap \mathbf{J}_{\boldsymbol{\sigma}}^{\kappa,\eta} \neq \emptyset$ and $\phi_n(\theta_{\boldsymbol{\sigma}_n'}T_n) \cap \mathbf{J}_{\boldsymbol{\sigma}}^{\kappa,\eta} \neq \emptyset$ for two distinct $\boldsymbol{\sigma}_n, \boldsymbol{\sigma}_n' \in \partial B_{T_n}(\rho_n, r)$. 
\begin{figure}[t]
\center
\includegraphics[width=0.6\textwidth]{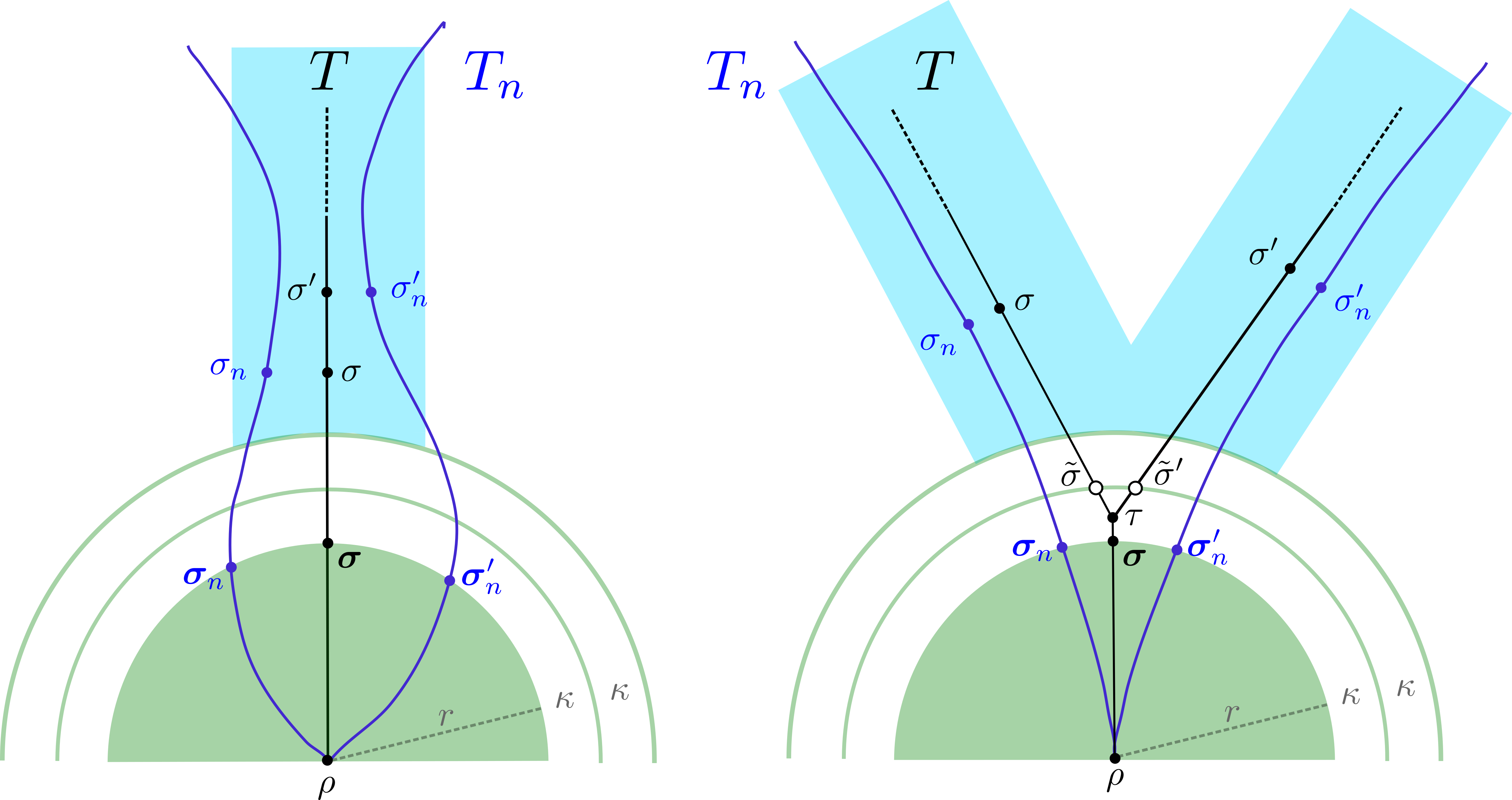}
\caption{$\tau \in \{\sigma, \sigma'\}$ (left) and $\tau \notin \{\sigma, \sigma'\}$ (right) in the proof of Lemma \ref{lemma:MEGV3}.}
\label{fig:Fig4}
\end{figure}
Then there must exist some points $\sigma, \sigma' \in \theta_{\boldsymbol{\sigma}}T \setminus B_T(\rho, r+2\kappa)$ (which are mapped into $\mathbf{J}_{\boldsymbol{\sigma}}^{\kappa,\eta}$ under $\phi$), and some points $\sigma_n \in \theta_{\boldsymbol{\sigma}_n}T_n \setminus B_{T_n}(\rho_n, r+2\kappa)$, $\sigma_n' \in \theta_{\boldsymbol{\sigma}_n'}T_n \setminus B_{T_n}(\rho_n, r+2\kappa)$, such that $\delta(\phi(\sigma), \phi_n(\sigma_n)) \leq \eta$ and $\delta(\phi(\sigma'), \phi_n(\sigma_n')) \leq \eta$. In particular, the triangle inequality gives $d_n(\sigma_n, \sigma_n') \leq 2\eta + d(\sigma, \sigma')$. Set $R := d(\rho, \sigma)$ and $R' := d(\rho, \sigma')$. Then $|d_n(\rho_n, \sigma_n) - R| \leq \eta$ and $|d_n(\rho_n, \sigma_n') - R'| \leq \eta$, and in particular $d_n(\sigma_n, \sigma_n')  = d_n(\sigma_n, \boldsymbol{\sigma}_n) + d_n(\boldsymbol{\sigma}_n, \boldsymbol{\sigma}_n') + d_n(\boldsymbol{\sigma}_n', \sigma_n') \geq R + R' - 2r - 2 \eta + d_n(\boldsymbol{\sigma}_n, \boldsymbol{\sigma}_n')$. Combining the inequalities gives
\begin{equation}
\label{eq:MEGV2a}
4\eta + d(\sigma, \sigma') \geq R + R' - 2r + d_n(\boldsymbol{\sigma}_n, \boldsymbol{\sigma}_n').
\end{equation}
Set $\tau := \sigma \wedge \sigma'$ (see Figure \ref{fig:Fig4}). Then $\tau \in \theta_{\boldsymbol{\sigma}}T$, and we can set $\tilde{r} := d(\rho, \tau)$ such that $r \leq \tilde{r} \leq \min\{R, R'\}$. Suppose for contradiction that $\tau \in \{\sigma, \sigma'\}$ (by symmetry we can assume without loss of generality that $\tau = \sigma$). Then $d(\sigma, \sigma') = R' - R$, and so \eqref{eq:MEGV2a} gives $4\eta \geq 2(R' - r) + d_n(\boldsymbol{\sigma}_n, \boldsymbol{\sigma}_n') \geq 4\kappa$ as $R' \geq r+2\kappa$. This contradicts the assumption that $\eta \in (0, \eta_0)$, and so we must have that $\tau \notin \{\sigma, \sigma'\}$. Hence, we see that $d(\sigma, \sigma') = d(\sigma, \tau) + d(\tau, \sigma') = R + R' - 2\tilde{r}$, and so \eqref{eq:MEGV2a} gives
\begin{equation}
\label{eq:MEGV2b}
4\eta \geq 2(\tilde{r} - r) + d_n(\boldsymbol{\sigma}_n, \boldsymbol{\sigma}_n').
\end{equation}
Now, as $\boldsymbol{\sigma}_n \neq \boldsymbol{\sigma}_n'$ we get from \eqref{eq:MEGV2b} that $\tilde{r} < r + \kappa$. As $\sigma, \sigma'$ are mapped into $\mathbf{J}_{\boldsymbol{\sigma}}^{\kappa,\eta}$ under $\phi$, there must exist check points $\tilde{\sigma}, \tilde{\sigma}' \in \circ^{r,\kappa}$ such that $\sigma \in \theta_{\tilde{\sigma}}T$ and $\sigma' \in \theta_{\tilde{\sigma}'}T$. But as the check points are located at level $r+\kappa$, this must necessarily imply that $\tilde{\sigma}, \tilde{\sigma}' \in \theta_{\tau}T$, since $\tau = \sigma \wedge \sigma'$. As $\tilde{\sigma}, \tilde{\sigma}' \in {\rm{Skel}}(T, \mu)$, this gives that $\tau \in {\rm{Bp}}({\rm{Skel}}(T, \mu))$. But now by \eqref{eq:MEGV2b}, $2\eta_0 > 2\eta \geq d(\boldsymbol{\sigma}, \tau)$. This yields a contradiction, as $\left(\overline{B}_T(\rho, r+2\eta_0) \setminus B_T(\rho, r)\right) \cap {\rm{Bp}}({\rm{Skel}}(T, \mu) =\emptyset$.
\end{proof}
Given $r > 0$ and $\kappa > 0$, fix some constant $K > 0$. We let $R_0(K) = R_0(r, \kappa, K)$ be chosen large enough such that $\mu(\theta_{\boldsymbol{\sigma}}T \cap \overline{B}_T(\rho, R- \kappa) \setminus B_T(\rho, r+3\kappa)) \geq K$ for all $R \geq R_0(K)$ and $\boldsymbol{\sigma} \in \bullet^r$.
\begin{lemma}
\label{lemma:MEGV4}
Let $\boldsymbol{\mu} = [T, d, \rho, \mu]$ and $\boldsymbol{\mu}_n = [T_n, d_n, \rho_n, \mu_n]$ be isometry classes in $\mathbb{T}_{\rm{min}}^{\rm{di}}$. Fix $r > 0$, $\kappa > 0$ and $\eta \in (0, \kappa)$, and let $K > 0$. Let $R \geq R_0(K)$. If $(\mathscr{Z}, \Phi, \phi)$ is an embedding system in which $\delta_{\rm{P}}( ( \mu_n \circ \phi_n^{-1} )^{\restr R}, ( \mu \circ \phi^{-1})^{\restr R}) \leq \eta$, then
\begin{equation}
\label{eq:MEGV4}
\mu_n \circ \phi_n^{-1}\left( \mathbf{J}_{\boldsymbol{\sigma}}^{\kappa,\eta} \cap \overline{B}_{\mathscr{Z}}(\varrho, R) \right) + \eta \geq K \quad \text{for all } \boldsymbol{\sigma} \in \bullet^r.
\end{equation}
\end{lemma}
\begin{proof}
Fix some $\boldsymbol{\sigma} \in \bullet^r$, and set $A = \phi(\theta_{\boldsymbol{\sigma}}T) \cap \overline{B}_{\mathscr{Z}}(\varrho, R- \eta) \setminus B_{\mathscr{Z}}(\varrho, r+2\kappa+\eta)$. By our choice of $R$, we immediately see as $\eta < \kappa$, that $\left( \mu \circ \phi^{-1} \right)^{\restr R}(A) = \mu \circ \phi^{-1}(A) \geq K$. Moreover, $A^{\eta} = \phi(\theta_{\boldsymbol{\sigma}}T)^{\eta} \cap \overline{B}_{\mathscr{Z}}(\varrho, R) \setminus B_{\mathscr{Z}}(\varrho, r+2\kappa) = \mathbf{J}_{\boldsymbol{\sigma}}^{\kappa,\eta} \cap \overline{B}_{\mathscr{Z}}(\varrho, R)$, and so by definition of the Prokhorov distance, $K \leq (\mu \circ \phi^{-1})^{\restr R}(A) \leq \mu_n \circ \phi_n^{-1}(\mathbf{J}_{\boldsymbol{\sigma}}^{\kappa,\eta} \cap \overline{B}_{\mathscr{Z}}(\varrho, R)) + \eta$, as desired.
\end{proof}
Note by Lemma \ref{lemma:GV2}$(ii)$ that if $d_{\rm{GV}}(\boldsymbol{\mu}_n, \boldsymbol{\mu}) \leq \frac{\eta}{2} e^{-R_0(K)}$, then there exists an embedding system $(\mathscr{Z}, \Phi, \phi)$ and some $R \geq R_0(K)$ such that $\delta_{\rm{P}}((\mu_n \circ \phi_n^{-1})^{\restr R}, (\mu \circ \phi^{-1})^{\restr R}) \leq \eta$.
\begin{cor}[Bijectivity of $u_n$]
\label{cor:MEGV5}
Let $\boldsymbol{\mu} = [T, d, \rho, \mu]$ and $\boldsymbol{\mu}_n = [T_n, d_n, \rho_n, \mu_n]$ be isometry classes in $\mathbb{T}_{\rm{min}}^{\rm{di}}$. Fix some $\boldsymbol{\mu}$-nice $r > 0$, and let $\kappa > 0$, $\eta \in (0, \eta_0)$, and $R \geq R_0(\eta_0)$. If $(\mathscr{Z}, \Phi, \phi)$ is an embedding system with $\delta_{\rm{P}}((\mu_n \circ \phi_n^{-1})^{\restr R}, (\mu \circ \phi^{-1})^{\restr R}) \leq \eta$, then $u_n \colon \bullet_n^{r,\kappa,\eta} \rightarrow \bullet^r$ is bijective.
\end{cor}
\begin{proof}
As we have chosen $K = \eta_0$, it follows immediately from Lemma \ref{lemma:MEGV4} that for all $\boldsymbol{\sigma} \in \bullet^r$, $\phi_n(\mathbf{J}_{\boldsymbol{\sigma}}^{\kappa,\eta} \cap \overline{B}_{\mathscr{Z}}(\varrho, R)) \geq \eta_0 - \eta > 0$, implying in particular that $\phi_n(T_n) \cap \mathbf{J}_{\boldsymbol{\sigma}}^{\kappa,\eta} \neq \emptyset$, giving the existence of some $\boldsymbol{\sigma}_n \in \bullet_n^{r,\kappa,\eta}$ with $u_n(\boldsymbol{\sigma}_n) = \boldsymbol{\sigma}$.
\end{proof}
Corollary \ref{cor:MEGV5} implies in particular that if $\boldsymbol{\mu}_n \stackrel{\rm{GV}}{\rightarrow} \boldsymbol{\mu}$ then for $R$ chosen as above and applying Lemma \ref{lemma:GV2}$(ii)$, it eventually holds that \eqref{eq:MEGV4} is true and $u_n$ is bijective, proving Proposition \ref{prop:MEGV2}. Moreover, bijectivity of course ensures that $|{\bullet_n^{r,\kappa,\eta}}| = |{\bullet^r}| < \infty$.
\begin{lemma}
\label{lemma:MEGV6}
Let $\boldsymbol{\mu} = [T, d, \rho, \mu]$ and $\boldsymbol{\mu}_n = [T_n, d_n, \rho_n, \mu_n]$ be isometry classes in $\mathbb{T}_{\rm{min}}^{\rm{di}}$. Fix some $\boldsymbol{\mu}$-nice $r > 0$, and let $\kappa > 0$ and $\eta \in (0, \eta_0)$. Let $R \geq R_0(2\eta_0)$, and suppose that $d_{\rm{GV}}(\boldsymbol{\mu}_n, \boldsymbol{\mu}) \leq \frac{\eta}{20} e^{-R}$. Then there exists some $R' \geq R$ and an embedding system $(\mathscr{Z}, \Phi, \phi)$ in which $\delta_{\rm{P}}((\mu_n \circ \phi_n^{-1})^{\restr R'}, (\mu \circ \phi^{-1})^{\restr R'}) \leq \eta$ and $\delta(\phi(\boldsymbol{\sigma}), \phi_n(u_n(\boldsymbol{\sigma}))) \leq 6\eta$ for every $\boldsymbol{\sigma} \in \bullet^r$.
\end{lemma}
\begin{proof}
By Lemma \ref{lemma:BFMeas3}$(ii)$ there exists some $R' \geq R$ such that $d_{\rm{GP}}(\boldsymbol{\mu}_n^{\restr R'}, \boldsymbol{\mu}^{\restr R'}) \leq \frac{\eta}{20} (1 + e^{-R}) \leq \frac{\eta}{10}$. Applying \cite{DW26} Proposition 4.15 in the same way as in the proof of \cite{DW26} Theorem 4.21, we get the existence of compact rooted subtrees $T' \subseteq T$ and $T_n' \subseteq T_n$ with $\mu^{\restr R'}(T \setminus T') \vee \mu_n^{\restr R'}(T_n \setminus T_n') \leq \frac{2\eta}{10}$, and an embedding system\footnote{Note that we have multiplied the Gromov--Prokhorov and Gromov--Hausdorff estimates by two, as we wish to identify the roots in the embedding system. Recall the discussion from Section \ref{subsec:Gromov}.} $(\mathscr{Z}, \Phi, \phi)$ such that $\delta_P(\mu_n^{\restr R'} \circ \phi_n^{-1}, \mu^{\restr R'} \circ \phi^{-1}) \leq \eta$ and $\delta_H(\phi_n(T_n'), \phi(T')) \leq \eta$. Note that the embedding system $(\mathscr{Z}, \Phi, \phi)$ only embeds the finite isometry classes $\boldsymbol{\mu}_n^{\restr R'}$ and $\boldsymbol{\mu}^{\restr R'}$, but that it can easily be extended to an embedding space of $\boldsymbol{\mu}_n$ and $\boldsymbol{\mu}$ (with no guarantee, however, that $\mu_n \circ \phi_n^{-1}$ and $\mu \circ \phi^{-1}$ are close above level $R'$). Let $\bullet_n^{r,\kappa,\eta}$ be the cut-off points in this embedding system. Then by Corollary \ref{cor:MEGV5} the map $u_n \colon \bullet_n^{r,\kappa,\eta} \rightarrow \bullet^r$ is bijective. Fix some $\boldsymbol{\sigma} \in \bullet^r$ and $\boldsymbol{\sigma}_n = u_n(\boldsymbol{\sigma})$. Then by Lemma \ref{lemma:MEGV4},
$$\mu_n\left(\theta_{\boldsymbol{\sigma}_n}T_n \cap \overline{B}_T(\rho, R') \setminus B_T(\rho, r+2\kappa) \right) \geq \mu_n \circ \phi_n^{-1}\left( \mathbf{J}_{\boldsymbol{\sigma}}^{\kappa,\eta} \cap \overline{B}_{\mathscr{Z}}(\varrho, R') \right) \geq 2\eta_0 - \eta > \eta > \frac{2}{10}\eta.$$ 
So indeed, we must have that $\boldsymbol{\sigma} \in T'$ and $\boldsymbol{\sigma}_n \in T_n'$, and moreover, there must exist points $\sigma \in T' \cap \theta_{\boldsymbol{\sigma}}T \cap \overline{B}_T(\rho, R') \setminus B_T(\rho, r+2\kappa)$ and $\sigma_n \in T_n' \cap \theta_{\boldsymbol{\sigma}_n}T_n \cap \overline{B}_{T_n}(\rho_n, R') \setminus B_{T_n}(\rho_n, r+2\kappa)$ such that $\delta(\phi(\sigma), \phi_n(\sigma_n)) \leq \eta$. Now, as $\delta_H(\phi_n(T_n'), \phi(T')) \leq \eta$, every point in the segment $\phi_n([\boldsymbol{\sigma}_n, \sigma_n])$ has to be within distance $\eta$ of $T'$ at all times. This will of course immediately be the case if every point in $\phi_n([\boldsymbol{\sigma}_n, \sigma_n])$ is at distance at most $\eta$ from some point in the subtree $\phi(\theta_{\boldsymbol{\sigma}}T')$ (which of course includes the segment $\phi([\boldsymbol{\sigma}, \sigma])$), in which case it is easily checked that $\delta(\phi(\boldsymbol{\sigma}), \phi_n(\boldsymbol{\sigma}_n)) \leq 2 \eta$. Suppose now instead that some points in $\phi_n([\boldsymbol{\sigma}_n, \sigma_n])$ are more than distance $\eta$ away from $\phi(\theta_{\boldsymbol{\sigma}}T')$ (see Figure \ref{fig:Fig5}). Then any such points must be within distance $\eta$ from $T' \setminus \theta_{\boldsymbol{\sigma}}T'$. In particular, there must exist some `intersection point' $\sigma_n' \in [\boldsymbol{\sigma}_n, \sigma_n]$ for the two respective $\eta$-neighbourhoods, such that for some $\sigma' \in \theta_{\boldsymbol{\sigma}}T'$ and some $\tilde{\sigma}' \in T' \setminus \theta_{\boldsymbol{\sigma}}T'$ it holds that $\delta(\phi_n(\sigma_n'), \phi(\sigma')) \leq \eta$ and $\delta(\phi_n(\sigma_n'), \phi(\tilde{\sigma}')) \leq \eta$. Denote by $\tilde{\boldsymbol{\sigma}}$ the unique point in $\partial B_T(\rho, r)$ such that $\tilde{\boldsymbol{\sigma}} \in [\rho, \tilde{\sigma}']$. Then we must have that $d(\sigma', \tilde{\sigma}') \leq 2 \eta$, but at the same time we know that $d(\sigma', \tilde{\sigma}') = d(\sigma', \boldsymbol{\sigma}) + d(\boldsymbol{\sigma}, \tilde{\boldsymbol{\sigma}}) + d(\tilde{\boldsymbol{\sigma}}, \tilde{\sigma}')$. In particular, setting $r' = d(\rho, \sigma')$, we get that $r' - r = d(\sigma', \boldsymbol{\sigma}) \leq 2\eta$ and also $d_n(\boldsymbol{\sigma}_n, \sigma_n') \leq r' - r +\eta \leq 3\eta$. Hence, we conclude that $\delta(\phi_n(\boldsymbol{\sigma}_n), \phi(\boldsymbol{\sigma})) \leq d_n(\boldsymbol{\sigma}_n, \sigma_n') + \delta(\phi_n(\sigma_n'), \phi(\sigma'))  + d(\sigma', \boldsymbol{\sigma}) \leq 6\eta$, as desired.
\begin{figure}[t]
\center
\includegraphics[width=0.6\textwidth]{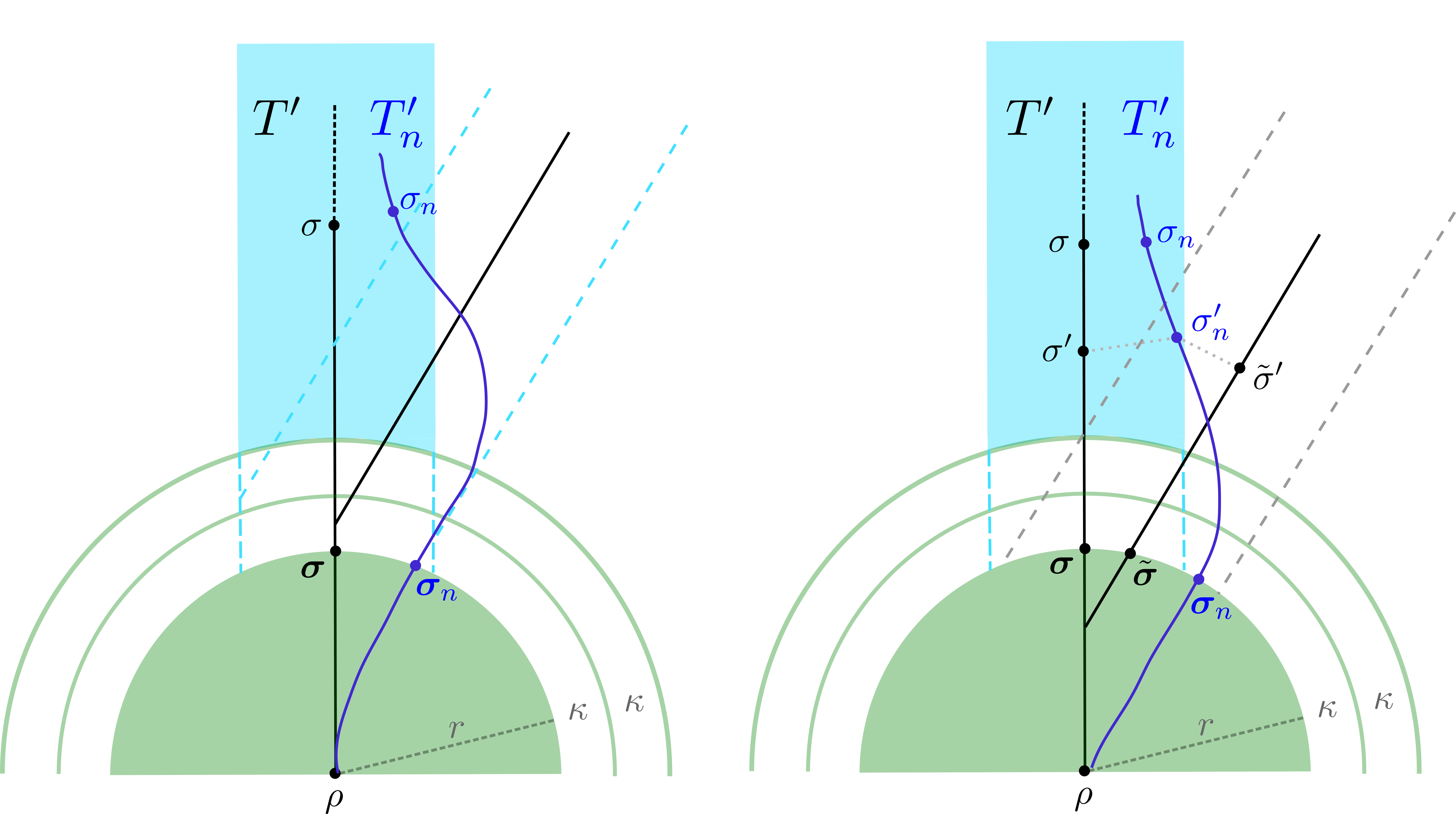}
\caption{$\phi_n([\boldsymbol{\sigma}_n, \sigma_n]) \subseteq \phi(\theta_{\boldsymbol{\sigma}}T')^{\eta}$ (left) and $\phi(\theta_{\boldsymbol{\sigma}}T')^{\eta} \setminus \phi_n([\boldsymbol{\sigma}_n, \sigma_n]) \neq \emptyset$ (right) in the proof of Lemma \ref{lemma:MEGV6}}
\label{fig:Fig5}
\end{figure} 
\end{proof}
\subsection{General cut-off points and b-tightness conditions}
\label{subsec:GeneralCutOff}
We may now prove that the three conditions outlined at the beginning of Section \ref{sec:5} are sufficient for any collection of cut-off points $\blackdiamond_n^r$ to coincide with the points $\bullet_n^{r,\kappa,\eta}$ from Section \ref{subsec:CutOffEmbed}. Recall the notion of \emph{compatible} cut-off points from Section \ref{subsec:MainConvME}. Recall also how the notion of b-tightness extends to the representative measures for a sequence of isometry classes.
\begin{prop}
\label{prop:MEGV8}
Let $\boldsymbol{\mu} = [T, d, \rho, \mu]$ and $\boldsymbol{\mu}_n = [T_n, d_n, \rho_n, \mu_n]$, $n \in \mathbb{N}$, be isometry classes in $\mathbb{T}_{\rm{min}}^{\rm{di}}$. Fix some $\boldsymbol{\mu}$-nice $r > 0$, and let $\kappa > 0$ and $\eta \in (0, \eta_0)$. For each $n \in \mathbb{N}$, let $\blackdiamond_n^r$ be $\bullet^r$-compatible cut-off points in $\partial B_{T_n}(\rho_n, r)$. If
\begin{enumerate}[$(i)$]
\item $\boldsymbol{\mu}_n \stackrel{\rm{GV}}{\rightarrow} \boldsymbol{\mu}$ as $n \rightarrow \infty$, and
\item $\left( \mu_n|_{{\rm{Cut}}(T_n, \blackdiamond_n^r)} \right)_{n \in \mathbb{N}}$ is b-tight,
\end{enumerate}
then for any embedding system $(\mathscr{Z}, \Phi, \phi)$ realizing condition $(ii)$, we have $\blackdiamond_n^r = \bullet_n^{r,\kappa,\eta}$ eventually.
\end{prop}
We first prove an exact estimate.
\begin{lemma}
\label{lemma:MEGV7}
Let $\boldsymbol{\mu} = [T, d, \rho, \mu]$ and $\boldsymbol{\mu}_n = [T_n, d_n, \rho_n, \mu_n]$ be isometry classes in $\mathbb{T}_{\rm{min}}^{\rm{di}}$. Fix some $\boldsymbol{\mu}$-nice $r > 0$, and let $\kappa > 0$ and $\eta \in (0, \eta_0)$. Let $\blackdiamond_n^r$ be $\bullet^r$-compatible cut-off points in $\partial B_{T_n}(\rho_n, r)$, and suppose that 
\begin{enumerate}[$(i)$]
\item $\mu_n\left( {\rm{Cut}}(T_n, \blackdiamond_n^r) \setminus \overline{B}_{T_n}(\rho_n, R_0) \right) \leq \eta_0$ for some $R_0 > 0$, and
\item $d_{\rm{GP}}\left( \boldsymbol{\mu}_n^{\restr R}, \boldsymbol{\mu}^{\restr R} \right) \leq \eta$ for some $R > R_0 \vee R_0(K)$ where $K = \mu(\overline{B}_T(\rho, R_0 + \eta_0)) + 3\eta_0$. 
\end{enumerate}
Then for any embedding system $(\mathscr{Z}, \Phi, \phi)$ realizing condition $(ii)$, it holds that $\blackdiamond_n^r = \bullet_n^{r,\kappa,\eta}$.
\end{lemma}
\begin{proof}
Let $(\mathscr{Z}, \Phi, \phi)$ be an embedding system realizing condition $(ii)$, meaning that it satisfies $d_{\rm{P}}((\mu_n \circ \phi_n^{-1})^{\restr R}, (\mu \circ \phi^{-1})^{\restr R}) \leq \eta$. Then by Corollary \ref{cor:MEGV5}, the map $u_n \colon \bullet_n^{r,\kappa,\eta} \rightarrow \bullet^r$ is bijective, so in particular $|{\bullet_n^{r,\kappa,\eta}}| = |{\bullet^r}|$. This moreover gives that $|\partial B_{T_n}(\rho_n, r)| \geq |{\bullet^r}|$, so by construction we must also have $|{\blackdiamond_n^r}| = |{\bullet^r}|$. Suppose for contradiction that $\blackdiamond_n^r \neq \bullet_n^{r,\kappa,\eta}$. Then as $|{\blackdiamond_n^r}| = |{\bullet_n^{r,\kappa,\eta}}|$ there must necessarily exist some $\boldsymbol{\sigma}_n \in \bullet_n^{r,\kappa,\eta} \setminus \blackdiamond_n^r$. By construction, this point must satisfy that $\theta_{\boldsymbol{\sigma}_n}T_n \subseteq {\rm{Cut}}(T_n, \blackdiamond_n^r)$, and so condition $(i)$ gives $\mu_n(\theta_{\boldsymbol{\sigma}_n}T_n \setminus \overline{B}_{T_n}(\rho_n, R_0)) \leq \eta_0$. At the same time, as $R \geq R_0(K)$, Lemma \ref{lemma:MEGV4} gives that $\mu_n(\theta_{\boldsymbol{\sigma}_n} T_n \cap \overline{B}_{T_n}(\rho_n, R)) \geq \mu_n \circ \phi_n^{-1}(\mathbf{J}_{\boldsymbol{\sigma}}^{\kappa,\eta} \cap \overline{B}_{\mathscr{Z}}(\varrho, R)) \geq K-\eta$ for $\boldsymbol{\sigma} = u_n^{-1}(\boldsymbol{\sigma}_n)$. But then using that $\mu_n(\theta_{\boldsymbol{\sigma}_n}T_n \cap \overline{B}_{T_n}(\rho_n, R)) \leq \mu_n(\overline{B}_{T_n}(\rho_n, R_0)) + \eta_0$ and applying the definition of $K$, we conclude that $\mu_n\left( \overline{B}_{T_n}(\rho_n, R_0) \right) > \mu\left( \overline{B}_T(\rho, R_0 + \eta_0) \right) + \eta_0$. This clearly violates condition $(ii)$, and so we conclude by contradiction that $\blackdiamond_n^r = \bullet_n^{r,\kappa,\eta}$.
\end{proof}
\begin{proof}[Proof of Proposition \ref{prop:MEGV8}]
By condition $(ii)$ $\limsup_{n \rightarrow \infty} \mu_n\left( {\rm{Cut}}(T_n, \blackdiamond_n^r) \setminus \overline{B}_{T_n}(\rho_n, R_0) \right) < \eta_0$ for some $R_0$ sufficiently large. Hence, we can find some $N_0$ sufficiently large such that for all $n \geq N_0$ it holds that $\mu_n\left( {\rm{Cut}}(T_n, \blackdiamond_n^r) \setminus \overline{B}_{T_n}(\rho_n, R_0) \right) \leq \eta_0$. Let $R > R_0 \wedge R_0(K)$ with $K$ given as in Lemma \ref{lemma:MEGV7}. Then condition $(i)$ gives in particular that $\boldsymbol{\mu}_n^{\restr R} \stackrel{\rm{GP}}{\rightarrow} \boldsymbol{\mu}^{\restr R}$, and so we can find some $N \geq N_0$ sufficiently large that $d_{\rm{GP}}(\boldsymbol{\mu}_n^{\restr R}, \boldsymbol{\mu}^{\restr R}) \leq \eta$ for all $n \geq N$. As all $n \geq N$ now satisfy conditions $(i)$ and $(ii)$ from Lemma \ref{lemma:MEGV7}, the conclusion follows.
\end{proof}
\subsection{Local Gromov--Hausdorff convergence of mass-erased subtrees}
The aim of this section is to prove local Gromov-Hausdorff convergence of mass-erased subtrees (Theorem \ref{thm:MEGV14}), and an analogous limit theorem for projected measures (Proposition \ref{prop:MEGV16}). We start by proving Gromov--Prokhorov convergence of the (isometry classes of) cut-off measures (Theorem \ref{thm:MEGV12}) through a series of lemmas. Observe that the contents of Theorem \ref{thm:MEGV12} are written out in terms of explicit estimates in Lemma \ref{lemma:MEGV11}. Going forward we will often write out explicit estimates as lemmas, as these will be useful in our companion paper \cite{Draft2}.
\begin{lemma}
\label{lemma:MEGV9}
Let $\boldsymbol{\mu} = [T, d, \rho, \mu]$ and $\boldsymbol{\mu}_n = [T_n, d_n, \rho_n, \mu_n]$ be isometry classes in $\mathbb{T}_{\rm{min}}^{\rm{di}}$, and let $(\mathscr{Z}, \Phi, \phi)$ be an embedding system. Fix some $\boldsymbol{\mu}$-nice $r > 0$, and let $\kappa > 0$ and $\eta \in (0, \frac{\kappa}{2})$. Then
\begin{align*}
\phi^{-1}\left(\phi_n\left( {\rm{Cut}}(T_n, \bullet_n^{r,\kappa,\eta}) \right)^{\eta}\right)  &\subseteq {\rm{Cut}}(T, \bullet^{r+\kappa}) \cup B_{T}(\rho, r+2\kappa) &&\text{and} \\
\phi_n^{-1}\left(\phi\left( {\rm{Cut}}(T, \bullet^r) \right)^{\eta}\right) &\subseteq {\rm{Cut}}(T_n, \bullet_n^{r+\kappa,\kappa,\eta}) \cup B_{T_n}(\rho_n, r+2\kappa). &&
\end{align*}
\end{lemma}
\begin{proof}
The first inclusion holds since $\phi_n\left( {\rm{Cut}}(T_n, \bullet_n^{r,\kappa,\eta}) \right)^{\eta} \cap \mathscr{J}_r^{\kappa,\eta} = \emptyset$ and $\phi^{-1}\left( \mathscr{Z} \setminus \mathscr{J}_r^{\kappa,\eta} \right) = {\rm{Cut}}(T, \bullet^{r+\kappa}) \cup B_T(\rho, r+2\kappa)$. For the second inclusion, suppose for contradiction that there exists some $\sigma_n' \in T_n \setminus \left( {\rm{Cut}}(T_n, \bullet_n^{r+\kappa,\kappa,\eta}) \cup B_{T_n}(\rho_n, r+2\kappa) \right)$ and $\sigma' \in {\rm{Cut}}(T, \bullet^r)$ such that $\delta(\phi(\sigma'), \phi_n(\sigma_n')) \leq \eta$. Denote by $\boldsymbol{\sigma}'$ the point at level $r$ in $[\rho, \sigma']$, and denote by $\boldsymbol{\sigma}_n$ the point at level $r$ in $[\rho, \sigma_n']$. As $\sigma_n' \notin {\rm{Cut}}(T_n, \bullet_n^{r+\kappa,\kappa,\eta})$ and it is above level $r+2\kappa$, we must have that $[\boldsymbol{\sigma}_n, \sigma_n'] \cap \bullet_n^{r+\kappa, \kappa, \eta} \neq \emptyset$. As $\bullet_n^{r+\kappa, \kappa, \eta} \subseteq \circ_n^{r,\kappa,\eta}$ this gives that $[\boldsymbol{\sigma}_n, \sigma_n'] \cap \circ_n^{r, \kappa, \eta} \neq \emptyset$, and so $\boldsymbol{\sigma}_n \in \bullet_n^{r,\kappa,\eta}$. Denote by $\tilde{\sigma}_n$ the check point in $\circ_n^{r,\kappa,\eta}$ intersecting $[\boldsymbol{\sigma}_n, \sigma_n']$. See Figure \ref{fig:Fig6}.

As $\boldsymbol{\sigma}_n \in \bullet_n^{r,\kappa,\eta}$ there must exist some $\boldsymbol{\sigma} \in \bullet^r$ such that $\phi_n(\theta_{\boldsymbol{\sigma}_n}T_n) \cap \mathbf{J}_{\boldsymbol{\sigma}}^{\kappa,\eta} \neq \emptyset$. Indeed, as $\tilde{\sigma}_n$ is a check point, there must be some check point $\tilde{\sigma} \in \circ^{r,\kappa} \cap \theta_{\boldsymbol{\sigma}}T$ and some points  $\sigma \in \theta_{\tilde{\sigma}}T \setminus B_T(\rho, r+2\kappa)$ and $\sigma_n \in \theta_{\tilde{\sigma}_n}T_n \setminus B_T(\rho, r+2\kappa)$ with $\delta(\phi(\sigma), \phi_n(\sigma_n)) \leq \eta$. Set $\tau_n = \sigma_n \wedge \sigma_n'$, and denote by $r_{\sigma} = d(\rho, \sigma)$ the level of the point $\sigma$ (and similarly for every other point in the construction). Then we easily see that $d(\sigma, \sigma') = r_{\sigma} + r_{\sigma'} - 2r + d(\boldsymbol{\sigma}, \boldsymbol{\sigma}')$. At the same time, we have that $d(\sigma, \sigma') \leq \delta(\phi(\sigma), \phi_n(\sigma_n)) + d_n(\sigma_n, \sigma_n') + \delta(\phi_n(\sigma_n'), \phi(\sigma')) \leq 2\eta + d_n(\sigma_n, \tau_n) + d_n(\tau_n, \sigma_n') = 2\eta + r_{\sigma_n} + r_{\sigma_n'} - 2r_{\tau_n}$. Combining the two estimates gives $2(r_{\tau_n} - r) + d(\boldsymbol{\sigma}, \boldsymbol{\sigma}') \leq 4\eta$. 

Now, note that $\tilde{\sigma}_n$ is a common ancestor of $\sigma_n$ and $\sigma_n'$. So $r_{\tau_n} \geq r_{\tilde{\sigma}_n} = r+\kappa$. But inserting this in the aforementioned estimate gives $2\kappa + d(\boldsymbol{\sigma}, \boldsymbol{\sigma}') \leq 4\eta$, which clearly contradicts the assumption $\eta < \frac{\kappa}{2}$.
\begin{figure}[t]
\center
\includegraphics[width=0.8\textwidth]{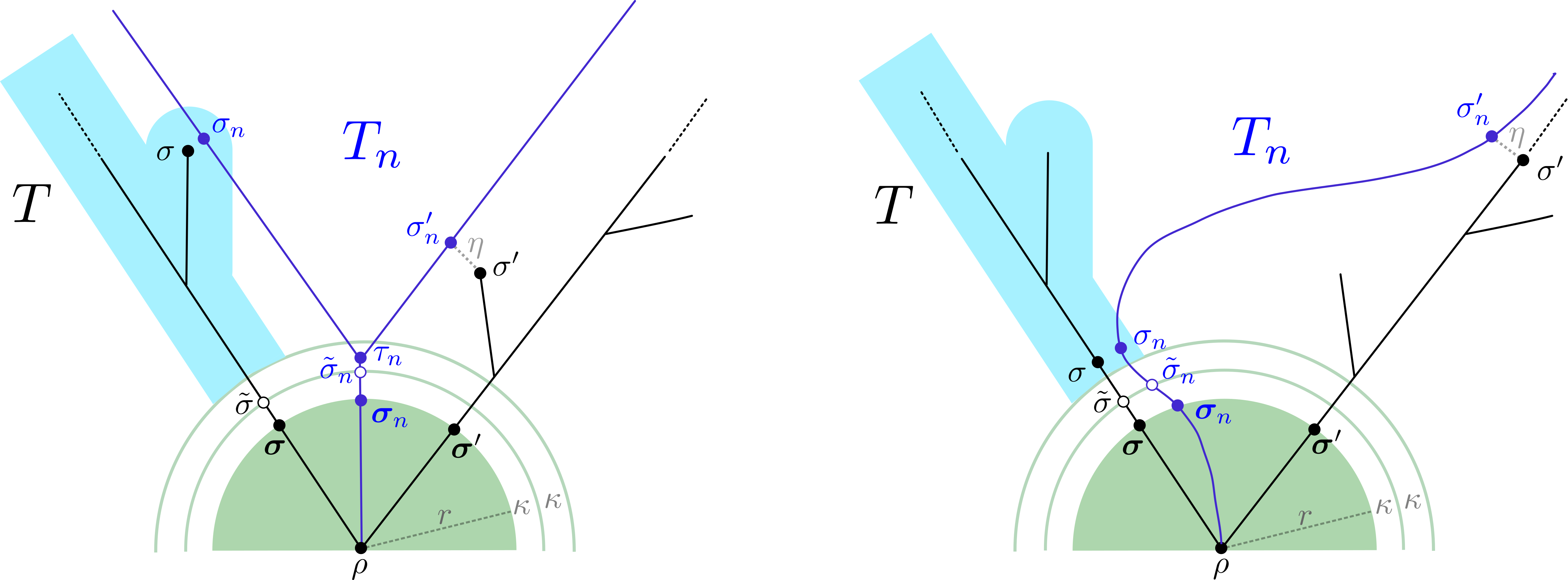}
\caption{$\tau_n \in \{\sigma_n, \sigma_n'\}$ (left) or $\tau_n \notin \{\sigma_n, \sigma_n'\}$ (right) in the proof of Lemma \ref{lemma:MEGV9}}
\label{fig:Fig6}
\end{figure}
\end{proof}
\begin{rmk}
\label{rmk:MEGV9}
Note that the radius $r$ in Lemma \ref{lemma:MEGV9} may of course be shifted by $-\kappa$. More specifically, if we have that $r- \kappa > 0$ is $\boldsymbol{\mu}$-nice, and $\eta \in \left( 0, \frac{\kappa}{2} \right)$, then $ \phi_n^{-1}\left(\phi\left( {\rm{Cut}}(T, \bullet^{r-\kappa}) \right)^{\eta}\right) \subseteq {\rm{Cut}}(T_n, \bullet_n^{r,\kappa,\eta}) \cup B_{T_n}(\rho_n, r+\kappa)$.
\demo
\end{rmk}
\begin{lemma}
\label{lemma:MEGV10}
Let $(T, d, \rho)$ be a Polish $\mathbb{R}$-tree with $\mu \in \mathscr{M}_{\rm{di}}(T)$, and let $r > 0$ be $\mu$-nice. Then for every $\varepsilon > 0$ there exists some $\kappa > 0$ sufficiently small that
$$ d_{\rm{P}}\left( \mu|_{{\rm{Cut}}(T, \bullet^{r-\kappa})}, \mu|_{{\rm{Cut}}(T, \bullet^r)} \right) \leq \varepsilon \quad \text{and} \quad d_{\rm{P}}\left( \mu|_{{\rm{Cut}}(T, \bullet^r)}, \mu|_{{\rm{Cut}}(T, \bullet^{r+\kappa})} \right) \leq \varepsilon. $$
\end{lemma}
\begin{proof}
We may assume without loss of generality that ${\rm{Skel}}(T, \mu) \neq \{\rho\}$, such that $|{\bullet^r}| \geq 1$, as the result is otherwise trivial. Let $A \subseteq T$ be a closed set. As ${\rm{Cut}}(T, \bullet^{r-\kappa}) \subseteq {\rm{Cut}}(T, \bullet^r) \subseteq {\rm{Cut}}(T, \bullet^{r+\kappa})$ we just have to argue that $\mu|_{{\rm{Cut}}(T, \bullet^{r+\kappa})}(A) \leq \mu|_{{\rm{Cut}}(T, \bullet^r)}(A^{\varepsilon}) + \varepsilon$ and $\mu|_{{\rm{Cut}}(T, \bullet^r)}(A) \leq \mu|_{{\rm{Cut}}(T, \bullet^{r-\kappa})}(A^{\varepsilon}) + \varepsilon$. This will be the case if we show that $\mu\left( {\rm{Cut}}(T, \bullet^{r+\kappa}) \setminus {\rm{Cut}}(T, \bullet^r) \right) \leq \varepsilon$ and $\mu\left( {\rm{Cut}}(T, \bullet^r) \setminus {\rm{Cut}}(T, \bullet^{r-\kappa}) \right) \leq \varepsilon$. We will only prove the first of these two estimates, as the argument from `below' is essentially analogous. Fix some $\boldsymbol{\sigma} \in \bullet^r$. Denote for each $\sigma' \in {\rm{Bp}}(T)$ by $\mathscr{C}_{\sigma'}$ the collection of connected components of $\theta_{\sigma'}T \setminus \{\sigma'\}$ entirely contained in $T \setminus {\rm{Skel}}(T, \mu)$, and set $\mathscr{C}(\kappa) = \bigcup_{\sigma' \in {\rm{Bp}}(T) \cap \overline{B}_T(\rho, r+\kappa) \cap \theta_{\boldsymbol{\sigma}}T} \mathscr{C}_{\sigma'}$. Then $\theta_{\boldsymbol{\sigma}}T \cap \left( {\rm{Cut}}(T, \bullet^{r+\kappa}) \setminus {\rm{Cut}}(T, \bullet^r) \right) \subseteq \left( \bigcup_{C \in \mathscr{C}(\kappa)} C \right) \cup \overline{B}_T(\rho, r+\kappa) \setminus \overline{B}_T(\rho, r)$. But $\mu\left( \left( \bigcup_{C \in \mathscr{C}(\kappa)} C \right) \cup \overline{B}_T(\rho, r+\kappa) \setminus \overline{B}_T(\rho, r) \right) < \infty$ since $\mu \in \mathscr{M}_{\rm{di}}(T)$, and so we may choose $\kappa$ sufficiently small that this mass becomes less than $\frac{\varepsilon}{|{\bullet^r}|}$.
\end{proof}
\begin{lemma}
\label{lemma:MEGV11}
Let $\boldsymbol{\mu} = [T, d, \rho, \mu]$ and $\boldsymbol{\mu}_n = [T_n, d_n, \rho_n, \mu_n]$ be isometry classes in $\mathbb{T}_{\rm{min}}^{\rm{di}}$. Fix some $\boldsymbol{\mu}$-nice $r > 0$, and let $\varepsilon > 0$. Choose $\kappa \in (0, r)$ such that $r- \kappa$ is $\boldsymbol{\mu}$-nice, the estimates from Lemma \ref{lemma:MEGV10} hold, and $\mu\left( \overline{B}_T(\rho, r+2\kappa) \setminus \overline{B}_T(\rho, r - \kappa) \right) \leq \varepsilon$. Let $\eta \in (0, \eta_0)$, and let $\blackdiamond_n^r$ be $\bullet^r$-compatible cut-off points in $\partial B_{T_n}(\rho_n, r)$. Suppose that
\begin{enumerate}[$(i)$]
\item $\mu_n\left( {\rm{Cut}}(T_n, \blackdiamond_n^r) \setminus \overline{B}_{T_n}(\rho_n, R_0) \right) \leq \varepsilon \wedge \eta_0$ and $\mu\left( {\rm{Cut}}(T, \bullet^r) \setminus \overline{B}_T(\rho, R_0) \right) \leq \varepsilon$ for some $R_0 > 0$, and

\item $d_{\rm{GV}}\left( \boldsymbol{\mu}_n, \boldsymbol{\mu} \right) \leq \frac{\eta}{20} e^{-R}$ for some $R > R_0 \vee R_0(K)$ where $K = \mu\left( \overline{B}_T(\rho, R_0 + \eta_0) \right) + 3\eta_0$.
\end{enumerate}
Then it follows for any $h_0 > 0$ that $d_{\rm{GP}}\left( {\rm{cut}}_{\blackdiamond_n^r}^{h_0}\boldsymbol{\mu}_n, {\rm{cut}}_{\bullet^r}^{h_0}\boldsymbol{\mu} \right) \leq 4\varepsilon + (1+ 6h_0|{\bullet^r}|) \eta$.
\end{lemma}
\begin{proof}
By Lemma \ref{lemma:MEGV6} there exists some $R' \geq R$ and an embedding system $(\mathscr{Z}, \Phi, \phi)$ such that $\delta_{\rm{P}}\left( (\mu_n \circ \phi_n^{-1})^{\restr R'}, (\mu \circ \phi^{-1})^{\restr R'} \right) \leq \eta$ and $\delta(\phi(\boldsymbol{\sigma}), \phi_n(u_n(\boldsymbol{\sigma}))) \leq 6\eta$ for every $\boldsymbol{\sigma} \in \bullet^r$. Moreover, by Lemma \ref{lemma:MEGV7} we have that $\blackdiamond_n^r = \bullet_n^{r,\kappa,\eta}$. Let $A \subseteq \mathscr{Z}$ be a closed set. Then we have by our assumptions, together with Lemma \ref{lemma:MEGV9}, Remark \ref{rmk:MEGV9} and Lemma \ref{lemma:MEGV10}, that
\begin{align*}
\lefteqn{\mu\left( \phi^{-1}(A) \cap {\rm{Cut}}(T, \bullet^r) \right)} \quad &\\ 
&\leq \mu(\phi^{-1}(A^{\varepsilon}) \cap {\rm{Cut}}(T, \bullet^{r-\kappa})) + \varepsilon \\
&\leq  \mu^{\restr R_0}\left( \phi^{-1}(A^{\varepsilon}) \cap {\rm{Cut}}(T, \bullet^{r-\kappa}) \right) + 2\varepsilon \\
&\leq \mu_n^{\restr R_0}\left(\phi_n^{-1}(A^{\varepsilon+ \eta}) \cap \phi_n^{-1}\left(\phi\left( {\rm{Cut}}(T, \bullet^{r-\kappa}) \right)^{\eta} \right)\right) + 2\varepsilon + \eta \\
&\leq \mu_n^{\restr R_0} \left(\phi_n^{-1}(A^{\varepsilon + \eta}) \cap \left( {\rm{Cut}}(T_n, \bullet_n^{r,\kappa,\eta}) \cup B_{T_n}(\rho_n, r+\kappa) \right) \right) + 2\varepsilon + \eta \\
&\leq \mu_n^{\restr R_0}\left( \phi_n^{-1}(A^{\varepsilon + \eta}) \cap {\rm{Cut}}(T_n, \bullet_n^{r,\kappa,\eta}) \right) + \mu_n^{\restr R_0}\left( \overline{B}_{T_n}(\rho_n, r+\kappa) \setminus B_{T_n}(\rho_n, r) \right) + 2\varepsilon + \eta \\
&\leq \mu_n\left( \phi_n^{-1}(A^{\varepsilon + \eta}) \cap {\rm{Cut}}(T_n, \bullet_n^{r,\kappa,\eta}) \right) + \mu^{\restr R_0}\left( \overline{B}_T(\rho, r+2\kappa) \setminus B_{T}(\rho, r-\kappa) \right) + 3\varepsilon + \eta \\
&\leq \mu_n\left( \phi_n^{-1}(A^{\varepsilon + \eta}) \cap {\rm{Cut}}(T_n, \bullet_n^{r,\kappa,\eta}) \right) + 4\varepsilon + \eta.
\end{align*}
In the other direction, we see that
\begin{align*}
\mu_n\left( \phi_n^{-1}(A) \cap {\rm{Cut}}(T_n, \bullet_n^{r,\kappa,\eta}) \right) &\leq \mu_n^{\restr R_0}\left( \phi_n^{-1}(A) \cap {\rm{Cut}}(T_n, \bullet_n^{r,\kappa,\eta}) \right) + \varepsilon \\
&\leq \mu^{\restr R_0}\left(\phi^{-1}(A^{\eta}) \cap \phi^{-1}\left( \phi_n\left( {\rm{Cut}}(T_n, \bullet_n^{r,\kappa,\eta}) \right)^{\eta} \right) \right) + \varepsilon + \eta \\
&\leq \mu^{\restr R_0}\left( \phi^{-1}(A^{\eta}) \cap \left( {\rm{Cut}}(T, \bullet^{r+\kappa}) \cup B_T(\rho, r+2\kappa) \right) \right) + \varepsilon + \eta \\
&\leq \mu\left( \phi^{-1}(A^{\eta}) \cap {\rm{Cut}}(T, \bullet^{r+\kappa}) \right) + 2\varepsilon + \eta \\
&\leq \mu\left(\phi^{-1}(A^{\varepsilon + \eta}) \cap {\rm{Cut}}(T, \bullet^r)\right) + 3\varepsilon + \eta.
\end{align*}
Combining the two estimates gives $\delta_{\rm{P}}\left( \phi(\mu|_{{\rm{Cut}}(T, \bullet^r)}), \phi_n\left( \mu_n|_{{\rm{Cut}}(T_n, \bullet_n^{r,\kappa,\eta})} \right) \right) \leq 4\varepsilon + \eta$.

Now, denote by $(\boldsymbol{\sigma}, \boldsymbol{\sigma}_n) = (\boldsymbol{\sigma}, u_n(\boldsymbol{\sigma}))$ each pair of cut-off points under the bijective map $u_n$. As $\delta(\phi(\boldsymbol{\sigma}), \phi_n(\boldsymbol{\sigma}_n)) \leq 6\eta$, it follows that the Dirac point masses are also close in the sense that $\delta_{\rm{P}}\left( \delta_{\phi(\boldsymbol{\sigma})}, \delta_{\phi_n(\boldsymbol{\sigma}_n)} \right) \leq 6 \eta$. As $u_n$ is bijective, this allows us to write for each $A \subseteq \mathscr{Z}$ closed
\begin{align*}
h_0 \sum_{\boldsymbol{\sigma}_n \in \blackdiamond_n^r} \delta_{\phi_n(\boldsymbol{\sigma}_n)}(A) &\leq h_0 \sum_{\boldsymbol{\sigma} \in \bullet^r} \left( \delta_{\phi(\boldsymbol{\sigma})}(A^{6\eta}) + 6\eta \right) = h_0 \sum_{\boldsymbol{\sigma} \in \bullet^r} \delta_{\phi(\boldsymbol{\sigma})}(A^{6\eta}) + 6h_0 |{\bullet^r}| \eta,
\end{align*}
and similarly the other way around using that $|{\bullet^r}| = |{\blackdiamond_n^r}|$. Combining with our previous observations we thus conclude that $ d_{\rm{GP}}\left({\rm{cut}}_{\blackdiamond_n^r}\boldsymbol{\mu}_n, {\rm{cut}}_{\bullet^r}\boldsymbol{\mu} \right) \leq 4\varepsilon + (1+ 6h_0|{\bullet^r}|) \eta$.
\end{proof}
\begin{thm}[Convergence of cut-off measures]
\label{thm:MEGV12}
Let $\boldsymbol{\mu} = [T, d, \rho, \mu]$ and $\boldsymbol{\mu}_n = [T_n, d_n, \rho_n, \mu_n]$, $n \in \mathbb{N}$, be isometry classes in $\mathbb{T}_{\rm{min}}^{\rm{di}}$. Fix some $\boldsymbol{\mu}$-nice $r > 0$. Let for each $n \in \mathbb{N}$, $\blackdiamond_n^r$ be $\bullet^r$-compatible cut-off points in $\partial B_{T_n}(\rho_n, r)$. Suppose that
\begin{enumerate}[$(i)$]
\item $\boldsymbol{\mu}_n \stackrel{\rm{GV}}{\rightarrow} \boldsymbol{\mu}$ as $n \rightarrow \infty$, and
\item $\left( \mu_n|_{{\rm{Cut}}(T_n, \blackdiamond_n^r)} \right)_{n \in \mathbb{N}}$ is b-tight.
\end{enumerate}
Then for any $h_0 > 0$ it follows that ${\rm{cut}}_{\blackdiamond_n^r}^{h_0} \boldsymbol{\mu}_n \stackrel{\rm{GP}}{\rightarrow} {\rm{cut}}_{\bullet^r}^{h_0}\boldsymbol{\mu}$ as $n \rightarrow \infty$.
\end{thm}
\begin{proof}[Proof of Theorem \ref{thm:MEGV12}]
Let $\varepsilon > 0$, let $\kappa \in (0,r)$ be given as in Lemma \ref{lemma:MEGV11}, and let $\eta \in (0, \eta_0 \wedge \varepsilon)$. By assumption $(ii)$, we can find some $R_0 > 0$ such that for all $n$ sufficiently large, condition $(i)$ of Lemma \ref{lemma:MEGV11} is true. Choosing $R > R_0 \vee R_0(K)$ with $K$ given as in Lemma \ref{lemma:MEGV11}, we also get by assumption $(i)$ that $d_{\rm{GV}}(\boldsymbol{\mu}_n, \boldsymbol{\mu}) \leq \frac{\eta}{20}e^{-R}$ for all $n$ sufficiently large. We thus conclude by Lemma \ref{lemma:MEGV11} that for all $n$ sufficiently large, $d_{\rm{GP}}\left( {\rm{cut}}_{\blackdiamond_n^r}\boldsymbol{\mu}_n, {\rm{cut}}_{\bullet^r}\boldsymbol{\mu} \right) \leq \left(5 + 6h_0 |{\bullet^r}| \right) \varepsilon$. As $|{\bullet^r}|$ and $h_0$ are fixed and finite, and $\varepsilon$ is arbitrary, we conclude that $d_{\rm{GP}}\left( {\rm{cut}}_{\blackdiamond_n^r}\boldsymbol{\mu}_n, {\rm{cut}}_{\bullet^r}\boldsymbol{\mu} \right) \rightarrow 0$.
\end{proof}
\begin{rmk}
Theorem \ref{thm:MEGV12} may be strengthened to an \emph{`if and only if'} statement by some minor adjustments. Indeed, if ${\rm{cut}}_{\blackdiamond_n^{r_k}}^{h_0} \boldsymbol{\mu}_n \stackrel{\rm{GP}}{\rightarrow} {\rm{cut}}_{\bullet^{r_k}}^{h_0}\boldsymbol{\mu}$ as $n \rightarrow \infty$ for all $k$ and some increasing sequence of radii $r_k \rightarrow \infty$, then it follows that $\boldsymbol{\mu}_n \stackrel{\rm{GV}}{\rightarrow} \boldsymbol{\mu}$ and $\left( \mu_n|_{{\rm{Cut}}(T_n, \blackdiamond_n^{r_k})} \right)_{n \in \mathbb{N}}$ is b-tight for all $k$. This is easily seen by restricting the cut-off measures to some smaller $\boldsymbol{\mu}$-nice radii $\tilde{r}_k < r_k$ in order to get Gromov-vague convergence. B-tightness on the cut-off trees follows immediately by definition of the cut-off measures and Gromov-Prokhorov convergence.
\demo
\end{rmk}
Using Theorem \ref{thm:MEGV12} we may prove local Gromov-Hausdorff convergence of the mass-erased subtrees.
\begin{thm}[Convergence of mass-erased subtrees 2]
\label{thm:MEGV14}
Let for all $n$, $\boldsymbol{\mu} = [T, d, \rho, \mu]$ and $\boldsymbol{\mu}_n = [T_n, d_n, \rho_n, \mu_n]$ be isometry classes in $\mathbb{T}_{\rm{min}}^{\rm{di}}$. Fix some $\boldsymbol{\mu}$-nice $r > 0$. Let for each $n$, $\blackdiamond_n^r$ be $\bullet^r$-compatible cut-off points in $\partial B_{T_n}(\rho_n, r)$. Let $h, h_n \in (0,\infty)$, $n \in \mathbb{N}$, and let $h_0 \geq \sup_{n \in \mathbb{N}} h_n$. Suppose that
\begin{enumerate}[$(i)$]
\item $\boldsymbol{\mu}_n \stackrel{\rm{GV}}{\rightarrow} \boldsymbol{\mu}$ and $\left( \mu_n|_{{\rm{Cut}}(T_n, \blackdiamond_n^r)} \right)_{n \in \mathbb{N}}$ is b-tight, and
\item $h_n \rightarrow h$ and the map $h' \mapsto R_{{\rm{cut}}_{\bullet^r}^{h_0}\mu, h'}(T)$ is $d_H$-continuous at $h$.\footnote{This condition may be removed if $h_n = h$ eventually.}
\end{enumerate}
Then $\lim_{n \rightarrow \infty} d_{\rm{GH}}\left( R_{\mu_n,h_n}(T_n) \cap {\rm{Cut}}(T_n, \blackdiamond_n^r), R_{\mu,h}(T) \cap {\rm{Cut}}(T, \bullet^r) \right) = 0$. If $(i)$--$(ii)$ are true for some increasing sequence of radii $r_k \rightarrow \infty$, then $\lim_{n \rightarrow \infty} d_{\rm{GH}}^{\rm{loc}}(R_{\mu_n, h_n}(T_n), R_{\mu,h}(T)) = 0$.
\end{thm}
\begin{lemma}
\label{lemma:MEGV13}
Let $\boldsymbol{\mu} = [T, d, \rho, \mu]$ and $\boldsymbol{\mu}_n = [T_n, d_n, \rho_n, \mu_n]$ be isometry classes in $\mathbb{T}_{\rm{min}}^{\rm{di}}$. Fix some $\boldsymbol{\mu}$-nice $r > 0$, and let $\kappa > 0$ and $\eta \in (0, \eta_0)$. Let for $n \in \mathbb{N}$, $\blackdiamond_n^r$ be $\bullet^r$-compatible cut-off points in $\partial B_{T_n}(\rho_n, r)$, and let $h_0 > 0$. Suppose that
\begin{enumerate}[$(i)$]
\item $\mu_n\left( {\rm{Cut}}(T_n, \blackdiamond_n^r) \setminus \overline{B}_{T_n}(\rho_n, R_0) \right) \leq \eta_0$ for some $R_0 > 0$, and
\item $d_{\rm{GP}}\left( \boldsymbol{\mu}_n^{\restr R}, \boldsymbol{\mu}^{\restr R} \right) \leq \eta$ for some $R > R_0 \vee R_0(K)$ where $K = \mu\left( \overline{B}_T(\rho, R_0 + \eta_0) \right) + 3 \eta_0 + h_0$.
\end{enumerate}
Then for every $h \in (0, h_0]$ it holds that $R_{\mu_n, h}(T_n) \cap {\rm{Cut}}(T_n, \blackdiamond_n^r) = R_{ {\rm{cut}}_{\blackdiamond_n^r}^{h_0}\mu_n, h}(T_n)$.
\end{lemma}
\begin{proof}
Let $(\mathscr{Z}, \Phi, \phi)$ be an embedding system realizing condition $(ii)$. Then by Lemma \ref{lemma:MEGV7}, $\blackdiamond_n^r = \bullet_n^{r,\kappa,\eta}$, and by Corollary \ref{cor:MEGV5} the map $u_n \colon \bullet_n^{r,\kappa,\eta} \rightarrow \bullet^r$ is bijective.

For the inclusion $(\subseteq)$, take some $\sigma_n \in R_{\mu_n, h}(T_n) \cap {\rm{Cut}}(T_n, \blackdiamond_n^r)$, and note that $\mu_n(\theta_{\sigma_n}T_n) \geq h$. If $\theta_{\sigma_n}T_n \cap \blackdiamond_n^r = \emptyset$, then by construction ${\rm{cut}}_{\blackdiamond_n^r}\mu_n(\theta_{\sigma_n}T_n) = \mu_n(\theta_{\sigma_n}T_n) \geq h$. If on the other hand $\theta_{\sigma_n}T_n \cap \blackdiamond_n^r \neq \emptyset$, then ${\rm{cut}}_{\blackdiamond_n^r}\mu_n(\theta_{\sigma_n}T_n) \geq h_0 \geq h$. In either case we have that $\sigma_n \in R_{ {\rm{cut}}_{\blackdiamond_n^r}\mu_n, h}(T_n)$.

For $(\supseteq)$, take some $\sigma_n \in R_{ {\rm{cut}}_{\blackdiamond_n^r}\mu_n, h}(T_n)$, and note that then ${\rm{cut}}_{\blackdiamond_n^r}\mu_n(\theta_{\sigma_n}T_n) \geq h$ (and so in particular $\sigma_n \in {\rm{Cut}}(T_n, \blackdiamond_n^r)$). If $\theta_{\sigma_n}T_n \cap \blackdiamond_n^r = \emptyset$, then as before $\mu_n(\theta_{\sigma_n}T_n) = {\rm{cut}}_{\blackdiamond_n^r}\mu_n(\theta_{\sigma_n}T_n) \geq h$. If on the other hand $\theta_{\sigma_n}T_n \cap \blackdiamond_n^r \neq \emptyset$, then there exists some $\boldsymbol{\sigma}_n \in \blackdiamond_n^r$ with $\theta_{\boldsymbol{\sigma}_n}T_n \subseteq \theta_{\sigma_n}T_n$. As $\blackdiamond_n^r = \bullet_n^{r,\kappa,\eta}$ and $u_n \colon \bullet_n^{r,\kappa,\eta} \rightarrow \bullet^r$ is bijective, there must exist some $\boldsymbol{\sigma} \in \bullet^r$ such that $\mathbf{J}_{\boldsymbol{\sigma}}^{\kappa,\eta} \cap \phi_n(T_n) \subseteq \phi_n(\theta_{\boldsymbol{\sigma}_n}T_n) \subseteq \phi_n(\theta_{\sigma_n}T_n)$. By Lemma \ref{lemma:MEGV4} we moreover have from our choice of constant $K$, that $\mu_n \circ \phi_n^{-1}\left( \mathbf{J}_{\boldsymbol{\sigma}}^{\kappa,\eta} \cap \overline{B}_{\mathscr{Z}}(\varrho, R) \right) > h_0$. We thus get that $\mu_n(\theta_{\sigma_n}T_n) \geq \mu_n \circ \phi_n^{-1}(\mathbf{J}_{\boldsymbol{\sigma}}^{\kappa,\eta}) > h_0 \geq h$, and so we conclude that $\sigma_n \in R_{\mu_n, h}(T_n)$, as desired.
\end{proof}
\begin{proof}[Proof of Theorem \ref{thm:MEGV14}]
We immediately get that $d_{\rm{GH}}\left( R_{ {\rm{cut}}_{\blackdiamond_n^r}\mu_n, h_n}(T_n), R_{ {\rm{cut}}_{\bullet^r}\mu, h}(T) \right) \rightarrow 0$ by Theorem \ref{thm:MEGV12} and \cite{DW26} Theorem 4.21$(ii)$. By Lemma \ref{lemma:MEGV13} we have $R_{\mu_n, h_n}(T_n) \cap {\rm{Cut}}(T_n, \blackdiamond_n^r) = R_{{\rm{cut}}_{\blackdiamond_n^r}\mu_n, h_n}(T_n)$ eventually, and similarly $R_{\mu,h}(T) \cap {\rm{Cut}}(T, \bullet^r) = R_{{\rm{cut}}_{\bullet^r}\mu, h}(T)$.
\end{proof}
We finally show Gromov-vague convergence of the isometry classes of projected measures.
\begin{prop}[Convergence of projected measures]
\label{prop:MEGV16}
Let for all $n$, $\boldsymbol{\mu} = [T, d, \rho, \mu]$ and $\boldsymbol{\mu}_n = [T_n, d_n, \rho_n, \mu_n]$ be isometry classes in $\mathbb{T}_{\rm{min}}^{\rm{di}}$. Let $r_k \rightarrow \infty$ be an increasing sequence of $\boldsymbol{\mu}$-nice radii. Let for each $n$ and $k$, $\blackdiamond_n^{r_k}$ be $\bullet^{r_k}$-compatible cut-off points in $\partial B_{T_n}(\rho_n, r_k)$. Let $h, h_n \in (0, \infty)$, $n \in \mathbb{N}$, let $h_0 \geq \sup_{n \in \mathbb{N}} h_n$, and suppose that for all $k$,
\begin{enumerate}[$(i)$]
\item $\boldsymbol{\mu}_n \stackrel{\rm{GV}}{\rightarrow} \boldsymbol{\mu}$ and $\left( \mu_n|_{{\rm{Cut}}(T_n, \blackdiamond_n^{r_k})} \right)_{n \in \mathbb{N}}$ is b-tight, and
\item $h_n \rightarrow h$ and the map $h' \mapsto R_{{\rm{cut}}_{\bullet^{r_k}}^{h_0}\mu, h'}(T)$ is $d_H$-continuous at $h$.\footnote{This condition may be removed if $h_n = h$ eventually.}
\end{enumerate}
Then $\left[{\rm{Pr}}_{R_{\mu_n,h_n}(T_n)}\mu_n\right] \stackrel{\rm{GV}}{\rightarrow} \left[{\rm{Pr}}_{R_{\mu,h}(T)}\mu\right]$,\footnote{Here, the square-brackets denote the Gromov-vague isometry classes generated by the projected measures.} and $\left( \left({\rm{Pr}}_{R_{\mu_n,h_n}(T_n)}\mu_n \right)|_{{\rm{Cut}}(T_n, \blackdiamond_n^{r_k})} \right)_{n \in \mathbb{N}}$ is b-tight for all $k$.
\end{prop}
\begin{lemma}
\label{lemma:MEGV15}
Let $\boldsymbol{\mu} = [T, d, \rho, \mu]$ and $\boldsymbol{\mu}_n = [T_n, d_n, \rho_n, \mu_n]$ be isometry classes in $\mathbb{T}_{\rm{min}}^{\rm{di}}$. Fix some $\boldsymbol{\mu}$-nice $r > 0$, and let $\kappa > 0$ and $\eta \in (0, \eta_0)$. Let $\blackdiamond_n^r$ be $\bullet^r$-compatible cut-off points in $\partial B_{T_n}(\rho_n, r)$, and let $h_0 > 0$. Suppose that
\begin{enumerate}[$(i)$]
\item $\mu_n\left( {\rm{Cut}}(T_n, \blackdiamond_n^r) \setminus \overline{B}_{T_n}(\rho_n, R_0) \right) \leq \eta_0$ for some $R_0 > 0$, and
\item $d_{\rm{GP}}\left( \boldsymbol{\mu}_n^{\restr R}, \boldsymbol{\mu}^{\restr R} \right) \leq \eta$ for some $R > R_0 \vee R_0(K)$ where $K = \mu\left( \overline{B}_T(\rho, R_0 + \eta_0) \right) + 3 \eta_0 + h_0$.
\end{enumerate}
Then for every $h \in (0, h_0]$ and $\tilde{r} \in (0, r)$, $\left({\rm{Pr}}_{R_{\mu_n,h}(T_n)}\mu_n\right)^{\restr \tilde{r}} = \left({\rm{Pr}}_{R_{{\rm{cut}}_{\blackdiamond_n^r}^{h_0}\mu_n, h}(T_n)} {\rm{cut}}_{\blackdiamond_n^r}^{h_0}\mu_n\right)^{\restr \tilde{r}}$.
\end{lemma}
\begin{proof}
Fix some $\tilde{r} \in (0, r)$, and let $(C_i)_{i \in I(\tilde{r})}$ be the connected components of $T_n \setminus R_{\mu_n, h}(T_n)$ which have closure points $(\sigma_i)_{i \in I(\tilde{r})}$ in $\overline{B}_{T_n}(\rho_n, \tilde{r})$. These are also the connected components of $T_n \setminus \left( R_{\mu_n,h}(T_n) \cap {\rm{Cut}}(T_n, \blackdiamond_n^r) \right)$ with closure points in $\overline{B}_{T_n}(\rho_n, \tilde{r})$. By Lemma \ref{lemma:MEGV13} we know that $ R_{\mu_n,h}(T_n) \cap {\rm{Cut}}(T_n, \blackdiamond_n^r) = R_{{\rm{cut}}_{\blackdiamond_n^r}\mu_n, h}(T_n)$. Hence, we see by Proposition \ref{prop:Proj2}$(i)$ that
\begin{align*}
\left({\rm{Pr}}_{R_{{\rm{cut}}_{\blackdiamond_n^r}\mu_n, h}(T_n)} {\rm{cut}}_{\blackdiamond_n^r}\mu_n\right)^{\restr \tilde{r}} &= {\rm{cut}}_{\blackdiamond_n^r}\mu_n\left( \cdot \cap R_{\mu_n,h}(T_n) \cap \overline{B}_{T_n}(\rho_n, \tilde{r}) \right) + \sum_{i \in I(\tilde{r})} {\rm{cut}}_{\blackdiamond_n^r}\mu_n(C_i) \delta_{\sigma_i} \\
&= \mu_n\left( \cdot \cap R_{\mu_n,h}(T_n) \cap \overline{B}_{T_n}(\rho_n, \tilde{r}) \right) + \sum_{i \in I(\tilde{r})} {\rm{cut}}_{\blackdiamond_n^r}\mu_n(C_i) \delta_{\sigma_i},
\end{align*}
using in the last step the definition of ${\rm{cut}}_{\blackdiamond_n^r}\mu_n$ and the fact that all cut-off points in $\blackdiamond_n^r$ are outside $\overline{B}_{T_n}(\rho_n, \tilde{r})$. What remains is to check that ${\rm{cut}}_{\blackdiamond_n^r}\mu_n(C_i) = \mu_n(C_i)$ for every $i \in I(\tilde{r})$. This will be the case if $\left( \bigcup_{i \in I(\tilde{r})} C_i \right) \cap \blackdiamond_n^r = \emptyset$, which holds since we know from the proof of Lemma \ref{lemma:MEGV13} that $\mu_n(\theta_{\boldsymbol{\sigma}_n}T_n) > h_0$ for every $\boldsymbol{\sigma}_n \in \blackdiamond_n^r$, and so in particular $\blackdiamond_n^r \subseteq R_{\mu_n,h}(T_n)$, which means that it cannot intersect any of the components $(C_i)_{i \in I(\tilde{r})}$. Plugging into our calculation, and applying Proposition \ref{prop:Proj2}(1) again, gives the desired.
\end{proof}
\begin{proof}[Proof of Proposition \ref{prop:MEGV16}]
Fix some $k$. Then we have by Theorem \ref{thm:MEGV12} and \cite{DW26} Theorem 4.21, that $\left[{\rm{Pr}}_{R_{{\rm{cut}}_{\blackdiamond_n^{r_k}}\mu_n, h}(T_n)} {\rm{cut}}_{\blackdiamond_n^{r_k}}\mu_n\right] \stackrel{\rm{GP}}{\rightarrow} \left[{\rm{Pr}}_{R_{{\rm{cut}}_{\bullet^{r_k}}\mu, h}(T)} {\rm{cut}}_{\bullet^{r_k}}\mu\right]$. Let $\varepsilon > 0$, and choose some ${\rm{Pr}}_{R_{{\rm{cut}}_{\bullet^{r_k}}\mu, h}(T)} {\rm{cut}}_{\bullet^{r_k}}\mu$-nice $\tilde{r}_k \in (r_k - \varepsilon, r_k)$. Then by Lemma \ref{lemma:MEGV15}, $\left[\left({\rm{Pr}}_{R_{\mu_n,h_n}(T_n)}\mu_n\right)^{\restr \tilde{r}_k}\right] \stackrel{\rm{GP}}{\rightarrow} \left[\left( {\rm{Pr}}_{R_{\mu,h}(T)}\mu \right)^{\restr \tilde{r}_k}\right]$. As $\tilde{r}_k \rightarrow \infty$, we conclude that $\left[{\rm{Pr}}_{R_{\mu_n,h_n}(T_n)}\mu_n\right] \stackrel{\rm{GV}}{\rightarrow} \left[{\rm{Pr}}_{R_{\mu,h}(T)}\right]$.

For b-tightness, fix some $r \in \{r_k ~|~ k \in \mathbb{N}\}$. Take some $\sigma_n \in T_n$ with $f_{R_{\mu_n,h_n}(T_n)}(\sigma_n) \in {\rm{Cut}}(T_n, \blackdiamond_n^r) \setminus \overline{B}_{T_n}(\rho_n, R)$. Then we must also have $\sigma_n \in {\rm{Cut}}(T_n, \blackdiamond_n^r) \setminus \overline{B}_{T_n}(\rho_n, R)$, as any point from outside ${\rm{Cut}}(T_n, \blackdiamond_n^r)$ can only be projected into ${\rm{Cut}}(T_n, \blackdiamond_n^r)$ if it is projected into a point at height $r$ or below. I.e. $f_{R_{\mu_n,h_n}(T_n)}^{-1}\left( {\rm{Cut}}(T_n, \blackdiamond_n^r) \setminus \overline{B}_{T_n}(\rho_n, R) \right) \subseteq {\rm{Cut}}(T_n, \blackdiamond_n^r) \setminus \overline{B}_{T_n}(\rho_n, R)$. As ${\rm{Pr}}_{R_{\mu_n,h_n}(T_n)}\mu_n = \mu_n \circ f_{R_{\mu_n,h_n}(T_n)}^{-1}$, we conclude that $ {\rm{Pr}}_{R_{\mu_n,h_n}(T_n)}\mu_n\left( {\rm{Cut}}(T_n, \blackdiamond_n^r) \setminus \overline{B}_{T_n}(\rho_n, R) \right) \leq \mu_n\left( {\rm{Cut}}(T_n, \blackdiamond_n^r) \setminus \overline{B}_{T_n}(\rho_n, R) \right)$, giving the desired.
\end{proof}
\subsection{Gromov-vague convergence in the sense of mass erasure}
Recall the definition of convergence in the sense of mass erasure from Definition \ref{def:MEConv}. We first prove the approximation property from Lemma \ref{lemma:MEGV18}, and the limit theorem from Theorem \ref{thm:MEGV17}.
\begin{proof}[Proof of Lemma \ref{lemma:MEGV18}]
For $(i)$ we observe that $(T, d, \rho)$ is a natural embedding space of $\mathscr{E}_{h_p}\mu$ and $\mu$. We can thus apply Lemma \ref{lemma:MEOp5b} directly to get that $\mathscr{E}_{h_p}\boldsymbol{\mu} \stackrel{\rm{GV}}{\rightarrow} \boldsymbol{\mu}$. For b-tightness, we apply a similar argument to that in Theorem \ref{thm:MEGV17} in order to get that $\mathscr{E}_{h_p}\mu\left( {\rm{Cut}}(T, \bullet^r) \setminus \overline{B}_{T}(\rho, R)\right) \leq \mu\left( {\rm{Cut}}(T, \bullet^r) \setminus \overline{B}_T(\rho, R) \right)$ for all $p$. As $\mu|_{{\rm{Cut}}(T, \bullet^r)} \in \mathscr{M}_{\rm{fin}}(T)$, we automatically have for any $\varepsilon > 0$ that $R > r$ can be chosen sufficiently large that $\mu\left( {\rm{Cut}}(T, \bullet^r) \setminus \overline{B}_T(\rho, R) \right) < \varepsilon$, as wanted.

For $(ii)$, observe that if $\mathscr{E}_{h_p}\boldsymbol{\nu} = \mathscr{E}_{h_p}\boldsymbol{\mu}$ for all $p$, then by $(i)$ we have $\mathscr{E}_{h_p}\boldsymbol{\mu} \stackrel{\rm{GV}}{\rightarrow} \boldsymbol{\mu}$ and $\mathscr{E}_{h_p}\boldsymbol{\mu} = \mathscr{E}_{h_p}\boldsymbol{\nu} \stackrel{\rm{GV}}{\rightarrow} \boldsymbol{\nu}$ as $p \rightarrow \infty$. Uniqueness of limits then gives $\boldsymbol{\mu} = \boldsymbol{\nu}$, as desired.
\end{proof}
\begin{proof}[Proof of Theorem \ref{thm:MEGV17}]
Let $h_0 \geq \sup_{n \in \mathbb{N}} h_n$. Fix some $k$. Then we immediately have by Theorem \ref{thm:MEGV12} and \cite{DW26} Theorem 4.21$(i)$ that $\mathscr{E}_{h_n}{\rm{cut}}_{\blackdiamond_n^{r_k}}\boldsymbol{\mu}_n \stackrel{\rm{GP}}{\rightarrow} \mathscr{E}_h{\rm{cut}}_{\bullet^{r_k}}\boldsymbol{\mu}$. Let $\varepsilon > 0$ and choose some $\mathscr{E}_h\mu$-nice $\tilde{r}_k \in (r_k - \varepsilon, r_k)$. Using Lemma \ref{lemma:MEGV13} and Lemma \ref{lemma:MEGV15}, and a similar argument as in Lemma \ref{lemma:MEOp5}, we get that $\left( \mathscr{E}_{h_n}\boldsymbol{\mu}_n \right)^{\restr \tilde{r}_k} = \left( \mathscr{E}_{h_n}{\rm{cut}}_{\blackdiamond_n^{r_k}}\boldsymbol{\mu}_n \right)^{\restr \tilde{r}_k}$ eventually. Moreover, we know automatically from Lemma \ref{lemma:MEOp5} that $\left( \mathscr{E}_h\boldsymbol{\mu} \right)^{\restr \tilde{r}_k} = \left( \mathscr{E}_h {\rm{cut}}_{\bullet^{r_k}} \boldsymbol{\mu} \right)^{\restr \tilde{r}_k}$. Hence, $\left( \mathscr{E}_{h_n}\boldsymbol{\mu}_n \right)^{\restr \tilde{r}_k} \stackrel{\rm{GP}}{\rightarrow} \left( \mathscr{E}_h\boldsymbol{\mu} \right)^{\restr \tilde{r}_k}$. As $\tilde{r}_k \rightarrow \infty$ is an increasing sequence of radii, we conclude that $\mathscr{E}_{h_n}\boldsymbol{\mu}_n \stackrel{\rm{GV}}{\rightarrow} \mathscr{E}_h\boldsymbol{\mu}$.

For b-tightness, fix some $r \in \{r_k ~|~ k \in \mathbb{N}\}$. Note that ${\rm{Cut}}(T_n, \blackdiamond_n^r) \setminus \overline{B}_{T_n}(\rho_n, R) = T_n \setminus \left( \overline{B}_{T_n}(\rho_n, R) \cup \bigcup_{\boldsymbol{\sigma}_n \in \blackdiamond_n^r} \theta_{\boldsymbol{\sigma}_n}T_n^{\circ} \right)$. But if $R > r$, then $ \overline{B}_{T_n}(\rho_n, R) \cup \bigcup_{\boldsymbol{\sigma}_n \in \blackdiamond_n^r} \theta_{\boldsymbol{\sigma}_n}T_n^{\circ}$ is a closed rooted subtree of $T_n$, and so we get by Proposition \ref{prop:MEOp2}$(iii)$ that $\mathscr{E}_h\mu_n\left( {\rm{Cut}}(T_n, \blackdiamond_n^r) \setminus \overline{B}_{T_n}(\rho_n, R) \right) \leq \mu_n\left( {\rm{Cut}}(T_n, \blackdiamond_n^r) \setminus \overline{B}_{T_n}(\rho_n, R) \right)$ for all $R > r$. Applying assumption $(i)$ then gives the desired.
\end{proof}
Our next goal will be to prove Theorem \ref{thm:MEGV23b} and Theorem \ref{thm:MEGV24}. We will do this through a series of lemmas, culminating in Theorem \ref{thm:MEGV22}. Given $r > 0$, $\kappa > 0$, and $K > 0$, we let $R_h(K) = R_h(r,\kappa,K)$ be chosen large enough such that $\mathscr{E}_h\mu\left( \theta_{\boldsymbol{\sigma}}T \cap \overline{B}_T(\rho, R - \kappa) \setminus B_T(\rho, r+3\kappa) \right) \geq K$ for all $R \geq R_h(K)$ and $\boldsymbol{\sigma} \in \bullet^r$.
\begin{lemma}
\label{lemma:MEGV19}
Let $\boldsymbol{\mu} = [T, d, \rho, \mu]$ and $\boldsymbol{\mu}_n = [T_n, d_n, \rho_n, \mu_n]$ be isometry classes in $\mathbb{T}_{\rm{min}}^{\rm{di}}$. Fix some $\boldsymbol{\mu}$-nice $r > 0$, and let $\kappa > 0$, $\eta \in (0, \eta_0)$ and $h > 0$. Let $R \geq R_h(\eta_0)$. If $(\mathscr{Z}, \Phi, \phi)$ is an embedding system in which $\delta_{\rm{P}}\left( (\mathscr{E}_h \mu_n \circ \phi_n^{-1})^{\restr R}, ( \mathscr{E}_h \mu \circ \phi^{-1} )^{\restr R} \right) \leq \eta$, then $\bullet_n^{r,\kappa,\eta}(\mathscr{E}_h\mu) = \bullet_n^{r,\kappa,\eta}(\mu)$, and the map $u_n \colon \bullet_n^{r,\kappa,\eta}(\mu) \rightarrow \bullet^r$ from Proposition \ref{prop:MEGV1} is bijective.
\end{lemma}
\begin{proof}
Recall that ${\rm{Skel}}(T, \mu) = {\rm{Skel}}(T, \mathscr{E}_h\mu)$, so the cut-off points $\bullet^r$ are the same in $(T, d, \rho, \mu)$ and $(R_{\mu,h+}(T), d, \rho, \mathscr{E}_h\mu)$. We may consider the contour area above level $r$ associated to the minimal $\mathbb{R}$-tree $(R_{\mu,h+}(T), d, \rho, \mathscr{E}_h\mu)$, by setting $\mathscr{J}_r^{\kappa,\eta}(h) := \phi\left(R_{\mu,h+}(T) \setminus {\rm{Cut}}^{r+\kappa}(T)\right)^{\eta} \setminus B_{\mathscr{Z}}(\varrho, r+2\kappa)$. Then $\mathscr{J}_r^{\kappa,\eta}(h) \subseteq \mathscr{J}_r^{\kappa, \eta}$. As $\sigma_n \in \circ_n^{r,\kappa,\eta}(\mathscr{E}_h\mu)$ if and only if $\phi_n(\theta_{\sigma_n}T_n \cap R_{\mu_n,h+}(T_n)) \cap \mathscr{J}_r^{\kappa,\eta}(h) \neq \emptyset$, and the latter implies $\phi_n(\theta_{\sigma_n}T_n) \cap \mathscr{J}_r^{\kappa,\eta} \neq \emptyset$, we must necessarily have that $\circ_n^{r, \kappa, \eta}(\mathscr{E}_h\mu) \subseteq \circ_n^{r,\kappa,\eta}(\mu)$, which gives $\bullet_n^{r,\kappa,\eta}(\mathscr{E}_h\mu) \subseteq \bullet_n^{r,\kappa,\eta}(\mu)$. As $u_n \colon \bullet_n^{r,\kappa,\eta}(\mu) \rightarrow \bullet^r$ is injective by Lemma \ref{lemma:MEGV3}, we thus conclude that $|{\bullet_n^{r,\kappa,\eta}(\mathscr{E}_h\mu)}| \leq |{\bullet_n^{r,\kappa,\eta}(\mu)}| \leq |{\bullet^r}|$. 

Now, considering $(\mathscr{Z}, \Phi, \phi)$ as an embedding system of $\mathscr{E}_h\boldsymbol{\mu}_n$ and $\mathscr{E}_h\boldsymbol{\mu}$ with contour area $\mathscr{J}_r^{\kappa,\eta}(h)$, we get by Corollary \ref{cor:MEGV5} that the map $u_n^h \colon \bullet_n^{r,\kappa,\eta}(\mathscr{E}_h\mu) \rightarrow \bullet^r$ is bijective, and so $|{\bullet_n^{r,\kappa,\eta}(\mathscr{E}_h\mu)}| =  |{\bullet^r}|$. By our previous inclusion, this is enough to conclude that $\bullet_n^{r,\kappa,\eta}(\mathscr{E}_h\mu) = \bullet_n^{r,\kappa,\eta}(\mu)$ and $u_n$ is bijective, as desired.
\end{proof}
\begin{lemma}
\label{lemma:MEGV20}
Let $\boldsymbol{\mu} = [T, d, \rho, \mu]$ and $\boldsymbol{\mu}_n = [T_n, d_n, \rho_n, \mu_n]$ be isometry classes in $\mathbb{T}_{\rm{min}}^{\rm{di}}$. Fix some $\boldsymbol{\mu}$-nice $r > 0$, and let $\kappa > 0$ and $\eta \in (0, \eta_0)$. Let $\blackdiamond_n^r$ be $\bullet^r$-compatible cut-off points in $\partial B_{T_n}(\rho_n, r)$. Fix some $h > 0$ and $h_0 > 0$. If
\begin{enumerate}[$(i)$]
\item $\mathscr{E}_h\mu_n\left( {\rm{Cut}}(T_n, \blackdiamond_n^r) \setminus \overline{B}_{T_n}(\rho_n, R_0) \right) \leq \eta_0$ for some $R_0 > 0$, and 
\item $d_{\rm{GP}}\left( (\mathscr{E}_h\boldsymbol{\mu}_n)^{\restr R}, (\mathscr{E}_h\boldsymbol{\mu})^{\restr R} \right) \leq \eta$ for some $R > R_h \vee R_h(K)$ where $K = \mathscr{E}_h\mu\left( \overline{B}_T(\rho, R_0 + \eta_0) \right) + 3\eta_0 + h_0$,
\end{enumerate}
then for every $h' \in (0, h_0]$ it holds that $R_{\mu_n,h'}(T_n) \cap {\rm{Cut}}(T_n, \blackdiamond_n^r) = R_{{\rm{cut}}_{\blackdiamond_n^r}^{h_0}\mu_n, h'}(T_n)$-
\end{lemma}
\begin{proof}
Let $(\mathscr{Z}, \Phi, \phi)$ be an embedding system realizing assumption $(ii)$. Then by Lemma \ref{lemma:MEGV19}, $\bullet_n^{r,\kappa,\eta}(\mathscr{E}_h\mu) = \bullet_n^{r,\kappa,\eta}(\mu)$ and $u_n \colon \bullet_n^{r,\kappa,\eta}(\mu) \rightarrow \bullet^r(T)$ is bijective. Hence, applying Lemma \ref{lemma:MEGV7} to $\mathscr{E}_h\boldsymbol{\mu}_n$ and $\mathscr{E}_h\boldsymbol{\mu}$ gives $\blackdiamond_n^r = \bullet_n^{r,\kappa,\eta}(\mu)$. The fact that $R_{\mu_n,h'}(T_n) \cap {\rm{Cut}}(T_n, \blackdiamond_n^r) \subseteq R_{{\rm{cut}}_{\blackdiamond_n^r}\mu_n, h'}(T_n)$ is shown exactly as in Lemma \ref{lemma:MEGV13}. For the opposite inclusion, take some $\sigma_n \in R_{{\rm{cut}}_{\blackdiamond_n^r}\mu_n, h'}(T_n)$. The case where $\theta_{\sigma_n} T_n \cap \blackdiamond_n^r = \emptyset$ is also treated as in the proof of Lemma \ref{lemma:MEGV13}. For the case where $\theta_{\sigma_n} T_n \cap \blackdiamond_n^r \neq \emptyset$, there must exist some $\boldsymbol{\sigma} \in \bullet^r(T)$ such that $\mathbf{J}_{\boldsymbol{\sigma}}^{\kappa, \eta}(h) \cap \phi_n(R_{\mu_n,h+}(T_n)) \subseteq \phi_n(\theta_{\sigma_n}T_n)$, where $\mathbf{J}_{\boldsymbol{\sigma}}^{\kappa,\eta}(h)$ denotes the part of the contour area $\mathscr{J}_r^{\kappa,\eta}(h)$ (defined as in the proof of Lemma \ref{lemma:MEGV19}) associated to $\boldsymbol{\sigma}$. Applying Lemma \ref{lemma:MEGV4} we thus have that $\mu_n(\theta_{\sigma_n}T_n) \geq \mathscr{E}_h\mu_n(\theta_{\sigma_n}T_n) \geq \mathscr{E}_h\mu_n \circ \phi_n^{-1}\left( \mathbf{J}_{\boldsymbol{\sigma}}^{\kappa,\eta}(h) \cap \overline{B}_{\mathscr{Z}}(\varrho, R) \right) > h_0$, and so $\sigma_n \in R_{\mu_n,h'}(T_n)$, as desired.
\end{proof}
\begin{lemma}
\label{lemma:MEGV21}
Let $\boldsymbol{\mu} = [T, d, \rho, \mu]$ and $\boldsymbol{\mu}_n = [T_n, d_n, \rho_n, \mu_n]$ be isometry classes in $\mathbb{T}_{\rm{min}}^{\rm{di}}$. Fix some $\boldsymbol{\mu}$-nice $r > 0$, and let $\varepsilon > 0$. Fix some $h_0 > h > 0$, and choose $\kappa \in (0, r)$ such that $r- \kappa$ is $\boldsymbol{\mu}$-nice, the estimates from Lemma \ref{lemma:MEGV10} hold for $\mathscr{E}_h\mu$, and $\mathscr{E}_h\mu\left( \overline{B}_T(\rho, r+2\kappa) \setminus \overline{B}_T(\rho, r- \kappa) \right) \leq \varepsilon$. Let moreover $\eta \in (0, \eta_0)$, and let $\blackdiamond_n^r$ be $\bullet^r$-compatible cut-off points in $\partial B_{T_n}(\rho_n, r)$. Suppose that
\begin{enumerate}[$(i)$]
\item $\mathscr{E}_h\mu_n\left( {\rm{Cut}}(T_n, \blackdiamond_n^r) \setminus \overline{B}_{T_n}(\rho_n, R_0) \right) \leq \varepsilon \wedge \eta_0$ and $\mathscr{E}_h\mu\left( {\rm{Cut}}(T, \bullet^r) \setminus \overline{B}_T(\rho, R_0) \right) \leq \varepsilon$ for some $R_0 > 0$, and

\item $d_{\rm{GV}}\left( \mathscr{E}_h\boldsymbol{\mu}_n, \mathscr{E}_h\boldsymbol{\mu} \right) \leq \frac{\eta}{20} e^{-R}$ for some $R > R_0 \vee R_h(K)$ where $K = \mathscr{E}_h\mu\left( \overline{B}_T(\rho, R_0 + \eta_0) \right) + 3\eta_0 + h_0$.
\end{enumerate}
Then it follows that $d_{\rm{GP}}\left( \mathscr{E}_h{\rm{cut}}_{\blackdiamond_n^r}^{h_0}\boldsymbol{\mu}_n, \mathscr{E}_h{\rm{cut}}_{\bullet^r}^{h_0}\boldsymbol{\mu} \right) \leq 4\varepsilon + (1+ 6h_0|{\bullet^r}|) \eta$.
\end{lemma}
\begin{proof}
By Lemma \ref{lemma:MEGV20} and \eqref{eq:MEOp}, we have $\mathscr{E}_h{\rm{cut}}_{\blackdiamond_n^r}^{h_0}\mu_n = \left( \mathscr{E}_h\mu_n \right)|_{{\rm{Cut}}(T_n, \blackdiamond_n^r)} + \sum_{\sigma \in \blackdiamond_n^r} (h_0 - h) \delta_{\sigma} = {\rm{cut}}_{\blackdiamond_n^r}^{h_0 - h} \mathscr{E}_h\mu_n$.
A similar calculation is carried out for $\mathscr{E}_h {\rm{cut}}_{\bullet^r}^{h_0}\mu$. Now, by Lemma \ref{lemma:MEGV11} we get that $d_{\rm{GP}}\left( {\rm{cut}}_{\blackdiamond_n^r}^{h_0 - h}\mathscr{E}_h\boldsymbol{\mu}_n, {\rm{cut}}_{\bullet^r}^{h_0 - h}\mathscr{E}_h\boldsymbol{\mu} \right) \leq 4\varepsilon + \left(1 + 6(h_0-h)|{\bullet^r}| \right)\eta$, giving the desired.
\end{proof}
\begin{thm}[Convergence in the sense of mass erasure of cut-off measures]
\label{thm:MEGV22}
Let $\boldsymbol{\mu} = [T, d, \rho, \mu]$ and $\boldsymbol{\mu}_n = [T_n, d_n, \rho_n, \mu_n]$, $n \in \mathbb{N}$, be isometry classes in $\mathbb{T}_{\rm{min}}^{\rm{di}}$. Fix some $\boldsymbol{\mu}$-nice $r > 0$. Let for each $n$, $\blackdiamond_n^r$ be $\bullet^r$-compatible cut-off points in $\partial B_{T_n}(\rho_n, r)$, and let $h_0 > 0$. If
\begin{enumerate}[$(i)$]
\item $\boldsymbol{\mu}_n \stackrel{\rm{GVme}}{\rightarrow} \boldsymbol{\mu}$ as $n \rightarrow \infty$, and 
\item $\left( \left( \mathscr{E}_h\mu_n \right)|_{{\rm{Cut}}(T_n, \blackdiamond_n^r)} \right)_{n \in \mathbb{N}}$ is b-tight for all $h > 0$,
\end{enumerate}
then ${\rm{cut}}_{\blackdiamond_n^r}^{h_0}\boldsymbol{\mu}_n \stackrel{\rm{GPme}}{\rightarrow} {\rm{cut}}_{\bullet^r}^{h_0}\boldsymbol{\mu}$ as $n \rightarrow \infty$.
\end{thm}
\begin{proof}[Proof of Theorem \ref{thm:MEGV22}]
This follows by applying Lemma \ref{lemma:MEGV21} in a similar way as in Theorem \ref{thm:MEGV12} for each fixed $h \in (0, h_0)$, to get that $d_{\rm{GP}}\left( \mathscr{E}_h {\rm{cut}}_{\blackdiamond_n^r}^{h_0}\boldsymbol{\mu}_n, \mathscr{E}_h{\rm{cut}}_{\blackdiamond_n^r}^{h_0}\boldsymbol{\mu} \right) \rightarrow 0$ as $n \rightarrow \infty$.
\end{proof}
Theorem \ref{thm:MEGV23b} now follows immediately by applying Theorem \ref{thm:MEGV22} and Lemma \ref{lemma:MEGV19} together with \cite{DW26} Lemma 4.31$(ii)$. In Proposition \ref{prop:MEGV23} we similarly obtain a limit theorem for the projections. The proof of Proposition \ref{prop:MEGV23} is a direct consequence of Theorem \ref{thm:MEGV22}, \cite{DW26} Lemma 4.31$(ii)$, and a straight-forward adaptation of Lemma \ref{lemma:MEGV15}.
\begin{prop}
\label{prop:MEGV23}
Let for all $n$, $\boldsymbol{\mu} = [T, d, \rho, \mu]$ and $\boldsymbol{\mu}_n = [T_n, d_n, \rho_n, \mu_n]$ be isometry classes in $\mathbb{T}_{\rm{min}}^{\rm{di}}$. Let $r_k \rightarrow \infty$ be an increasing sequence of $\boldsymbol{\mu}$-nice radii. Let for each $n$ and $k$, $\blackdiamond_n^{r_k}$ be $\bullet^{r_k}$-compatible cut-off points in $\partial B_{T_n}(\rho_n, r_k)$. Let $h, h_n \in (0, \infty)$, $n \in \mathbb{N}$, let $h_0 \geq \sup_{n \in \mathbb{N}} h_n$, and suppose that for all $k$,
\begin{enumerate}[$(i)$]
\item $\boldsymbol{\mu}_n \stackrel{\rm{GVme}}{\rightarrow} \boldsymbol{\mu}$ and $\left( \left( \mathscr{E}_h\mu_n \right)|_{{\rm{Cut}}(T_n, \blackdiamond_n^{r_k})} \right)_{n \in \mathbb{N}}$ is b-tight for all $h > 0$, and 
\item $h_n \rightarrow h$ and the map $h' \mapsto R_{{\rm{cut}}_{\bullet^{r_k}}^{h_0}\mu, h'}(T)$ is $d_H$-continuous at $h$.
\end{enumerate}
Then $\left[{\rm{Pr}}_{R_{\mu_n,h_n}(T_n)}\mu_n\right] \stackrel{\rm{GV}}{\rightarrow} \left[{\rm{Pr}}_{R_{\mu,h}(T)}\mu\right]$, and $\left(\left( {\rm{Pr}}_{R_{\mu_n,h_n}(T_n)} \mu_n \right)|_{{\rm{Cut}}(T_n, \blackdiamond_n^{r_k})}\right)_{n \in \mathbb{N}}$ is b-tight for all $k$.
\end{prop}
We may also use Theorem \ref{thm:MEGV22} to prove Theorem \ref{thm:MEGV24}.
\begin{proof}[Proof of Theorem \ref{thm:MEGV24}]
For the direction $(a) \Rightarrow (b)$, we immediately get from $(a)$ that $\boldsymbol{\mu}_n \stackrel{\rm{GVme}}{\rightarrow} \boldsymbol{\mu}$ and $(\mathscr{E}_h\boldsymbol{\mu}_n)_{n \in \mathbb{N}}$ is b-tight on the cut-off trees for all $h > 0$, by Theorem \ref{thm:MEGV17}. Moreover, Theorem \ref{thm:MEGV12} gives that ${\rm{cut}}_{\blackdiamond_n^{r_k}}\boldsymbol{\mu}_n \stackrel{\rm{GP}}{\rightarrow} {\rm{cut}}_{\bullet^{r_k}}\boldsymbol{\mu}$. Considering some representative and using the continuous mapping theorem, thus yields that $\lambda_{{\rm{cut}}_{\blackdiamond_n^{r_k}}\boldsymbol{\mu}_n} \stackrel{\rm{wk}}{\rightarrow} \lambda_{{\rm{cut}}_{\bullet^{r_k}}\boldsymbol{\mu}}$ in $\mathscr{M}_{\rm{fin}}([0,\infty))$.

For the opposite direction, we know from Theorem \ref{thm:MEGV22} that ${\rm{cut}}_{\blackdiamond_n^{r_k}}\boldsymbol{\mu}_n \stackrel{\rm{GPme}}{\rightarrow} {\rm{cut}}_{\bullet^{r_k}}\boldsymbol{\mu}$ in $\mathbb{T}_{\rm{min}}^{\rm{fin}}$. So the assumption $\lambda_{{\rm{cut}}_{\blackdiamond_n^{r_k}}\boldsymbol{\mu}_n} \stackrel{\rm{wk}}{\rightarrow} \lambda_{{\rm{cut}}_{\bullet^{r_k}}\boldsymbol{\mu}}$ implies by \cite{DW26} Theorem 4.35 that ${\rm{cut}}_{\blackdiamond_n^{r_k}}\boldsymbol{\mu}_n \stackrel{\rm{GP}}{\rightarrow} {\rm{cut}}_{\bullet^{r_k}}\boldsymbol{\mu}$. Restricting to balls of $\boldsymbol{\mu}$-nice radius $\tilde{r}_k < r_k$ thus gives $\boldsymbol{\mu}_n \stackrel{\rm{GV}}{\rightarrow} \boldsymbol{\mu}$. As the cut-off measures have support on the cut-off trees and are automatically b-tight under Gromov--Prokhorov convergence, it also follows that $\lim_{R \rightarrow \infty} \limsup_{n \rightarrow \infty} \mu_n\left( {\rm{Cut}}(T_n, \blackdiamond_n^{r_k}) \setminus \overline{B}_{T_n}(\rho_n, R) \right) = 0$ for all $k$, as desired.
\end{proof}
\begin{rmk}
\label{rmk:MEGV24b}
Fix some $r > 0$, and let for all $n$, $\blackdiamond_n^r$ be $\bullet^r$-compatible cut-off points in $\partial B_{T_n}(\rho_n, r)$. Then $\left( \left( \mathscr{E}_h\mu_n \right)_{{\rm{Cut}}(T_n, \blackdiamond_n^r)} \right)_{n \in \mathbb{N}}$ is b-tight for all $h > 0$ if and only if it is b-tight for all $h$ in some decreasing sequence of parameters $h_p \searrow 0$. The proof of this statement is completely analogous to Remark \ref{rmk:MEOp10b}, utilizing the semigroup property together with the technique from the proof of Proposition \ref{prop:MEOp6}$(ii)$.
\demo
\end{rmk}
Note also that one may alternatively impose a condition of relative compactness together with Gromov-vague convergence in the sense of mass erasure, to obtain Gromov-vague convergence of isometry classes. This is a generalization of \cite{DW26} Theorem 4.32$(vii)$.
\begin{lemma}
Let $\boldsymbol{\mu}_n = [T_n, d_n, \rho_n, \mu_n]$, $n \in \mathbb{N}$, and $\boldsymbol{\mu} = [T, d, \rho, \mu]$ be isometry classes in $\mathbb{T}_{\rm{min}}^{\rm{di}}$. Let $r_k \rightarrow \infty$ be an increasing sequence of $\boldsymbol{\mu}$-nice radii. Let for each $n$ and $k$, $\blackdiamond_n^{r_k}$ be $\bullet^{r_k}$-compatible cut-off points in $\partial B_{T_n}(\rho_n, r_k)$. Suppose that $\left( \mu_n|_{{\rm{Cut}}(T_n, \blackdiamond_n^{r_k})} \right)_{n \in \mathbb{N}}$ is b-tight for all $k$. Then $\boldsymbol{\mu}_n \stackrel{\rm{GV}}{\rightarrow} \boldsymbol{\mu}$ if and only if $\boldsymbol{\mu}_n \stackrel{\rm{GVme}}{\rightarrow} \boldsymbol{\mu}$ and $\{\boldsymbol{\mu}_n ~|~ n \in \mathbb{N}\}$ is relatively compact in $(\mathbb{T}_{\rm{min}}^{\rm{di}}, d_{\rm{GV}})$.
\end{lemma}
\begin{proof}
The implication $(\Rightarrow)$ is a straightforward consequence of Theorem \ref{thm:MEGV17}. For the opposite implication, let $\boldsymbol{\mu}' \in \mathbb{T}_{\rm{min}}^{\rm{di}}$ be a $d_{\rm{GV}}$-limit point of $(\boldsymbol{\mu}_n)_{n \in \mathbb{N}}$. Then by the assumption of relative compactness there exists a subsequence $(\boldsymbol{\mu}_{n_k})_{k \in \mathbb{N}}$ such that $\boldsymbol{\mu}_{n_k} \stackrel{\rm{GV}}{\rightarrow} \boldsymbol{\mu}'$. Applying now the assumption of b-tightness together with Theorem \ref{thm:MEGV17} then yields $\boldsymbol{\mu}_{n_k} \stackrel{\rm{GVme}}{\rightarrow} \boldsymbol{\mu}'$ as $k \rightarrow \infty$. But as $\boldsymbol{\mu}_n \stackrel{\rm{GVme}}{\rightarrow} \boldsymbol{\mu}$ by assumption, uniqueness of limits clearly gives that $\boldsymbol{\mu}' = \boldsymbol{\mu}$. Hence, $(\boldsymbol{\mu}_n)_{n \in \mathbb{N}}$ has $\boldsymbol{\mu}$ as its unique $d_{\rm{GV}}$-limit point, and so we conclude that $\boldsymbol{\mu}_n \stackrel{\rm{GV}}{\rightarrow} \boldsymbol{\mu}$, as desired.
\end{proof}
Before moving on to the next section, we prove Lemma \ref{lemma:MEGV24b}, and discuss some relative compactness conditions.
\begin{proof}[Proof of Lemma \ref{lemma:MEGV24b}]
Suppose for $\boldsymbol{\mu}_n \in \mathbb{T}_{\rm{min}}^{\rm{di}}$, $n \in \mathbb{N}$, that $(\mathscr{E}_h\boldsymbol{\mu}_n)_{n \in \mathbb{N}}$ is Cauchy with respect to the Gromov-vague topology. Then, as $(\mathbb{T}_{\rm{min}}^{\rm{bf}}, d_{\rm{GV}})$ is complete, there exists some $\boldsymbol{\nu} \in \mathbb{T}_{\rm{min}}^{\rm{bf}}$ such that $\mathscr{E}_h\boldsymbol{\mu}_n \stackrel{\rm{GV}}{\rightarrow} \boldsymbol{\nu}$. Let $(\mathscr{Z}, \Phi, \phi)$ denote an embedding system realizing the convergence, as constructed in Proposition \ref{prop:GV3}. We wish to argue that $\boldsymbol{\nu} \in \mathbb{T}_{\rm{min}}^{\rm{di}}$. 

Suppose for contradiction that some representative $(T, d, \rho, \nu)$ of $\boldsymbol{\nu}$ satisfies that either i) ${\rm{Skel}}(T, \nu) \cap \overline{B}_T(\rho, r)$ is \emph{not} of finite type or ii) $\nu\left( \bigcup_{\sigma \in \partial B_T(\rho, r) \setminus {\rm{Skel}}(T, \nu)} \theta_{\sigma}T \right) = +\infty$, for some $\nu$-nice $r > 0$. Then for any $M \in \mathbb{N}$ we may choose some $\nu$-nice $R > r$ large enough that $\nu\left( \theta_{\sigma_i}T \cap \overline{B}_T(\rho, R) \right) > 0$ for at least $M$ points in $\partial B_T(\rho, r)$. Denote these $M$ points by $\sigma^{(1)}, \dots, \sigma^{(M)}$. Let $\varepsilon > 0$ be chosen small enough such that $\phi(\theta_{\sigma^{(i)}}T)^{\varepsilon}$ are disjoint in $(\mathscr{Z}, \delta)$ and such that $\nu(\theta_{\sigma^{(i)}}T \cap \overline{B}_T(\rho, R)) > 2\varepsilon$ for all $1 \leq i \leq M$. Then we can find some $N$ sufficiently large that $\delta_{\rm{P}}\left( (\nu \circ \phi^{-1})^{\restr R}, ( \mathscr{E}_h\mu_n \circ \phi_n^{-1})^{\restr R} \right) \leq \varepsilon$ for all $n \geq N$, implying in particular that $2 \varepsilon < \nu^{\restr R}(\theta_{\sigma^{(i)}}T) \leq (\mathscr{E}_h\mu_n \circ \phi_n^{-1})^{\restr R}(\phi(\theta_{\sigma^{(i)}}T)^{\varepsilon}) + \varepsilon$, and so $(\mathscr{E}_h\mu_n)^{\restr R}\left(\phi_n^{-1}\left( \phi(\theta_{\sigma^{(i)}}T)^{\varepsilon} \right) \right) > \varepsilon > 0$. In particular there must exist $M$ distinct points in ${\rm{Span}}\left( {\rm{supp}}(\mathscr{E}_h\mu_n) \right) \cap \partial B_{T_n}(\rho_n, r-\varepsilon) = R_{\mu_n, h+}(T_n) \cap \partial B_{T_n}(\rho_n, r-\varepsilon) $, and so as $R_{\mu_n,h+}(T_n) \subseteq R_{\mu_n,h}(T_n)$, we get that $R_{\mu_n,h}(T_n) \cap \overline{B}_T(\rho, r-\varepsilon)$ must eventually have at least $M$ leaves on $\partial B_{T_n}(\rho_n, r-\varepsilon)$. But then $(n(\rho, R_{\mu_n,h}(T_n)) - 1) + \sum_{\sigma \in {\rm{Bp}}(R_{\mu_n,h}(T_n)) \cap B_T(\rho, r-\varepsilon)} (n(\sigma, R_{\mu_n,h}(T_n)) - 2) \geq M-1$, implying in particular that $\mathscr{E}_h\mu_n(\overline{B}_T(\rho, r)) \geq \mathscr{E}_h\mu_n(\overline{B}_T(\rho, r-\varepsilon)) \geq h(M-1)$ eventually. As the argument can be repeated for $M$ arbitrarily large, we get $\mathscr{E}_h\mu_n(\overline{B}_T(\rho, r)) \rightarrow \infty$, contradicting our assumptions.
\end{proof}
We finally consider some generalizations of the relative compactness conditions from \cite{DW26} Proposition 4.13 and Theorem 4.32$(v)$. Recall from Proposition \ref{prop:GV4} that relative compactness in the Gromov-vague sense is equivalent to Gromov--Prokhorov relative compactness for the restrictions of isometry classes to closed balls. For a metric space $(\mathscr{X}, d)$ and a subset $A \subseteq \mathscr{X}$, we denote by $N(A, \varepsilon)$ the minimal number of open balls centered in $A$ with radius $\varepsilon$ necessary to cover $A$. From \cite{DW26} p. 33 we recall the definition of a slight variation of the essential covering number for $\mathbb{R}$-trees, defined for every $\boldsymbol{\mu} = [T, d, \rho, \mu] \in \mathbb{T}_{\rm{min}}^{\rm{fin}}$ by
\begin{align*}
\texttt{ess-tree}(\boldsymbol{\mu}, \varepsilon) := \left\{ N(T', \varepsilon) + {\rm{Ht}}(T') ~|~ T' \subseteq T \text{ is a compact rooted subtree}: \mu(T \setminus T') \leq \varepsilon \right\}.
\end{align*}
For convenience, we define also the smaller $\boldsymbol{\mu}$-essential covering number
$$ \widetilde{\texttt{ess-tree}}(\boldsymbol{\mu}, \varepsilon) := \left\{ N(T', \varepsilon) ~|~ T' \subseteq T \text{ is a compact rooted subtree}: \mu(T \setminus T') \leq \varepsilon \right\}. $$
\begin{lemma}
\label{lemma:precompact}
Let $\boldsymbol{\Pi} \subseteq \mathbb{T}_{\rm{min}}^{\rm{bf}}$ be non-empty. Then $\boldsymbol{\Pi}$ is relatively compact in $(\mathbb{T}_{\rm{min}}^{\rm{bf}}, d_{\rm{GV}})$ if and only if for some increasing sequence of radii $r_k \rightarrow \infty$ it holds for each $k$ that
\begin{enumerate}[$(i)$]
\item $\sup_{\boldsymbol{\mu} \in \boldsymbol{\Pi}} \mu\left( \overline{B}_T(\rho, r_k) \right) < \infty$, and
\item $\sup_{\boldsymbol{\mu} \in \boldsymbol{\Pi}} \widetilde{\texttt{\emph{ess-tree}}}(\boldsymbol{\mu}^{\restr r_k}, \varepsilon) < \infty$ for all $\varepsilon > 0$.
\end{enumerate}
\end{lemma}
\begin{proof}
By Proposition \ref{prop:GV4}, $\boldsymbol{\Pi}$ is relatively compact in $(\mathbb{T}_{\rm{min}}^{\rm{bf}}, d_{\rm{GV}})$ if and only if for some increasing sequence of radii $r_k \rightarrow \infty$, the restriction $\boldsymbol{\Pi}^{\restr r_k}$ is $d_{\rm{GP}}$-relatively compact. Now, as $\boldsymbol{\Pi}^{\restr r_k} \subseteq \mathbb{T}_{\rm{min}}^{\rm{fin}}$, we know from \cite{DW26} Proposition 4.13 that $\boldsymbol{\Pi}^{\restr r_k}$ is relatively compact if and only if $\sup_{\boldsymbol{\mu} \in \boldsymbol{\Pi}} \mu(\overline{B}_T(\rho, r_k)) < \infty$ and $\sup_{\boldsymbol{\mu} \in \boldsymbol{\Pi}} \texttt{ess-tree}(\boldsymbol{\mu}^{\restr r_k}, \varepsilon) < \infty$ for all $\varepsilon > 0$. Observe that any compact rooted subtree of $\overline{B}_T(\rho, r_k)$ will have height at most $r_k$. Hence, $\texttt{ess-tree}(\boldsymbol{\mu}^{\restr r_k}, \varepsilon) \leq \widetilde{\texttt{ess-tree}}(\boldsymbol{\mu}^{\restr r_k}, \varepsilon) + r_k$. From this we conclude the desired.
\end{proof}
\begin{lemma}
Let $\boldsymbol{\Pi} \subseteq \mathbb{T}_{\rm{min}}^{\rm{di}}$ be non-empty. Fix $h > 0$. Then $\mathscr{E}_h(\boldsymbol{\Pi})$ is relatively compact in $(\mathbb{T}_{\rm{min}}^{\rm{di}}, d_{\rm{GV}})$ if and only if for some increasing sequence of radii $r_k \rightarrow \infty$ it holds for each $k$ that
\begin{enumerate}[$(i)$]
\item $\sup_{\boldsymbol{\mu} \in \boldsymbol{\Pi}} \mathscr{E}_h\mu(\overline{B}_T(\rho, r_k)) < \infty$, and
\item $\sup_{\boldsymbol{\mu} \in \boldsymbol{\Pi}} |{{\rm{Lf}}\left( R_{\mu,h}(T) \cap \overline{B}_T(\rho, r_k) \right)}| < \infty$.
\end{enumerate}
\end{lemma}
\begin{proof}
Note that $(\mathscr{E}_{h_p}\boldsymbol{\mu})^{\restr r_k}$ has support contained in $R_{\mu,h}(T) \cap \overline{B}_T(\rho, r_k)$, which is known to be a subtree of finite type. So in particular, $ N(R_{\mu,h}(T) \cap \overline{B}_T(\rho, r_k), \varepsilon) \leq \frac{r_k}{\varepsilon} |{{\rm{Lf}}\left( R_{\mu,h}(T) \cap \overline{B}_T(\rho, r_k) \right)}|$. Now, by Lemma \ref{lemma:precompact}, $(i)$ and $(ii)$ are true if and only if $\mathscr{E}_h(\boldsymbol{\Pi})$ is relatively compact in $(\mathbb{T}_{\rm{min}}^{\rm{bf}}, d_{\rm{GV}})$. By Lemma \ref{lemma:MEGV24b}, this holds if and only if $\mathscr{E}_h(\boldsymbol{\Pi})$ is relatively compact in $(\mathbb{T}_{\rm{min}}^{\rm{di}}, d_{\rm{GV}})$.
\end{proof}
\subsection{Existence of a limit}
\label{subsec:GVExist}
We start by generalizing Proposition \ref{prop:MEOp13} to the Gromov-vague setting.
\begin{prop}[Monotone convergence of mass erasures]
\label{prop:MEGV26}
Let $\boldsymbol{\mu}_n = [T_n, d_n, \rho_n, \mu_n] \in \mathbb{T}_{\rm{min}}^{\rm{di}}$, $n \in \mathbb{N}$, and let $h_n \searrow 0$ be strictly decreasing in $(0, \infty)$. Suppose that
\begin{enumerate}[$(i)$]
\item $\mathscr{E}_{h_n - h_{n+1}}\boldsymbol{\mu}_{n+1} = \boldsymbol{\mu}_n$ for all $n \in \mathbb{N}$, and
\item $\left( \mu_n|_{{\rm{Cut}}(T_n, \bullet^{r_k}(\mu_n))} \right)_{n \in \mathbb{N}}$ is b-tight for some $r_k \rightarrow \infty$ increasing.
\end{enumerate}
Then there exists a unique $\boldsymbol{\mu} \in \mathbb{T}_{\rm{min}}^{\rm{di}}$ such that $\mathscr{E}_{h_n}\boldsymbol{\mu} = \boldsymbol{\mu}_n$ for all $n \in \mathbb{N}$, and $\boldsymbol{\mu}_n \stackrel{\rm{GV}}{\rightarrow} \boldsymbol{\mu}$.
\end{prop}
\begin{proof}[Proof of Proposition \ref{prop:MEGV26}]
From assumption $(i)$ we get the existence of a bijective isometry $ (T_n, d_n, \rho_n, \mu_n) \rightarrow \left( R_{\mu_{n+1}, (h_n - h_{n+1})+}(T_{n+1}), d_{n+1}, \rho_{n+1}, \mathscr{E}_{h_n - h_{n+1}}\mu_{n+1} \right)$.
From this it is possible to construct a Polish $\mathbb{R}$-tree $(T, d, \rho)$ and isometries $\phi_n \colon T_n \rightarrow T$, $n \in \mathbb{N}$, such that setting $\mu_n' := \mu_n \circ \phi_n^{-1}$, we have that $\mathscr{E}_{h_n - h_{n+1}}\mu_{n+1}' = \mu_n'$ for all $n \in \mathbb{N}$ in $(T, d, \rho)$. Indeed, $(T, d, \rho)$ acts as an embedding space for the isometry classes $(\boldsymbol{\mu}_n)_{n \in \mathbb{N}}$, and is constructed exactly as in the proof of \cite{DW26} Lemma 4.29. This moves us back to the setting of Section \ref{sec:4}, and so a direct application of Proposition \ref{prop:MEOp13} to the measures $(\mu_n')_{n \in \mathbb{N}}$ on $(T, d, \rho)$ gives the existence of a measure $\mu \in \mathscr{M}_{\rm{di}}(T)$ with $\mathscr{E}_{h_n}\mu = \mu_n'$ for all $n \in \mathbb{N}$ and $\mu_n' \stackrel{\rm{vg}}{\rightarrow} \mu$. Letting $\boldsymbol{\mu}$ denote the isometry class of $\mu$ in $\mathbb{T}_{\rm{min}}^{\rm{di}}$ gives the desired.
\end{proof}
We may finally prove Theorem \ref{thm:MEGV27}.
\begin{proof}[Proof of Theorem \ref{thm:MEGV27}]
By the semigroup property, Theorem \ref{thm:MEGV17}, and assumption $(i)$, we have $\mathscr{E}_{h_p}\boldsymbol{\mu}_n = \mathscr{E}_{h_p - h_{p+1}}\mathscr{E}_{h_{p+1}}\boldsymbol{\mu}_n \stackrel{\rm{GV}}{\rightarrow} \mathscr{E}_{h_p - h_{p+1}}\boldsymbol{\nu}_{p+1}$ as $n \rightarrow \infty$. By uniqueness of limits, $\mathscr{E}_{h_p - h_{p+1}}\boldsymbol{\nu}_{p+1} = \boldsymbol{\nu}_p$ for all $p \in \mathbb{N}$. Using assumption $(ii)$ we conclude by Proposition \ref{prop:MEGV26} that there exists a unique isometry class $\boldsymbol{\mu} \in \mathbb{T}_{\rm{min}}^{\rm{di}}$ such that $\mathscr{E}_{h_p}\boldsymbol{\mu} = \boldsymbol{\nu}_p$ for all $p \in \mathbb{N}$ and $\boldsymbol{\nu}_p \stackrel{\rm{GV}}{\rightarrow} \boldsymbol{\mu}$ as $p \rightarrow \infty$. We conclude that $\mathscr{E}_{h_p}\boldsymbol{\mu}_n \stackrel{\rm{GV}}{\rightarrow} \mathscr{E}_{h_p}\boldsymbol{\mu}$, and so $\boldsymbol{\mu}_n \stackrel{\rm{GVme}}{\rightarrow} \boldsymbol{\mu}$ as $n \rightarrow \infty$.
\end{proof}
We finally note that Proposition \ref{prop:MEGV25} may be shown in a completely analogous manner to Proposition \ref{prop:MEOp11}.
\section*{Acknowledgements}
I would like to thank my PhD-supervisor Matthias Winkel for many insightful discussions on the ideas presented in this paper, and for reading through and providing invaluable feedback on countless drafts produced throughout the writing process. I would also like to thank Thomas Duquesne for hosting my research stay at the Sorbonne Université in Paris from January to June 2026, during which parts of this paper was drafted.

I would moreover like to express my gratitude to the Anglo--Danish Society for their financial support during my first year of PhD-studies, and Exeter College Oxford, London Mathematical Society, William Demant Fonden, Knud Højgårds Fond, Christian og Ottilia Brorsons Rejselegater, and Carl og Ellen Hertz's Legat, for supporting my research stay in Paris.

This research is supported by the EPSRC Centre for Doctoral Training in Mathematics of Random Systems: Analysis, Modelling and Simulation (ESPRC Grant EP/S023925/1).

For the purpose of Open Access, the author has applied a CC BY public copyright licence to
any Author Accepted Manuscript (AAM) version arising from this submission.
\end{document}